\PassOptionsToPackage{table}{xcolor}

\documentclass[11pt]{siamart190516}

\usepackage{graphicx}
\usepackage[nameinlink]{cleveref} 
\usepackage{bm}
\usepackage{bbm}
\usepackage{comment}
\usepackage{enumitem}
\usepackage{diagbox}
\usepackage{multirow}
\usepackage[ruled,vlined,algo2e]{algorithm2e}

\crefname{equation}{}{}
\crefname{algocf}{algorithm}{algorithms}
\graphicspath{{./Figures/}}
\definecolor{CeruleanRef}{RGB}{12,127,172}
\hypersetup{colorlinks=true,allcolors=CeruleanRef}
\hypersetup{pdfstartview=FitB,pdfpagemode=UseNone}
\newsiamremark{remark}{Remark}
\newsiamremark{example}{Example}

\usepackage{accents}
\DeclareMathAccent{\wtilde}{\mathord}{largesymbols}{"65}

\usepackage{caption}
\usepackage{subcaption}
\usepackage{hhline}

\renewcommand{\qed}{\hfill\blacksquare}

\usepackage{tikz}
\usetikzlibrary{decorations.pathreplacing}
\usetikzlibrary{calligraphy}
\usetikzlibrary{cd}
\usepackage{pgfplots}
\pgfplotsset{compat=1.18}
\definecolor{color0}{rgb}{0.7843, 0.7843, 0.7843}
\definecolor{color1}{rgb}{0, 0.4470, 0.7410}
\definecolor{color2}{rgb}{0.8500, 0.3250, 0.0980}
\definecolor{color3}{rgb}{0.9290, 0.6940, 0.1250}
\definecolor{color4}{rgb}{0.7060, 0.3840, 0.7650}
\definecolor{color5}{rgb}{0.4660, 0.6740, 0.1880}
\definecolor{color6}{rgb}{0.3010, 0.7450, 0.9330}
\definecolor{color7}{rgb}{0.6350, 0.0780, 0.1840}
\definecolor{color8}{rgb}{0.0, 0.4078, 0.3412}

\usepackage{amsmath,amssymb}
\DeclareMathOperator*{\argmin}{\vphantom{p}arg\,min}

\DeclareMathOperator*{\esssup}{\vphantom{p}ess\,sup}
\DeclareMathOperator*{\essinf}{\vphantom{p}ess\,inf}

\DeclareMathOperator{\fa}{~for~all~}
\DeclareMathOperator{\lnit}{lnit}
\DeclareMathOperator{\sigmoid}{expit} 
\DeclareMathOperator{\prox}{prox}

\DeclareMathOperator{\atanh}{arctanh}
\DeclareMathOperator{\interior}{int}
\DeclareMathOperator{\bd}{bd}
\DeclareMathOperator{\cl}{cl}
\DeclareMathOperator{\supp}{supp}
\DeclareMathOperator{\dom}{dom}
\DeclareMathOperator{\dd}{d\!}
\let\div\relax \DeclareMathOperator{\div}{div}

\let\inf\relax \DeclareMathOperator*\inf{\vphantom{p}inf}
\let\min\relax \DeclareMathOperator*\min{\vphantom{p}min}
\let\max\relax \DeclareMathOperator*\max{\vphantom{p}max}
\let\tilde\widetilde
\let\hat\widehat

\newsiamthm{claim}{Claim}

\allowdisplaybreaks

\begin{document}

\title{
    Proximal Galerkin: A structure-preserving finite element method for pointwise bound constraints\thanks{Submitted to the editors \today.
    \funding{BK was supported in part by the U.S.\ Department of Energy Office of Science, Early Career Research Program under Award Number DE-SC0024335.
    BK was also supported in part by the LLNL-LDRD Program under Project Tracking No.\ 22-ERD-009.}
    }
}

\author{
    Brendan~Keith\thanks{\protect
        Division of Applied Mathematics,
        Brown University,
        Providence, RI 02912 USA
        (\email{brendan\_keith@brown.edu}).
    }
    \and Thomas~M.~Surowiec\thanks{\protect
        Simula Research Laboratory,
        Department of Numerical Analysis and Scientific Computing,
        Kristian Augusts gate 23,
        0164, Oslo, Norway, 
        (\email{thomasms@simula.no})
    }
}

\date{\today}

\maketitle

\begin{center}
    \small
    \medskip
  \emph{Dedicated with respect and admiration to Leszek Demkowicz on the occasion of his 70th birthday.}
  \medskip

    \end{center}

\begin{abstract}
The proximal Galerkin finite element method is a high-order, low iteration complexity, nonlinear numerical method that preserves the geometric and algebraic structure of pointwise bound constraints in infinite-dimensional function spaces.
This paper introduces the proximal Galerkin method and applies it to solve free boundary problems, enforce discrete maximum principles, and develop a scalable, mesh-independent algorithm for optimal design with pointwise bound constraints.
This paper also introduces the latent variable proximal point (LVPP) algorithm, from which the proximal Galerkin method derives.

When analyzing the classical obstacle problem, we discover that the underlying variational \emph{inequality} can be replaced by a sequence of second-order partial differential \emph{equations} (PDEs) that are readily discretized and solved with, \textit{e.g.}, the proximal Galerkin method.
Throughout this work, we arrive at several contributions that may be of independent interest.
These include (1) a semilinear PDE we refer to as the \emph{entropic Poisson equation}; (2) an algebraic/geometric connection between high-order positivity-preserving discretizations and certain infinite-dimensional Lie groups; and (3) a gradient-based, bound-preserving algorithm for two-field, density-based topology optimization.
The complete proximal Galerkin methodology combines ideas from nonlinear programming, functional analysis, tropical algebra, and differential geometry and can potentially lead to new synergies among these areas as well as within variational and numerical analysis.
Open-source implementations of our methods accompany this work to facilitate reproduction and broader adoption.

\end{abstract}



\section{Introduction} \label{sec:introduction} 

Although the origins of variational analysis can be traced back at least to the seventeenth century \cite{oden2012variational}, its role in the modern study of partial differential equations (PDEs) only truly began to take shape around 1847 once William Thomson introduced what is now known as the Dirichlet principle.
In contemporary language, this energy principle states that for all functions $f\in L^2(\Omega)$ and $g\in H^1(\Omega)$, the (weak) solution of Poisson's equation over a Lipschitz domain $\Omega \subset \mathbb{R}^n$,
\begin{equation}
\label{eq:PoissonEquation}
	-\Delta u = f
	\quad \text{in~} \Omega,
	\qquad
	u = g \quad \text{on~} \partial\Omega,
\end{equation}
can be obtained as the $H^1(\Omega)$-minimizer of the Dirichlet energy,
\begin{equation}
\label{eq:DirichletEnergy}
	E(v)
	=
	\frac{1}{2}
	\int_\Omega |\nabla v|^2 \dd x
	-
	\int_\Omega v f \dd x
	\,,
\end{equation}
confined to the constraint set $H^1_g(\Omega) = g + H^1_0(\Omega) = \{ v \in H^1(\Omega) \mid v = g \text{~on~} \partial \Omega\}$.

Owing to the fact that $H^1_g(\Omega)$ is nonempty, closed, and convex, it is a consequence of the Lions--Stampacchia theorem \cite{stampacchia1963equations,lions1967variational} that the energy minimizer $u^\ast \in K = H^1_g(\Omega)$ is the unique solution to the variational inequality (VI)
\begin{equation}
\label{eq:DirichletVI}
			\int_\Omega  \nabla u^\ast \cdot (\nabla v - \nabla u^\ast ) \dd x \geq \int_\Omega (v-u^\ast) f \dd x
	~\fa v \in K.
\end{equation}
It happens that the boundary condition in~\cref{eq:PoissonEquation} is an equality constraint that induces an \emph{affine} structure on the feasible set.
Moreover, it is this particular algebraic structure that can be exploited to show that the minimizer $u^\ast \in H^1_g(\Omega)$ is also uniquely characterized by a variational \emph{equation}.
In the setting above, we have
\begin{equation}
\label{eq:DirichletVE}
	\int_\Omega \nabla u^\ast \cdot \nabla w \dd x
	=
	\int_\Omega f w \dd x
	~\fa w \in H^1_0(\Omega)
	\,.
\end{equation}
To arrive at this conclusion from~\cref{eq:DirichletVI}, the key idea is to notice that $H^1_g(\Omega) + H^1_0(\Omega) = H^1_g(\Omega)$ and, therefore, one may replace $v$ in~\cref{eq:DirichletVI} by $u^\ast\pm w$, for any $w \in H^1_0(\Omega)$.
For further details, see, e.g., \cite[Proposition~9.22]{brezis2011functional} and \cite[Theorem~1.2.2]{ciarlet2002finite}.

A related additive structure appears if we consider imposing the pointwise non-negativity constraint, $u^\ast \geq 0$, and setting $g \equiv 0$.
Hereafter, let $H^1_+(\Omega) = \{ v \in H^1(\Omega) \mid v \geq 0 \text{~a.e.}\}$.
In this setting, the feasible set
\begin{equation}
\label{eq:NonNegativeConstraintSet}
	K = \{ v \in H^1_0(\Omega) \mid v \geq 0 \text{~a.e.}\} = H^1_0(\Omega) \cap H^1_+(\Omega)
\end{equation}
forms a closed convex cone in $H^1(\Omega)$.
It is well-known that the \emph{conic} structure of $K$ allows us to write
\begin{equation}
	\int_\Omega \nabla u^\ast \cdot \nabla v \dd x \geq \int_\Omega f v \dd x
	~\fa v \in K
	\,,
\end{equation}
with equality holding for $v = u^\ast$; see, e.g., \cite[Theorem~1.1.2]{ciarlet2002finite}.
Thus, we encounter another simplification directly from the algebraic structure of the feasible set.
This work pursues a third type of algebraic structure.
\smallskip

Our aim is to provide a high-order, \emph{multiplicative} structure-preserving approach to solving bound-constrained optimization problems and variational inequalities in Sobolev spaces.
This will lead us to working in \emph{Banach algebras}, which are Banach spaces that are closed under a continuous multiplication operation \cite{conway2019course}.
Instead of performing Lagrangian relaxation or relying on penalty functions, the key component of our approach is an adaptive form of \emph{entropy regularization}.
Through entropy regularization, we will find, \textit{e.g.}, that minimizing the Dirichlet energy over functions in $H^1_g(\Omega) \cap H^1_+(\Omega)$ can be reduced to solving a sequence of second-order semilinear PDEs where each right-hand side is conditioned by the prior solution.
\smallskip

\Cref{alg:HomogeneousAlg} outlines the meta-algorithm for minimizing the Dirichlet energy under the pointwise non-negativity constraint $u^\ast \geq 0$ considered above when $f \in L^\infty(\Omega)$ and $g_{|\partial\Omega} \in C(\partial\Omega)$ satisfies $\min_{\partial\Omega} g > 0$.
Note that, unlike, \textit{e.g.}, descent methods \cite[Section~3]{wright2022optimization}, this algorithm converges for \emph{all} step size values $\alpha > 0$; cf.~\Cref{thm:ConvergenceContinuousLevel}.
A practical version of the algorithm appears via a \emph{saddle-point} reformulation of the semilinear subproblem~\cref{eq:EPEIntro}, leading to the finite element method that gives this paper its name.

\begin{algorithm2e}\DontPrintSemicolon
	\caption{\label{alg:HomogeneousAlg} Entropic proximal point algorithm for Dirichlet energy minimization with a non-negativity constraint.}
	\SetKwInOut{Input}{input}
		\BlankLine
	\Input{Step size parameter $\alpha>0$ and initial solution guess $w \in H^1_g(\Omega)\cap L^\infty(\Omega)$ satisfying $\essinf w > 0$.}
		\BlankLine
		\Repeat{a convergence test is satisfied}
	{
		Solve the entropic Poisson equation,
		\begin{equation}
		\label{eq:EPEIntro}
																											\left\{
				\begin{aligned}
					\,-\Delta u + \alpha^{-1}\ln u
					&=
					f + \alpha^{-1}\ln w
					~~ &&\text{in~} \Omega\,,
					\\
					u
					&=
					g ~~ &&\text{on~} \partial\Omega\,.
				\end{aligned}
			\right.
		\end{equation}
				\;
		\vspace*{-\baselineskip}
		Assign $w \leftarrow u$.\;
	}
		\BlankLine
\end{algorithm2e}

\begin{remark}[Latent variable proximal point vs.\ proximal Galerkin]
	Two general approximation techniques are introduced in this work: the latent variable proximal point (LVPP) algorithm and the proximal Galerkin finite element method.
	LVPP is an infinite-dimensional optimization algorithm, equivalent to a saddle-point reformulation of the (entropic) Bregman proximal point algorithm \cite{teboulle2018simplified}, which dates back to early work by Bregman, Nemirovskij, and Yudin \cite{bregman1967relaxation,nemirovskij1983problem}.
		The proximal Galerkin finite element method is a nonlinear mixed method derived by discretizing the LVPP subproblems with the Galerkin finite element method.
	Although other numerical methods, \textit{e.g.}, finite difference, finite volume, spectral, spline, or even neural network methods could also be used to discretize LVPP, we focus our attention on the aforementioned finite element approach.
	We also acknowledge the existence of significantly different ``proximal Galerkin'' method proposed in \cite{mackey2014compressive}, yet choose to use the same name for our method without any cause for confusion.
		\end{remark}

\begin{remark}[Why pursue high-order discretizations?]
	The solutions to variational inequality problems, as with other nonlinear problems, often have \emph{low regularity}, potentially leading to low-order accuracy.
	In the early days of the finite element method, this fact led much of the community to conclude that low-order methods ``are sufficient for all practical purposes'' \cite[Chapter~5]{ciarlet2002finite}.
		Much to the contrary, high-order discretizations are now widely-used to solve nonlinear problems \cite{wang2013high,cottrell2009isogeometric} owing, in part, to more the accurate computational geometries \cite{cottrell2009isogeometric} and higher efficiency on modern computing architectures; see, \textit{e.g.}, \cite{andrej2024high}.
	The benefits of high-order methods can be further improved when combined with adaptive mesh refinement, such as $hp$-refinement \cite{schwab1998,demkowicz2006computing}.
	This often (sometimes provably \cite{verfurth2013posteriori}) results in high-order accuracy, similar to if the underlying solution was smooth.
			These aspects all support the present investigation into a high-order methodology for discretizing variational inequalities.
\end{remark}

\subsection{Notation} \label{sub:notation}
Our notation is rather standard for the finite element literature. Norms are denoted by
$\| \cdot \|_{X}$, inner products by $(\cdot,\cdot)_{X}$, and duality pairings by 
$\langle \cdot,\cdot \rangle_{X',X}$ for Banach spaces $X$ and its paired topological 
dual $X'$. Whenever it is clear in context, we leave off or abbreviate the subscripts in a natural way.
For weak convergence, we use the standard notation,
$\stackrel{X}{\rightharpoonup}$ or $\rightharpoonup$.  For subsets $C$ of infinite 
dimensional spaces, we denote the closure by $\cl C$,
the boundary by $\bd C$, and the interior by $\interior C$.
For a mapping
$F$ between normed linear spaces $X$ and $Y$, the Fr\'echet derivative of $F$
at $x$ is indicated by $F'(x)$.

For an open bounded domain $\Omega \subset \mathbb R^n$, $n \in \{1,2,...\}$, $
L^p(\Omega)$, $p \in [1,\infty]$, denotes the usual Lebesgue space of
(equivalence classes of) $p$-integrable functions when $p \in [1,\infty)$, and essentially 
bounded functions when $p = \infty$, respectively.
Furthermore, we simplify the $L^2(\Omega)$-inner product notation to $(u,v) = (u,v)_{L^2(\Omega)}$ and define
\[
L^p_+(\Omega) := \left\{ u \in L^p(\Omega) \mid u \ge 0 \text{ a.e.~in } \Omega \right\}
\]
for any $p \in [1,\infty]$.

For domains $\Omega$ in $\mathbb R^n$,
we denote the boundary by $\partial \Omega$ and the closure by $\overline{\Omega}$.
The spaces $L^p(\partial \Omega)$ are defined in the usual way. When needed, we indicate
the surface measure for $\partial \Omega$ by $d \mathcal{H}_{n-1}$.
The space of continuous functions on $\overline{\Omega}$ is written $C(\overline{\Omega})$.
Similarly, $C^m(\overline{\Omega})$, $m \in \mathbb N \cup \{\infty\}$, is the 
space of all $m$-times continuously differentiable functions. The space of smooth 
compactly supported test functions on $\Omega$ is given by $C^{\infty}_{c}(\Omega)$. 
The Sobolev space of $L^2(\Omega)$ functions with $L^2(\Omega)$ integrable 
weak derivatives is denoted by $H^1(\Omega)$ and its closed subspace of functions $u$
with trace $\gamma u = 0$ is denoted by $H^1_0(\Omega)$. We use $H^{s}(\partial \Omega)$,
$s \in (0,1)$, for the usual Sobolev--Slobodeckij space on $\partial \Omega$.
We refer the reader to a 
standard text on function spaces for further details, e.g., \cite{adams2003sobolev,mclean2000strongly}.
Finally, we adopt the following notational conventions: $\mathbb{R}_+ = [0,\infty)$, $\mathbb{R}_{++} = (0,\infty)$, $0 \ln 0 := 0$, and $\|v\|_{H^{-1}(\Omega)} := \sup_{w \in H^1_0(\Omega)} \frac{(v,w)}{\|\nabla w\|_{L^2(\Omega)}}$, with $w=0$ ignored in this and similar expressions.

\subsection{Outline} \label{sub:intro_overview}

We have attempted to provide a scaffolded presentation of our findings.
To this end, \Cref{sec:structure} presents preliminary concepts and provides further motivation for this work.
Next, \Cref{sec:literature_review} reviews the literature and summarizes our main contributions.
\Crefrange{sec:the_obstacle_problem}{sec:new_algorithms_for_topology_optimization} present the essential features of proximal Galerkin methods for the obstacle problem, the advection-diffusion equation, and topology optimization, respectively.
Each of these sections contains an algorithm that is designed be implemented in a production-level finite element code.
The reader is encouraged to compare these algorithms with our publicly available implementations \cite{Keith2023ObstacleCode,Keith2023TopOptCode,ZenodoCode}.
The main paper closes with \Cref{sec:conclusion}, which contains a small number of concluding remarks, and then proceeds to two technical appendices.
\Cref{sec:preliminaries} contains the continuous-level mathematical analysis and \Cref{sec:the_entropic_finite_element_method} contains the discrete-level, finite element theory.
\Cref{sec:preliminaries,sec:the_entropic_finite_element_method} are the most specialized sections of the paper and may be skipped by a casual reader.

\section{Preserving multiplicative structure} \label{sec:structure}

The proximal Galerkin finite element method is a nonlinear numerical method that preserves the algebraic and geometric structure of bound constraints in infinite-dimensional function spaces.
It relies on reformulating subproblems like~\cref{eq:EPEIntro} into a saddle-point form and discretizing the resulting system of equations.
In this section, we outline the essential features of the method using Dirichlet energy minimization~\cref{eq:DirichletEnergy} as a motivating example.
First, however, we study the multiplicative structure of non-negative functions in order to illustrate how proximal Galerkin preserves this structure.

\subsection{Deconstructing the semiring of non-negative functions} \label{sub:the_cole_hopf_transform_and_the_semiring_of_non_negative_functions}

Here, we discuss the natural logarithmic transformation between non-negative functions and extended real-valued functions that may take the value $-\infty$. This also first introduces the latent variable $\psi$. We claim this provides a basis for the use of logarithmic transformations to analyze and solve PDEs, an idea that goes back at least to work by Schr\"odinger in 1926. Finally, we address the somewhat unnatural conditions this transformation imposes on the solution spaces and variational equations themselves. In turn, we show how a simple regularization of the transformed equations remedies these concerns. We use this discussion to motivate the natural function spaces for pointwise bound constraints in $H^1(\Omega)$ and construct a direct link to entropy regularization.
\smallskip

Let $\mathcal{X}$ be a set equipped with two binary operations: addition $\oplus \colon \mathcal{X} \times \mathcal{X} \to \mathcal{X}$ and multiplication $\odot \colon \mathcal{X} \times \mathcal{X} \to \mathcal{X}$.
\begin{definition}[Semiring]
\label{def:Semiring}
We say that $\mathcal{X}$ is a semiring if the following conditions are satisfied \cite{golan2013semirings,gondran2008graphs}:
\begin{itemize}
	\item
	Addition $\oplus$ and multiplication $\odot$ are associative;
	\item
	Addition $\oplus$ is commutative;
		\item
	Multiplication $\odot$ is distributive with respect to addition $\oplus$.
\end{itemize}
We say that $\mathcal{X}$ is a commutative semiring if the conditions above are satisfied and, moreover, multiplication $\odot$ is commutative.
\end{definition}

It is easy to check that the set of non-negative Lebesgue measurable functions,
\begin{equation}
	M_+(\Omega)
	=
	\big\{ v\colon\Omega\to \mathbb{R}_+ \mid \{v > c\} \text{~is~Lebesgue~measurable~} \forall c > 0 \big\}
	,
	\end{equation}
forms a commutative semiring under the standard binary operations of pointwise addition and multiplication.
In particular, note that for any $u,v \in M_+(\Omega)$, we have
\begin{equation}
	u + v
		\in M_+(\Omega)
	\,,
	\qquad
				uv \in M_+(\Omega)
	\,.
\end{equation}

There is an interesting identification between $M_+(\Omega)$ and the space of extended real-valued measurable functions
\begin{equation}
	M_{\max}(\Omega)
	=
	\big\{ \varphi\colon\Omega\to \mathbb{R}\cup\{-\infty\} \mid \{\varphi > c\} \text{~is~Lebesgue~measurable~} \forall c \in \mathbb{R} \big\}
	,
\end{equation}
induced by the (pointwise) logarithm and exponential operators.
Namely, for all $u \in M_+(\Omega)$, $\psi \in M_{\max}(\Omega)$, and $\alpha > 0$, we have that $\alpha^{-1}\ln u \in M_{\max}(\Omega)$ and $\exp(\alpha\psi) \in M_+(\Omega)$ under the convention that $\ln 0 = -\infty$ and, likewise, $\exp (-\infty) = 0$.
Such logarithmic transformations provide a family of semiring isomorphisms between $M_+(\Omega)$ and $M_{\max}(\Omega)$, where $M_{\max}(\Omega)$ is endowed with the following (generalized) addition and multiplication operations:
\begin{equation}
\label{eq:BinaryOperationsLog}
	\psi \oplus \varphi = \alpha^{-1}\ln (\exp (\alpha \psi) + \exp (\alpha \varphi))
	\,,
	\qquad
	\psi \odot \varphi = \psi + \varphi
	\,,
\end{equation}
respectively \cite{maslov1987new,maslov1987newprinciple}.
Moreover, in the limit $\alpha \to \infty$, the generalized addition operation~\cref{eq:BinaryOperationsLog} becomes the pointwise maximum operation \cite{litvinov2007maslov,maclagan2021introduction}; namely,
\begin{equation}
\label{eq:BinaryOperationsTropical}
	\psi \oplus \varphi \to \max\{\psi,\varphi\}
	\,.
\end{equation}

Logarithmic transformations of the above form have been used famously over the last century to analyze differential equations in quantum mechanics \cite{schrodinger1926quantisierung}, fluid flow \cite{hopf1950partial,cole1951quasi}, and electrical engineering \cite{scharfetter1969large,markowich1985stationary}, and, more recently, to study stochastic PDEs \cite{bertini1997stochastic,hairer2013solving}.
Given that they appear to capture certain key algebraic properties of the set of non-negative functions, it is tempting to use logarithmic transformations to enforce non-negativity constraints on function spaces.
Unfortunately, however, special care is required to apply a logarithmic transformation to a non-negative solution variable in a free-boundary problem.

For illustration, consider minimizing the Dirichlet energy~\cref{eq:DirichletEnergy} over the set of non-negative functions
\begin{equation}
\label{eq:NonNegativeConstraintSet_gNonZero}
	K = \{ v \in H^1_g(\Omega) \mid v \geq 0 \text{~a.e.}\} = H^1_g(\Omega) \cap H^1_+(\Omega) 
	.
\end{equation}
Assuming that $f \in L^2(\Omega)$ and $u^\ast \in H^2(\Omega)$, the well-known complementarity conditions for the solution are as follows \cite[p.~79]{kinderlehrer2000introduction}: 
\begin{equation}\label{eq:complementarity}
	u^\ast \geq 0
	,\quad
	-\Delta u^\ast-f\geq 0
	,\quad
	(\Delta u^\ast+f)\,u^\ast = 0
	~~\text{a.e. in~}\Omega
	\,.
\end{equation}
Another perspective uses a dual variable $\lambda^\ast$, also known as a Lagrange multiplier, to formulate \cref{eq:complementarity} as 
a mixed complementarity problem of the form:
\begin{equation}
\label{eq:IntroLagrangeMultiplier}
\aligned
-\Delta u^* - \lambda^* = f,
\quad
u^\ast \ge 0,
\quad
\lambda^\ast \ge 0,
\quad
\langle u^\ast,\lambda^\ast \rangle = 0\,.
\quad
\endaligned
\end{equation}
The Lagrange multiplier exhibits rather low regularity for general domains $\Omega$, so the term ``$\lambda^\ast \ge 0$'' is actually understood
to mean $\langle \lambda^\ast, w \rangle \ge 0$ for all $w \in H^1_0(\Omega)$ with $w \ge 0$ a.e.\ in $\Omega$, i.e., without further regularity assumptions $\lambda^\ast$ is merely a nonnegative Radon measure on $\Omega$. See \cite[Chap. II, Sec. 6]{kinderlehrer2000introduction} for details.

If we wish to study this problem under a logarithmic transformation, then a formal computation using the substitution $u^\ast = \exp\psi^\ast$ leads to the observation that
\begin{equation}
\label{eq:DegeneratePoisson}
	\psi^\ast = -\infty
	\quad
	\text{~or}
	\quad
	-\div (\exp\psi^\ast\nabla \psi^\ast) = f
		\,,
\end{equation}
at almost every point in $\Omega$.
Analyzing these equations presents challenges, in part, because it requires moving away from well-studied Sobolev spaces \cite{adams2003sobolev} and, instead, working in a latent space of extended real-valued functions \cite{kolokoltsov1997idempotent} endowed with the metric
\begin{equation}
\label{eq:LatentVariableNonnegativityMetric}
	d(\psi,\varphi)
	=
	\|\nabla \exp\psi - \nabla \exp\varphi\|_{L^2(\Omega)}
	\,.
\end{equation}
More specifically, it is difficult to imagine how to rigorously discretize a space of functions that are allowed to take the value $-\infty$ on sets of positive Lebesgue measure.

One conclusion of this work is that the above concerns are alleviated by a simple regularization of the degenerate PDE in~\cref{eq:DegeneratePoisson}.
In particular, we show that for all bounded $f \in L^\infty(\Omega)$, the latent solution variable $\psi^\ast = \ln u^\ast$ is recovered as the $\alpha\to \infty$ limit (with respect to the metric~\cref{eq:LatentVariableNonnegativityMetric}) of a family of regularized solutions $\psi \in H^1(\Omega)\cap L^\infty(\Omega)$ satisfying
\begin{equation}
\label{eq:DegenerateEntropicPoisson}
		-\div (\exp\psi\nabla \psi) + \alpha^{-1}\psi = f
		\,.
\end{equation}
Moreover, the latent variable iteration $\psi^0 \in H^1(\Omega)\cap L^\infty(\Omega)$,
\begin{equation}
\label{eq:LatentVariableIteration}
	-\div (\exp\psi^k\nabla \psi^k) + \alpha^{-1}\psi^k = f + \alpha^{-1}\psi^{k-1}
	\,,
	\quad
	k=1,2,\ldots,
\end{equation}
formerly written with primal variables in~\Cref{alg:HomogeneousAlg}, converges to $\psi^\ast$ for all finite $\alpha > 0$; cf.~\Cref{thm:ConvergenceContinuousLevel}.

The ambient function space for the regularized latent variable $\psi$ is interesting from an algebraic point of view because it is a \emph{Banach algebra}.
Indeed, the Sobolev subspace $H^1(\Omega)\cap L^\infty(\Omega)$, whose norm is
\begin{equation}
	\|v\|_{H^1(\Omega)\cap L^\infty(\Omega)}
	=
	\max\{\|v\|_{H^1(\Omega)},\|v\|_{L^\infty(\Omega)}\}
	\,,
\end{equation}
is closed under the standard operations of pointwise addition and multiplication \cite[Proposition~9.4]{brezis2011functional}.
Maintaining closure under multiplication is desirable, in part, because it often allows one to construct a smooth exponential map \cite{glockner2002algebras,glockner2003lie}.
Indeed, of particular interest to this work is the Nemytskii operator
\begin{equation}
\label{eq:Exponential_H1Linfty}
	\exp \colon H^1(\Omega)\cap L^\infty(\Omega) \to H^1(\Omega) \cap \interior L^\infty_+(\Omega),
\end{equation}
which is an isomorphism between $H^1(\Omega)\cap L^\infty(\Omega)$ and the \emph{Banach--Lie group}
\begin{equation}
	H^1(\Omega) \cap \interior L^\infty_+(\Omega)
			=
	\{ w \in H^1(\Omega) \cap L^\infty(\Omega) \mid \essinf w > 0\}
	\,;
\end{equation}
cf.~\Cref{prop:logexpChainRule}.
Since the range of this isomorphism is contained in the $H^1(\Omega) \cap L^\infty(\Omega)$-interior of the set of essentially bounded non-negative $H^1(\Omega)$ functions, we find that the primal iterates, $$u^k = \exp \psi^k \in H^1(\Omega) \cap \interior L^\infty_+(\Omega) \subset \interior( H^1(\Omega) \cap L^\infty(\Omega) )\,,$$ will always be \emph{interior points}.
In the next motivational subsection, we explain that an identical sequence of interior points $u^k \stackrel{H^1(\Omega)}{\longrightarrow} u^\ast$ can be found by regularizing the Dirichlet energy with an appropriate \emph{entropy} functional.

\subsection{Dirichlet free energy} \label{sub:dirichlet_free_energy}

Only special function spaces are endowed with a norm topology that permits a continuous multiplication operator.
Indeed, it is well-known that $H^1(\Omega)$ is only closed under multiplication when $n = 1$ \cite{adams2003sobolev}.
Moreover, it is easy to show that $\interior H^1_+(\Omega) = \emptyset$ for all $n \geq 2$, which makes it impossible to define an $H^1(\Omega)$-interior point in any of its subsets (cf.~\Cref{rem:H1InteriorNonnegativeFunctions}).
Because $H^1(\Omega) \cap L^\infty(\Omega)$ bypasses both of these topological issues, it is appealing to restrict the feasible set $K$ in~\cref{eq:NonNegativeConstraintSet_gNonZero} to essentially bounded functions when minimizing the Dirichlet energy.

Unfortunately, requiring the feasible set to be the intersection of $K$ and $L^{\infty}(\Omega)$ would cause the direct method of calculus of variations \cite{bartels2015numerical} to fail.
This is because the Dirichlet energy does not provide control over point-wise values of the solution and $K\cap L^{\infty}(\Omega)$ is not closed in the $H^1(\Omega)$ norm topology.
Therefore, one may conclude that maintaining some important mathematical structures is in conflict with the classical energy principle.

Fortunately, it turns out there is resolution to this conflict that exposes the missing algebraic structure; namely, minimizing the \emph{Dirichlet free energy},
\begin{equation}
\label{eq:DirichletFreeEnergy}
	A(u) = 
	E(u) + \theta S(u)
						.
\end{equation}
Here, $\theta = \alpha^{-1}>0$ is a non-dimensional ``temperature'' parameter and
\begin{equation}
\label{eq:EntropyFunctional}
	S(u) =
	\int_\Omega u\ln u- u \dd x
	\,,
\end{equation}
is the (negative) entropy functional.
As we show in the proof of \Cref{thm:PrimalProblem}, the unique minimizer of~\cref{eq:DirichletFreeEnergy} lies in $K\cap \interior L^\infty_+(\Omega)$.
In particular, for all $\theta > 0$, the function
\begin{equation}
\label{eq:DirichletFreeEnergyMinimization}
	u = \argmin_{v \in K}
	A(v)
	= \argmin_{v \in K\cap L^\infty(\Omega)}
	A(v)
	\,,
\end{equation}
is essentially bounded away from zero.
Moreover, \Cref{cor:EntropicPoissonConvergence} shows that $u = u(\theta)$ converges linearly with respect to $\sqrt{\theta}$ to the unique non-negative minimizer of~\cref{eq:DirichletEnergy}, $u^\ast = \argmin_{v \in K} E(v)$.
More specifically,
\begin{equation}
\label{eq:EPE_ConvergenceWRTTemperature}
	\|\nabla u^\ast - \nabla u\|_{L^2(\Omega)}^2 \leq 2\theta (S(u^\ast) + |\Omega|)
	\,,
\end{equation}
whenever $f\in L^\infty(\Omega)$.

Finally, as a consequence of $u$ belonging to the interior of $L^\infty_+(\Omega)$, the general VI that characterizes the solution of~\cref{eq:DirichletFreeEnergyMinimization}, i.e.,
\begin{equation}
\label{eq:DirichletVI_nonnegativity}
	\int_\Omega  \nabla u \cdot \nabla v \dd x
	+
	\theta
	\int_\Omega  v\ln u \dd x
		\geq \int_\Omega f v \dd x
	~\fa v \in K-u
							\,,
\end{equation}
can be replaced by a variational equality; namely, the weak form of a semilinear PDE we call the \emph{entropic Poisson equation}, $-\Delta u + \theta\ln u = f$.
Indeed, \Cref{thm:PrimalProblem} shows that
\begin{equation}
\label{eq:EntropicPoissonWeakForm}
	\int_\Omega \nabla u \cdot \nabla w \dd x
	+
	\theta
	\int_\Omega w \ln u \dd x
	=
	\int_\Omega f w \dd x
	~\fa w \in H^1_0(\Omega)
	\,.
				\end{equation}
The entropic Poisson equation is the primal form of~\cref{eq:DegenerateEntropicPoisson} and has numerous interesting properties that we exploit in this work.
For completeness, we also note that $\theta \ln (1/u)$ approximates the true Lagrange multiplier $\lambda^\ast$ introduced in~\cref{eq:IntroLagrangeMultiplier} above. 

The essential idea presented above is expanded on in \Cref{sec:maximum_principles,sec:new_algorithms_for_topology_optimization} to accommodate bound constraints for general VIs that do not appear as a result of energy principles, as well as those that appear in topology optimization with a view toward other bound-constrained optimization problems.
Crucially, and unlike traditional penalty or barrier methods \cite{nocedal1999numerical,Bertsekas99,Wriggers2006}, it is \emph{not necessary} to take $\theta \to 0$ in order to get an arbitrarily accurate approximation of $u^\ast$.
Indeed, the simple adaptive entropic regularization algorithm given in \Cref{alg:HomogeneousAlg} (see also~\cref{eq:LatentVariableIteration}), which comes from regularizing the Dirichlet energy with a \emph{relative entropy} functional, is far more appealing and is derived in \Cref{sec:the_obstacle_problem}.
\Cref{fig:Trinity} provides a diagrammatic reference for the main elements of the continuous-level algorithm in the case $\alpha = 1$.

\begin{figure}
\centering
	\includegraphics[width=0.95\linewidth]{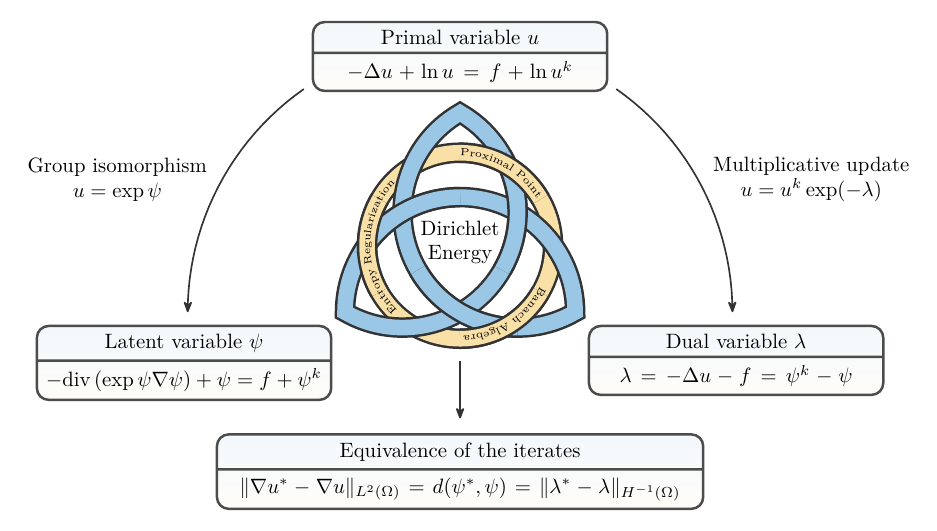}
	\caption{
	A trinity is formed by the three isomorphic representations of the iterates in the latent variable proximal point method.
	In this figure, equations for the three representations are given for the problem of minimizing the Dirichlet energy~\cref{eq:DirichletEnergy} over non-negative functions $u \in H^1_g(\Omega) \cap H^1_+(\Omega)$.
	Note that, for simplicity, the step size here is set to $\alpha = 1$.
		See \Cref{thm:PrimalProblem,thm:ConvergenceContinuousLevel} for further details and consequences for variable step sizes.
	\label{fig:Trinity}}
\end{figure}

\begin{remark}[Dirichlet free energy]
	We propose the name ``Dirichlet free energy'' for the functional in~\cref{eq:DirichletFreeEnergy} by analogy with the Helmholtz free energy from statistical mechanics \cite{pathria2016statistical}, $A = E - T S$, where $E$ denotes total system energy, $T$ denotes absolute temperature, and $S$ denotes thermodynamic entropy.
\end{remark}

\subsection{Pointwise-positivity for every polynomial degree} \label{sub:pointwise_positive_discretizations}

The majority of this paper is based on pursing the aforementioned observation that the solution of VIs for bound constraints, including~\cref{eq:DirichletVI_nonnegativity}, can be approximated arbitrarily accurately by variational \emph{equations} like~\cref{eq:EntropicPoissonWeakForm}.
Leveraging this observation for computational purposes leads to a new class of high-order, nonlinear finite element methods we refer to as proximal Galerkin methods.
In turn, taking advantage of the multiplicative structure of the solution space leads to non-standard approximation spaces that naturally preserve pointwise positivity at the discrete level.

As we shall also show in \Cref{sec:the_obstacle_problem}, a very convenient Galerkin discretization of the entropic Poisson equation~\cref{eq:EPEIntro} is found by introducing a pair of linear subspaces $V_h \subset H^1_0(\Omega)$ and $W_h \subset L^\infty(\Omega)$ --- for instance, spaces of high-degree piecewise polynomials --- and simultaneously approximating the solution $u$ in both the primal space and via the latent variable in a potentially different function space; namely,
\begin{equation}
	u \approx u_h
	\quad
	\text{and}
	\quad
	u \approx \exp\psi_h
		\,,
\end{equation}
where $u_h \in g_h + V_h$, $\psi_h \in W_h$, and $g_h \in H^1(\Omega)$ provides an approximation of the boundary values $g_{h|\partial\Omega} \approx g_{|\partial\Omega}$.
The basic method is outlined in \Cref{alg:EntropicGalerkinIntro}.

\begin{algorithm2e}\DontPrintSemicolon
	\caption{\label{alg:EntropicGalerkinIntro} 
	Proximal Galerkin method for Dirichlet energy minimization with a pointwise non-negativity constraint.
	}
	\SetKwInOut{Input}{Input}
	\SetKwInOut{Output}{Output}
	\BlankLine
	\Input{Step size parameter $\alpha > 0$, linear subspaces $V_h \subset H^1_0(\Omega)$ and $W_h \subset L^\infty(\Omega)$, and initial solution guess $\psi_h \in W_h$.}
	\Output{Two approximate solutions, $u_{h}$ and $\widetilde{u}_h = \exp\psi_h$, and an approximate Lagrange multiplier, $\lambda_h = (\varphi_h - \psi_h)/\alpha$.}
		\BlankLine
	\Repeat{a convergence test is satisfied}
	{
		Assign $\varphi_h \leftarrow \psi_{h}$.\;
		Solve the following (nonlinear) discrete saddle-point problem:
		\begin{gather*}
					\left\{
			\begin{aligned}
				\,&\text{Find}~
				u_{h}\in g_h + V_{h} ~\text{and}~\psi_{h} \in W_{h}
								~\text{such that~}
				\\
				&\begin{alignedat}{4}
					\int_\Omega \alpha \nabla u_h\cdot \nabla v \dd x + \int_\Omega \psi_h v \dd x &= \int_\Omega (\alpha f + \varphi_h)\, v \dd x
					&&~\fa v \in V_h
					\,,
					\\
					\int_\Omega u_h w \dd x - \int_\Omega \exp(\psi_h) w \dd x &= 0
					&&~\fa w \in W_h
																																			\,.
				\end{alignedat}
			\end{aligned}
			\right.
		\end{gather*}
		\;
		\vspace*{-\baselineskip}
	}
\end{algorithm2e}

A novelty of the approximate solution $\widetilde{u}_h = \exp\psi_h$ is that it is \emph{guaranteed} to deliver \emph{pointwise positivity}.
We exploit and extend the above property throughout this work to develop some of the first high-order \emph{bound-preserving} finite element methods for a variety of benchmark problems.
Another important property of this exponential discretization is that it preserves the \emph{multiplicative} group structure of the set
$\interior L^\infty_+(\Omega) = \{ w \in L^\infty(\Omega) \mid \essinf w > 0\}$.
More specifically,
\begin{equation}
	\exp \psi_h \exp \varphi_h = \exp(\psi_h + \varphi_h)
	\in
	\exp(W_h)
	\subset
	\interior L^\infty_+(\Omega)
	\,,
\end{equation}
for all $\psi_h$ and $\varphi_h\in W_h$.
Before expanding further on this and other topics, we present a comprehensive review of the literature and an itemized list of contributions.

\section{Contributions and related work} \label{sec:literature_review}

Numerical methods for pointwise bound constraints have been a topic of investigation for decades.
With this is mind, it is important to distinguish proximal Galerkin from the many other numerical methods for bound-constrained variational problems developed over this time frame.
Moreover, proximal Galerkin is as much a finite element method as it is an optimization algorithm.
In turn, we choose to survey the optimization literature as well as the numerical PDE literature in \Cref{sub:bound_constrained_optimization,sub:bound_constrained_finite_element_methods}, respectively.
The main contributions of this work are highlighted and itemized in~\Cref{sub:contributions}.
We encourage readers to move directly to the latter subsection if are they not interested in the historical context of the method.

\subsection{Optimization methods for pointwise bound constraints} \label{sub:bound_constrained_optimization}

Bound-\linebreak{}constrained variational problems arise in many subjects. These include, but are not limited to, contact mechanics  \cite{kikuchi1988contact,Wriggers2006}, financial mathematics \cite[Chap. 12]{tankov2015financial}, mathematical image processing \cite{LAmbrosio_VMTortorelli_1990}, and the geosciences, such as glaciology \cite{Zwinger2009}.
It is here that we are often confronted with the requirement that the solution be pointwise bounded from above or below by some critical threshold over at least a portion of the physical domain or its boundary. In PDE-constrained optimization and optimal control, bounds on the solution of the PDE, i.e., state constraints, naturally arise as a modeling requirement, see the well-known monographs \cite{Troltzsch_book2010,MHinze_RPinnau_MUlbrich_SUlbrich_2009}  and the references therein, especially \cite{ECasas_1986,ECasas_1993,ECasas_1997}. Consequently, a great deal of effort has been spent on treating bound constraints in infinite dimensions. 

We mainly restrict our overview to the numerical solution of the obstacle problem \cref{eq:DirichletVI}, with $K \subset \{v \in H^1(\Omega) \mid v \geq \phi\}$ for some $\phi \in H^1(\Omega)\cap L^\infty(\Omega)$, since the available solvers capture the main essences of the common techniques for other bound-constrained problems, however, we note that a number of the optimization algorithms listed below are applicable far beyond this setting.
Perhaps the most direct approach begins by prescribing a finite-dimensional subspace of $H^1(\Omega)$ for the discrete solution and then solving the associated variational problem by methods of nonlinear programming.
In this ``first-discretize-then-optimize'' class of approaches, the finite-dimensional reformulation typically amounts to a strongly convex quadratic program or a discrete strongly monotone variational inequality.
The fact that higher-order basis functions face numerous challenges when used to enforce pointwise bound constraints limits the benefits of these approaches; for further discussion, see~\Cref{sub:bound_constrained_finite_element_methods}.
However, a wealth of viable algorithms from nonlinear programming can be applied to lowest-order discretizations; see, e.g., \cite{Bertsekas99,nocedal1999numerical}.
Nevertheless, at least for active set-based approaches, such as in \cite{Hintermller2002}, one will almost certainly observe \emph{mesh-dependence}.

Mesh-dependence means that the number of nonlinear solver iterations required to reach a prescribed stopping tolerance (using the appropriate function space norm) will grow without bound on successively finer meshes.
Nevertheless, mesh-dependence can be computationally mitigated 
by appealing to multigrid methods, as was done in the celebrated papers 
\cite{Brandt1983,Hackbusch1983,RHWHoppe_1987,Hoppe1990,Hoppe_1994,Kornhuber_1994,Kornhuber_1996,Kornhuber_2001,BIWohlmuth_RHKrause_2003}; see \cite{CGraeser_RKornhuber_2009} for a comprehensive review. Despite the favorable behavior of these multigrid methods, there is no proof of mesh-independence in general. In particular, there is no guarantee that a given sequence of meshes will not miss low-dimensional portions (sets of positive capacity \cite{kinderlehrer2000introduction}) of the active set. To be fair, despite strong evidence of mesh independence for our method, which is bolstered by the convergence theory of the infinite-dimensional optimization algorithm underpinning the numerical method, a rigorous mesh-independence proof leveraging the ideas in \cite{weiser2005asymptotic} is postponed to future work.  

Mesh-dependence in active set methods arises from a lack of generalized differentiability of the (nonsmooth) residual in the function space setting, cf.~\cite{Hintermller2002,Ulbrich_2002}. This has motivated researchers to propose and analyze algorithms for bound-constrained problems in the continuous, i.e., infinite-dimensional setting. If an algorithm can be shown to converge in the continuous setting and the problem of interest exhibits sufficient stability properties around its solution, then this convergence will carry over to perturbed problems. At least for conforming discretizations, the associated finite-dimensional problem can be viewed as such a perturbation provided the discretization is sufficiently fine.
For further material on this topic, we refer the reader to the detailed discussions and references to applications in \cite{weiser2005asymptotic} and the pioneering works \cite{McCormick_1978,Allgower_1986,Allgower_1987}. 

Infinite-dimensional algorithms follow their finite-dimensional counterparts and can be roughly split into several categories: penalty methods, barrier methods, augmented Lagrangian methods, and first-order methods of convex optimization. For penalty (approximation) methods, we point the interested reader to the well-known monograph \cite{Glowinski_1984}, which claims these techniques go back to \cite{JLLions_1968,JLLions_1969}. However, we note that the numerical methods in \cite{Glowinski_1984}, e.g., coordinate descent, are not seen to be competitive with more recent developments in the subsequent decades after its publication.

Quadratic penalty methods are used widely in PDE-constrained optimization, see, e.g.,~\cite{hintermuller2006feasible,hintermuller2006path,Hinterm_ller_2009} and readily extendible to numerous applications; see, e.g., \cite{Kunisch_2012,Keuthen_2014,Adam_2018}. These are often referred to as ``Moreau--Yosida''-based approaches because the quadratic penalty can be viewed as the Moreau envelope of the indicator function for the bound constraints. The downside of these methods is the requirement to drive the penalty parameter to infinity to restore feasibility. Mirroring their finite-dimensional equivalents, interior point methods have also been investigated in detail for certain classes of PDE-constrained problems, see, e.g., \cite{Schiela2006,Ulbrich2007,Wollner2008,Hinze2009,Schiela2011}. Our method is closer to interior point methods, due to the entropy term \cite{sturmfels2022toric,pavlov2022gibbs}, and somewhat related to the first-order methods in \cite{Tran2015,Zosso2017}. However, in contrast to traditional interior point methods, the entropy functions employed in the text below do not exclude points from the feasible sets as they are still well-defined for feasible solutions that exhibit contact on sets of positive measure (or capacity).
Recently, entropy regularization has become a popular technique to promote exploration in reinforcement learning \cite{ahmed2019understanding,li2021quasi,landajuela2021improving}.
The same technique is also used in semidefinite programming \cite{lindsey2023fast} and optimal transport \cite{cuturi2013sinkhorn}.
An early comparison of infinite-dimensional interior point versus quadratic penalty approaches can be found in \cite{Bergounioux2000}. We also point to more recent work \cite{Tran2015,Zosso2017} on new penalty methods that appear to be mesh-dependent. Finally, though not expressly developed for bound-constrained problems in infinite-dimensional spaces, proximal point methods will play a central role in our method. This is discussed in detail in \Cref{sub:proximal_point} below.

Augmented Lagrangian approaches have also been developed for variational inequalities and PDE-constrained optimization; see, e.g., early work in \cite{Ito1990,Ito1990a,bergounioux1993augmented,Bergounioux1997} along with the monographs \cite{Glowinski_1984,KIto_KKunisch_2008,Wriggers2006,kikuchi1988contact} and the many references therein. Recent work has extended these methods to more general problems in abstract Banach spaces while simultaneously exploiting advances in matrix-free, inexact subproblem solvers in constrained optimization (such as \cite{heinkenschloss2014matrix,kouri2018inexact}), see \cite{birgin2014practical,kanzow2018augmented,Antil2023}.  In finite dimensions, augmented Lagrangian approaches are generally superior to penalty-based methods in the sense that the penalty parameter does not need to be driven to infinity to guarantee feasibility. Moreover, the penalty function in the subproblems is adaptively updated by the dual variables at each iteration. However, as observed in \cite[Sec.~5]{Antil2023}, the situation is more delicate in infinite dimensions, e.g., it may be necessary that some of the penalty parameters need to pass to infinity to guarantee the generation of a sequence of iterates with feasible accumulation points and the dual variables may not be bounded in those function spaces which are more easily treated numerically.

\subsection{Numerical methods for pointwise bound constraints} \label{sub:bound_constrained_finite_element_methods}

The development of bound-preserving numerical methods for PDEs began in the early days of scientific computing \cite{lax1954weak,godunov1959finite} and has remained an important pursuit ever since.
Although the present paper focuses on an entirely different category of PDE problems, hyperbolic conservation laws have provided a major source of motivation for research on the topic \cite{crandall1980monotone,harten1983upstream}, and have inspired many bound-preserving techniques now applied to other classes of PDEs.
In many situations, the challenge lies in the fact that standard high-order numerical methods do not preserve key invariant domain properties of the underlying physics \cite{guermond2016invariant}, such as pointwise positivity \cite{sun2018discontinuous}, yet, such properties are often required for numerical stability \cite{wu2021geometric}.

Some of the earliest attempts to ensure bound constraints involved using artificial viscosity to dampen oscillations that would lead to negativity and other spurious solution features \cite{vonneumann1950method,lax1960systems}.
Later on, more sophisticated ``high resolution'' flux- and slope-limiting strategies emerged \cite{boris1973flux,van1979towards,harten1983high,sweby1984high}; see also \cite{leveque1992numerical} for a classical overview and further references.

One of the most popular approaches to designing high-order bound-preserving methods is flux-corrected transport \cite{boris1973flux,zalesak1979fully,kuzmin2001positive,kuzmin2012flux}.
The general idea relies on forming a convex combination of a desired high-order solution and a bound-preserving low-order solution.
The method then selects the high-order solution wherever the constraint is satisfied and locally transitions to the low-order solution wherever it is necessary to avoid constraint violations.
A more recent popular approach \cite{zhang2010positivity,zhang2010maximum,zhang2011maximum}, which can be traced back to \cite{perthame1996positivity}, relies on developing high-order schemes with positive cell averages.
If such a high-order scheme can be found, the local solution need only be rescaled towards its (positive) mean wherever the constraints are violated.

The majority of high-order bound-preserving numerical methods for PDEs, including the two methods just described for hyperbolic conservation laws, do not constrain the solution to the continuous-level feasible set.
This is due, in part, to the fact that checking pointwise bound violations with an arbitrary polynomial is prohibitively expensive \cite{lasserre2007sum}.
Instead, almost all modern methods involve one of two common strategies: (1) enlarging the feasible set by only constraining the values of the solution at quadrature or nodal points \cite{zhang2010positivity,sun2018discontinuous,lin2023positivity,dzanic2022positivity,barrenechea2023nodally} or (2) diminishing the feasible set by constraining the solution's basis function coefficients \cite{abgrall2010example,anderson2017high,lohmann2017flux,abgrall2022reinterpretation}.
The former strategy results in a relaxation of the underlying problem that allows for solutions that are not truly positive \emph{pointwise}.
The latter strategy typically involves discretizing the solution with a positive basis that guarantees, e.g., that the solution is non-negative \emph{if} its coefficient are non-negative; see \cite{sukumar2004construction,ortiz2010maximum,ainsworth2011bernstein,cottrell2009isogeometric,allen2022bounds,dahlke2022wavelet} for the properties of various choices.
If a high-order discretization is used, both strategies lead to basis-dependent solutions, instead of solely approximation space-dependent solutions.

Since limiters tend to have a minimal number of hyperparameters, enforcing bound constraints using many of the techniques above may, at most, reduce to only solving a single-variable optimization problem at each element.
Recently, however, optimization-based methods have been explored to enlarge the solution space \cite{yee2020quadratic,bochev2020optimization}.
In these methods, a nonlinear program is solved at each element.
Likewise, global optimization approaches have also been explored, but, possibly owing to the cost, we are only aware of investigations with simple model problems \cite{liska2008enforcing,evans2009enforcement}.

Finally, logarithm-transformation methods, which date back at least to \cite{ilinca1996methodes}, have been known in the literature for some time \cite{ilinca1998unified,Carrillo2001Posit-6428,luo2003computation}.
Yet, they have taken on new interest in recent years \cite{metti2016energetically,liu2020exponential,fu2022high,vijaywargiya2023two,dzanic2023bounds}.
Other earlier work of related interest include \cite{degroen1979error,brezzi1989two,minor1993exponential,lazghab2002adaptive,fattal2004constitutive}.
The appeal is that discretizing a transformed variable may deliver an approximation that is intrinsically structure-preserving and basis-independent because it encodes geometry of the feasible set.
However, as we have already described in detail in \Cref{sub:the_cole_hopf_transform_and_the_semiring_of_non_negative_functions}, naively transforming a PDE variable leads to theoretical concerns when the solution is permitted to reach the boundary of the feasible set.
Therefore, implementing these methods in practice can be challenging, and may require ad-hoc assembly rules for the degenerate PDEs that arise, as noted in \cite{vijaywargiya2023two}.

\subsection{Contributions of the present work} \label{sub:contributions}

This paper focuses on establishing a mathematical foundation for the proximal Galerkin finite element method and exploring some of its applications.
The main technical results are developed specifically for the obstacle problem.
Yet, \Cref{sec:maximum_principles,sec:new_algorithms_for_topology_optimization} provide further applications with an eye towards future work.
In order to distinguish our work from previous and parallel efforts described in the literature above, we itemize our primary contributions:
\begin{itemize}[leftmargin=0cm,itemindent=1cm,itemsep=3pt,topsep=3pt]
	\item
		We introduce a new numerical method to treat infinite-dimensional bound-constrained variational inequality problems.
	The method hinges on an adaptive entropy regularization technique that was introduced by Nemirovskij and Yudin in \cite{nemirovskij1983problem} for general \emph{reflexive} Banach spaces, but has been primarily explored as an efficient optimization algorithm for finite-dimensional problems \cite{teboulle2018simplified}.
	Moreover, the nature of the functionals involved in our approach indicate that we need to work in a non-standard, \emph{non-reflexive} setting that is nevertheless natural for entropy regularization in infinite dimensions.
		\item
		The adaptive entropy regularization technique explored in this paper indicates the potential for a broad methodology in which the nonlinearity arising from the variational derivative of the entropy term can be replaced by a slack variable --- which we call the ``latent'' variable --- that is isomorphic to the regularized primal variable. This ultimately delivers a greater degree of flexibility in the choice of discretization scheme as the isomorphism naturally facilitates structure-preserving discretizations. We coin this framework the latent variable proximal point (LVPP) methodology.
	\item
	We apply the entropy regularization technique to the obstacle problem, and in doing so derive (distributional forms of) the \emph{entropic Poisson equation},
	\begin{subequations}
	\begin{equation}\label{eq:entrop-poisson}
		-\Delta u + \ln u = f
		\,,
			\end{equation}
	and the \emph{binary-entropic Poisson equation},
	\begin{equation}
				-\Delta u + \atanh u = f
				\,.
	\end{equation}
	\end{subequations}
	When understood as arising from the Euler--Lagrange equations for the regularized energy functionals, these appear to be novel semilinear elliptic PDEs. Though a similar equation to \cref{eq:entrop-poisson} has been investigated in \cite{MMontenegro_OSantanadeQueiroz_2009} and the nonlinearities are, at least when restricted to their \textit{domains}, smooth and monotone, the Nemytskii operators induced by $\ln$ and $\atanh$ require special care as they have \textit{restricted} domains when defined from the original real-valued functions; cf.~\cite{ambrosetti1995primer} and related literature for the analysis of Nemytskii, i.e., nonlinear superposition operators.
		\item
	Motivated by the analysis of the entropic Poisson equation, we establish a non-trivial geometric connection between non-negativity-constrained optimization and group theory.
	Further geometric connections are established via entropy functionals for other bound constraints.
	\item
	We present a novel finite element method to solve the entropic Poisson equation and perform preliminary \textit{a priori} error analysis on the resulting nonlinear mixed method.
	Our numerical experiments indicate that the method is mesh-independent when comparing the number of iterates required to reach a certain solution tolerance; see, e.g.,~\Cref{sub:obstacle_numerical_experiments}.
	\item
	We extend the contributions above to arrive at a novel approach to enforce discrete maximum principles on non-symmetric elliptic PDE, e.g., the advection-diffusion equation.
	\item
	We introduce \emph{two} different types of stable finite element pairs for proximal Galerkin discretizations of second-order elliptic VIs with pointwise bound constraints.
	The first type employs a \emph{discontinuous} latent variable $\psi_h$; cf.~\Cref{sub:finite_element_subspaces}.
	These finite elements lead to a primal solution $u_h$ with a bound-preserving cell average; cf.~\Cref{rem:PositiveCellAverage}.
	The second type uses a $C^0(\Omega)$-\emph{continuous} latent variable; cf.~\Cref{sub:StableElement_ContinuousLatentVariable}.
	In this case, a well-chosen quadrature rule induces a nodally bound-preserving primal solution; cf.~\Cref{rem:MassLumping}.
	Both types of proximal Galerkin discretizations lead to a secondary solution variable $\tilde{u}_h$ that is \emph{pointwise} bound-preserving throughout the domain.
	\item
	We present a new algorithm for topology optimization to showcase the breadth of applicability of the geometric optimization techniques developed in this work.
	The algorithm is efficient and relatively simple to implement.
	Our results indicate that it is also \emph{mesh-independent}.
	\item
	We release our code \cite{Keith2023ObstacleCode,Keith2023TopOptCode,ZenodoCode}, implemented in part using the finite element software FEniCSx in Python and, otherwise, with the MFEM library in C++ \cite{anderson2021mfem}, to facilitate broader adoption in the community.
\end{itemize}

\section{The obstacle problem} \label{sec:the_obstacle_problem}

In \Cref{sec:structure}, we surveyed several structural properties that entropy regularization brings to a specific form of the obstacle problem,
\begin{equation}
\label{eq:ObstacleProblem}
	\min_{u\in H^1_g(\Omega)}
	~
	\frac{1}{2}
	\int_\Omega |\nabla u|^2 \dd x
	-
	\int_\Omega f u \dd x
	~~\text{subject to~}
		u \geq \phi
	~\text{in~}\Omega
	\,,
\end{equation}
where $\phi = 0$.
In this section, we return to the same motivating example to review these properties in greater detail and extend our conclusions in order to analyze nonzero obstacle functions $\phi \neq 0$.
The main theoretical results in this section are the representation theorem, \Cref{lem:EntropyDifferentiability}, the characterization theorem, \Cref{thm:PrimalProblem}, and the convergence theorem, \Cref{thm:ConvergenceContinuousLevel}.
The section closes with a proximal Galerkin algorithm to solve the obstacle problem (\Cref{alg:ObstacleProblem}) and a report of our numerical experiments with it (\Cref{sub:obstacle_numerical_experiments}).

\subsection{The entropy gradient} \label{sub:entropy_regularity}

Before we can properly investigate entropy regularization and its role in treating the obstacle problem~\cref{eq:ObstacleProblem}, we must closely analyze the regularity of the entropy functional~\cref{eq:EntropyFunctional} in Lebesgue spaces.
Doing so will guide us toward the key geometric structure encoded in this functional.
As a pedagogical instrument, we proceed by building an analogy to the finite-dimensional setting.

Let $x\in \mathbb{R}^N$ denote the $N$-dimensional vector $(x_1,\ldots, x_N)$ and
denote the nonnegative orthant in $\mathbb R^N$ by
\begin{equation}
	\mathbb{R}^N_{+}
			= \{ (x_1,\ldots,x_N) \in \mathbb{R}^N \mid x_i \ge 0 \text{~for all~} i=1,\ldots N\}
	\,.
\end{equation}
Now, consider the corresponding finite-dimensional entropy function 
$s:\mathbb R^N_+ \to \mathbb R$ defined by $s(x) = \sum_{i=1}^N x_i\ln x_i - x_i$, wherein we remind the reader of our simplifying convention $0 \ln 0 := 0$.
It is easy to see that $s(x)$ is continuous and strictly convex on $\mathbb R^N_+$,
but only differentiable on its interior,
\begin{equation}
	\interior \mathbb{R}^N_{+} = \{ (x_1,\ldots,x_N) \in \mathbb{R}^N \mid x_i > 0 \text{~for all~} i=1,\ldots N\}
	\,,
\end{equation}
due to the logarithmic singularity in the gradient $\nabla s(x) = (\ln x_1, \ldots, \ln x_N)$.
A careful analysis is required to determine what the effect of the same type of logarithmic singularity will be at the function space level when analyzing the entropy functional $S$ in~\cref{eq:EntropyFunctional}.

As our first key structural result shows, $L^\infty_+(\Omega)$ and $\interior L^\infty_+(\Omega)$ reflect the roles played above in finite-dimensions by $\mathbb{R}^N_{+}$ and $\interior \mathbb{R}^N_{+}$, respectively. The proof is deferred to the outset of \Cref{sub:regularity_of_H}. 

\begin{theorem}[Gradient representation]
\label{lem:EntropyDifferentiability}
Let  $S: L^{p}(\Omega) \to  \mathbb R \cup \left\{+\infty\right\}$, $p \in [1,\infty]$, be the negative entropy functional defined by 
	 \[
	 S(u) = \left\{
	 		\begin{array}{cc}
	 		\int_\Omega u\ln u- u \dd x,& \text{ if $u \in L^p_+(\Omega)$,}\\
			 +\infty,& \text{otherwise.}
			 \end{array}\right.
	\]
	\begin{enumerate}[leftmargin=0.75cm,itemindent=0cm,itemsep=3pt]
	\item\label{item:EntropyDifferentiability_Part_1}
	If $p \in [1,\infty]$, then $S$ is strictly convex and lower semicontinuous. 
	\item\label{item:EntropyDifferentiability_Part_2}
	If $p \in (1,\infty]$, then $S$ is continuous on $L^p_+(\Omega)$.
	\item\label{item:EntropyDifferentiability_Part_3}
	If $p = \infty$, then $S$ is continuously
	Fr\'echet differentiable on $\interior L^p_+(\Omega)$
	with respect to the $L^p(\Omega)$-norm topology.
	In particular, the $L^\infty(\Omega)$-Fr\'echet derivative of $S$ can be uniquely characterized by the variational equation
	\begin{equation}
	\label{eq:EntropyDifferentiability_variations}
		\langle {S}^\prime(u), v\rangle = \int_\Omega v \ln u \dd x
						~~
		\text{for all~}
		u \in \interior L^\infty_+(\Omega)
		\text{~and~}
		v \in L^\infty(\Omega)
		.
	\end{equation}
	Moreover, $\|{S}^\prime(u)\|_{(L^\infty(\Omega))^\prime} = \|\nabla {S}(u)\|_{L^1(\Omega)}$, where
	\begin{equation}
	\label{eq:EntropyDifferentiability_gradient}
		\nabla {S}(u) = \ln u
		\in L^\infty(\Omega)
			\end{equation}
	is the unique primal representation (i.e., gradient) of 
	$S'(u)$ and is uniquely
		determined by the variational equation
	\begin{equation}
	\label{eq:EntropyDifferentiability_primal}
		( \nabla {S}(u), v) = \langle {S}^\prime(u), v\rangle
		~~
		\text{for all~}
		u\in \interior L^\infty_+(\Omega)
		\text{~and~} v \in L^1(\Omega)
		.
	\end{equation}
	\end{enumerate}
\end{theorem}

At a first glance, it is tempting to define $S$ from $L^1(\Omega)$ into $\mathbb R \cup \left\{+\infty\right\}$. 
This is the perspective taken in much of the literature on infinite-dimensional convex analysis; see, in particular, \cite{borwein1994strong,bauschke2001essential}. 
In this setting, it is shown that we have strict convexity and lower semicontinuity. However, as noted in \cite[Remark 5.7]{bauschke2001essential}, there
are some issues with this viewpoint. For example, $S$ would be nowhere continuous, but it would admit subgradients of the form $\ln u$ whenever $u \in \interior L^{\infty}_+(\Omega)$. 

As claimed above, and proven in \Cref{sub:regularity_of_H}, we see that $S \colon L^p(\Omega) \to \mathbb R \cup \left\{+\infty\right\}$ is in fact continuous on $L^p_+(\Omega)$ provided $p > 1$ and even continuously Fr\'echet differentiable when we take $p = \infty$ and $u \in \interior L^{\infty}_+(\Omega)$. Moreover, the derivative $S'(u)$ admits a ``primal'' representation of the form $\ln u$, which connects back to the convex analysis literature.  Our proof techniques, however, are not based on duality arguments or the properties of subgradients.  

Since $S$ will be used to define a Bregman distance below, whose domain needs to fit together with the typical regularity spaces for partial differential operators, we can safely choose any $p \in [1,\infty]$ so that the regularity space is continuously embedded into $L^p(\Omega)$, even if this initially appears to rule out certain functions in the domain of $S$. For example, if we are working with $u \in H^1(\Omega)$, then we can select $p \in [1,2]$, regardless of the dimension of $\Omega$ or regularity of $\partial \Omega$. On the other hand, if the dimension of $\Omega$ is $n=2$ or higher, then $H^1(\Omega)$ does not continuously embed into $L^{\infty}(\Omega)$. 

Finally, the properties of $S$ given in \Cref{lem:EntropyDifferentiability} indicate that $S : L^{\infty}_+(\Omega)\to\mathbb R$ is part of an important class of \emph{essentially smooth} functions introduced by Rockafellar \cite[Section~26]{rockafellar1970convex} (in finite dimensions) known as \emph{Legendre functions}, which are extended to infinite dimensions in \cite{borwein1994strong,bauschke2001essential}. 
As discussed in, e.g., \cite[Section~2.3]{teboulle2018simplified}, Legendre functions play a crucial role in proximal algorithms for finite-dimensional convex optimization.

To prepare us for non-trivial obstacles $\phi \neq 0$, we have the following corollary to \Cref{lem:EntropyDifferentiability} pertaining to the shifted entropy functional $S_{\phi}(u) = S(u - \phi)$.
As with \Cref{lem:EntropyDifferentiability}, the proof of this result is deferred to~\Cref{sub:regularity_of_H}.
       
\begin{corollary}[Gradient of the shifted entropy functional]
\label{cor:shift-ent}
	Let $\phi \in L^\infty(\Omega)$.
	The shifted negative entropy functional $S_{\phi}(u)$
	is strictly convex on
	\begin{equation}\label{eq:closure_inhom}
		L^{\infty}_{\phi,+}(\Omega)
		=
		\{ w \in L^\infty(\Omega) \mid w \geq \phi \}
		\,.
	\end{equation}
	and Fr\'echet differentiable on
	\begin{equation}
		\interior L^{\infty}_{\phi,+}(\Omega)
		=
		\{w \in L^{\infty}_{\phi,+}(\Omega) \mid \essinf (w-\phi) > 0 \}
	\end{equation}
	with respect to the norm topology on $L^{\infty}(\Omega)$.
	The Fr\'echet derivative of $S_{\phi}$ can be uniquely characterized by the variational equation
	\begin{equation}
	\label{eq:EntropyDifferentiability_variations_inhom}
		\langle S^\prime_{\phi}(u), v\rangle = \int_\Omega v \ln (u-\phi) \dd x
		~~
		\text{for all~}
		u\in \interior L^\infty_{\phi,+}(\Omega)
		\text{~and~} v \in L^\infty(\Omega)
		.
	\end{equation}
	Moreover, $\|S^\prime_{\phi}(u)\|_{(L^\infty(\Omega))^\prime} = \|\nabla {S}_{\phi}(u)\|_{L^1(\Omega)}$, where
	\begin{equation}
	\label{eq:EntropyDifferentiability_gradient_inhom}
		\nabla {S}_{\phi}(u) = \ln(u-\phi)
		\in L^\infty(\Omega)
		\,,
	\end{equation}
	is the unique primal representation (i.e., gradient) of $S^\prime_{\phi}\colon \interior L^\infty_{\phi,+}(\Omega) \to (L^\infty(\Omega))^\prime$ in $L^\infty(\Omega)$, determined by the variational equation
	\begin{equation}
	\label{eq:EntropyDifferentiability_primal_inho}
		( \nabla {S}_{\phi}(u), v) = \langle {S}^\prime_{\phi}(u), v\rangle
		\quad
		\text{for all~}
		u\in \interior L^\infty_{\phi,+}(\Omega)
		\text{~and~} v \in L^1(\Omega)
		.
	\end{equation}
\end{corollary}

\begin{remark}[Empty interior in the $H^1(\Omega)$ topology]
\label{rem:H1InteriorNonnegativeFunctions}
We point out that if $K = \{ u \in H^1_g(\Omega) \mid u \geq 0 \}  = H^1_g(\Omega) \cap H^1_+(\Omega)$ and $\Omega \subset \mathbb R^n$ with $n > 1$, then $\interior K = \emptyset$.
This is a non-trivial consequence of the fact that $H^1(\Omega)$ contains unbounded functions (see, e.g., \cite[4.43~Example]{adams2003sobolev}), and so we can get arbitrarily close to any $u \in K$ in the $H^1$-norm with an unbounded function lying outside of $K$. For example, if $\Omega \subset \mathbb R^2$ is the open unit ball centered at zero, then  $\varphi(x) := \min(0,-\ln | \ln (\|x\|) |)$, $ c > 0$, is in $H^1(\Omega)$, but is unbounded at zero. Fix $u \in K$ and assume wlog that $u$ is bounded on a neighborhood of $0$. Then $u + \epsilon \varphi \not\in K$ can be made to be arbitrarily close to $u$ by modulating $\epsilon > 0$. 
\end{remark}

\begin{remark}[No Riesz representation theorem]
When inspecting \Cref{lem:EntropyDifferentiability} and \Cref{cor:shift-ent}, the reader should note that $L^p(\Omega)$ is a Banach algebra only in the case $p=\infty$ and we only prove that $u\mapsto S_{\phi}(u)$ is Fr\'echet differentiable with respect to variations in this set; see~\cref{eq:EntropyDifferentiability_variations_inhom}.
In fact, there is a key step in our proof of \Cref{lem:EntropyDifferentiability} that requires all functions $u$ where the functional $S(u)$ is differentiable to have a multiplicative inverse $1/u \in L^\infty(\Omega)$; see \cref{eq:EntropyDifferentiability_critical}.
Based in part on this requirement, we continue to work directly with $L^\infty(\Omega)$, which is a \emph{non-reflexive} Banach space \emph{without} a corresponding Riesz representation theorem \cite{adams2003sobolev}.
It is, therefore, not a trivial consequence of differentiability that the Fr\'echet derivative $S^\prime_{\phi}(u) \in (L^\infty(\Omega))^\prime$ has the unique function space representation $\nabla {S}_{\phi}(u) \in L^\infty(\Omega)$ given by~\cref{eq:EntropyDifferentiability_primal_inho}.
In fact, the derivative of general functionals on $L^\infty(\Omega)$ lie in $(L^\infty(\Omega))^\prime$, which is the space of absolutely continuous, finitely additive set functions of bounded total variation on $\Omega$; cf.~\cite[p.~118]{yosida2012functional}.
Throughout this work, we consciously choose to refer to $\nabla S_{\phi}\colon \interior L^\infty_{\phi,+}(\Omega) \to L^\infty(\Omega)$ as the \emph{gradient} of the (shifted) entropy functional, even though we are well aware that the term ``gradient'' is typically understood as a Hilbert space concept. 
\end{remark}

\subsection{The entropy gradient is an isomorphism} \label{sub:entropy_and_its_geometry}

Let us return to the finite-dimensional entropy function $s(x) = \sum_{i=1}^N x_i\ln x_i - x_i$ introduced at the beginning of the previous subsection and focus on its properties in the strictly positive orthant $\interior \mathbb{R}^N_+ \subset \mathbb{R}^N$.
In this case, the reader should note that $x \mapsto \nabla s(x) = (\ln x_1, \ldots, \ln x_N)$, is a bijection between the set of component-wise positive vectors $x \in \mathbb{R}^N_+$ and the entire vector space $\mathbb{R}^N$.

This correspondence has a special algebraic significance if we view $\interior \mathbb{R}^N_+$ as a Lie group under the operation of componentwise multiplication,
\begin{equation}
	x \otimes y = (x_1 y_1,\ldots,x_N y_N),
\end{equation}
and view $\mathbb{R}^N$ as its associated Lie algebra under addition; cf.~\cite[Example~7.4~(b)]{lee2012introduction}.
Indeed, the smooth map $\nabla s\colon \interior \mathbb{R}^N_+ \to \mathbb{R}^N$ given above is a Lie group isomorphism because
\begin{equation}
	\nabla s(x) + \nabla s(y) = (\ln x_1 + \ln y_1, \ldots, \ln x_N + \ln y_N)
	=
	\nabla s(x \otimes y)
	.
\end{equation}
It is trivial to see that the same structure is replicated at the infinite-dimensional level between the Banach--Lie algebra $L^\infty(\Omega)$ and its Banach--Lie group $\interior L^\infty_+(\Omega)$ since
\begin{equation}
	\nabla S(u) + \nabla S(v) = \ln u + \ln v = \nabla S(uv).
\end{equation}

A deeper geometric meaning to this correspondence is revealed if we draw upon the well-known result in differential geometry that all finite-dimensional Lie groups are associated to their Lie algebra by an exponential map \cite[Proposition~20.8]{lee2012introduction}.
In the case of the Lie group $\interior \mathbb{R}^N_+$, it may be checked that the inverse of $\nabla s$, defined $(\nabla s)^{-1}(x) = (\exp x_1, \ldots, \exp x_N)$, is precisely this map.
Conveniently, the finite-dimensional result extends to the Banach--Lie group $\interior L^\infty_+(\Omega)$ \cite{glockner2002algebras,glockner2003lie}, and we are left with a similar geometric interpretation (cf.~\cref{fig:ExponentialMap}) of the isomorphism induced by the gradient of the entropy functional $\nabla S\colon \interior L^\infty_+(\Omega) \to L^\infty(\Omega)$ and its inverse,
\begin{equation}
\label{eq:InverseOfTheEntropyGradient}
	(\nabla S)^{-1}(\varphi) = \exp \varphi
	.
\end{equation}
Moreover, it can be shown that restricting the exponential map~\cref{eq:InverseOfTheEntropyGradient} to the subalgebra $H^1(\Omega)\cap L^\infty(\Omega)$ induces an isomorphism with the subgroup $H^1(\Omega) \cap \interior L^\infty_+(\Omega)$.
For further details, see \Cref{prop:Equivalence,prop:logexpChainRule}.

\begin{figure}
\centering
	\includegraphics[width=0.6\linewidth]{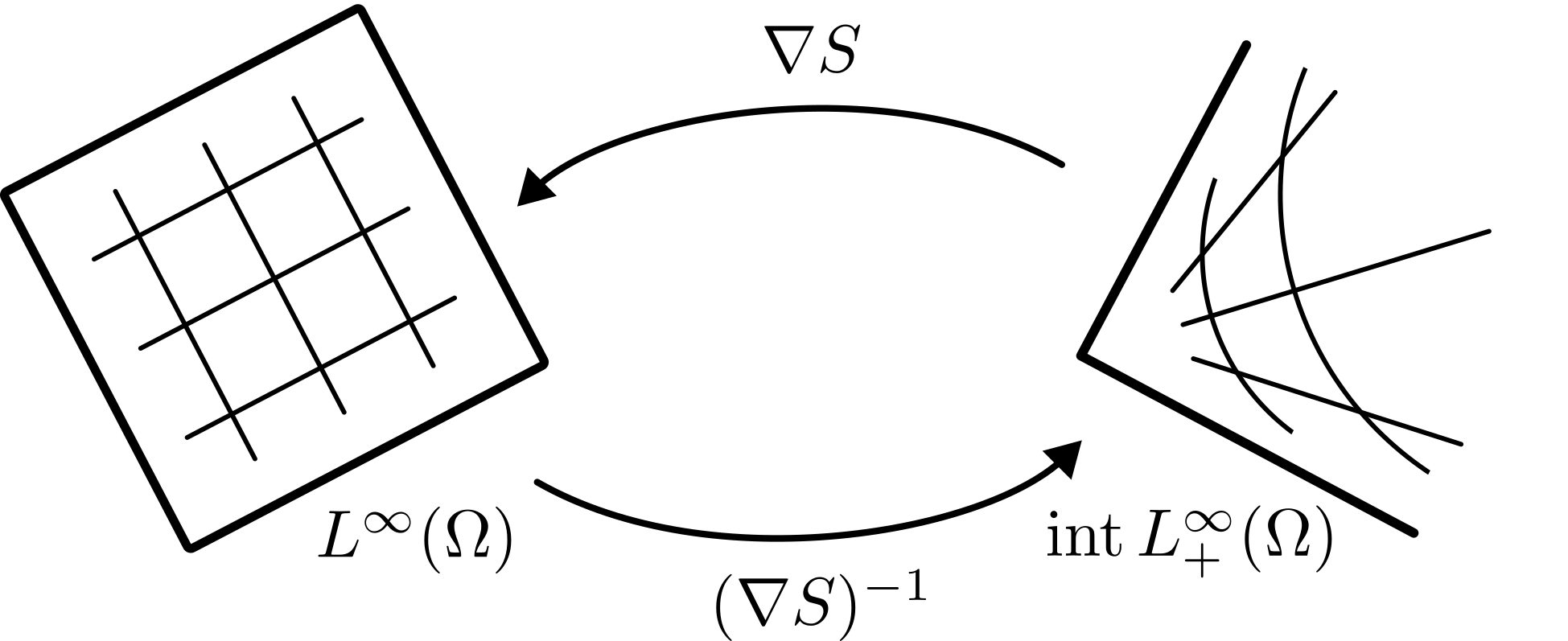}
	\caption{
	The exponential map $(\nabla S)^{-1}(\varphi) = \exp \varphi$ is an analytic isomorphism between the Banach algebra $L^\infty(\Omega)$ and the Banach--Lie group $\interior L^\infty_+(\Omega) = \{v \in L^\infty(\Omega) \mid \essinf v > 0 \}$; see~\Cref{prop:Equivalence}.
	Moreover, its restriction to the subalgebra $H^1(\Omega)\cap L^\infty(\Omega)$ forms an isomorphism with the subgroup $H^1(\Omega) \cap \interior L^\infty_+(\Omega)$; see~\Cref{prop:logexpChainRule}.
		\label{fig:ExponentialMap}}
\end{figure}

\begin{remark}[Exploiting the geometry of the feasible set]
From the optimization point-of-view, there is great value in the isomorphism $\nabla S$ residing in the fact that $L^\infty(\Omega)$ is a Banach space and Banach spaces are natural spaces in which to construct additive update formulas (they are complete, normed, and closed under addition).
Many competitive algorithms for \emph{unconstrained} optimization problems, such as gradient descent and Newton methods, are additive update formulas that leverage this linear structure in some way \cite{nocedal1999numerical}.
Likewise, when dealing with constrained optimization problems, most algorithms appeal to the linear structure of the ambient space containing the feasible set.
In \Cref{sub:proximal_point}, we will show how the isomorphism $\nabla S\colon \interior L^\infty_+(\Omega) \to L^\infty(\Omega)$ allows us to ignore the ambient space the original problem is posed in and work instead with the intrinsic geometry of the constraint set.
This, in turn, will allow us to treat constrained optimization problems in Sobolev spaces with methods originally designed only for the unconstrained setting.
\end{remark}

\subsection{Relative entropy} \label{sub:relative_entropy}

Entropy not only delivers an isomorphism between the Banach--Lie group $\interior L^\infty_+(\Omega)$ and its Banach algebra $L^\infty(\Omega)$.
It also induces a valuable distance function called the \emph{relative entropy} or \emph{(extended) Kullback--Leibler divergence}.

We assume below that $V$ is a Banach space.
For any smooth convex function $G\colon V \to \mathbb{R}$, its Bregman divergence is defined by the formula
\begin{equation}\label{eq:bregman-trad}
	D_G(u,v) = G(u) - G(v) - \langle G^\prime(v), u - v \rangle
	\,.
\end{equation}
Encoded in this definition is the important observation that, because $G$ is convex, the graph $\{(u,G(u)) \mid u \in V\}$ will always lie on or above its supporting hyperplanes, $\{ (u,G(v) + \langle G^\prime(v), u - v \rangle) \mid u \in V \}$, for every $v\in V$ at which $G^\prime(v)$ exists, see \cite{bregman1967relaxation} for this and related insights.
The Bregman divergence $D_G\colon \dom G \times\dom G^\prime \to \mathbb{R}$ measures the vertical distance between these two sets.

Loosely speaking, a Bregman divergence is a generalization of the squared distance between two points in a Hilbert space.
Indeed, it is a straightforward exercise to check, e.g., that if $G\colon H^1_0(\Omega) \to \mathbb{R}$, with $G(u) = \frac{1}{2}\|\nabla u\|_{L^2(\Omega)}^2$, the associated Bregman divergence is $D_G(u,v) = \frac{1}{2}\|\nabla u-\nabla v\|_{L^2(\Omega)}^2$.

The relative entropy $D \colon L^p_+(\Omega) \times \interior L^\infty_+(\Omega) \to \mathbb{R}$, for $p \in [1,\infty]$, is the Bregman divergence induced by the entropy functional $S$.
Given its importance to this work, we neglect to write the subscript-$S$ when working with this measure of distance.
In turn, we may select any $u \in L^p_+(\Omega)$ and $v \in \interior L^\infty_+(\Omega)$ to explicitly derive the relative entropy as follows,
\begin{equation}
\label{eq:RelativeEntropy}
	D(u,v) = S(u) - S(v) - ( \nabla S(v), u - v )
	=
	\int_\Omega u \ln\frac{u}{v} - u  + v \dd x
	\,.
\end{equation}
An illustration of the Bregman divergence of the finite-dimensional entropy function $s(x) = \sum_{i=1}^N x_i\ln x_i - x_i$ is given in \Cref{fig:BregmanPlot} for the case $N=1$. We initially use the right-hand side of \cref{eq:RelativeEntropy} in our study below without requiring its definition as a Bregman divergence. After a careful analysis shows that the relevant solutions are in $L^{\infty}_+(\Omega)$, we then employ the usual properties of Bregman divergences where required in several convergence proofs. This frees us from the rigid structures of convex analysis, e.g., that often fix the domain $V$ in the beginning and require us to work only in this space and its given topology.
\begin{figure}
\centering
	\includegraphics[width=0.5\linewidth]{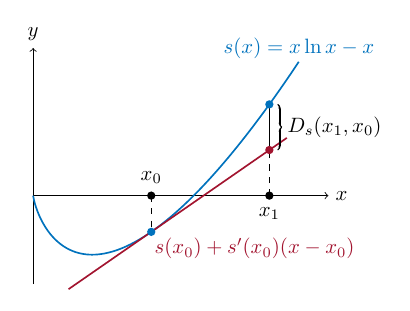}
	\caption{
	The convex function $s(x) = x \ln x - x$, its supporting hyperplane $\{s^\prime(x_0) + s^\prime(x_0)(x - x_0) \mid x \in \mathbb{R}\}$, and its Bregman divergence $D_s(x_1,x_0) = x_1 \ln ({x_1}/{x_0}) - x_1 + x_0$.
	\label{fig:BregmanPlot}}
\end{figure}

Along with other statistical distances, the relative entropy has a rich history of being used to encode geometric structure in analysis within statistics, probability theory, and information theory \cite{amari2000methods,amari2016information,nielsen2020elementary}.
Although a Bregman divergence is not symmetric, i.e., $D_G(u,v) \neq D_G(v,u)$ in general, it will satisfy the following important properties when $G$ is strictly convex \cite{bregman1967relaxation,chen1993convergence}:

\begin{proposition}[Properties of Bregman divergences]
\label{prop:BregmanDivergenceProperties}
Let $G\colon V \to \mathbb{R}$ be smooth and strictly convex.
Then the following properties hold:
\medskip

\noindent\textsl{Non-negativity.}
$D_G(u,v) \geq 0$ for all $u\in \dom G$ and $v \in \dom G^\prime$.
\medskip

\noindent\textsl{Positivity.}
$D_G(u,v) = 0$ if and only if $u=v$.
\medskip

\noindent\textsl{Convexity.}
$D_G(u,v)$ is strictly convex in its first argument.
Moreover, if $G$ is strongly convex, then so is $u \mapsto D_G(u,v)$.
\medskip

\noindent\textsl{Linearity.}
Let $F\colon V \to \mathbb{R}$ be smooth and strictly convex and $\lambda \geq 0$.
Then
\begin{equation}
	D_{G+\lambda F}(u,v) = D_{G}(u,v) + \lambda D_{F}(u,v)
	\,.
\end{equation}

\noindent\textsl{Three points identity.}
For all $u \in \dom G $ and $v,w \in \dom G^\prime$, it holds that
\begin{equation}
\label{eq:CosineIdentity}
	D_G(u,v) - D_G(u,w) + D_G(v,w)
	=
	\langle G^\prime(v) - G^\prime(w), v - u \rangle
	\,.
\end{equation}

\end{proposition}

\subsection{Proximal point} \label{sub:proximal_point}

Recall that in~\Cref{sub:dirichlet_free_energy} we proposed the regularized Dirichlet free energy functional
\begin{equation}
	A(u) = E(u) + \theta S(u)
	\,,
\end{equation}
and asserted that its minimizer will converge to the solution of the obstacle problem in the limit $\theta \to 0$; cf.~\cref{eq:EPE_ConvergenceWRTTemperature}.
Although this approach to solving the non-negative obstacle problem is viable, there is a much more numerically stable alternative.
Indeed, it turns out that we can just as readily generate a sequence of positive functions $u^k \to u^\ast$  by recursively regularizing the Dirichlet energy $E(u)$ with the Bregman divergence $D(u,u^{k})$.
The idea is relatively old in finite dimensions \cite{censor1992proximal,teboulle1992entropic,chen1993convergence,teboulle2018simplified}, and well-explored in reflexive Banach spaces \cite{darbon2021efficient,darbon2021accelerated}.
However, given that the algorithm is not well-known in the finite element community, we present a classical description that begins with a Hilbert space framework.

We now introduce the so-called \emph{proximal point algorithm} \cite{rockafellar2009variational,parikh2014proximal,teboulle2018simplified}, due to Marinet \cite{martinet1970regularisation}.
In turn, let $H$ be a Hilbert space and $\alpha>0$ be a positive step size parameter.
The \emph{proximal operator}, introduced in \cite{moreau1965proximite} by Moreau, is defined for every proper lower semi-continuous function $F\colon H \to \mathbb{R}\cup\{\infty\}$ as follows,
\begin{equation}
\label{eq:ProximalMap}
	{\prox}_{\alpha F}(v)
	=
	\argmin_{u \in H}
	\Big\{
		F(u) + \frac{1}{2\alpha}\|u - v\|_{H}^2
	\Big\}
	.
\end{equation}
The utility of this operator lies largely in the fact that the $\|\cdot\|_H^2$-regularization term in~\cref{eq:ProximalMap} transforms $F$ (which may not be differentiable) into a finite-valued function,
\begin{equation}
	F_\alpha(v) = \min_{u \in H} \Big\{ F(u) + \frac{1}{2\alpha}\|u - v\|_{H}^2 \Big\}
	,
\end{equation}
with an $\alpha^{-1}$-Lipschitz continuous gradient \cite{rockafellar2009variational}.
Moreover, when $F$ is convex, minimizing either $F$ or $F_\alpha$ is equivalent in the sense that
\begin{equation}
	\inf_{u\in H} F_\alpha(u) = \inf_{u\in H} F(u)
	\,.
	\quad
\end{equation}
In fact, the set of minimizers, $\argmin_{u\in H} F(u)$, coincides with the set of fixed points $u^\ast \in H$ that satisfy $u^\ast = {\prox}_{\alpha F}(u^\ast)$; see, e.g., \cite[Prop. 12.28]{HHBauschke_PLCombettes_2011}.

Choosing to iterate this fixed point equation with variable step sizes $\alpha_k > 0$ delivers the \emph{proximal point algorithm} \cite{martinet1970regularisation,rockafellar1976monotone}, written explicitly as
\begin{equation}
\label{eq:ProximalUpdate}
	u^0 \in H
	,
	\quad
	u^{k+1} = \prox_{\alpha_{k+1} F}(u^k)
	,
	\quad
	k=0,1,\ldots
	\end{equation}
It is well-known (see, e.g., \cite{guler1991convergence}), that $F(u^k)$ converges to $F(u^\ast)$ at a rate inversely proportional to the sum of the step sizes.
More explicitly, it holds that
\begin{equation}
\label{eq:ConvergenceRatePPA}
	F(u^k) - F(u^\ast) \leq \frac12 \frac{\|u^\ast - u^0\|^2_H}{\sum_{\ell=1}^k\alpha_\ell}
	\,.
\end{equation}
Thus, the function values of proximal iterates~\cref{eq:ProximalUpdate} can converge ``arbitrarily'' fast (by increasing $\alpha_\ell$), and the asymptotic complexity of the iteration~\cref{eq:ProximalUpdate} is determined by the complexity of the method used to solve each subproblem~\cref{eq:ProximalMap}.
Convergence of the function values carries over to convergence of the iterates provided an estimate of the type 
\[
\sigma(\| u - v\|) \le F(u) - F(v)
\]
holds, where $\sigma$ is monotone and invertible on $\mathbb R_+$, e.g., if $F$ is strongly convex. 

The potentially arbitrary order of convergence in~\cref{eq:ConvergenceRatePPA} makes the proximal point algorithm an attractive candidate to solve many optimization problems.
The drawback, however, is that each iteration of the algorithm requires the solution of a nonsmooth optimization problem that may be just as difficult to solve as the original problem; cf.~\cref{rem:ProximalMap}.
At the same time, the proximal operator~\cref{eq:ProximalMap} and fixed point iterations~\cref{eq:ProximalUpdate} are fundamental to a broad selection of modern optimization algorithms; see e.g., \cite{HHBauschke_PLCombettes_2011,ABeck_2017,teboulle2018simplified,GLan_2020} and the many references therein. They also play a deep role in augmented Lagrangian methods, as recognized at least as early as \cite{RTRockafellar_1976}, which have seen a resurgence in interest due, in part, to their applicability for infinite dimensional problems \cite{kanzow2018augmented,Antil2023}.

It turns out many of the most important properties of the proximal point algorithm also hold if $\frac{1}{2}\|u-v\|_H^2$ in~\cref{eq:ProximalMap} is replaced by a Bregman divergence $D_G(u,v)$ \cite{teboulle2018simplified}.
Indeed, if we assume that $G\colon V\to \mathbb{R}$ is a strictly convex functional on a Banach space $V$, we may define the \emph{Bregman proximal operator}
\begin{equation}
\label{eq:BregmanProximalMap}
	{\prox}_{\alpha F}^G(v)
	=
	\argmin_{u \in \dom F \cap\dom G }
	\big\{
		F(u) + \alpha^{-1}D_G(u,v)
	\big\}
	,
\end{equation}
and the corresponding \emph{Bregman proximal point algorithm}
\begin{equation}
\label{eq:BregmanProximalUpdate}
	u^0 \in \dom F \cap\dom G^\prime
	,
	\quad
	u^{k+1} = \prox_{\alpha_{k+1} F}^G(u^k)
	,
	\quad
	k=0,1,\ldots
	\end{equation}

\Cref{fig:BregmanPlot} illustrates the execution of this algorithm for the one-dimensional energy function $e(x) = \frac{1}{2}x^2 + x$ and the relative entropy $D_s(x,y) = x\ln(x/y) - x + y$.
Note that under the definitions above, one can show that~\cref{eq:ConvergenceRatePPA} generalizes as follows \cite{chen1993convergence},
\begin{equation}
\label{eq:ConvergenceRateBPPA}
	F(u^k) - F(u^\ast)
	\leq
	\frac{D_G(u^\ast,u^0)}{\sum_{\ell=1}^k \alpha_\ell}
	\,.
\end{equation}
See also \Cref{thm:ConvergenceContinuousLevel}.

\begin{figure}
\centering

						\includegraphics[height=5cm]{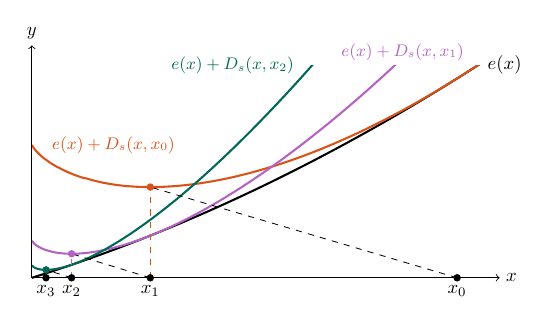}
	\caption{
	Illustration of convergence to the solution $x^\ast = 0$ for the constrained minimization problem $\min_{x\in[0,\infty)}\, e(x)$, where $e(x) = \frac{1}{2}x^2 + x$, by solving the sequence of minimization problems $x_{k+1} = \argmin_{x\in[0,\infty)}\, \{ e(x) + D_s(x,x_k) \}$ starting at $x_0 = 1$.
		\label{fig:proxPlot}}
\end{figure}

Our contribution is to show that the proximal operator~\cref{eq:BregmanProximalMap}, with an appropriately defined Bregman divergence, transforms the solution of an infinite-dimensional constrained optimization problem into a sequence of semi-linear PDEs whose solutions converge to the solution of the underlying VI.
In the case of the positive obstacle problem (i.e., $F = E$ and $G = S$), this conclusion hinges on the following result.
When combined with~\cref{eq:BregmanProximalUpdate}, \Cref{thm:PrimalProblem} leads us directly to \Cref{alg:HomogeneousAlg}, which forms the basis for the proximal Galerkin finite element method.
We note that the proof is technical and saved until \Cref{sub:Characterization}.

\begin{theorem}[Solution characterization]
\label{thm:PrimalProblem}
	Assume $\Omega \subset \mathbb{R}^n$ is an open, bounded Lipschitz domain, $n \ge 1$.
	Let $K = \{ v \in H^1_g(\Omega) \mid v \geq 0\} = H^1_g(\Omega) \cap H^1_+(\Omega)$, where $g \in H^1(\Omega) \cap C(\overline{\Omega})$ satisfies $\min_{\partial\Omega} g_{|_{\partial\Omega}} > 0$.
										Moreover, given $f \in L^\infty(\Omega)$, set
	\[
	E(v) = \frac{1}{2}\int_\Omega |\nabla v|^2 \dd x - \int_\Omega fv \dd x,
	\] 
	and for $w \in \interior L^\infty_{+}(\Omega)$
		set $D(v,w) = \int_\Omega v\ln(v/w) - v + w \dd x.$
	Then, for any
	step size $\alpha > 0$, the (relative) Dirichlet free energy minimization problem,
	\begin{equation}\label{eq:subprob_obs}
		\min_{v\in K}
		~
					 A_{\alpha}(v,w) :=
			E(v) + \alpha^{-1} D(v,w)
										,
	\end{equation}
	has a unique solution $u \in H^1_g(\Omega)\cap \interior L^\infty_{+}(\Omega)$ that is (uniquely) characterized by the weak form of the entropic Poisson equation; namely,
																						\begin{equation}
	\label{eq:PrimalProblemVE}
		(\alpha\nabla u, \nabla v)
				+
		(\ln u, v)
		=
		(\alpha f, v) + (\ln w, v)
				\quad
		\text{for all~}
		v \in H^1_0(\Omega)
		\,.
	\end{equation}
																														\end{theorem}

\begin{remark}[Adaptive entropy regularization]
\label{rem:MultiplierRepresentation}
Similar to the free energy formulation~\cref{eq:DirichletFreeEnergy}, where the Lagrange multiplier $\lambda$ is approximated by $\theta \ln (1/u)$, we see that the subproblems \cref{eq:PrimalProblemVE} give rise to an approximation of the form $\alpha^{-1} \ln (w/u)$. Recalling that $\theta = \alpha^{-1}$ we see that there is fundamental difference in the two approximations given by the inclusion of the function $w$.
Chosen correctly, as with $u^k$ in~\cref{eq:BregmanProximalUpdate}, this function can act as an informative prior on the sequence of subproblems.
More specifically, $w = u^k$ allows us to view the Bregman divergence $v \mapsto D(v,u^k) = \int_\Omega v \ln(v/u^k) - v + u^k \dd x$ as a biased barrier function that is updated adaptively at each iteration $k$ so that $u^k \to u^\ast$ without sending the step size $\alpha \to \infty$.
\end{remark}

From now on, we mainly focus on inhomogeneous obstacle problems; i.e., $\phi \neq 0$.
Therefore, we close this subsection with an important corollary for this case.
Before we state the result, we note that 
\[
S_{\phi}(u) - S_{\phi}(v) - (\nabla S_{\phi}(v),u-v) = D(u - \phi,v - \phi)
\]
whenever $u \in L^{p}_{\phi,+}(\Omega)$ and $v \in \interior L^{\infty}_{\phi,+}(\Omega)$. 
Therefore, $(u,v) \mapsto D(u - \phi,v - \phi)$ is a Bregman divergence on $L^{p}_{\phi,+}(\Omega) \times \interior L^{\infty}_{\phi,+}(\Omega)$.
For technical reasons, we require the obstacles to be in a particular subset of $H^1(\Omega)$ defined by
\begin{equation}
\label{eq:ObstacleSet}
	\mathcal{O} := \{ \phi \in H^1(\Omega) \cap C(\overline{\Omega}) \mid \Delta \phi \in L^{\infty}(\Omega) \}.
\end{equation}
Moreover, like~\cref{thm:PrimalProblem}, the proof of \Cref{cor:PrimalProblem_inhom} is delayed until \Cref{sub:Characterization}.

\begin{corollary}[Solution characterization for inhomogeneous obstacles]
\label{cor:PrimalProblem_inhom}
	In addition to the assumptions of \Cref{thm:PrimalProblem}, let  $\phi \in \mathcal{O}$ such that
	 $\min_{\partial\Omega} (g - \phi)_{|_{\partial\Omega}} > 0$ on $\partial \Omega$ and define
	 \[
		K_{\phi} :=  \{ u \in H^1_g(\Omega) \mid u \geq \phi ~\text{~a.e.~in~} \Omega\}.
	\]
	Then for any step size $\alpha > 0$ and $w \in \interior L^{\infty}_{\phi,+}(\Omega)$ the optimization problem
	\begin{equation}\label{eq:subproblem_inhom}
		\min_{v\in K_{\phi}}
		~
								E(v) + \alpha^{-1} D(v-\phi,w-\phi)
										,
	\end{equation}
	has a unique solution $u \in H^1_g(\Omega)\cap \interior L^\infty_{\phi,+}(\Omega)$ that satisfies the weak form of the (generalized) entropic Poisson equation; namely,
	\begin{equation}
	\label{eq:PrimalProblemVE_inhom}
		(\alpha\nabla u, \nabla v)
				+
		(\ln (u-\phi), v)
		=
		(\alpha f, v) + (\ln (w-\phi), v)
				\quad
		\text{for all~}
		v \in H^1_0(\Omega)
		.
	\end{equation}
\end{corollary}

\begin{remark}[Delicate analysis]
Semilinear mixed variational inequalities of obstacle type have been thoroughly studied, as detailed in the famous monograph by J.-F.~Rodrigues, \cite[Chap. 4.6]{rodrigues1987obstacle}. This includes regularity theory and a maximum principle that relates the solution of the VI to the obstacle, forcing term, and boundary values.  The techniques go back to the seminal work by Stampacchia \cite{stampacchia1965probleme}, Murty and Stampacchia \cite{Murthy1972} and can also be found in \cite{kinderlehrer2000introduction}. However, the VI associated with our problem is only valid if we can differentiate the ``extra'' nonlinearity in the entropy term. This in turn requires the solution $u$ of each subproblem to be essentially bounded and strictly above the obstacle, so we need to resort to a more delicate analysis solely based on the properties of the optimization problem. 
\end{remark}

\begin{remark}[Challenges of the Hilbert space setting]
\label{rem:ProximalMap}
		Let $\chi_K\colon H^1(\Omega) \to \mathbb{R}\cup\{\infty\}$ denote the indicator function $\chi_K(x) = 0$ if $x \in K$ and $\chi_K(x) = \infty$ otherwise.
	It is interesting to compare the operator
	\begin{equation}
	\label{eq:ProximalMapIdicator}
		\prox_{\alpha E + \chi_K}(v)
		=
		\argmin_{u \in {K}}
		\big\{
			E(u) + \frac{1}{2\alpha}\|u - v\|_{H^1}^2
		\big\}
		,
	\end{equation}
	to $\prox_{\alpha E}^S(v)$.
		Indeed, unlike~\cref{eq:PrimalProblemVE}, the subproblems that~\cref{eq:ProximalMapIdicator} induces each require the solution of their own VI,
	\begin{equation}
	\label{eq:L2ProxVI}
		\int_\Omega
		\nabla( (1+\alpha)u - v) \cdot \nabla w
		\dd x
		+
		\int_\Omega (u-v-\alpha f) w  \dd x
		\geq
		0
		~\fa w \in K-u,
													\end{equation}
	that is at least as difficult to solve as the original VI defining $u^\ast$; cf.~\cref{eq:DirichletVI}.
	Similar issues tend to appear whenever squared norm regularization terms are used to design proximal point algorithms for infinite-dimensional bound constraints. 
	
	Alternatively, one can use a penalty method to solve the original problem by considering instead a $C^{1,1}$-quadratic penalty term of the type 
	\[
	\frac{1}{2\alpha}\int_{\Omega} \max\{0, \phi - u\}^2 \, \dd x.
	\]
	This functional is in fact the prox-operator (in the $L^2(\Omega)$ topology) of the indicator function for the larger feasible set $\left\{u \in L^2(\Omega) \left|\; u \ge \phi \right.\right\}$. See \cite{hintermuller2006path,hintermuller2006feasible} for details including second-order algorithms and an analytical path-following scheme for $\alpha$. Note that the subproblems using a quadratic penalty would be semismooth semilinear elliptic PDEs. However, since the nonlinearity does not arise from a strictly monotone continuous function, we cannot derive a similar latent  variable formulation.
	\end{remark}

\begin{remark}[Comparison to the augmented Lagrangian method]
It is possible to view classical augmented Lagrangian methods as penalty methods that adaptively change the penalty function and associated penalty parameter via the behavior of the dual variable, i.e., Lagrange multiplier. Aside from identifying an efficient subproblem solver, the challenge is usually to find an appropriate combination of update strategies that allow for inexact subproblem solves and conservative parameter update rules that still 
exhibit rapid convergence behavior in practice. The method described in this work follows a similar strategy. Indeed, the role of the penalty function is played by a Bregman distance, which is adaptively updated via the primal variable, and the penalty parameter is given by $\alpha$. Bregman distances allow us to better exploit the geometry of the feasible set and the convergence theory of the proximal point method provides a clear connection to convergence rates that even allows for $\alpha$ to remain constant.
\end{remark}

\subsection{Latent variable proximal point} \label{sub:latent_variable_proximal_point}

An appealing feature of the entropic Poisson equation~\cref{eq:PrimalProblemVE_inhom} is that its solution permits \emph{two} additional representations; cf.~\Cref{fig:Trinity}.
In both cases, we take advantage of the entropy gradient $\nabla S(v) = \ln v$ being an isomorphism (cf.~\Cref{sub:entropy_and_its_geometry}).
First, we may introduce the latent variable representation,
\begin{equation}
	\psi = \ln (u-\phi)
	\quad
	\iff
	\quad
	u = \exp\psi + \phi
		\,,
\end{equation}
by simply applying the entropy gradient transformation to the primal solution $u$.
Second, as already noted in~\Cref{rem:MultiplierRepresentation}, we may construct a dual variable representation which, for the inhomogeneous obstacle problem, is written as follows:
\begin{equation}
\label{eq:ApproximateLagrangeMultiplierObstacleProblem}
	\lambda = \alpha^{-1}\ln\frac{w-\phi}{u - \phi}
	\quad
	\iff
	\quad
	u = (w - \phi)\exp(-\alpha\lambda) + \phi
	\,.
\end{equation}

The utility of these representations is witnessed if we consider how to solve either of the primal subproblems~\cref{eq:PrimalProblemVE} or \cref{eq:PrimalProblemVE_inhom}.
Indeed, due to the logarithmic terms, these semi-linear PDEs are only defined if $\essinf ( u - \phi) > 0$, which appears to rule out most efficient root-finding algorithms, such as Newton's method, and discretization choices, such as the Galerkin method.
Fortunately, the alternative solution representations above provide saddle-point relaxations of the entropic Poisson equation that do not suffer from these two drawbacks.

We are now ready to state the final main theoretical result, which also establishes explicit bounds on the optimization error for the latent variable proximal point (LVPP) algorithm, defined via~\cref{eq:ConvergenceContinuousLevel_VE} below.
The proof is given in~\Cref{sub:proximal_methods}.

\begin{theorem}[Convergence of LVPP]
\label{thm:ConvergenceContinuousLevel}
	Assume $\alpha_{k+1} > 0$, $k = 0,1,\ldots$, is a sequence of positive step size parameters.
	Furthermore, assume $\Omega \subset \mathbb{R}^n$ is an open, bounded Lipschitz domain, $n \ge 1$, $\phi \in \mathcal{O}$, and let $g \in H^1(\Omega) \cap C(\overline{\Omega})$ such that $\min_{\partial\Omega} (g - \phi) > 0$.
														Fix $\psi^0 \in H^1(\Omega) \cap L^\infty(\Omega)$ and consider the sequence of functions $u^k$, $\psi^k$ solving the following coupled system of variational equations:
			\begin{equation}
	\label{eq:ConvergenceContinuousLevel_VE}
		\left\{
		\begin{aligned}
			\,&\text{Find}~
			u^{k+1}\in  H^1_g(\Omega) ~\text{and}~\psi^{k+1} \in L^\infty(\Omega)
						~\text{such that~}
			\\
			&\begin{alignedat}{4}
				( \alpha_{k+1}\nabla u^{k+1}, \nabla v) + (\psi^{k+1}, v) &= (\alpha_{k+1} f + \psi^k, v)
				&&~\fa v \in H^1_0(\Omega)
				\,,
				\\
				(u^{k+1}, w) - (\exp\psi^{k+1}, w) &= (\phi, w)
				&&~\fa w \in L^2(\Omega)
				\,.
			\end{alignedat}
		\end{aligned}
		\right.
	\end{equation}
	Then the Dirichlet energy of the primal iterates is monotonically non-increasing, i.e.,
	\begin{equation}
	\label{eq:ConvergenceContinuousLevel_Monotonicity}
		E(u^{k+1}) \leq E(u^k)
		\,.
	\end{equation}
	Moreover, if $\sum_{j=1}^k \alpha_j \to \infty$ as $k \to \infty$, then the subproblem solutions $u^k$ converge in $H^1(\Omega)$ to
	\begin{equation}
		u^\ast
		=
		\argmin_{u\in H^1(\Omega)}
		~
		E(u)
		~~\text{subject to~}
		u \geq \phi
		~\text{in~}\Omega
		~\text{and~}
		u = g
		~\text{on~}\partial\Omega
		\,.
	\end{equation}
	Furthermore, the functions $\lambda^{k+1} = (\psi^k - \psi^{k+1})/\alpha_{k+1} $ converge strongly in $H^{-1}(\Omega)$ to the Lagrange multiplier $\lambda^\ast = - \Delta u^\ast - f$.
		In fact, the optimization error in both $u^k$ and $\lambda^k$ are equal and converge at the following arbitrary rate determined by the sequence of step-sizes $\alpha_k > 0$,
	\begin{equation}
	\label{eq:ConvergenceContinuousLevel_Rate}
		\frac12\|\lambda^\ast - \lambda^k \|^2_{H^{-1}(\Omega)}
		=
		\frac12\|\nabla u^\ast - \nabla u^k\|_{L^2(\Omega)}^2
		\leq
		\frac{D(u^\ast-\phi,u^0-\phi)}{\sum_{j=1}^k \alpha_{j}}
		.
	\end{equation}
																																					\end{theorem}

\begin{remark}[Arbitrary orders of convergence]
\label{rem:ConvergenceOrders}
	\Cref{thm:ConvergenceContinuousLevel} shows that the iteration complexity of LVPP depends on the choice of the step sizes $\alpha_k$.
	The consequences of different step size choices is summarized in~\Cref{cor:ConvergenceRates}.
		For example, we find that constant step sizes lead to sublinear convergence and geometrically increasing step sizes lead to first-order convergence.
	Even faster growing step size sequences will achieve superlinear convergence.
	See also~\Cref{rem:StrictComplementarity}.
\end{remark}

\begin{remark}[Convergence in the $H^1(\Omega)$-norm]
At first glance, control over the full $H^1(\Omega)$ norm of $u^k$ appears problematic because \cref{eq:ConvergenceContinuousLevel_Rate} does not include the full norm on $H^1_g(\Omega)$. However, in light of the Poincar\'e inequality and $u^\ast - u^k \in H^1_0(\Omega)$, we also obtain 
\[
	\frac12\| u^\ast -  u^k\|_{L^2(\Omega)}^2
	\leq
	c \frac{D(u^\ast-\phi,u^0-\phi)}{\sum_{j=1}^k \alpha_{j}},
\]
where $c > 0$ is an embedding constant independent of $k$.
\end{remark}

\begin{remark}[Convergence of the latent variable]
	If we adopt the conventions $\ln 0 = - \infty$ and $\exp(-\infty) = 0$, we may define $\psi^\ast = \ln (u^\ast-\phi)$ as an extended real-valued function on $\Omega$; cf.~\Cref{sub:the_cole_hopf_transform_and_the_semiring_of_non_negative_functions}.
	Likewise, we may understand convergence of the latent variable $\psi \to \psi^\ast$ under the metric implied by this transformation.
	Indeed, consider the metric $ d(\psi,\varphi) = \|\nabla\exp\psi - \nabla\exp\varphi\|_{L^2(\Omega)}$, first introduced in~\cref{eq:LatentVariableNonnegativityMetric}.
	Clearly,
	\begin{align}
		d(\psi^\ast,\psi)
		=
						\|\nabla(\exp\psi^\ast+\phi) - \nabla(\exp\psi+\phi)\|_{L^2(\Omega)}
				&=
		\|\nabla u^\ast - \nabla u\|_{L^2(\Omega)}
		\,,
	\end{align}
	which converges to zero as $k \to \infty$ by~\cref{eq:ConvergenceContinuousLevel_Rate}.
\end{remark}

\begin{remark}[Dual variable mixed formulation]
	The formulation~\cref{eq:ConvergenceContinuousLevel_VE} is derived by setting $w = u^k$, $\alpha = \alpha_{k+1}$, and substituting the equation $\psi^{k+1} = \ln(u^k - \phi)$ into~\cref{eq:PrimalProblemVE_inhom}.
	If, instead, we considered the dual variable substitution $\lambda^{k+1} = \ln ((u^{k}-\phi)/(u^{k+1}-\phi))/\alpha_{k+1}$, we would arrive at the following alternative formulation:
	\begin{equation}
	\label{eq:ConvergenceContinuousLevel_VE_alternative}
		\left\{
		\begin{aligned}
			\,&\text{Find}~
			u^{k+1}\in  H^1_g(\Omega) ~\text{and}~\lambda^{k+1} \in L^\infty(\Omega)
			~\text{such that~}
			\\
			&\begin{alignedat}{4}
				( \nabla u^{k+1}, \nabla v) - (\lambda^{k+1}, v) &= ( f, v)
				&&~\fa v \in H^1_0(\Omega)
				\,,
				\\
				(u^{k+1}, w) - (u^k\exp(- \alpha_{k+1}\lambda^{k+1}), w) &= (\phi,w)
				&&~\fa w \in L^2(\Omega)
				\,.
			\end{alignedat}
		\end{aligned}
		\right.
	\end{equation}
	Although this is equivalent to~\cref{eq:ConvergenceContinuousLevel_VE} at the continuous level, it will induce a different Galerkin method; cf.~\Cref{sub:proximal_galerkin}.
	We leave the study of such dual variable proximal Galerkin methods for future research.

\end{remark}

\begin{remark}[Strict complementarity]
\label{rem:StrictComplementarity}
	Although~\Cref{thm:ConvergenceContinuousLevel} allows us to establish arbitrary orders of convergence (see \Cref{cor:ConvergenceRates}), it still represents the worst-case iteration complexity.
	In particular, our numerical experiments in \Cref{ssub:experiment_2_kkt_conditions}, suggest that an improved result may be possible if the solution $u^\ast$ exhibits strict complementarity.
\end{remark}

\subsection{Proximal Galerkin} \label{sub:proximal_galerkin}

Motivated by~\Cref{thm:ConvergenceContinuousLevel}, it is natural to use finite-dimensional subspaces $V_h \subset H^1_0(\Omega)$ and $W_h \subset L^\infty(\Omega)$ in order to form a Galerkin discretization of~\cref{eq:ConvergenceContinuousLevel_VE}.
Thus, we arrive at \Cref{alg:ObstacleProblem}, which may be seen as a natural extension of~\Cref{alg:EntropicGalerkinIntro} to the inhomogeneous obstacle problem.

\begin{algorithm2e}[ht]
\DontPrintSemicolon
	\caption{\label{alg:ObstacleProblem} 
	Proximal Galerkin method for the obstacle problem.
	}
	\SetKwInOut{Input}{Input}
	\SetKwInOut{Output}{Output}
	\BlankLine
	\Input{Linear subspaces $V_h \subset H^1_0(\Omega)$ and $W_h \subset L^\infty(\Omega)$, initial solution guess $\psi_h^0 \in W_h$, unsummable sequence of step sizes $\alpha_k>0$.}
	\Output{Two approximate solutions, $u_{h}$ and $\widetilde{u}_h = \phi + \exp\psi_h$, and an approximate Lagrange multiplier, $\lambda_h = (\psi_h^{k-1} - \psi_h)/\alpha_k$.}
	\BlankLine
	Initialize $k = 0$.\;
	\Repeat{a convergence test is satisfied}
	{
		Solve the following (nonlinear) discrete saddle-point problem:
		\begin{gather}
		\label{eq:ObstacleDiscreteNonlinearSaddlePoint}
			\left\{
			\begin{aligned}
				\,&\text{Find}~
				u_{h}\in g_h + V_{h} ~\text{and}~\psi_{h} \in W_{h}
								~\text{such that~}
				\\
				&\begin{alignedat}{4}
				    (\alpha_{k+1} \nabla u_h, \nabla v) + (\psi_h, v) &= ( \alpha_{k+1} f + \psi_h^{k}, v)
					&&~\fa v \in V_h
					\,,
					\\
					(u_h, w) - (\exp\psi_h, w) &= (\phi, w)
					&&~\fa w \in W_h
					\,.
				\end{alignedat}
			\end{aligned}
			\right.
		\end{gather}
		\;
		\vspace*{-\baselineskip}
		Assign $\psi^{k+1}_h \leftarrow \psi_{h}$ and $k \leftarrow k+1$.\;
	}
\end{algorithm2e}
Just like~\Cref{alg:EntropicGalerkinIntro}, we find that~\Cref{alg:ObstacleProblem} delivers \emph{two} distinct approximations of the exact solution; $u_{h} \in V_h$ and $\widetilde{u}_h \in \phi + \exp(W_h)$.
The second of these approximations is unusual because it is guaranteed to satisfy the inequality $\widetilde{u}_h > \phi$.
Moreover, like the continuous-level algorithm in~\Cref{thm:ConvergenceContinuousLevel}, it also produces an approximate Lagrange multiplier,
\begin{equation}
	\lambda_h = (\psi_h^{k-1} - \psi_h^{k})/\alpha_{k}
	\,,
\end{equation}
where $k$ denotes the final iterate where the abstract convergence test in~\Cref{alg:ObstacleProblem} is satisfied.

Finite element methods typically lead to piece-wise polynomial approximations of the exact solution.
Given that $\widetilde{u}_h = \phi + \exp\psi_h$ relies on a non-standard type of exponential function approximation, it is natural ask whether $\widetilde{u}_h$ can produce an accurate approximation of the continuous-level solution $u$.
The following result provides a partial positive answer to this question.
The proof is given in~\Cref{sub:nonlinear_approximability}.

\begin{proposition}[Approximability]
\label{lem:UtildeBound}
	Let $u \in \interior L^\infty_+(\Omega)$ and define $\psi = \ln u$.
	Moreover, let $\psi_h \in W_h$ and $\widetilde{u}_h = \exp \psi_h$.
	The following identity holds:
	\begin{equation}
		\|u - \widetilde{u}_h\|_{L^\infty(\Omega)}
		\leq
				\|u\|_{L^\infty(\Omega)}
		\big(
			\exp\|\psi - \psi_h\|_{L^\infty(\Omega)} - 1
		\big).
	\end{equation}
\end{proposition}

The next ordinary concern would be the stability of the discretization~\cref{eq:ObstacleDiscreteNonlinearSaddlePoint}.
In the next subsection, we propose stable pairs of finite elements that can be used to construct $V_h$ and $W_h$.

\subsection{Stable pairs of finite elements I: Discontinuous latent variable} \label{sub:finite_element_subspaces}

Subspace pairings determine the stability of finite element methods for saddle-point problems~\cite{boffi2013mixed}.
Thus, it should come as no surprise that choosing compatible subspaces $V_h$ and $W_h$ is central to the proximal Galerkin method.
In this section, we choose to focus on stable pairs of finite elements with discontinuous latent variables $\psi_h$, as these appear to provide the most efficient approximations per degree of freedom.
The elements we propose are defined in~\cref{eq:SubspacePairs}, below.

\begin{remark}[Continuous latent variables]
In \Cref{sub:StableElement_ContinuousLatentVariable}, we propose a simple type of equal-order Lagrange element discretization employing a continuous latent variable $\psi_h$.
Although this compatible discretization has numerous interesting properties, including providing a feasible, $H^1(\Omega)$-conforming discrete solution $\tilde{u}_h$, we have been unable prove that these elements are stable under as general a set of assumptions as the elements proposed here.
The elements proposed in \Cref{sub:StableElement_ContinuousLatentVariable} are also less efficient with respect to degrees of freedom.
\end{remark}

\begin{remark}[Other stable pairs]
	Alternative finite element discretizations employing macroelement partitions (e.g., \cite{stenberg1984analysis}), various non-conforming approximation techniques (e.g., \cite{crouzeix1973conforming,cockburn2009unified}), or even spline-based approximation spaces (cf.~\cite{hughes2005isogeometric}) all provide possible alternatives for the design of proximal Galerkin methods.
	Due to space in this manuscript and the limitations of our present software, we leave these and other possible constructions to future studies.
\end{remark}

\subsubsection{Finite element notation} \label{ssub:finite_element_notation}

Here and throughout, $\mathcal{T}_h$ always denotes a shape-regular \cite[Chapter~11.1]{ern2021finite} partition of the domain $\Omega \subset \mathbb{R}^n$, $n = 2,3$, into finitely many open connected triangular/tetrahedral or quadrilateral/hexahedral mesh cells $T$ with Lipschitz boundaries $\partial T$ such that $\Omega$ is the union of the closure of all mesh cells $T$ in $\mathcal{T}_h$.
Following convention, $h > 0$ denotes the mesh size $h = \max_{T\in\mathcal{T}_h} h_T$, where $h_T = \operatorname{diam}(T)$.
Let $\mathbb{P}_{p}(T)$ denote the space of polynomials of total degree up to and including $p$ on a triangle/tetrahedron $T$.
Likewise, let $\mathbb{Q}_{p}(T)$ denote the space of tensor-product polynomials of partial degree up to and including $p$ on a quadrilateral/hexahedron $T$ \cite[Chapter~6.4]{ern2021finite}.
For any space $\mathbb{X}(T)$ of polynomials over an element $T \in \mathcal{T}_h$, we abuse notation to denote the corresponding space of ``broken'' polynomials $\mathbb{X}(\mathcal{T}_h) = \{\varphi \in L^\infty(\Omega) \mid \varphi_{|T} \in \mathbb{X}(T) \text{~for every~} T \in \mathcal{T}_h\}$.

We will require spaces of degree-$q$ polynomials whose traces on the cell boundary $\partial T$ have lower polynomial degree $p < q$.
To this end, define the sets of so-called bubble functions in $\mathbb{P}_{q}(T)$ and $\mathbb{Q}_{q}(T)$ to be $\mathring{\mathbb{P}}^{q}(T) = \{ \varphi \in \mathbb{P}_{q}(T) \mid \varphi_{|\partial T} =  0\}$ and $\mathring{\mathbb{Q}}^{q}(T) = \{ \varphi \in \mathbb{Q}_{q}(T) \mid \varphi_{|\partial T} =  0\}$, respectively.
Accordingly, define $\hat{\mathbb{P}}_{p}(T) = \mathbb{P}_{p}(T) \setminus \mathring{\mathbb{P}}^{p}(T)$ and $\hat{\mathbb{Q}}_{p}(T) = \mathbb{Q}_{p}(T) \setminus \mathring{\mathbb{Q}}^{p}(T)$.
Finally, let
\begin{equation}
	\mathbb{P}_p^q(T) = \hat{\mathbb{P}}_{p}(T) \oplus \mathring{\mathbb{P}}^{q}(T)
	\quad
	\text{and}
	\quad
	\mathbb{Q}_p^q(T) = \hat{\mathbb{Q}}_{p}(T) \oplus \mathring{\mathbb{Q}}^{q}(T)
	\,.
\end{equation}

\subsubsection{Element definitions and properties} \label{ssub:element_definitions}

\begin{figure}
\centering
	\subcaptionbox*{$\mathbb{P}_1\text{-bubble}$}{
		\includegraphics[width=0.2\linewidth]{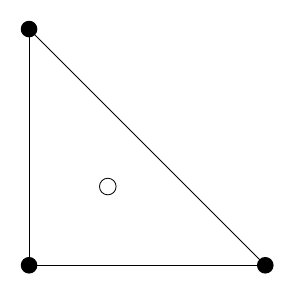}
	}
	\subcaptionbox*{$\mathbb{P}_0\text{-broken}$}{
		\includegraphics[width=0.2\linewidth]{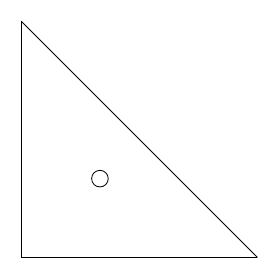}
	}
	\qquad
	\subcaptionbox*{$\mathbb{Q}_1\text{-bubble}$}{
		\includegraphics[width=0.2\linewidth]{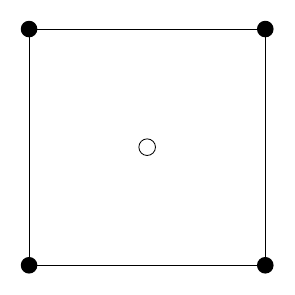}
	}
	\subcaptionbox*{$\mathbb{Q}_0\text{-broken}$}{
		\includegraphics[width=0.2\linewidth]{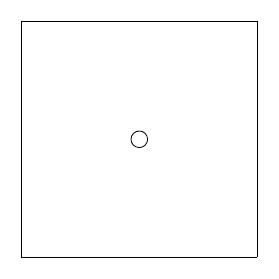}
	}
	\caption{
	Conventional representation of the $(\mathbb{P}_1\text{-bubble},\mathbb{P}_0\text{-broken})$ and $(\mathbb{Q}_1\text{-bubble},\mathbb{Q}_0\text{-broken})$ finite elements in two dimensions.
	The central degree of freedom in each element can be understood as the average over the individual mesh cell.
	\label{fig:P1BP0}}
\end{figure}

We are now ready to define the finite element spaces, which we chose based on \textit{a priori} analysis of a simple linearization of subproblem~\cref{eq:ObstacleDiscreteNonlinearSaddlePoint}.
For further details of the analysis, see~\Cref{sec:the_entropic_finite_element_method}.
\smallskip

For any integer $p \geq 1$, we define the following two pairs of spaces:
\begin{subequations}
\label{eq:SubspacePairs}
\smallskip

\noindent\textsl{Triangular/Tetrahedral elements.} We refer to the following as the $(\mathbb{P}_p\text{-bubble},\linebreak\mathbb{P}_{p-1}\text{-broken})$ pairing:
\begin{equation}
\label{eq:SubspacePair1}
	V_h = \mathbb{P}_{p}^{p+n}(\mathcal{T}_h)\cap H^1_0(\Omega)
	\,,\qquad
				W_h = \mathbb{P}_{p-1}(\mathcal{T}_h)
	\,.
\end{equation}

\noindent\textsl{Quadrilateral/Hexahedral elements.} We refer to the following as the $(\mathbb{Q}_p\text{-bubble},\linebreak\mathbb{Q}_{p-1}\text{-broken})$ pairing:
\begin{equation}
\label{eq:SubspacePair2}
	V_h = \mathbb{Q}_{p}^{p+1}(\mathcal{T}_h)\cap H^1_0(\Omega)
	\,,\qquad
						W_h = \mathbb{Q}_{p-1}(\mathcal{T}_h)
	\,.
\end{equation}
\end{subequations}
\Cref{fig:P1BP0} provides a visual representation of the lowest-order versions of these elements.

\begin{remark}[Positive cell average]
\label{rem:PositiveCellAverage}
	Assume $\phi = 0$.
	Although the piecewise polynomials $u_h$ that arise from solving the subproblems~\cref{eq:ObstacleDiscreteNonlinearSaddlePoint} can not be guaranteed to preserve pointwise positivity, they are guaranteed to have positive cell averages.
	Indeed, notice that the subspaces $W_h$ in~\cref{eq:SubspacePairs} always include piecewise constant functions.
	Therefore, we may consider the second equation in~\cref{eq:ObstacleDiscreteNonlinearSaddlePoint} with $w = 1$ in $T$ and $w = 0$ otherwise.
	Testing with this particular function implies that
	\begin{equation}
		\int_T u_h \dd x = \int_T \exp\psi_h > 0
		\,.
	\end{equation}
	If $\phi \neq 0$, then a similar argument implies that each cell average of $u_h$ lies above the corresponding cell average of $\phi$.
\end{remark}

\begin{remark}[Alternative subspaces]
\label{rem:AlternativeSubspaces1}
	Although variable-order spaces like $V_h$ in~\cref{eq:SubspacePairs} are supported in some software \cite{schoberl2014c++,fuentes2015orientation}, they may not available in the preferred software of many users.
	For this reason, we also propose the following less efficient alternative pairings:
	\begin{subequations}
	\label{eq:SubspacePairs_alt}
	\smallskip

	\noindent\textsl{Alternative triangular/tetrahedral elements.} We refer to the following as the $(\mathbb{P}_{p+n},\linebreak\mathbb{P}_{p-1}\text{-broken})$ subspaces:
	\begin{equation}
	\label{eq:SubspacePair3}
		V_h = \mathbb{P}_{p+n}(\mathcal{T}_h)\cap H^1_0(\Omega)
		\,,\qquad
		W_h = \mathbb{P}_{p-1}(\mathcal{T}_h)
		\,.
	\end{equation}

	\noindent\textsl{Alternative quadrilateral/hexahedral elements.} We refer to the following as the $(\mathbb{Q}_{p+1},\linebreak\mathbb{Q}_{p-1}\text{-broken})$ subspaces:
	\begin{equation}
	\label{eq:SubspacePair4}
		V_h = \mathbb{Q}_{p+1}(\mathcal{T}_h)\cap H^1_0(\Omega)
		\,,\qquad
		W_h = \mathbb{Q}_{p-1}(\mathcal{T}_h)
		\,.
	\end{equation}
	\end{subequations}
	Since~\cref{eq:SubspacePairs} are stable (cf.~\Cref{lem:Fortin}), it is a straightforward consequence of the inclusions $\mathbb{P}_{p}^{p+n}(T) \subset \mathbb{P}_{p+n}(T)$ and $\mathbb{Q}_{p}^{p+1}(T) \subset \mathbb{Q}_{p+1}(T)$ that~\cref{eq:SubspacePairs_alt} are also stable.
	For further details, see~\Cref{rem:AlternativeSubspaces2}.

\end{remark}

\subsection{Numerical experiments} \label{sub:obstacle_numerical_experiments}

We performed five sets of numerical experiments in order to validate the proximal Galerkin method.
The first experiment involves a smooth biactive manufactured solution that allows us to verify the (mesh-independent) iteration complexity predicted by \Cref{thm:ConvergenceContinuousLevel}, in addition to high-order convergence rates with respect to the polynomial order of the finite element subspaces.
In the second experiment, we check the discrete Karush--Kuhn--Tucker (KKT) conditions on a manufactured solution exhibiting strict complementarity.
In this case, we observe \emph{better} iteration complexity than predicted by \Cref{thm:ConvergenceContinuousLevel}.
We conjecture that this improved convergence order holds in general whenever a strict complementarity condition is satisfied; cf.~\Cref{rem:StrictComplementarity}.
The third experiment involves a non-smooth biactive solution and is included to further stress test the proximal Galerkin method.
In the fourth experiment, we consider a benchmark obstacle problem from the literature and demonstrate our ability solve this problem with the highest order finite elements currently supported in our MFEM implementation ($p=12$); cf.~\Cref{rem:Implementations}.
Finally, our fifth experiment demonstrates good performance with a \emph{discontinuous} obstacle on an adaptively-refined mesh.
Thus, the regularity restriction $\Delta \phi \in L^\infty(\Omega)$ required in the theory above is \emph{not} essential in practice.

Each of our experiments were conducted on standard sequences of uniformly refined nested meshes $\mathcal{T}_{h}, \mathcal{T}_{h/2}, \mathcal{T}_{h/4}, \ldots$ conforming to unit ball domains in $\mathbb{R}^2$.
The experiments with the triangular elements (FEniCSx) used an $\ell^\infty$-unit ball (i.e., square) domain,
\begin{subequations}
\begin{equation}
	\Omega_\infty
	=
	\{
		(x,y) \in \mathbb{R}^2 \mid \max\{|x|,|y|\} < 1
	\}
	\subset \mathbb{R}^2
	\,,
\end{equation}
with initial mesh size denoted $h = h_\infty$.
Meanwhile, the experiments with the quadrilateral elements (MFEM) used an $\ell^2$-unit ball (i.e., circular) domain,
\begin{equation}
	\Omega_2 
	=
	\{
		(x,y) \in \mathbb{R}^2 \mid x^2 + y^2 < 1
	\}
	\subset \mathbb{R}^2
	\,,
\end{equation}
\end{subequations}
with initial mesh size denoted $h = h_2$, which was uniformly refined using a standard transfinite interpolation rule to handle the curvilinear element mappings \cite{gordon1973construction}.
The initial meshes in these sequences are depicted in~\Cref{fig:ObstacleMesh}.

\begin{figure}
	\centering
	\begin{tikzpicture} [dot/.style={draw,rectangle,minimum size=4mm,inner sep=0pt,outer sep=0pt,thick}, step=2cm]
		\draw [thin] (0,0) grid (4,4);
		\draw [thin] (0,2) -- (2,4);
		\draw [thin] (0,0) -- (4,4);
		\draw [thin] (2,0) -- (4,2);
		\node [font=\normalsize] at (0.4,3.7) {$\Omega_\infty$};
																					\draw [thin,<->] (0,-0.3) -- (4,-0.3) node[draw=none,fill=none,font=\footnotesize,midway,below] {2};
		\draw [thin,<->] (-0.3,0) -- (-0.3,4) node[draw=none,fill=none,font=\footnotesize,midway,left] {2};
	\end{tikzpicture}
	\qquad\qquad
	\begin{tikzpicture}
		\draw[thin] (2,2) circle (2 cm);
		\draw[thin] (1.29289321882,1.29289321882) -- (1.29289321882,2.70710678118);
		\draw[thin] (1.29289321882,1.29289321882) -- (2.70710678118,1.29289321882);
		\draw[thin] (2.70710678118,2.70710678118) -- (2.70710678118,1.29289321882);
		\draw[thin] (2.70710678118,2.70710678118) -- (1.29289321882,2.70710678118);
		\draw[thin] (0.58578643763,0.58578643763) -- (1.29289321882,1.29289321882);
		\draw[thin] (3.41421356236,3.41421356236) -- (2.70710678118,2.70710678118);
		\draw[thin] (3.41421356236,0.58578643763) -- (2.70710678118,1.29289321882);
		\draw[thin] (0.58578643763,3.41421356236) -- (1.29289321882,2.70710678118);
				\draw [thin,<->] (0,-0.3) -- (4,-0.3) node[draw=none,fill=none,font=\footnotesize,midway,below] {2};
		\node [font=\normalsize] at (0.4,3.7) {$\Omega_2$};
    \end{tikzpicture}
	\caption{\label{fig:ObstacleMesh}Initial finite element meshes for computational domains $\Omega_\infty$ and $\Omega_2$. We denote their mesh sizes $h = h_\infty$ and $h = h_2$, respectively. Left: Initial triangular mesh for the $(\mathbb{P}_p\text{-bubble},\mathbb{P}_{p-1}\text{-broken})$ subspace pair on the domain $\Omega_\infty$ (FEniCSx experiments). Right: Initial five-element curvilinear quadrilateral mesh for the $(\mathbb{Q}_{p+1},\mathbb{Q}_{p-1}\text{-broken})$ subspace pair on the domain $\Omega_2$ (MFEM experiments). Across our experiments, we consider various polynomial orders $p \geq 1$ for both of these subspace pairs.}
\end{figure}
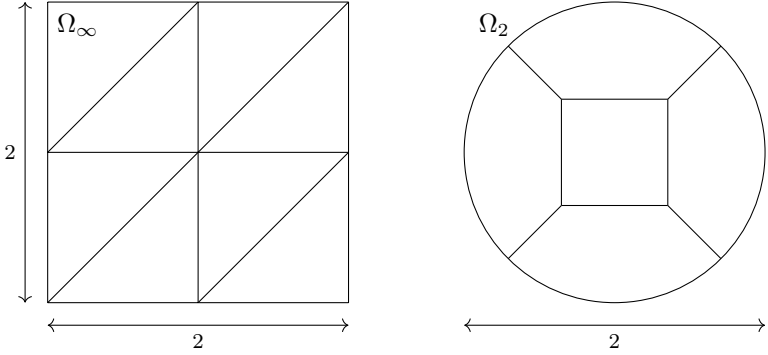

\begin{remark}[Software implementations and reproducibility]
\label{rem:Implementations}
	We conducted numerical experiments across two separate code bases, FEniCSx \cite{scroggs2022construction} and MFEM \cite{anderson2021mfem}, and have released our implementations to the public \cite{Keith2023ObstacleCode,ZenodoCode}.
	The FEniCSx implementation \cite{ZenodoCode} is Python-based and uses the $(\mathbb{P}_p\text{-bubble},\mathbb{P}_{p-1}\text{-broken})$ triangular elements proposed in~\Cref{sub:finite_element_subspaces}.
	The MFEM implementation is written in C++.
	Because MFEM does not currently support bubble function enrichment, the MFEM implementation \cite{Keith2023ObstacleCode} uses $(\mathbb{Q}_{p+1},\mathbb{Q}_{p-1}\text{-broken})$ quadrilateral elements (see~\Cref{rem:AlternativeSubspaces1}) instead of the $(\mathbb{Q}_p\text{-bubble}, \mathbb{Q}_{p-1}\text{-broken})$ elements also proposed in~\Cref{sub:finite_element_subspaces}.

	We have taken care to ensure that the numerical experiments in this section are easily reproducible.
	Indeed, all results from Experiments 1, 2, 3, and 5 can be reproduced by running the FEniCSx script \texttt{obstacle.py} or compiling and running the MFEM codes \texttt{obstacle.cpp} or \texttt{obstacle\_discontinuous.cpp}, all of which are available at \cite{ZenodoCode}.
	Experiment 4 now exists as an official part of the MFEM codebase called MFEM Example~36 \cite{Keith2023ObstacleCode}.
\end{remark}

\subsubsection{Experiment 1: Smooth biactive solution} \label{ssub:experiment_1_biactivity}

In this experiment, we set $\phi = 0$ and $g=u$, where $u(x,y)$ is the smooth manufactured solution
\begin{equation}
\label{eq:Biactivity_SmoothManufacturedSolution}
	u(x,y)
	=
	\begin{cases}
		0 & \text{if~} x < 0\,,\\
		x^4 & \text{otherwise,}\\
	\end{cases}
	\quad
		\text{implied by}
		\quad
	f(x,y)
	=
	\begin{cases}
		0 & \text{if~} x < 0\,,\\
		-12x^2 & \text{otherwise.}\\
	\end{cases}
\end{equation}
See \Cref{fig:BiactiveIterationComplexity,fig:BiactivePolynomialOrderConvergenceRates} for depictions of the exact solution on $\Omega_\infty$ and $\Omega_2$, respectively.

This problem is specifically chosen to exhibit \emph{biactivity}; i.e., both the inequality constraint $u \geq 0$ and the associated Lagrange multiplier are simultaneously equal to zero on a set of positive measure; i.e., on the set $\{ (x,y) \mid x < 0 \}$.
Such problems are notoriously difficult to solve for certain classes of algorithms, such as active set methods.
Biactivity, also know as weak activity or lack of strict complementarity, is a notion from nonlinear optimization that indicates a kind of degenerate nonsmoothness of the primal-dual system of equations used to calculate the solution. It is often associated with a lack of stability with respect to perturbations of the data, as well. We refer the interested reader to any standard text of numerical optimization; see, e.g., \cite[Definition~12.5]{nocedal1999numerical}.

Our first aim is to use this challenging example to illustrate \emph{mesh-independence} of the proximal Galerkin method.
To this end, we use \Cref{tab:BiactivityMeshIndependence} to record the values of the increments $\|u_h^{k} - u_h^{k-1}\|_{H^1(\Omega_\infty)}$ taken from a sequence of refined meshes with polynomial orders $p=1,2$ from our FEniCSx implementation \cite{ZenodoCode}.
The specific step size rule used to generate this data is chosen based on \Cref{cor:ConvergenceRates} to deliver \emph{superlinear} convergence (in iterations), and is given as follows:
\begin{subequations}
\begin{equation}
\label{eq:ObstacleStepSizes_superexponential}
	\alpha_1 = 1\,,
	\quad
	\alpha_k
	=
	\min\bigl\{ \max\bigl\{\alpha_1,r^{q^{k-1}} - \alpha_{k-1}\bigr\} , 10^{10} \bigr\}
	\,,
	\quad
	k = 2,3,\ldots,
\end{equation}
where $r = q = 1.5$.
Note that, for each iteration $k$, the increments in \Cref{tab:BiactivityMeshIndependence} converge to fixed values as the mesh is refined or the polynomial order is raised.
Moreover, the total number of linear equation solves remains bounded.
Both of these characteristics are emblematic of a \emph{mesh-independent} numerical method.

\begin{table}
\centering
\footnotesize
\renewcommand{\arraystretch}{1.3}
\begin{tabular}{ |c|c|c|c|c|c|c| }
 \hhline{|>{\arrayrulecolor{white}}-->{\arrayrulecolor{black}}|-----|}
 \multicolumn{2}{c|}{} & \multicolumn{5}{c|}{\cellcolor{lightgray!15} \small\raisebox{5pt}{\vphantom{f}} Progress of the iterates $\|u_h^{k} - u_h^{k-1}\|_{H^1(\Omega_\infty)}$ for various $h$ and $p$}\\[3pt]
 \hhline{|>{\arrayrulecolor{white}}-->{\arrayrulecolor{black}}|-----|}
 \multicolumn{2}{c|}{}& \multicolumn{3}{c|}{\cellcolor{lightgray!05} Polynomial order $p = 1$} & \multicolumn{2}{c|}{\cellcolor{lightgray!10} Polynomial order $p = 2$}\\
 \hline
 \rowcolor{lightgray!15}
 $k$ & $\alpha_{k}$ & $h_\infty/16$ & $h_\infty/32$ & $h_\infty/64$ & $h_\infty/16$ & $h_\infty/32$ \\
 \hline
\cellcolor{lightgray!05}   1 & \cellcolor{lightgray!01}   $1.0$              & $2.10\cdot 10^{0}$   & $2.10\cdot 10^{0}$   & $2.10\cdot 10^{0}$   & $2.10\cdot 10^{0}$   & $2.10\cdot 10^{0}$   \\
\cellcolor{lightgray!05}   2 & \cellcolor{lightgray!01}   $1.0$              & $6.45\cdot 10^{-1}$  & $6.45\cdot 10^{-1}$  & $6.45\cdot 10^{-1}$  & $6.45\cdot 10^{-1}$  & $6.45\cdot 10^{-1}$  \\
\cellcolor{lightgray!05}   3 & \cellcolor{lightgray!01}   $1.49$             & $1.73\cdot 10^{-1}$  & $1.73\cdot 10^{-1}$  & $1.73\cdot 10^{-1}$  & $1.73\cdot 10^{-1}$  & $1.73\cdot 10^{-1}$  \\
\cellcolor{lightgray!05}   4 & \cellcolor{lightgray!01}   $2.43$             & $1.10\cdot 10^{-1}$  & $1.10\cdot 10^{-1}$  & $1.10\cdot 10^{-1}$  & $1.10\cdot 10^{-1}$  & $1.10\cdot 10^{-1}$  \\
\cellcolor{lightgray!05}   5 & \cellcolor{lightgray!01}   $5.35$             & $7.76\cdot 10^{-2}$  & $7.77\cdot 10^{-2}$  & $7.77\cdot 10^{-2}$  & $7.77\cdot 10^{-2}$  & $7.77\cdot 10^{-2}$  \\
\cellcolor{lightgray!05}   6 & \cellcolor{lightgray!01}   $1.64\cdot 10^{1}$ & $4.76\cdot 10^{-2}$  & $4.77\cdot 10^{-2}$  & $4.77\cdot 10^{-2}$  & $4.77\cdot 10^{-2}$  & $4.77\cdot 10^{-2}$  \\
\cellcolor{lightgray!05}   7 & \cellcolor{lightgray!01}   $8.50\cdot 10^{1}$ & $2.24\cdot 10^{-2}$  & $2.25\cdot 10^{-2}$  & $2.25\cdot 10^{-2}$  & $2.25\cdot 10^{-2}$  & $2.25\cdot 10^{-2}$  \\
\cellcolor{lightgray!05}   8 & \cellcolor{lightgray!01}   $9.35\cdot 10^{2}$ & $5.82\cdot 10^{-3}$  & $5.84\cdot 10^{-3}$  & $5.85\cdot 10^{-3}$  & $5.85\cdot 10^{-3}$  & $5.85\cdot 10^{-3}$  \\
\cellcolor{lightgray!05}   9 & \cellcolor{lightgray!01}   $3.17\cdot 10^{4}$ & $6.04\cdot 10^{-4}$  & $6.07\cdot 10^{-4}$  & $6.07\cdot 10^{-4}$  & $6.07\cdot 10^{-4}$  & $6.07\cdot 10^{-4}$  \\
\cellcolor{lightgray!05}  10 & \cellcolor{lightgray!01}   $5.85\cdot 10^{6}$ & $1.80\cdot 10^{-5}$  & $1.81\cdot 10^{-5}$  & $1.81\cdot 10^{-5}$  & $1.81\cdot 10^{-5}$  & $1.81\cdot 10^{-5}$  \\
\cellcolor{lightgray!05}  11 & \cellcolor{lightgray!01}   $1\cdot 10^{10}$   & $9.41\cdot 10^{-8}$  & $9.47\cdot 10^{-8}$  & $9.49\cdot 10^{-8}$  & $9.50\cdot 10^{-8}$  & $9.50\cdot 10^{-8}$  \\
\cellcolor{lightgray!05}  12 & \cellcolor{lightgray!01}   $1\cdot 10^{10}$   & $2.10\cdot 10^{-12}$ & $2.00\cdot 10^{-12}$ & $1.96\cdot 10^{-12}$ & $1.92\cdot 10^{-12}$ & $1.95\cdot 10^{-12}$ \\
 \hline
 \rowcolor{lightgray!05}
 \multicolumn{2}{|c|}{\cellcolor{lightgray!15} Tot. linear solves} & $21$ & $20$ & $19$ & $19$ & $19$ \\
 \hline
\end{tabular}
\caption{\label{tab:BiactivityMeshIndependence}
\textbf{Experiment 1: Smooth biactive solution}.
Table of increments $\|u_h^{k} - u_h^{k-1}\|_{H^1(\Omega_\infty)}$ for various mesh sizes $h$ and polynomial orders $p$ using the triangular element $(\mathbb{P}_p\text{-bubble},\mathbb{P}_{p-1}\text{-broken})$ discretization.
The initial degrees of freedom for $u_h$ and $\psi_h$ were set to zero at the beginning of each run.
Between eight and ten Newton iterations performed by the PETSc Newton solver used for each initial subproblem solve and then only one Newton solve was used for each of the following subproblems.
The convergence of the increments for each fixed $k$ and the boundedness of the number of linear solves indicates \emph{mesh-independence}.
}
\end{table}

Our next aim is to verify the convergence orders predicted by \Cref{thm:ConvergenceContinuousLevel} and \Cref{cor:ConvergenceRates}.
In doing so, we consider the double-exponential step size rule~\cref{eq:ObstacleStepSizes_superexponential} alongside the geometric rule
\begin{equation}
\label{eq:ObstacleStepSizes_geometric}
	\alpha_k = r^{k-1}
	\,,
	\quad
	k = 1,2,\ldots,
\end{equation}
\end{subequations}
with $r = 2$, and the fixed step size rule $\alpha_k = 1$, for all $k = 1,2,\ldots$
The results in \Cref{fig:BiactiveIterationComplexity} agree precisely with the predictions made later on in \Cref{cor:ConvergenceRates}.
In particular, notice that the fixed step rule $\alpha_k = 1$ leads to \emph{sublinear} convergence, the geometric rule \cref{eq:ObstacleStepSizes_geometric} induces \emph{linear} convergence, and the double-exponential step size rule~\cref{eq:ObstacleStepSizes_superexponential} delivers \emph{superlinear} convergence.

\begin{figure}
	\centering
	\begin{tikzpicture}
		\node at (-3.35,0.5) {\includegraphics[height=1.5in]{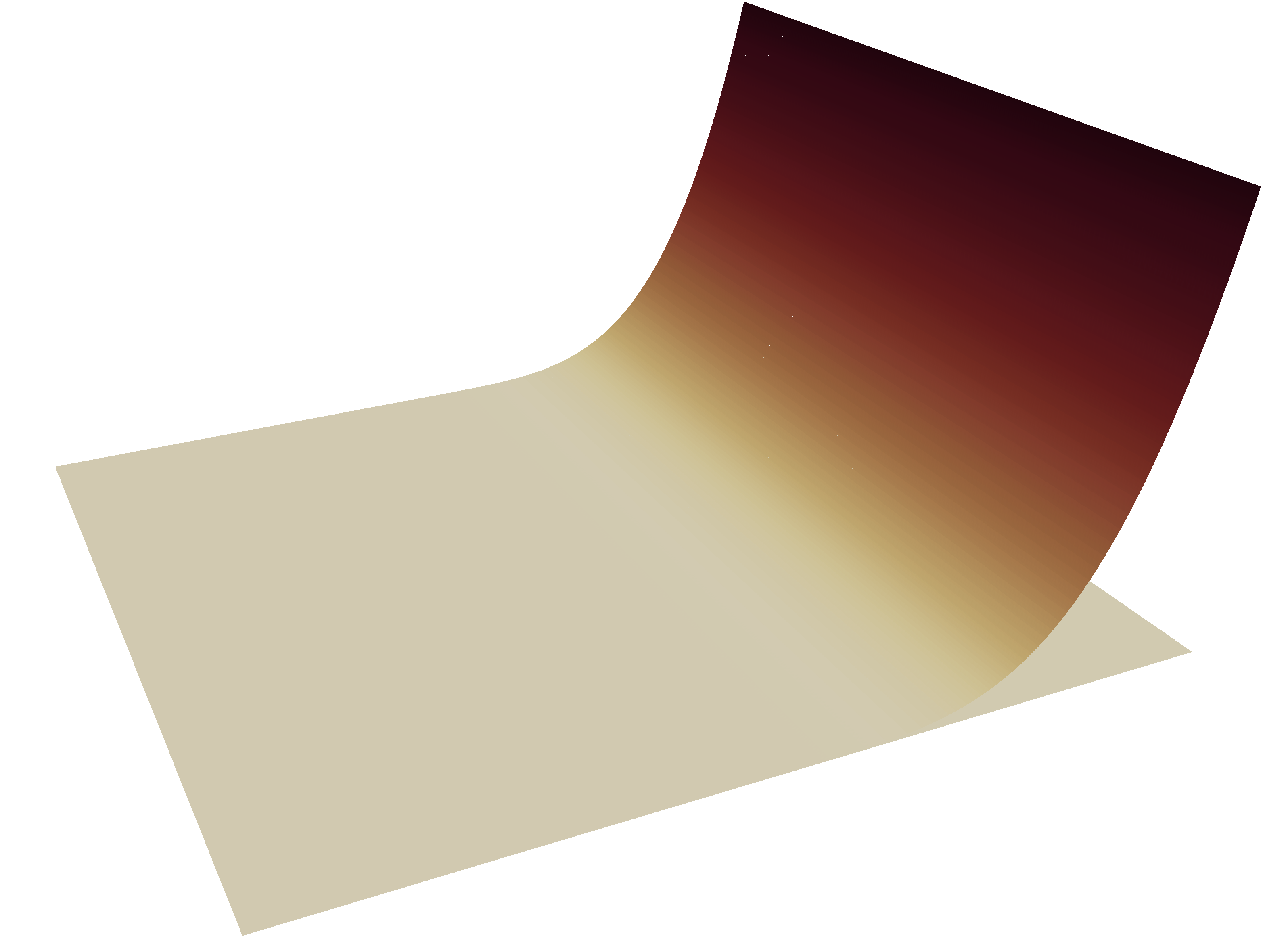}};
		\node at (3.25,0.4) {
			   \includegraphics[clip=true, trim= 0 0 0.35cm 0, height=2.1in]{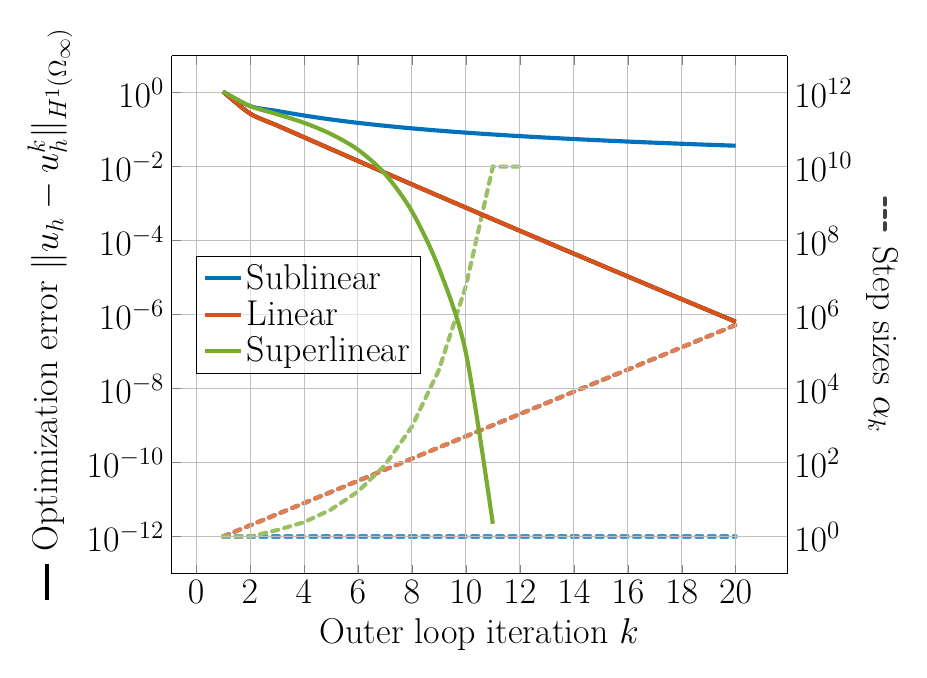}
		};
	    \node at (-0.5,0.4) {\includegraphics[height=2cm]{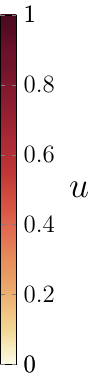}};
	    \node at (-2.2,-1.35) {\scriptsize $
	    	u(x,y)
	    	=
	    	\begin{cases}
	    		0 & \text{if~} x < 0\\
	    		x^4 & \text{otherwise}\\
	    	\end{cases}
	    $};
    \end{tikzpicture}
	\caption{
	\label{fig:BiactiveIterationComplexity}
	\textbf{Experiment 1: Smooth biactive solution}.
	Verifying the convergence orders predicted by \Cref{cor:ConvergenceRates} with the $(\mathbb{P}_1\text{-bubble},\mathbb{P}_{0}\text{-broken})$ discretization (FEniCSx).
	Left: The exact solution.
	Right: Plots of the optimization error $\|u_h - u_h^k\|_{H^1(\Omega_\infty)}$ and corresponding step size $\alpha_k$ when $h = h_\infty / 16$.
	The blue curve tracks the \emph{sublinear} convergence induced by the fixed step size rule $\alpha_k = 1$.
	Meanwhile, the step size rules \cref{eq:ObstacleStepSizes_geometric} and \cref{eq:ObstacleStepSizes_superexponential} induce \emph{linear} (red) and \emph{superlinear} (green) convergence, respectively.
	The results are similar on finer meshes due to mesh-independence; cf.~\Cref{tab:BiactivityMeshIndependence}.
	}
\end{figure}

The final aim of this experiment is to demonstrate that high-order convergence rates (with respect to the mesh size $h$) can be achieved using polynomial orders $p > 1$.
To this end, we use our high-order MFEM implementation to solve for the biactive solution~\cref{eq:Biactivity_SmoothManufacturedSolution} on the circular domain $\Omega = \Omega_2$.
In \Cref{fig:BiactivePolynomialOrderConvergenceRates}, we plot the approximation errors of the discrete solutions $u_h$, $\tilde{u}_h$, and $\lambda_h$.
From these results, we witness that high-order convergence rates can, indeed, be achieved with the proximal Galerkin method.

\begin{figure}
	\centering
	\begin{minipage}[c]{0.49\textwidth}
	\small
		\centering
				\begin{tikzpicture}
			\node at (-1.35,0.5) {\includegraphics[clip=true, trim= 14cm 0 20cm 0, height=1.7in]{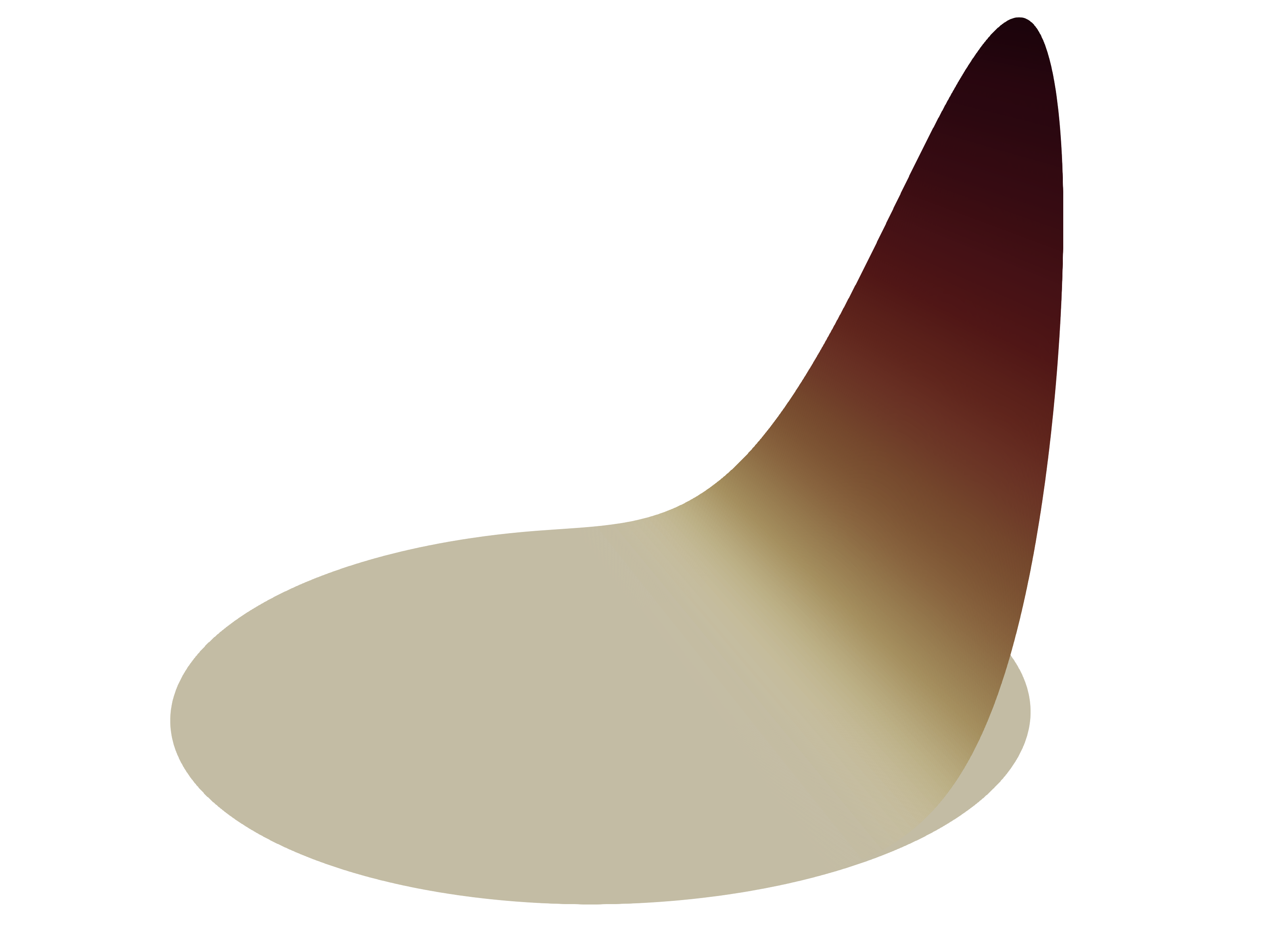}};
		    \node [left] at (1.65,0.4) {\includegraphics[height=2.2cm]{colorbar0.pdf}};
		    \node [right] at (-4.25,1.5) {\scriptsize $
		    	u(x,y)
		    	=
		    	\begin{cases}
		    		0 & \text{if~} x < 0\\
		    		x^4 & \text{otherwise}\\
		    	\end{cases}
		    $};
		    \node [right] at (-4.25,0.7) {\scriptsize $
		    	\lambda(x,y) =0
		    $};
	    \end{tikzpicture}
	\end{minipage}
	\begin{minipage}[c]{0.49\textwidth}
	\small
		\centering
		\includegraphics[width=\textwidth]{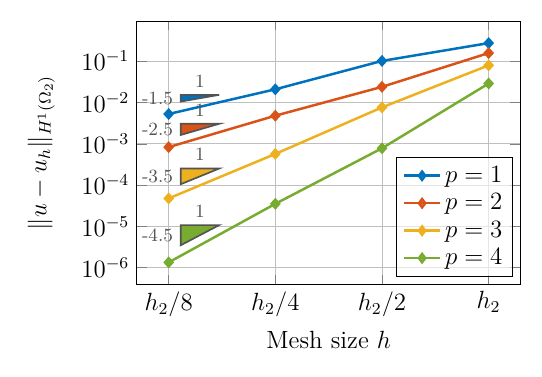}
	\end{minipage}
	\\
	\begin{minipage}[c]{0.49\textwidth}
	\small
		\centering
		\includegraphics[width=\textwidth]{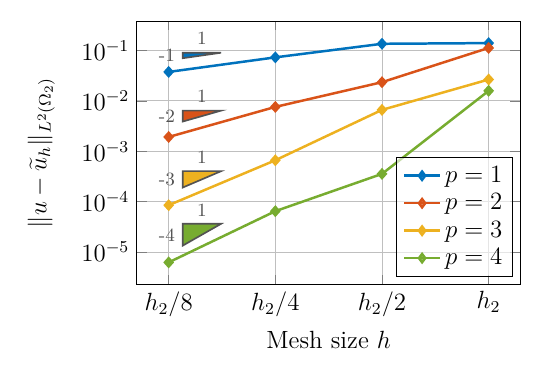}
	\end{minipage}
	\begin{minipage}[c]{0.49\textwidth}
	\small
		\centering
		\includegraphics[width=\textwidth]{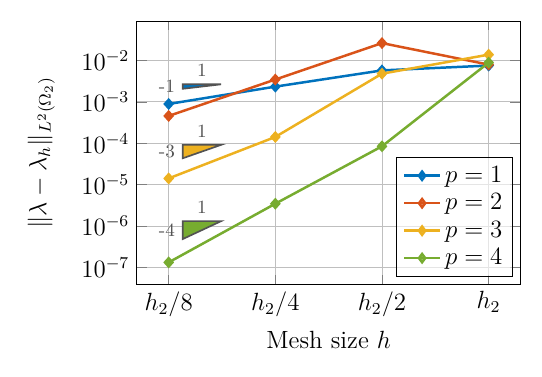}
	\end{minipage}
	\caption{
	\textbf{Experiment 1: Smooth biactive solution}.
	Verifying high-order convergence of the approximation error with various polynomial order $(\mathbb{Q}_{p+1},\mathbb{Q}_{p-1}\text{-broken})$-discretizations on $\Omega = \Omega_2$.
	Top left: The exact solution.
	Top right: The approximation error of the discrete solution $u_h$ in the $H^1$-norm.
	Here, the extra $1/2$-order of convergence is made possible from the $(p+1)$-degree approximation space.
	Bottom left: The approximation error of the discrete solution $\tilde{u}_h = \exp\psi_h$ in the $L^2$-norm.
	Bottom right: The approximation error of the Lagrange multiplier $\lambda_h$ in the $L^2$-norm.
	Notice that each of the convergence rates for this smooth solution grow with the polynomial order.
	\label{fig:BiactivePolynomialOrderConvergenceRates}}
\end{figure}

\subsubsection{Experiment 2: Strict complementarity} \label{ssub:experiment_2_kkt_conditions}

In this experiment, we set $\phi = g = 0$ and define 
\begin{equation}
\label{eq:StrictComplementarity_RHS}
	f(x,y)
	=
	2 \pi^2 \sin(\pi x)\sin(\pi y)
	\,.
\end{equation}
See~\Cref{fig:StrictComplementarityConvergenceOrder} for a fine mesh ($h = h_\infty/128$) solution $u_h$ as well as the associated Lagrange multiplier $\lambda_h$.

When viewed from the perspective of continuum mechanics, the multiplier $\lambda$ is a resolvent force. It is therefore rare that we would see biactivity of the type in the previous experiments on such large domains as it would correspond to contact without any opposing force resulting from the obstacle. 

The first aim of this experiment is to revisit the convergence orders predicted by \Cref{thm:ConvergenceContinuousLevel} and \Cref{cor:ConvergenceRates} and demonstrate that they are overly pessimistic for this more typical type of problem.
Indeed, as demonstrated in \Cref{fig:StrictComplementarityConvergenceOrder}, we see that linear convergence is achieved using only a fixed step size.
In turn, superlinear convergence can be achieved using any unbounded step size rule.
For illustration, we have added results using the geometric rule~\cref{eq:ObstacleStepSizes_geometric}
with various growth parameters $r \in \{1.05,1.1,2\}$.

\begin{figure}
	\centering
	\begin{tikzpicture}
		\node at (-3.35,1) {\includegraphics[clip=true, trim= 4cm 0 0 0, height=2.3in]{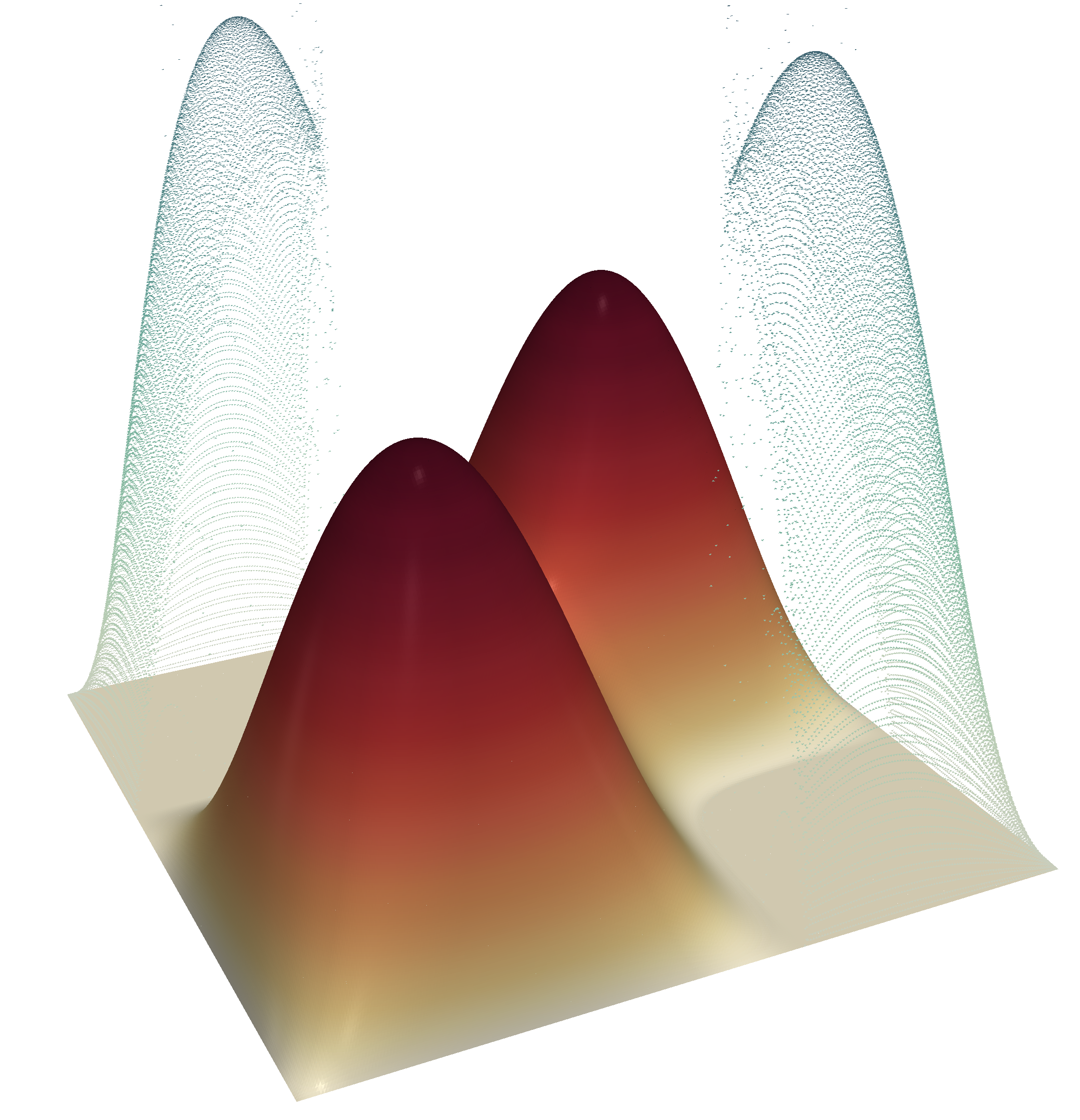}};
				\node at (3,0.4) {
		   \includegraphics[clip=true, trim= 0 0 0.35cm 0, height=2.1in]{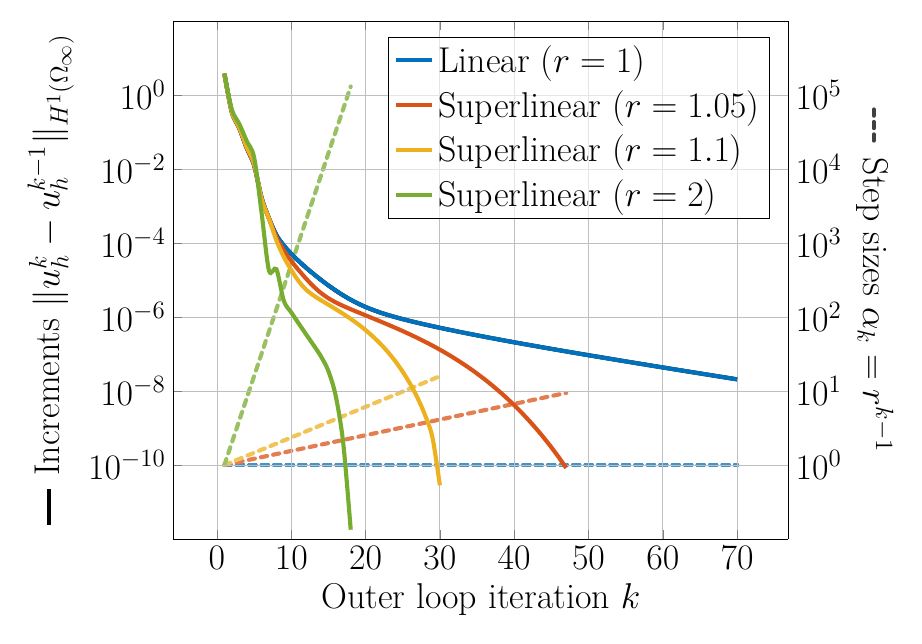}
		};
	    	    \node at (-5.8,-1.2) {\includegraphics[height=2cm]{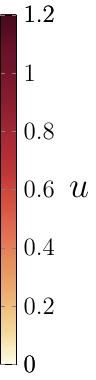}};
	    \node at (-0.95,2.2) {\includegraphics[height=2cm]{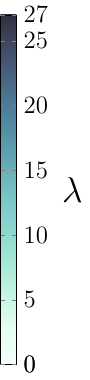}};
    \end{tikzpicture}
	\caption{
	\textbf{Experiment 2: Strict complementarity}.
	Surpassing the convergence orders predicted by \Cref{cor:ConvergenceRates} with the $(\mathbb{P}_1\text{-bubble},\mathbb{P}_{0}\text{-broken})$ discretization (FEniCSx).
	Left: A high-resolution image of the exact solution $u$ and associated Lagrange multiplier $\lambda$.
		Right: Plots of the primal variable increments $\|u_h^{k} - u_h^{k-1}\|_{H^1(\Omega_\infty)}$ and corresponding geometric step sizes $\alpha_k = r^{k-1}$ for mesh size $h = h_\infty / 16$.
			The results are similar on all finer meshes due to mesh-independence; cf.~\Cref{tab:BiactivityMeshIndependence}.
	Analogous results can also be obtained for higher-order and quadrilateral element discretizations (not shown).
	\label{fig:StrictComplementarityConvergenceOrder}}
\end{figure}

The second aim of this experiment is to check convergence of the discrete solution via the KKT conditions.
This is useful to assess \textit{a posteriori} the optimality of the discrete solution when the true solution is unknown.
To this end, we consider the complementarity condition $\big|\int_{\Omega} \lambda u \dd x\big| = 0$, the primal feasibility condition $\int_{\Omega} \max\{-u,0\} \dd x = 0$, and the dual feasibility condition $\int_{\Omega} \max\{-\lambda,0\} \dd x = 0$.
We note that the discrete solution $\tilde{u}_h = \exp\psi_h$ always satisfies the primal feasibility condition by construction; i.e., $\int_{\Omega} \max\{-\tilde{u}_h,0\} \dd x = 0$.
Therefore, in order to glean more interesting information about the proximal Galerkin solution, we focus on discrete versions of the KKT conditions formulated in terms of $u_h$.
The discrete KKT conditions that we checked are recorded in \Cref{tab:KKT}.
From the results in this table, we see that discrete primal feasibility, $\int_{\Omega} \max\{-u_h,0\} \dd x = 0$, is achieved only in the limit $h \to 0$.
However, discrete complementarity, $\big|\int_{\Omega} \lambda_h u_h \dd x\big| = 0$, and discrete dual feasibility, $\int_{\Omega} \max\{-\lambda_h,0\} \dd x = 0$, appear to hold for all mesh sizes.

\begin{table}
\centering
\small
\renewcommand{\arraystretch}{1.3}
\begin{tabular}{ |c|c|c|c| }
 \hline
    \rowcolor{lightgray!10}
  &&&
  \\
 \rowcolor{lightgray!10}
  \multirow{-2}{*}{$h$} & \multirow{-2}{3cm}{\centering {\small Complementarity} $\big|\int_{\Omega_\infty} \lambda_h u_h \dd x\big|$} & \multirow{-2}{3cm}{\centering {\small Primal feasibility} $\int_{\Omega_\infty} \max\{-u_h,0\} \dd x$} & \multirow{-2}{3cm}{\centering {\small Dual feasibility} $\int_{\Omega_\infty} \max\{-\lambda_h,0\} \dd x$}\\
   \hline
       $h_\infty$     & \multirow{6}{*}{($\text{all less than~} 10^{-14}$)} & $6.97 \cdot 10^{-3}$ & \multirow{6}{*}{($\text{all less than~} 10^{-12}$)} \\
 $h_\infty/2$   &  & $9.09 \cdot 10^{-3}$ & \\
 $h_\infty/4$   &  & $1.16 \cdot 10^{-3}$ & \\
 $h_\infty/8$   &  & $1.69 \cdot 10^{-4}$ & \\
 $h_\infty/16$  &  & $4.08 \cdot 10^{-5}$ & \\
 $h_\infty/32$  &  & $4.53 \cdot 10^{-6}$ & \\
 \hline
\end{tabular}
\caption{\label{tab:KKT}
	\textbf{Experiment 2: Strict complementarity}.
	Checking the discrete KKT conditions for the proximal Galerkin solution owing to~\cref{eq:StrictComplementarity_RHS}.
	Here, we see that primal feasibility is achieved in the limit $h\to 0$.
	Meanwhile, complementary and dual feasibility holds on all meshes.
}
\end{table}

\subsubsection{Experiment 3: Biactive solution, nonsmooth multiplier} \label{ssub:experiment_3_biactivity_revisited}

In this experiment, we set $\phi = 0$ and $g=u$, where $u(x,y)$ is the smooth manufactured solution
\begin{subequations}
\begin{equation}
	u(x,y)
	=
	\begin{cases}
		(1 - 4x^2 - 4y^2)^4 & \text{if~} x^2 + y^2 < 1/4\,,\\
		0 & \text{otherwise,}\\
	\end{cases}
\end{equation}
implied by the forcing function
\begin{equation}
	f(x,y)
	=
	-\Delta u(x,y)
	-
	\begin{cases}
		1 & \text{if~} x^2 + y^2 > 3/4\,,\\
		0 & \text{otherwise.}\\
	\end{cases}
\end{equation}
Clearly, this is another solution exhibiting biactivity.
In this case, however, the multiplier,
\begin{equation}
	\lambda(x,y) =
	\begin{cases}
		1 & \text{if~} x^2 + y^2 > 3/4\,,\\
		0 & \text{otherwise,}\\
	\end{cases}
\end{equation}
\end{subequations}
is discontinuous.
See \Cref{fig:Biactivity2} for a depiction of the exact solution $u$ as well as the associated Lagrange multiplier $\lambda$ on the domain $\Omega = \Omega_\infty$.

We use this experiment to inspect the approximation error of the $(\mathbb{P}_1\text{-bubble},\linebreak \mathbb{P}_{0}\text{-broken})$ discretization.
See \Cref{fig:Biactivity2} for our results.
As expected, unlike for the biactive solution studied in \Cref{ssub:experiment_1_biactivity}, the $L^2$-error of the Lagrange multiplier does not decay to zero linearly.
We observed no other adverse effects from this non-smooth manufactured solution.

\begin{figure}
	\centering
	\begin{tikzpicture}
	\footnotesize
		\node at (-3.5,1) {\includegraphics[clip=true, trim= 1cm 0cm 1cm 0cm, height=1.75in]{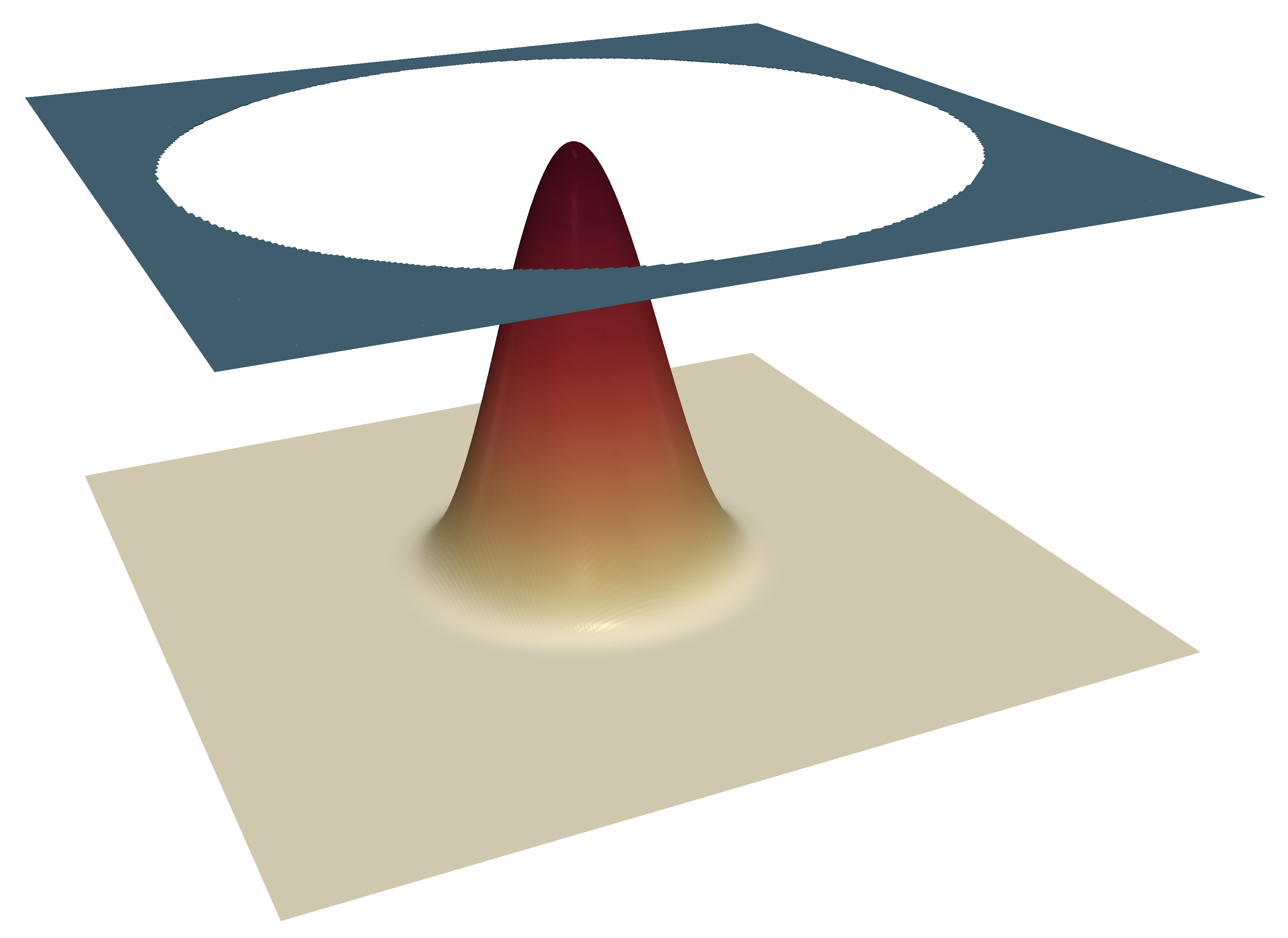}};
						\node at (3,1.1) {
			\renewcommand{\arraystretch}{1.3}
			\setlength{\tabcolsep}{4pt}
			\begin{tabular}{ |c|c|c| }
				\hhline{|>{\arrayrulecolor{white}}->{\arrayrulecolor{black}}|--|}
				\multicolumn{1}{c|}{} & \multicolumn{2}{c|}{\cellcolor{lightgray!10} \small\raisebox{5pt}{\vphantom{f}} Discretization errors}\\[3pt]
				\hline
				\rowcolor{lightgray!10}
				\vphantom{$\Big|$} $h$ & $\|u - u_h\|_{H^1(\Omega_\infty)}$ & $\|u - \tilde{u}_h\|_{L^2(\Omega_\infty)}$ \\
				\hline
				$h_\infty/16$  & $3.102\cdot 10^{-1}$ & $2.864\cdot 10^{-2}$\\
				$h_\infty/32$  & $1.566\cdot 10^{-1}$ & $1.405\cdot 10^{-2}$\\
				$h_\infty/64$  & $7.849\cdot 10^{-2}$ & $6.991\cdot 10^{-3}$\\
				$h_\infty/128$ & $3.927\cdot 10^{-2}$ & $3.491\cdot 10^{-3}$\\
				\hline
				\cellcolor{lightgray!10} Rate & $1$ & $1$ \\
				\hline
				\hline
				\rowcolor{lightgray!10}
				\vphantom{$\Big|$} $h$ & $\|u - u_h\|_{L^2(\Omega_\infty)}$ & $\|\lambda - \lambda_h\|_{L^2(\Omega_\infty)}$ \\
				\hline
				$h_\infty/16$  & $7.326\cdot 10^{-3}$ & $2.827\cdot 10^{-1}$\\
				$h_\infty/32$  & $1.875\cdot 10^{-3}$ & $1.823\cdot 10^{-1}$\\
				$h_\infty/64$  & $4.717\cdot 10^{-4}$ & $1.121\cdot 10^{-1}$\\
				$h_\infty/128$ & $1.118\cdot 10^{-4}$ & $8.646\cdot 10^{-2}$\\
				\hline
				\cellcolor{lightgray!10} Rate & $2$ & $< 1$ \\
				\hline
			\end{tabular}
		};
	    \node at (-5.6,-0.2) {\includegraphics[height=2cm]{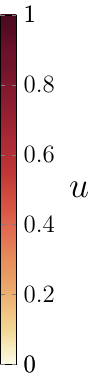}};
	    \node at (-0.7,2.5) {\includegraphics[height=2cm]{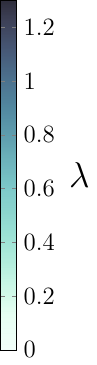}};
	    \node at (-2.4,-1.1) {\scriptsize $
	    	\lambda(x,y)
	    	=
	    	\begin{cases}
	    		1 & \text{if~} x^2+y^2 > 3/4\\
	    		0 & \text{otherwise}\\
	    	\end{cases}
	    $};
    \end{tikzpicture}
	\caption{
	\label{fig:Biactivity2}
	\textbf{Experiment 3: Biactive solution, nonmsooth multiplier}.
	Checking the discretization errors in various standard norms.
	Notice that the error in the Lagrange multiplier variable does not decay linearly with respect to the $L^2$-norm.
		}
\end{figure}

\subsubsection{Experiment 4: Spherical obstacle} \label{ssub:experiment_4_spherical_obstacle}

Our next experiment is motivated by an exact solution in \cite{gustafsson2017finite}.
Here, we set both $f = 0$ and $g = 0$ and define the obstacle to be the upper surface of a sphere of radius $1/2$, namely
\begin{equation}
\label{eq:SphericalObstacle}
	\phi(x,y) = \sqrt{ 1/4 - x^2 - y^2 }
	\,,
				\end{equation}
if $\sqrt{x^2 + y^2} \leq 1/2$, and assume that $\phi$ is sufficiently negative when $\sqrt{x^2 + y^2} > 1/2$ so that no contact happens on that subdomain.
Exploiting radial symmetry, the exact solution on the circular domain $\Omega = \Omega_2$ is found to be
\begin{equation}
	u(x,y) =
	\begin{cases}
		A \ln\sqrt{x^2 + y^2} & \text{if~} \sqrt{x^2 + y^2} > a,\\
		\phi(x,y) & \text{otherwise},\\
	\end{cases}
\end{equation}
where $a = \exp\big(W_{-1}\big(\frac{-1}{2e^2}\big)/2 + 1\big) \approx 0.34898$, $A = {\sqrt{1/4 - a^2}}/{\ln a} \approx -0.34012$, and $W_{j}(\cdot)$ is the $j$-th branch of the Lambert W‐function.

\Cref{fig:ObstacleSphere} presents the very high polynomial degree ($p=12$) proximal Galerkin solutions $u_h$ and $\tilde{u}_h$ to this problem on the coarsest mesh, $h = h_2$, which has only five elements.
We note that the solution to this problem is in $H^{5/2- \epsilon}(\Omega)$, $\epsilon > 0$, which is the highest regularity guaranteed by a smooth obstacle $\phi$ and load $f$ \cite{brezis1971nouveaux}.
Thus, we do not expect to achieve high-order accuracy with such a high degree discretization unless the mesh is adaptively refined \cite{banz2015biorthogonal}.
Developing an \textit{a posteriori} error estimator for proximal Galerkin in order to enact such adaptive mesh refinement strategies will be the subject of future work.

\Cref{tab:ObstableErrorsp1} compares the subproblem error, $\|u-u^k_h\|_{H^1(\Omega_2)}$, on a sequence of uniformly-refined meshes.
In order to include the more $h$-refinements, we consider only the lowest-degree discretization ($p=1$).
However, we find similar results with higher $p$; cf.~\Cref{tab:BiactivityMeshIndependence}.
To demonstrate that proximal Galerkin is practical without large step sizes, we fix $\alpha_k = 1$ for all iterations.
In this case, we note that the method is now expected to converge with linear complexity due to the strict complementarity of the solution; cf.~\Cref{ssub:experiment_2_kkt_conditions}.
From \Cref{tab:ObstableErrorsp1}, we see that if the number of outer iterations $k$ is fixed, then the error converges to a single value as the mesh is refined.
This is an important hallmark of mesh-independence.

\begin{figure}
	\centering
	\begin{minipage}[c]{0.475\textwidth}
	\small
		\centering
		\includegraphics[clip=true, trim= 0 1cm 0 0, height=1.25in]{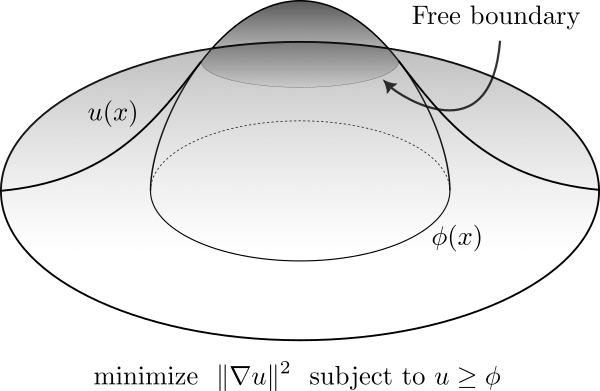}
		\\[0.4em]
		Diagram of problem
	\end{minipage}											\begin{minipage}[c]{0.475\textwidth}
	\small
		\centering
		\begin{minipage}[c]{0.85\textwidth}
		\centering
			\includegraphics[clip=true, trim = 0cm 2cm 0cm 1cm, width=0.85\linewidth]{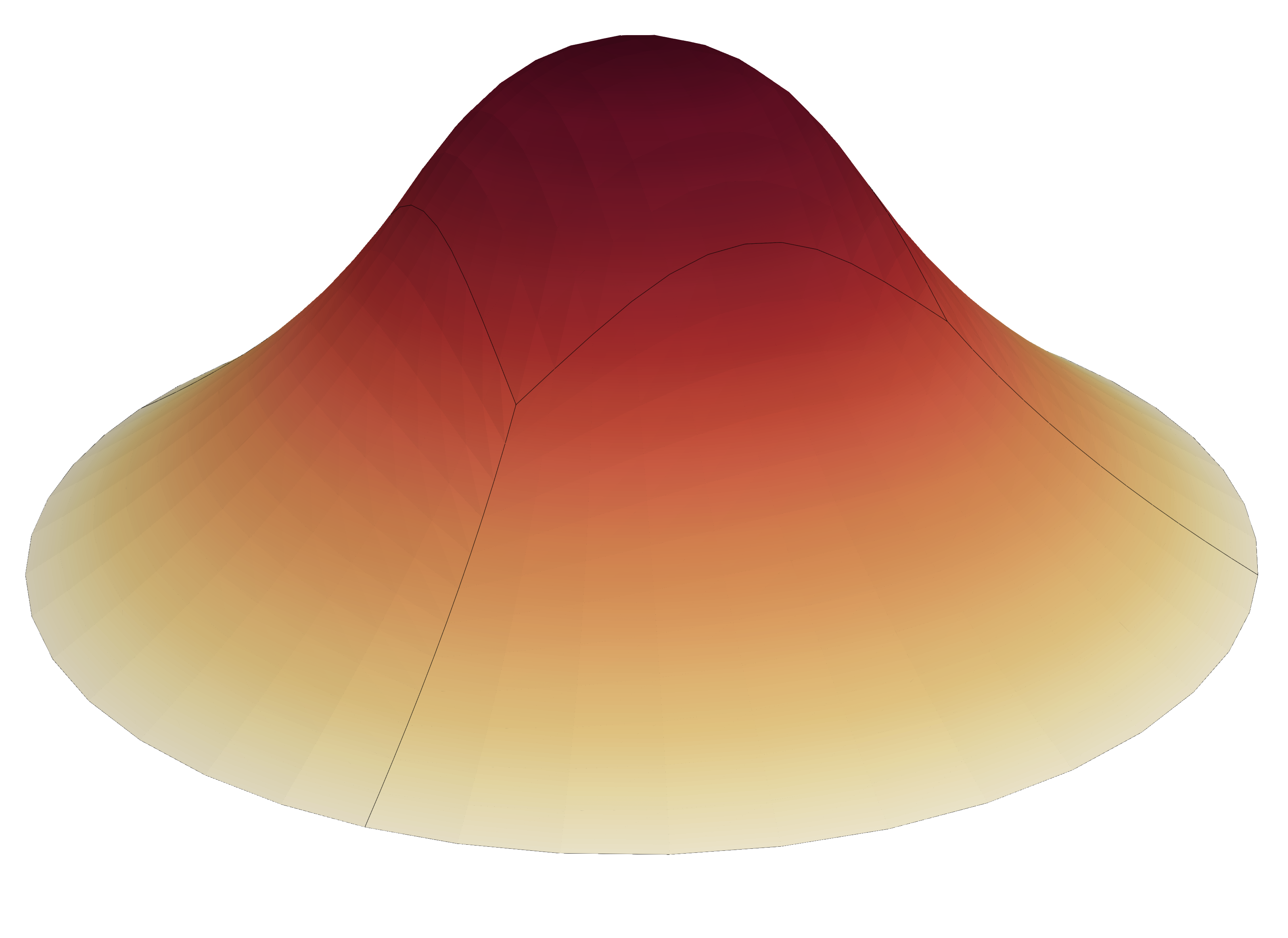}
			\includegraphics[clip=true, trim = 0cm 1cm 0cm 2cm, width=0.85\linewidth]{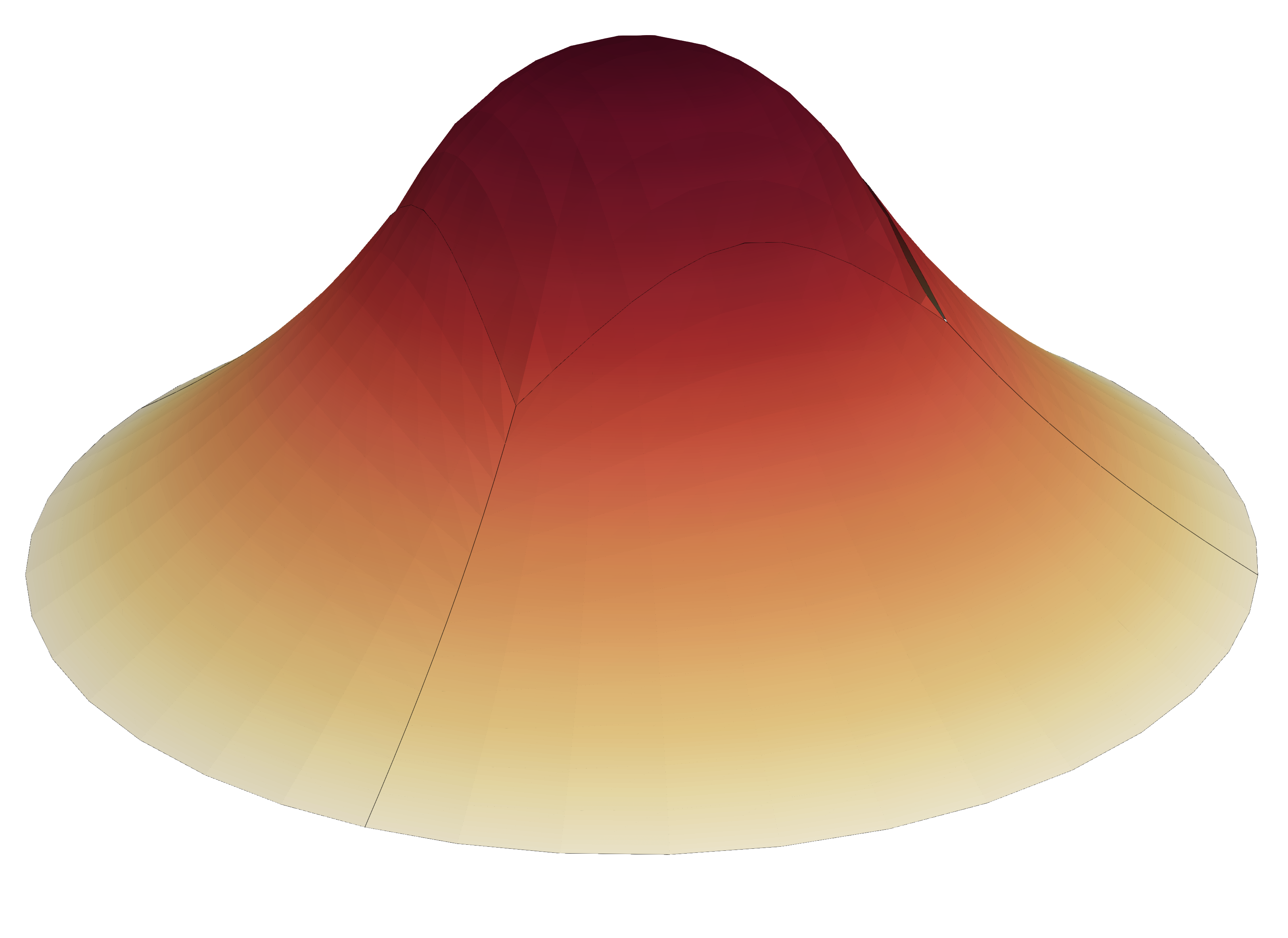}
								\end{minipage}		\begin{minipage}[c]{0.075\textwidth}
			\centering
			\includegraphics[width=\linewidth]{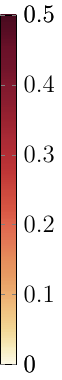}
								\end{minipage}		\\
		Very high order ($p=12$) proximal Galerkin solutions $u_h$ (top) and $\tilde{u}_h$ (bottom)
	\end{minipage}
																															\caption{
	\textbf{Experiment 4: Spherical obstacle}.
	Benchmark obstacle problem from \cite{gustafsson2017finite}.
	Left: Diagram of the problem set-up.
		Right: Five-element proximal Galerkin solutions $u_h$ and $\tilde{u}_h$.
				\label{fig:ObstacleSphere}}
\end{figure}

\begin{table}
\centering
\footnotesize
\renewcommand{\arraystretch}{1.3}
\begin{tabular}{ |c|c|c|c|c|c|c| }
  \hhline{|>{\arrayrulecolor{white}}-->{\arrayrulecolor{black}}|-----|}
 \multicolumn{2}{c|}{} & \multicolumn{5}{c|}{\cellcolor{lightgray!15} \small\raisebox{5pt}{\phantom{f}} Primal errors $\|u - u_h^{k}\|_{H^1(\Omega_2)}$ for $p=1$ \raisebox{5pt}{\phantom{f}}}\\[3pt]
 \hline
 \rowcolor{lightgray!15}
 $k$ & Linear solves & $h_2/8$ & $h_2/16$ & $h_2/32$ & $h_2/64$ & $h_2/128$ \\
 \hline
 \cellcolor{lightgray!05} 1 & \cellcolor{lightgray!01} 3 & $2.72\cdot 10^{-1}$ & $2.70\cdot 10^{-1}$ & $2.70\cdot 10^{-1}$ & $2.70\cdot 10^{-1}$ & $2.70\cdot 10^{-1}$ \\
 \cellcolor{lightgray!05} 2 & \cellcolor{lightgray!01} 1 & $1.37\cdot 10^{-1}$ & $1.38\cdot 10^{-1}$ & $1.38\cdot 10^{-1}$ & $1.38\cdot 10^{-1}$ & $1.38\cdot 10^{-1}$ \\
 \cellcolor{lightgray!05} 3 & \cellcolor{lightgray!01} 1 & $3.62\cdot 10^{-2}$ & $3.33\cdot 10^{-2}$ & $3.31\cdot 10^{-2}$ & $3.31\cdot 10^{-2}$ & $3.31\cdot 10^{-2}$ \\
  \vdots & \vdots & \vdots & \vdots & \vdots & \vdots & \vdots \\[4pt]
 \hline
 \rowcolor{lightgray!05}
 \multicolumn{2}{|c|}{\cellcolor{lightgray!15} Total iterations $k$} & $11$ & $11$ & $11$ & $11$ & $11$ \\
 \hline
 \rowcolor{lightgray!05}
 \multicolumn{2}{|c|}{\cellcolor{lightgray!15} Total linear solves} & $13$ & $13$ & $13$ & $13$ & $13$ \\
 \hline
 \rowcolor{lightgray!05}
 \multicolumn{2}{|c|}{\cellcolor{lightgray!15} Final error} & $1.98\cdot 10^{-2}$ & $8.73\cdot 10^{-3}$ & $3.49\cdot 10^{-3}$ & $1.18\cdot 10^{-3}$ & $3.85\cdot 10^{-4}$ \\
 \hline
\end{tabular}
\caption{\label{tab:ObstableErrorsp1}
	\textbf{Experiment 4: Spherical obstacle}.
	Checking the subproblem error, $\|u-u^k_h\|_{H^1(\Omega_2)}$, for various mesh sizes using the $(\mathbb{Q}_2,\mathbb{Q}_{0}\text{-broken})$ discretization.
	We used $\alpha_k = 1$ for all $k=1,\ldots$ and stopped the algorithm when $\|u_h^{k} - u_h^{k-1}\|_{L^2(\Omega_2)} < 10^{-6}$.
}
\end{table}

\subsubsection{Experiment 5: Discontinuous obstacle} \label{ssub:experiment_5_discontinuous_obstacle}

Our final experiment serves to test an important limitation of the mathematical theory presented in this section.
In particular, we explore the performance of proximal Galerkin when the regularity condition $\Delta\phi \in L^\infty(\Omega)$ (cf.~\cref{eq:ObstacleSet}) is {not} satisfied.
For simplicity, we set $p=1$.

In this experiment, we consider the \emph{discontinuous} obstacle,
\begin{equation}
\label{eq:DiscontinuousObstacle}
	\phi(x,y) =
	\begin{cases}
		1 &\text{if } x^2 + y^2 \leq 1/4\,,\\
		0 &\text{otherwise,}\\
	\end{cases}
\end{equation}
on the circular domain $\Omega = \Omega_2$, with no body force ($f = 0$), and homogeneous Dirichlet boundary conditions $u = 0$ on $\partial \Omega_2$.
In this case, one can show that the exact solution is 
\begin{equation}
\label{eq:DiscontinuousObstacle_Solution}
	u(x,y) =
	\begin{cases}
		1 &\text{if } x^2 + y^2 \leq 1/4\,,\\
		\ln(x^2 + y^2)/\ln(1/4)  &\text{otherwise.}\\
	\end{cases}
\end{equation}
To obtain an accurate discrete solution, we pre-process the mesh with adaptive refinements until the data oscillation error \cite{morin2000data} in the obstacle,
\[
	\operatorname{osc}(\phi,\mathcal{T}_h)
	:=
	\bigg(
		\sum_{T \in \mathcal{T}_h}
		\| h_T (I - \mathcal{I}) \phi \|_{L^2(T)}^2
	\bigg)^{1/2}
	\leq \mathtt{TOL}
	\,,
\]
is below $\mathtt{TOL} = 10^{-4}$.
The corresponding mesh $\mathcal{T}_h$, obstacle $\phi$, and solution $u_h$ are depicted in~\Cref{fig:DiscontinuousObstacle}.
Starting with the same initial guess as in Experiment 4, we witnessed convergence to the depicted solution in only $3$ outer iterations of~\Cref{alg:ObstacleProblem} with $4$ total linear solves (inner Newton iterations).
The corresponding $L^2$-error is $\|u - u_h\|_{L^2(\Omega_2)} = 2.16\cdot 10^{-3}$.

\begin{figure}
	\centering
	\begin{minipage}[c]{0.29\textwidth}
	\small
		\centering
		\includegraphics[height=3.1cm]{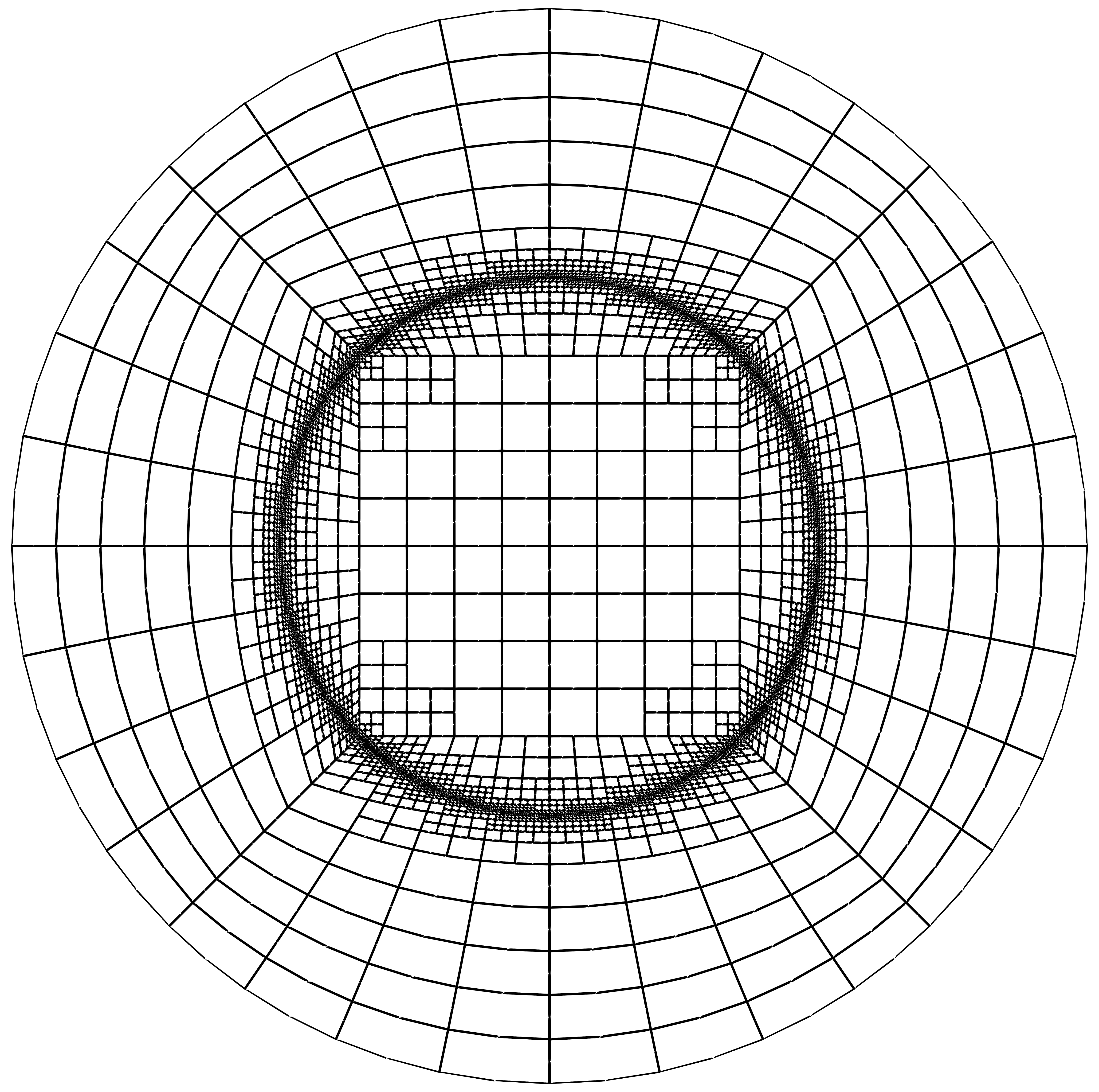}
					\end{minipage}	\begin{minipage}[c]{0.29\textwidth}
	\small
		\centering
		\includegraphics[height=3cm]{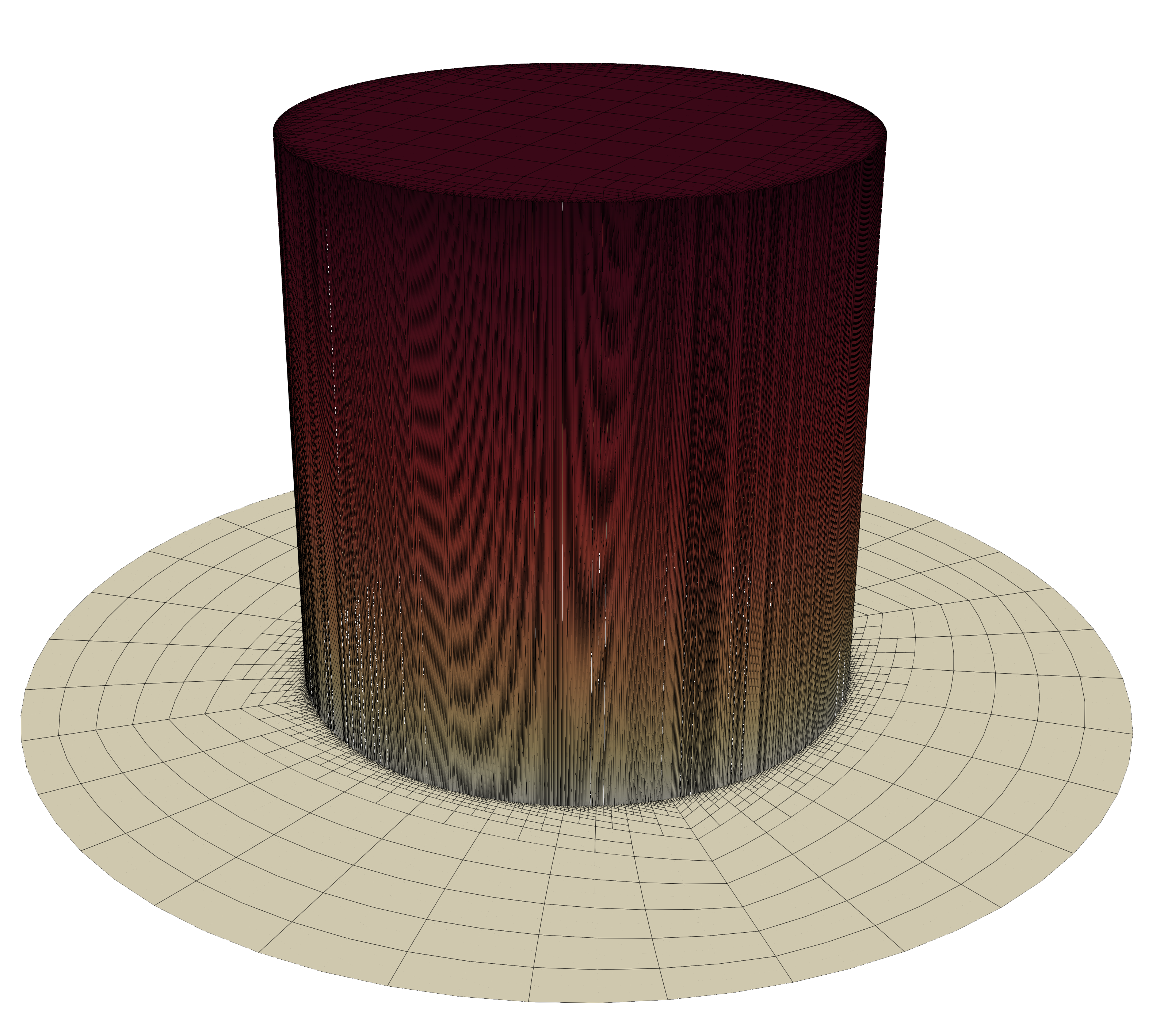}
					\end{minipage}	\begin{minipage}[c]{0.38\textwidth}
	\small
		\centering
		\begin{minipage}[c]{0.7\textwidth}
		\centering
			\includegraphics[height=3cm]{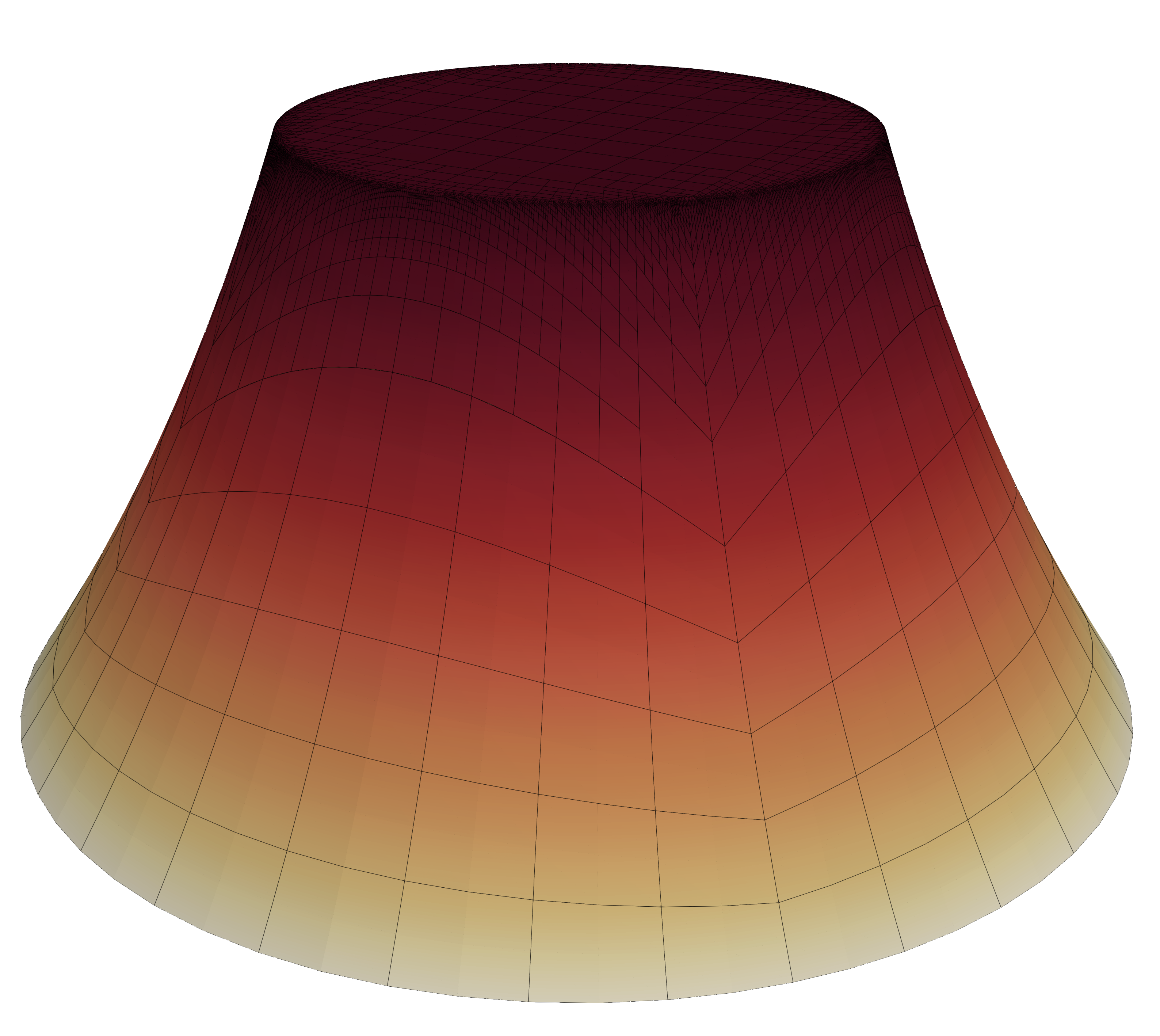}
		\end{minipage}		\begin{minipage}[c]{0.15\textwidth}
			\centering
			\includegraphics[height=2.2cm]{colorbar3.pdf}
		\end{minipage}					\end{minipage}
	\caption{
	\textbf{Experiment 5: Discontinuous obstacle}.
	Lowest-order proximal Galerkin solution from an adaptively-refined mesh and a discontinuous obstacle.
	Left: Adaptively-refined mesh $\mathcal{T}_h$.
	Middle: Discontinuous obstacle $\phi \not\in \mathcal{O}$.
	Right: Proximal Galerkin solution $u_h$.
	\label{fig:DiscontinuousObstacle}}
\end{figure}

\section{Extensions I: More general bound constraints and variational inequalities with an application to enforcing discrete maximum principles} \label{sec:maximum_principles}

The purpose of this section is to move beyond the framework developed in \Cref{sec:the_obstacle_problem} for energy principles with pointwise lower bound constraints.
To this end, we aim to answer the following two necessary questions:
\smallskip
\begin{enumerate}[leftmargin=0.75cm,itemindent=0cm,itemsep=3pt]
	\item
		Can proximal Galerkin be used to \emph{simultaneously} enforce pointwise upper and lower bound constraints?
	\item
	Can proximal Galerkin be applied to variational inequalities that \emph{do not} arise from an energy minimization principle?
\end{enumerate}
\smallskip
The answer to both of these questions is \emph{yes}.
\smallskip

We use our answers to these questions to construct a discrete maximum principle-preserving proximal Galerkin method for the advection-diffusion equation,
\begin{equation}
\label{eq:AdvectionDiffusionEquation}
	-\epsilon\Delta u + \beta\cdot\nabla u = f
	\quad \text{in~} \Omega,
	\qquad
	u = g \quad \text{on~} \partial\Omega
	\,,
\end{equation}
where $\epsilon > 0$ and $\beta \in \mathbb{R}^d$ are fixed, $f \in L^\infty(\Omega)$, and $g \in H^1(\Omega) \cap C(\overline{\Omega})$.
Along the way, we introduce the \emph{binary entropy} (\Cref{sub:binary_entropy}), a proximal algorithm for VIs with \emph{non-symmetric} coercive bilinear forms (\Cref{sub:non_symmetric_bilinear_forms}), and an alternative type of proximal Galerkin discretization employing a \emph{continuous} latent variable (\Cref{sub:StableElement_ContinuousLatentVariable}).
The section features an implementable algorithm and closes with a brief survey of numerical experiments.

\subsection{Binary entropy} \label{sub:binary_entropy}

In the previous section, enforcing a pointwise lower bound on the minimizer of the Dirichlet energy led us to consider a sequence of entropy-regularized energy minimization problems.
We now consider the situation of enforcing pointwise upper and lower bounds \emph{simultaneously}.
For simplicity, we illustrate the approach on the so-called double-obstacle problem written in~\cref{eq:DoubleObstacleProblem} below.

Let $\phi_1, \phi_2 \in H^1(\Omega)\cap L^\infty(\Omega)$ with $\esssup (\phi_2 - \phi_1)  > 0$ and $\esssup \gamma(\phi_1 - g) < 0 < \essinf \gamma(\phi_2 - g)$, and consider minimizing the Dirichlet energy under the pointwise bound constraints $\phi_1 \leq v \leq \phi_2$.
More specifically,
\begin{equation}
\label{eq:DoubleObstacleProblem}
	u^\ast = \argmin_{v\in K} E(v)
	\,,
\end{equation}
where $E(v) = \frac{1}{2} \int_\Omega |\nabla v|^2 \dd x - \int_\Omega v f \dd x$ and $K = \{ v \in H^1_g(\Omega) \mid \phi_1 \leq v \leq \phi_2 \}$.
A natural way to apply entropy regularization to~\cref{eq:DoubleObstacleProblem} is revealed if we rewrite the problem with two new variables $v_1 = v - \phi_1$ and $v_2 = \phi_2 - v$.
Doing so, we arrive at the equivalent equality-constrained optimization problem
\begin{equation}
	(u_1^\ast,u_2^\ast) = 
	\argmin_{(v_1,v_2) \in K_1\times K_2} E(v_1+\phi_1)
	~~\text{subject to~}
	v_1 + v_2 = \phi_2 - \phi_1
	\,,
\end{equation}
where $K_1 = \{ v_1 \in H^1_{g - \phi_1}(\Omega) \mid v_1 \geq 0 \}$ and $K_2 = \{ v_2 \in H^1_{\phi_2 - g}(\Omega) \mid v_2 \geq 0  \}$.
It can be readily verified that $u_1^\ast = u^\ast - \phi_1$ and $u_2^\ast = \phi_2 - u^\ast$.

Following our treatment of entropy regularization for pointwise non-negativity constraints in \Cref{sec:the_obstacle_problem}, it stands to consider the sequence $u^k = u_1^k - \phi_1 = \phi_1 - u_2^k \to u^\ast$ defined
\begin{subequations}
\begin{align}
\label{eq:DoubleObstacleEntropyRewriteA}
	(u_1^k,u_2^k) = 
	&\argmin_{(v_1,v_2) \in K_1\times K_2}
	\big\{
		E(v_1+\phi_1) + \alpha_k^{-1}\big(D(v_1, u_1^{k-1}) + D(v_2, u_2^{k-1})\big)
	\big\}
	\\
\label{eq:DoubleObstacleEntropyRewriteB}
	&\text{subject to~}
	v_1 + v_2 = \phi_2 - \phi_1
	\,.
\end{align}
\end{subequations}
We may now resubstitute $v_1 = v - \phi_1$ and $v_2 = \phi_2 - v$ into~\cref{eq:DoubleObstacleEntropyRewriteA}, which leads to
\begin{equation}
\label{eq:DoubleObstacleEntropy}
	u^k = \argmin_{v\in K} \big\{ E(v) + \alpha_k^{-1}D_B(v,u^{k-1})\big\}
	\,,
\end{equation}
where
\begin{equation}
\label{eq:RelativeBinaryEntropy}
	D_B(v,w)
	=
	\int_\Omega (v - \phi_1) \ln\Big| \frac{v - \phi_1}{w - \phi_1} \Big| + (\phi_2-v) \ln\Big| \frac{\phi_2-v}{\phi_2-w} \Big| \dd x
	,
\end{equation}
is the Bregman divergence of the (generalized) \emph{binary entropy}
\begin{equation}
\label{eq:NegativeBinaryEntropy}
	B(v)
	=
	\int_\Omega (v - \phi_1) \ln |v - \phi_1 | + (\phi_2-v) \ln |\phi_2-v| \dd x
		.
\end{equation}
A straight-forward computation shows that
\begin{equation}
	\nabla B(v) = \ln\frac{v-\phi_1}{\phi_2-v}
	\quad
	\text{and}
	\quad
	(\nabla B)^{-1}(\varphi) = \frac{\phi_1 + \phi_2\exp \varphi}{\exp \varphi + 1}
	\,.
\end{equation}

The cases $(\phi_1,\phi_2) = (0,1)$ and $(\phi_1,\phi_2) = (-1,1)$ are somewhat special for the binary entropy functional~\cref{eq:NegativeBinaryEntropy}.
In the first of these,~\cref{eq:NegativeBinaryEntropy} is usually referred to as the (negative) Fermi--Dirac or {electronic} entropy \cite{teboulle2018simplified}.
As these particular upper and lower bounds will appear prominently later on, we choose to adopt special notation for the corresponding entropy gradient and its inverse; namely,
\begin{equation}
\label{eq:linit_expit}
	\nabla B(v) = \lnit v := \ln\frac{v}{1-v}
	\quad
	\text{and}
	\quad
	(\nabla B)^{-1}(\varphi) = \sigmoid \varphi := \frac{\exp \varphi}{\exp \varphi + 1}
	\,.
\end{equation}
The second case, $(\phi_1,\phi_2) = (-1,1)$, provides a gradient that is an explicit diffeomorphism between the $L^\infty(\Omega)$-unit ball, denoted $\mathcal{B}^\infty(\Omega) = \{ v \in L^\infty(\Omega) \mid \|v\|_{L^\infty(\Omega)} < 1 \}$, and the entire Banach algebra $L^\infty(\Omega)$.
More explicitly, we write
\begin{equation}
\label{eq:arctanh_expit}
	\nabla B(v) = 2\atanh v
	\quad
	\text{and}
	\quad
	(\nabla B)^{-1}(\varphi) = \tanh (\varphi/2)
	\,,
\end{equation}
with the diffeomorphism illustrated visually in \Cref{fig:ExponentialMap2}.
For posterity, we use the latter case of the binary entropy functional to define a canonical \emph{binary-entropic Poisson equation},
\begin{equation}
	-\Delta u + \atanh u = f
	\,,
\end{equation}
which follows from writing the strong form of the first-order optimality condition for~\cref{eq:DoubleObstacleEntropy} with $\phi_1 = -1$, $\phi_2 = 1$, $\alpha_k = 1$, $u^{k-1} = 0$, and removing the factor of $2$ in~\cref{eq:arctanh_expit}.

\begin{figure}
\centering
	\includegraphics[width=0.6\linewidth]{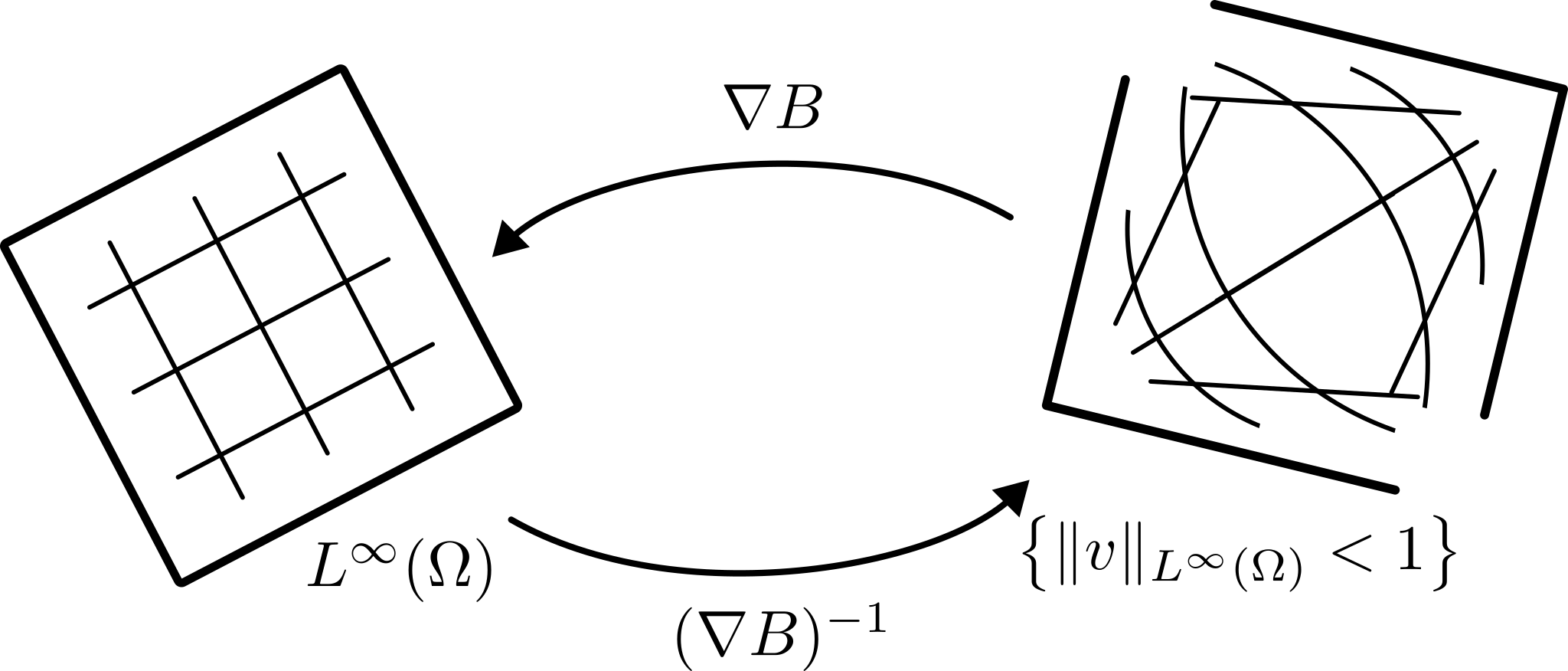}
	\caption{
	The sigmoid map $(\nabla B)^{-1}(\varphi) = \tanh (\varphi/2)$ is a diffeomorphism between the Banach algebra $L^\infty(\Omega)$ and the $L^\infty(\Omega)$-unit ball.
		\label{fig:ExponentialMap2}}
\end{figure}

\subsection{Variational inequalities with non-symmetric bilinear forms} \label{sub:non_symmetric_bilinear_forms}

In order to develop a proximal Galerkin method for the advection-diffusion equation~\cref{eq:AdvectionDiffusionEquation}, we first propose a continuous-level proximal algorithm for non-symmetric bilinear forms.
The approach is based on the proof technique for the classical theorem of Lions and Stampacchia on the existence of solutions to elliptic variational inequalities with non-symmetric bilinear forms, which can be understood as employing a ``linearized'' version of algorithm~\cref{eq:ProximalUpdate}.

Assume that $H \subset L^2(\Omega)$ is a real separable Hilbert space and $a\colon H\times H \to \mathbb{R}$ is bilinear, continuous, and coercive.
In particular, assume that there exist constants $C_a, c_a > 0$ such that
\begin{equation}
	a(w,v) \leq C_a\|w\|_H \|v\|_H
	\quad
	\text{and}
	\quad
	c_a \|v\|_H^2 \leq a(v,v)
	\,,
\end{equation}
for all $w,v \in H$. 
For any nonempty closed convex set $K \subset H$ and function $f \in L^2(\Omega)$,
the Lions--Stampacchia theorem \cite{lions1967variational} states that
the following variational inequality is well-posed:
\begin{equation}
\label{eq:GeneralVI}
	\left\{
	\begin{aligned}
		\,&\text{Find~}
		u^\ast \in K
		\text{~such that}
		\\
		& a(u^\ast, v-u^\ast) \geq ( f, v-u^\ast )
		~\fa v \in K
		\,.
	\end{aligned}
	\right.
\end{equation}
The proof works by arguing that for $\rho \in (0, 2c_{a}/C_a^2)$,
the mapping $Q_{\rho} : H \to K$, defined as the unique solution of the problem
\begin{subequations}
\begin{equation}
\label{eq:fixed_point_VI}
	\left\{
	\begin{aligned}
		\,&\text{Given $u \in H$, find}~
		w \in K
		\text{~such that}
		\\
		&(w,v-w)_{H} \ge (u,v-w)_{H} - \rho[a(u,v-w) - (f,v-w)]
		~\fa v \in K
		\,,
	\end{aligned}
	\right.
\end{equation}
is a contraction.
Note that we may equivalently write
\begin{equation}
\label{eq:fixed_point_VI_prox_form}
	Q_\rho(u)
	=
	\argmin_{v \in K}
	\Big\{
	a(u,v) - (f,v)
	+
	\frac{1}{2\rho} \|v-u\|_H^2
	\Big\}
	\,,
\end{equation}
which illustrates the following relationship to the proximal operator introduced in~\Cref{sub:proximal_point}, $Q_\rho(u) = \prox_{\rho [a(u,\cdot) - (f,\cdot)]}(u)$.
Clearly, if we find the fixed point $u = Q_{\rho}(u)$, then \cref{eq:fixed_point_VI} reduces to \cref{eq:GeneralVI} and we deduce that $u = u^\ast$.

This method of successive approximations,
\begin{equation}
	u^0 \in H,
	\quad
	u^{k+1} = Q_{\rho}(u^k)
	\,,
	\quad
	k = 0,1,\ldots
\end{equation}
\end{subequations}
is well-known and has been analyzed in, e.g., \cite{RTremolieres_JLLions_RGlowinski_2011}. 
However, it is not exactly amenable to computation because it requires a separate subproblem solver for each of the VIs \cref{eq:fixed_point_VI}.
Since $\rho$ is typically not known in practice, it is common to select diminishing step sizes $\rho_k$ to analyze the convergence behavior. 
In addition, different strategies may be employed using extrapolation steps, see related methods in, e.g., \cite{GLan_2020}. We postpone 
a deeper analysis for future work.
Given a set $K$ with appropriate structure, we can circumvent this issue using the proximal Galerkin methodology.

We begin by regularizing the continuous-level subproblems~\cref{eq:fixed_point_VI}.
A first approach would be to use the Bregman proximal point algorithm~\cref{eq:BregmanProximalUpdate} to solve the subproblems to an iteration-dependent tolerance $\mathtt{tol}_k > 0$.
This would result in a sequence of \emph{inexact} successive approximations $\| u^{k+1} - Q_\rho(u^k) \|_V \leq \mathtt{tol}_k$ that could converge to $u^\ast$ if the sequence of tolerances decays to zero as $k \to \infty$.
The potential drawback of this approach is that it creates an additional nested sequence of iterations.
In turn, generating each iterate $u^{k+1}$ may require numerous individual proximal point iterations~\cref{eq:BregmanProximalUpdate} for every inexact fixed point iteration $u^{k+1} \approx Q_\rho(u^k)$.

Instead of using the Bregman proximal point algorithm as a subproblem solver, we propose to modify the original fixed point map~\cref{eq:fixed_point_VI_prox_form} by adding an additional Bregman divergence term.
More specifically, we propose considering the alternative fixed-point iteration
\begin{subequations}
\label{eqs:Bregman_fixed_point}
\begin{equation}
	u^0 \in \dom G^\prime,
	\quad
	u^{k+1} = Q_{\rho}^{\alpha_{k+1}}(u^k)
	\,,
	\quad
	k = 0,1,\ldots
\end{equation}
where $G \colon \dom G \to \mathbb{R} \cup \{\infty\}$ is a strictly convex entropy functional associated to the feasible set $K \supset \interior \dom G$ and $Q_{\rho}^\alpha : \interior \dom G \to \interior \dom G$ is an operator formally defined for all $\rho,\alpha > 0$ as follows:
\begin{equation}
	Q_\rho^\alpha(u)
	=
	\argmin_{v \in K}
	\Big\{
	a(u,v) - (f,v)
	+
	\frac{1}{2\rho}
	\|v-u\|_H^2
	+
	\frac{1}{\alpha}
	D_G(v,u)
	\Big\}
	\,.
\end{equation}
\end{subequations}

Returning to the advection-diffusion problem~\cref{eq:AdvectionDiffusionEquation}, we now assume that $0 \leq f \leq 1$ a.e., $H = H^1(\Omega)$, and $K = \{ v \in H^1_g(\Omega) \mid 0 \leq v \leq  1 \}$, where
$g$ is such that and $\esssup \gamma(0 - g) < 0 < \essinf \gamma(1- g)$.
Instead of iterating the fixed point operator~$Q_{\rho} : H \to K$ for some $\rho > 0$, we propose to generate a sequence of iterates $\{u^k\}$ from~\cref{eqs:Bregman_fixed_point} with $G = B$ set to be the binary entropy considered in~\cref{eq:linit_expit}.
More explicitly, we choose an appropriate $\rho > 0$ and a sequence $\left\{\alpha_k\right\}$ of positive real numbers, and then solve the sequence of resulting subproblems
\begin{equation}
\label{eq:proxima_residual_VI}
	\left\{
	\begin{aligned}
		\,&\text{Given $u^{k-1}$, find}~
		u^k \in H^1_g(\Omega)\cap L^\infty(\Omega)
		\text{~such that}
		\\
		&\begin{aligned}
			\frac{\alpha_k}{\rho}(\nabla u^{k},\nabla v)_{L^2(\Omega)} +
			(\lnit(u^k),v) &=
			\frac{\alpha_k}{\rho}(\nabla u^{k-1}, \nabla v)_{L^2(\Omega)} + (\lnit(u^{k-1}),v)
			\\&\phantom{=} - \alpha_k[a(u^{k-1},v) - (f,v)]
			~~\fa v \in H^1_0(\Omega)
			\,.
		\end{aligned}
	\end{aligned}
	\right.
\end{equation}
Once discretized with a slack variable $\psi_h$, we arrive at the algorithm written below.

\begin{algorithm2e}[H]
\DontPrintSemicolon
	\caption{\label{alg:AdvectionDiffusion} 
	A Proximal Galerkin method for advection-diffusion.
	}
	\SetKwInOut{Input}{Input}
	\SetKwInOut{Output}{Output}
	\BlankLine
		\Input{Linear subspaces $V_h \subset H^1_0(\Omega)$ and $W_h \subset L^\infty(\Omega)$, initial solution guesses $u_h^0 \in V_h$, $\psi_h^0 \in W_h$, step sizes $\alpha_k>0$,
	$\rho > 0$.\;}
	\Output{Two approximate solutions, $u_{h}$ and $\widetilde{u}_h = \sigmoid\psi_h$.}
	\BlankLine
	Initialize $k = 0$.\;
	\Repeat{a convergence test is satisfied}
	{
		Solve the following (nonlinear) discrete saddle-point problem:
		\begin{gather}
		\label{eq:AdvectionDiffusionDiscreteNonlinearSaddlePoint}
			\left\{
			\begin{aligned}
				\,&\text{Find}~
				u_{h}\in g_h + V_{h} ~\text{and}~\psi_{h} \in W_{h}
				~\text{such that~}
				\\
				&\begin{alignedat}{4}
					\frac{\alpha_k}{\rho}(\nabla u_h, \nabla v) + (\psi_h, v) &= \alpha_k L(u_h^k,v) + (\psi_h^{k}, v)
					&&~\fa v \in V_h
					\,,
					\\
					(u_h, w) - (\sigmoid\psi_h, w) &= 0
					&&~\fa w \in W_h
					\,.
				\end{alignedat}
			\end{aligned}
			\right.
		\end{gather}
		where
		\begin{equation*}
			L(u,v)
			=
			(1/\rho-\epsilon)(\nabla u, \nabla v) - (\beta\cdot\nabla u - f,v)
			\,.
		\end{equation*}
		\;
		\vspace*{-\baselineskip}
		Assign $\psi^{k+1}_h \leftarrow \psi_{h}$ and $k \leftarrow k+1$.\;
	}
\end{algorithm2e}

\subsection{Stable elements II: Continuous latent variable} \label{sub:StableElement_ContinuousLatentVariable}

Constructing a stable finite element discretization for~\Cref{alg:AdvectionDiffusion} has the same challenges we witnessed in solving the obstacle problem with~\Cref{alg:ObstacleProblem}.
Namely, we must construct a \emph{stable} pair of finite element subspaces $V_h$ and $W_h$.
To this end, \Cref{sub:finite_element_subspaces} introduced a class of possible pairings based on the requirement that the latent variable $\psi_h$ be \emph{discontinuous}.
We could use the same finite elements here because the saddle-point problem~\cref{eq:AdvectionDiffusionDiscreteNonlinearSaddlePoint} has the same structure after linearization as~\cref{eq:ObstacleDiscreteNonlinearSaddlePoint}.
Instead, however, we use this subsection to introduce the following alternative class of equal-order finite element pairings where $\psi_h$ is \emph{continuous}.
\smallskip

For any integer $p \geq 1$, we define the following two pairs of spaces:
\begin{subequations}
\label{eqs:SubspacePairs_EqualOrder}
\smallskip

\noindent\textsl{Triangular elements.} We refer to the following as the $(\mathbb{P}_p,\mathbb{P}_{p})$ pairing:
\begin{equation}
\label{eq:SubspacePair1_EqualOrder}
	V_h = \mathbb{P}_{p}(\mathcal{T}_h)\cap H^1_0(\Omega)
	\,,\qquad
	W_h = V_h
	\,.
\end{equation}

\noindent\textsl{Quadrilateral elements.} We refer to the following as the $(\mathbb{Q}_p,\mathbb{Q}_p)$ pairing:
\begin{equation}
\label{eq:SubspacePair2_EqualOrder}
	V_h = \mathbb{Q}_{p}(\mathcal{T}_h)\cap H^1_0(\Omega)
	\,,\qquad
	W_h = V_h
	\,.
\end{equation}
\end{subequations}
\Cref{fig:P1P1} provides a visual representation of the lowest-order versions of these elements.

\begin{figure}
\centering
	\subcaptionbox*{$\mathbb{P}_1$}{
		\includegraphics[width=0.2\linewidth]{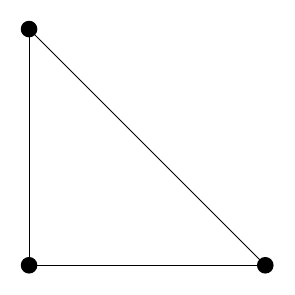}
	}
	\subcaptionbox*{$\mathbb{P}_1$}{
		\includegraphics[width=0.2\linewidth]{P1.pdf}
	}
	\qquad
	\subcaptionbox*{$\mathbb{Q}_1$}{
		\includegraphics[width=0.2\linewidth]{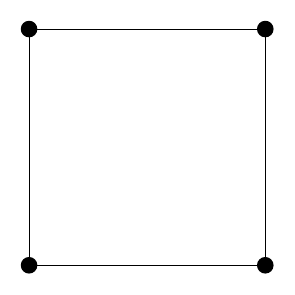}
	}
	\subcaptionbox*{$\mathbb{Q}_1$}{
		\includegraphics[width=0.2\linewidth]{Q1.pdf}
	}
	\caption{
	Conventional representation of the $(\mathbb{P}_1,\mathbb{P}_1)$ and $(\mathbb{Q}_1,\mathbb{Q}_1)$ subspace pairs in two dimensions.
		\label{fig:P1P1}}
\end{figure}

The equal-order finite element subspaces in~\cref{eqs:SubspacePairs_EqualOrder} are appealing because they can be formed from standard $C^0(\Omega)$-finite elements that are found in nearly all production codes.
However, due to limitations of the stability proof outlined in~\Cref{rem:Stability_EqualOrder}, we only advocate for using these elements when the mesh sequence $\{\mathcal{T}_h\}$ is \emph{quasi-uniform} \cite[Definition~22.20]{ern2021finite}.
\begin{definition}[Quasi-uniformity]
A shape-regular sequence of meshes is called quasi-uniform when there exists a mesh-independent constant $c > 0$ such that $h_T \geq c h$ for all $T \in \mathcal{T}_h$ and all $h> 0$ in the index set.
\end{definition}
A careful inspection of~\Cref{sec:the_entropic_finite_element_method} shows that we do not need to place such a restriction on the mesh when using the non-standard elements proposed in~\Cref{sub:finite_element_subspaces}.

In addition to the quasi-uniformity assumption above, the reader may notice that the number of degrees of freedom per element pair in~\cref{eqs:SubspacePairs_EqualOrder} is larger than that of the analogous order-$p$ pairs proposed in~\cref{eq:SubspacePairs}.
We pause to point out that neither of these factors preclude using~\cref{eqs:SubspacePairs_EqualOrder} in practical applications.
Indeed, many practical applications are solved on quasi-uniform meshes.
Moreover, it turns out that the additional computational cost can be mitigated via a mass lumping technique described in~\Cref{rem:MassLumping} below.

As  also explained in this remark, a particularly interesting consequence of mass lumping is that it induces a \emph{nodally} bound-preserving primal solution $u_h$.
Moreover, because nodal boundedness extends to pointwise boundedness when $p=1$, the primal solution $u_h$ will be \emph{pointwise} bound-preserving for any lowest-order $(\mathbb{P}_1,\mathbb{P}_1)$ or $(\mathbb{Q}_1,\mathbb{Q}_1)$ proximal Galerkin discretization with box constraints.
In contrast, the lowest-order elements in~\cref{eq:SubspacePairs} only induce a primal discretization with a bound-preserving cell average; cf.~\Cref{rem:PositiveCellAverage}.

\begin{remark}[Stability]
	\label{rem:Stability_EqualOrder}
	As described in~\Cref{sec:the_entropic_finite_element_method}, uniform stability of the proximal Galerkin discretization rests on satisfying the Ladyzhenskaya--Babu\v{s}ka--\linebreak{}Brezzi (LBB) condition
	\begin{equation}
	\label{eq:UniformStabilityConstant_Section5}
		\inf_{w \in W_h} \sup_{v \in V_h}
		\frac{(v,w)}{\|\nabla v\|_{L^2(\Omega)}\|w\|_{H^{-1}(\Omega)}}
		\geq
		\beta_0
		> 0
		\,,
	\end{equation}
	with $\beta_0$ independent of the mesh size $h>0$.
		Verifying this condition is often nontrivial.
	However, it reduces to a one-line argument given in~\cref{eq:LBB_EqualOrderProof} below if the global $L^2$-orthogonal projection $\mathcal{P}_h \colon H^1_0(\Omega) \to V_h$, defined
	\begin{subequations}
	\label{eqs:GlobalL2OrthogonalProjection}
	\begin{equation}
	\label{eq:GlobalL2OrthogonalProjection_Definition}
		(\mathcal{P}_h v , w) = (v,w)
		~~
		\fa w \in V_h
		\,,
	\end{equation}
	is stable in the (equivalent) $H^1(\Omega)$-norm; i.e., if there exists a constant $c>0$, independent of $h$, such that
	\begin{equation}
	\label{eq:GlobalL2OrthogonalProjection_H1Stability}
		\| \nabla (\mathcal{P}_h v) \|_{L^2(\Omega)}
		\leq
		c \| \nabla v \|_{L^2(\Omega)}
		\,,
	\end{equation}
	for all $v \in H^1_0(\Omega)$.
	\end{subequations}

	As shown in, e.g., \cite[Proposition~22.21]{ern2021finite}, quasi-uniformity of the mesh sequence implies~\cref{eqs:GlobalL2OrthogonalProjection} for $V_h = \mathbb{P}_{p}(\mathcal{T}_h)\cap H^1_0(\Omega)$ and $V_h = \mathbb{Q}_{p}(\mathcal{T}_h)\cap H^1_0(\Omega)$. 
	Therefore, if we assume $\{\mathcal{T}_h\}$ is quasi-uniform and we are using any of the equal-order pairs in~\cref{eqs:SubspacePairs_EqualOrder}, then $\mathcal{P}_h(H^1_0(\Omega)) \subset V_h$ and~\cref{eqs:GlobalL2OrthogonalProjection} imply that there exists $\beta_0 = 1/c > 0$ such that
	\begin{equation}
	\label{eq:LBB_EqualOrderProof}	
		\sup_{v \in V_h}
		\frac{(v,w)}{\|\nabla v\|_{L^2}}
		\geq
		\sup_{v \in H^1_0}
		\frac{(\mathcal{P}_h v,w)}{\|\nabla (\mathcal{P}_h v)\|_{L^2}}
		\geq
		\beta_0 \sup_{v \in H^1_0}
		\frac{(v,w)}{\|\nabla v\|_{L^2}}
		=
		\beta_0 \|w\|_{H^{-1}}
		\,,
	\end{equation}
	for all $w \in V_h = W_h$, as necessary.
\end{remark}

\begin{remark}[Nodal bound preservation]
\label{rem:MassLumping}
	We choose to focus this remark on the saddle-point problem in~\Cref{alg:AdvectionDiffusion}.
	Yet, similar conclusions could be drawn about~\Cref{alg:ObstacleProblem} and, potentially, other future proximal Galerkin algorithms for second-order elliptic VIs.

	Consider the equal-order finite element subspaces in~\cref{eqs:SubspacePairs_EqualOrder} and let $\{\varphi_i\}_{i=1}^N$ be a basis for $V_h = W_h$.
	It follows that there exist coefficients $\mathsf{c}_j,$ and $\mathsf{d}_j$, $j=1,2,\ldots,N$, such that $u_h(x) = \sum_{j=1}^N \mathsf{c}_j \varphi_j(x)$ and $\psi_h(x) = \sum_{j=1}^N \mathsf{d}_j \varphi_j(x)$.
	Substituting these expressions into the second variational equation in~\cref{eq:AdvectionDiffusionDiscreteNonlinearSaddlePoint}, setting $\varphi = \varphi_i$, and replacing the Lebesgue integral $\int_\Omega \varphi(x) \dd x$ with a global quadrature rule $\sum_{l=1}^M w_l \varphi(x_l)$, where $w_l \neq 0$ and $x_l \in \overline{\Omega}$ for $l=1,2,\ldots,M$, we find that
					\begin{equation}
	\label{eq:NodalFeasibilityQuadrature}
		\sum_{l=1}^M \sum_{j=1}^N
		w_l \mathsf{c}_j \varphi_j(x_l)\varphi_i(x_l)
		=
		\sum_{l=1}^M
		w_l \sigmoid\Biggl(\;\sum_{j=1}^N \mathsf{d}_j\varphi_j(x_l)\Bigg)\varphi_i(x_l)
			\end{equation}
	for each index $i = 1,\ldots, N$.
												
	We now employ the nodal-quadrature mass lumping technique \cite{fried1975finite} that is commonly used in, e.g., spectral element methods \cite{duczek2019mass}.
	Although this quadrature technique is applied to the linearized form of~\cref{eq:NodalFeasibilityQuadrature}, the effect on the solution is equivalent to modifying the nonlinear equation directly, as described below.
	Assume $M=N$ and that $\varphi_j$ are formed from a nodal basis with nodes corresponding to the quadrature points $x_j$ (e.g., a Lagrange basis with Gauss--Lobatto nodes \cite{pozrikidis2005introduction}).
	Thus, $\varphi_j(x_l) = \delta_{jl}$ for all $j,l = 1,\ldots, N$.
								Moreover, we find that~\cref{eq:NodalFeasibilityQuadrature} reduces to
	\begin{equation}
						\mathsf{c}_i
		=
		\sigmoid(\mathsf{d}_i)
						\,.
	\end{equation}
	Finally, notice that $\mathsf{c}_i = u_h(x_i)$ and $\mathsf{d}_i = \psi_h(x_i)$ since we have assumed the basis $\{\varphi_i\}_{i=1}^N$ is nodal.
	Thus, the primal variable $u_h$ is nodally bound preserving; i.e., $0 \leq u_h(x_i) \leq 1$ at all points $x_i$, $i = 1,\ldots, N$.

\end{remark}

\subsection{Numerical experiments} \label{sub:max_ppl_numerical_experiments}

In this set of experiments, we follow \cite[Section~4.1]{chan2014robust} and consider the exact solution of a model problem attributed to Eriksson and Johnson \cite{eriksson1993adaptive}.
In particular, we set $\Omega = (0,1)^2$, $f = 0$, and $\beta = (1,0)^\top$ in~\cref{eq:AdvectionDiffusionEquation} and, therefore, write
\begin{equation}
\label{eq:EJProblem}
	-\epsilon\bigg(\frac{\partial^2 u}{\partial x^2} + \frac{\partial^2 u}{\partial y^2}\bigg) + \frac{\partial u}{\partial x} = 0
	\quad \text{in~} \Omega,
	\qquad
	u = g \quad \text{on~} \partial\Omega
	\,.
\end{equation}
Exact solutions of this problem for arbitrary boundary data $g$ can be derived using the separation of variables technique.
We choose to isolate the solutions satisfying $u = 0$ where $x = 1$ and $\nabla u \cdot n = 0$ where $y = 0$ and $1$.
Doing so generates the following series expansion:
\begin{equation}
\label{eq:SeriesExpansion}
	u =
	\sum_{n=1}^\infty C_n
	\frac{\exp(r_2(x-1)) - \exp(r_1(x-1))}{r_1\exp(-r_2) - r_2\exp(-r_1)} \cos(n\pi y)
	\,,
\end{equation}
where $r_{1,2} = \frac{1\pm \sqrt{1+4\epsilon \lambda_n}}{2\epsilon}$ and $\lambda_n = n^2\pi^2\epsilon$.
Note that the constants $C_n$ can be determined from the prescribed values of $g$ on $\{(x,y) \in \overline{\Omega} \mid x = 0\} \subset \partial \Omega$.
Since we have not yet prescribed $g$ on this part of the boundary, we define $g$ there to be~\cref{eq:SeriesExpansion} with $C_1 = 1$ and $C_{n\neq 1} = 0$; i.e, we treat the first mode in this series as a manufactured solution for~\cref{eq:EJProblem}.

\Cref{fig:EJProblem} places the standard FEM solution of this problem for $\epsilon = 10^{-2}$ alongside the corresponding maximum principle-preserving proximal Galerkin solution $\widetilde{u}_h$.
Here, the proximal Galerkin solution can be found by running only two iterations of \Cref{alg:AdvectionDiffusion} with constant values for $\rho$ and $\alpha_k$, and using the standard first-order FEM solution to provide an initial guess for $u^0_h$ and $\psi^0_h$.
Note that the two discrete solutions are similar, except the proximal Galerkin solution preserves the maximum principle $0 \leq u \leq 1$.
We also experimented with using the mass lumping technique described in~\Cref{sub:StableElement_ContinuousLatentVariable}.
The results of this experiment are given in \Cref{fig:EJProblem2}.
Here, note that both the the primal solution $u_h$ and the latent variable solution $\tilde{u}_h$ are bound-preserving in this case.
For further details, or to reproduce the our experiments, the interested reader is directed to our open-source FEniCSx and MFEM implementations found at \cite{ZenodoCode}.

\begin{figure}
	\centering
	\begin{minipage}{0.3\textwidth}
	\small
		\centering
		\includegraphics[width=\linewidth]{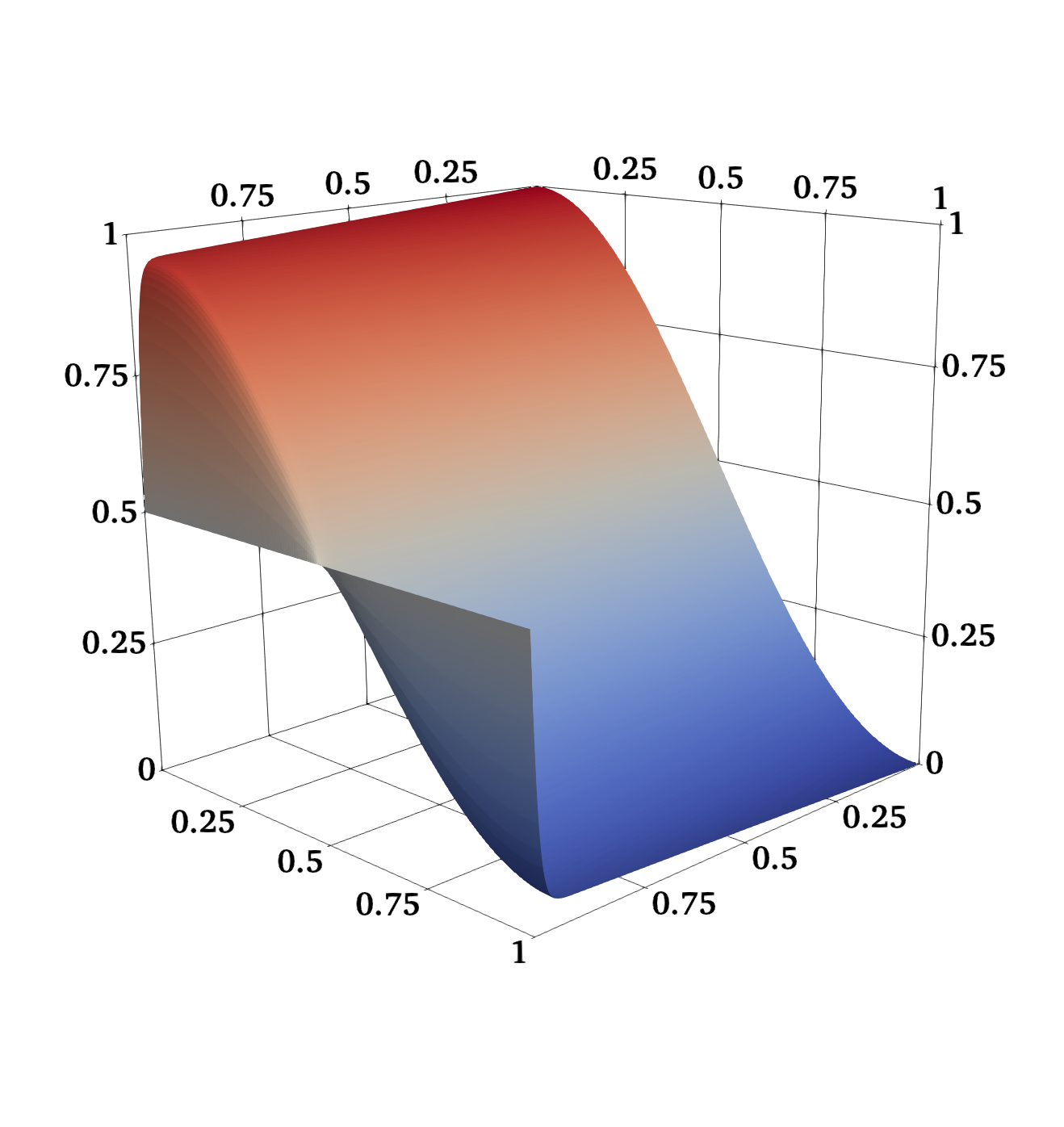}
		\\
		Exact solution $u$
	\end{minipage}
	\
	\begin{minipage}{0.03\textwidth}
	\small
		\centering
		\includegraphics[width=\linewidth]{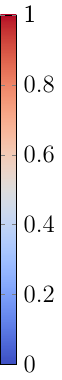}
		\\[1.5em]
		~
	\end{minipage}%
	~~~
	\begin{minipage}{0.3\textwidth}
	\small
		\centering
		\includegraphics[width=\linewidth]{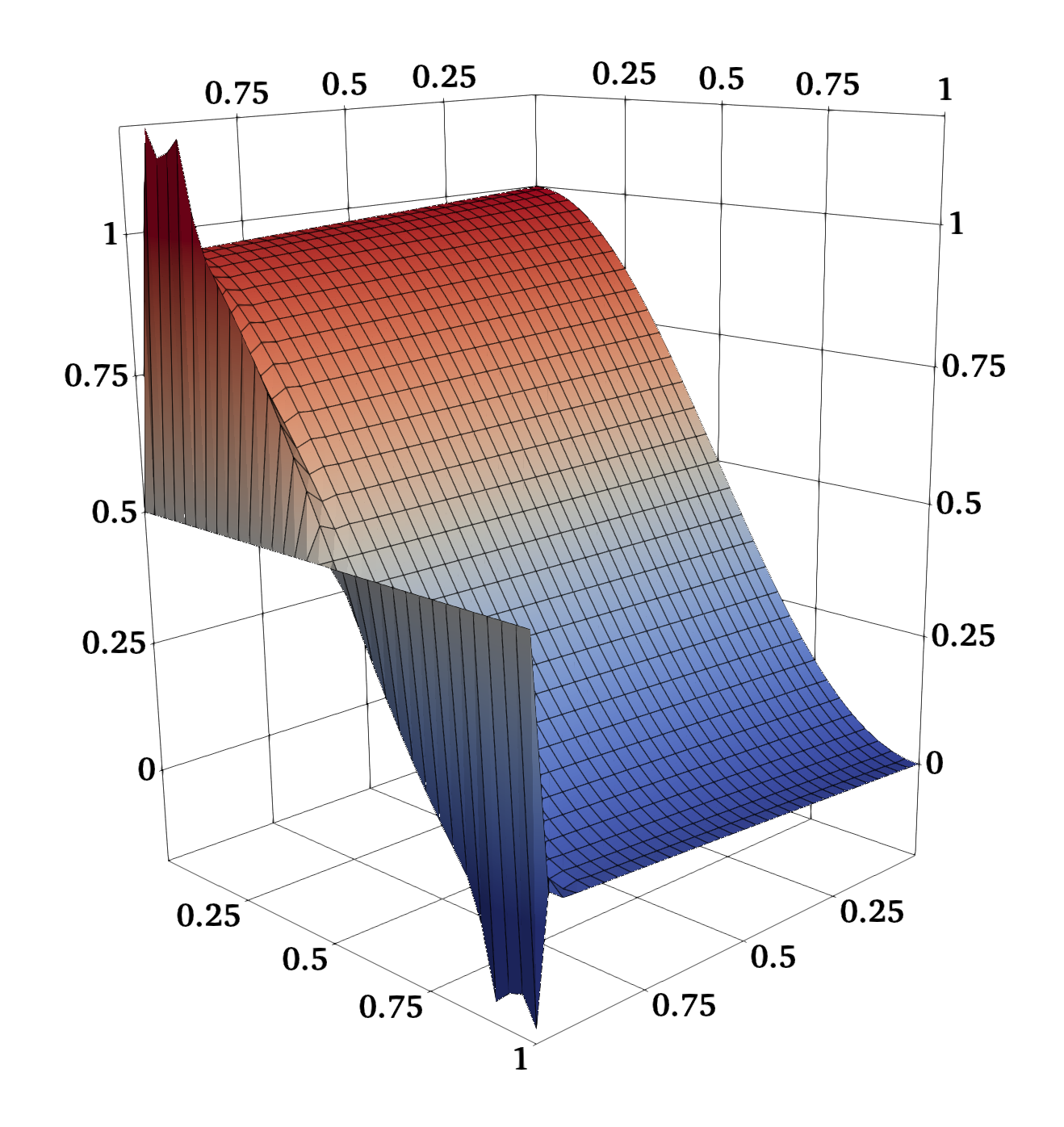}
		\\
		FEM solution
	\end{minipage}%
	~
	\begin{minipage}{0.3\textwidth}
	\small
		\centering
		\includegraphics[width=\linewidth]{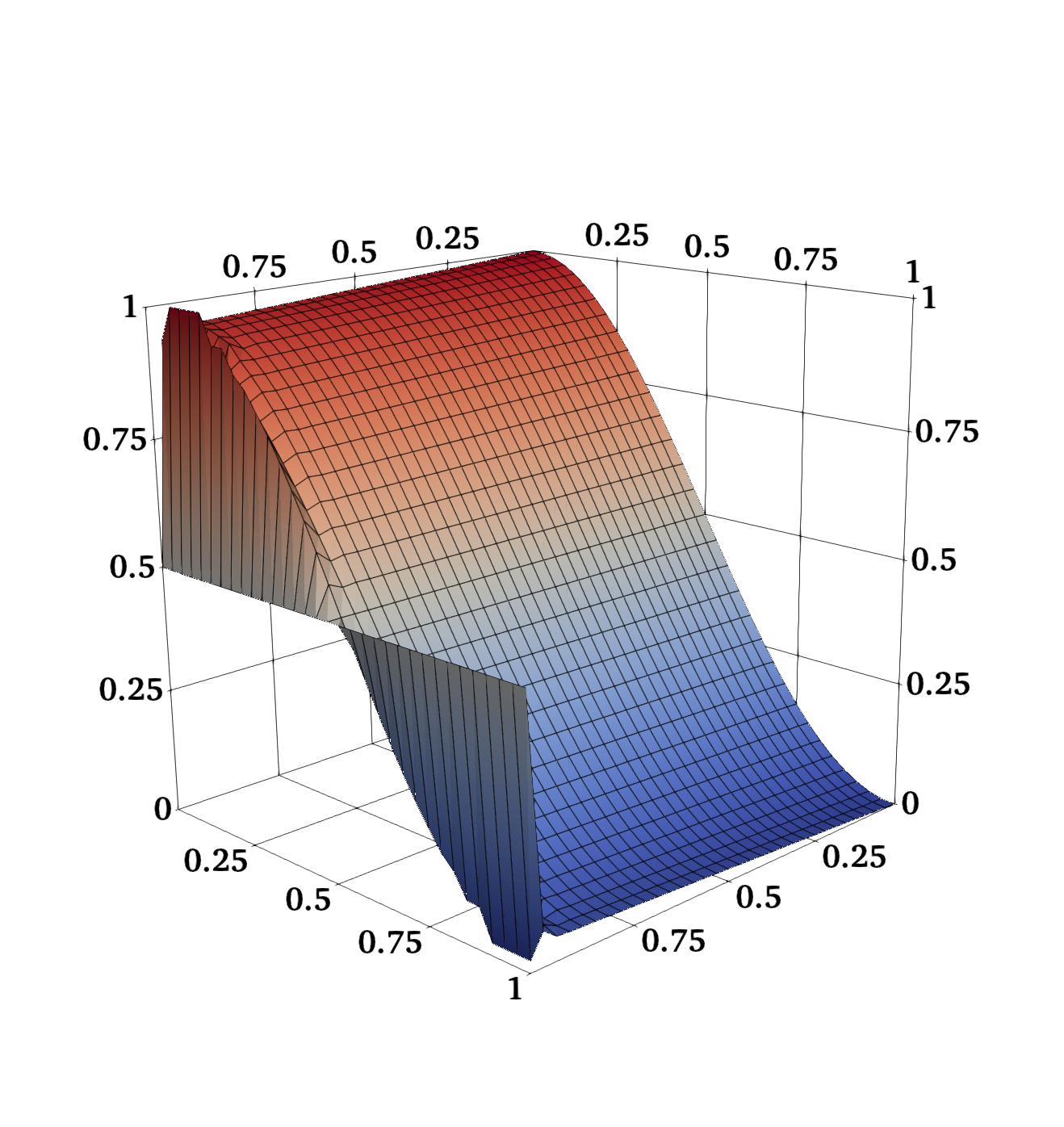}
		\\
		Proximal Galerkin solution $\tilde{u}_h$
	\end{minipage}	\caption{
	The Erikkson--Johnson problem~\cref{eq:EJProblem} for $\epsilon = 10^{-2}$.
	Left: The exact solution.
	Middle: A first-order Bubnov--Galerkin numerical solution that clearly violates the strong maximum principle $0 \leq u(x) \leq 1$.
	Right: The corresponding $(\mathbb{Q}_1,\mathbb{Q}_1)$-proximal Galerkin solution $\tilde{u}_h = \sigmoid(\psi_h)$ satisfies the strong maximum principle, by construction.
		These results can be reproduced by running the MFEM code \texttt{advection\_diffusion.cpp} available at \cite{ZenodoCode}.
		\label{fig:EJProblem}}
\end{figure}

\begin{figure}
	\centering
	\begin{minipage}{0.22\textwidth}
	\flushright
		Without lumping:
	\end{minipage}
	~
	\begin{minipage}{0.3\textwidth}
	\small
		\centering
		\includegraphics[clip, trim = 0 2cm 0 0, width=\linewidth]{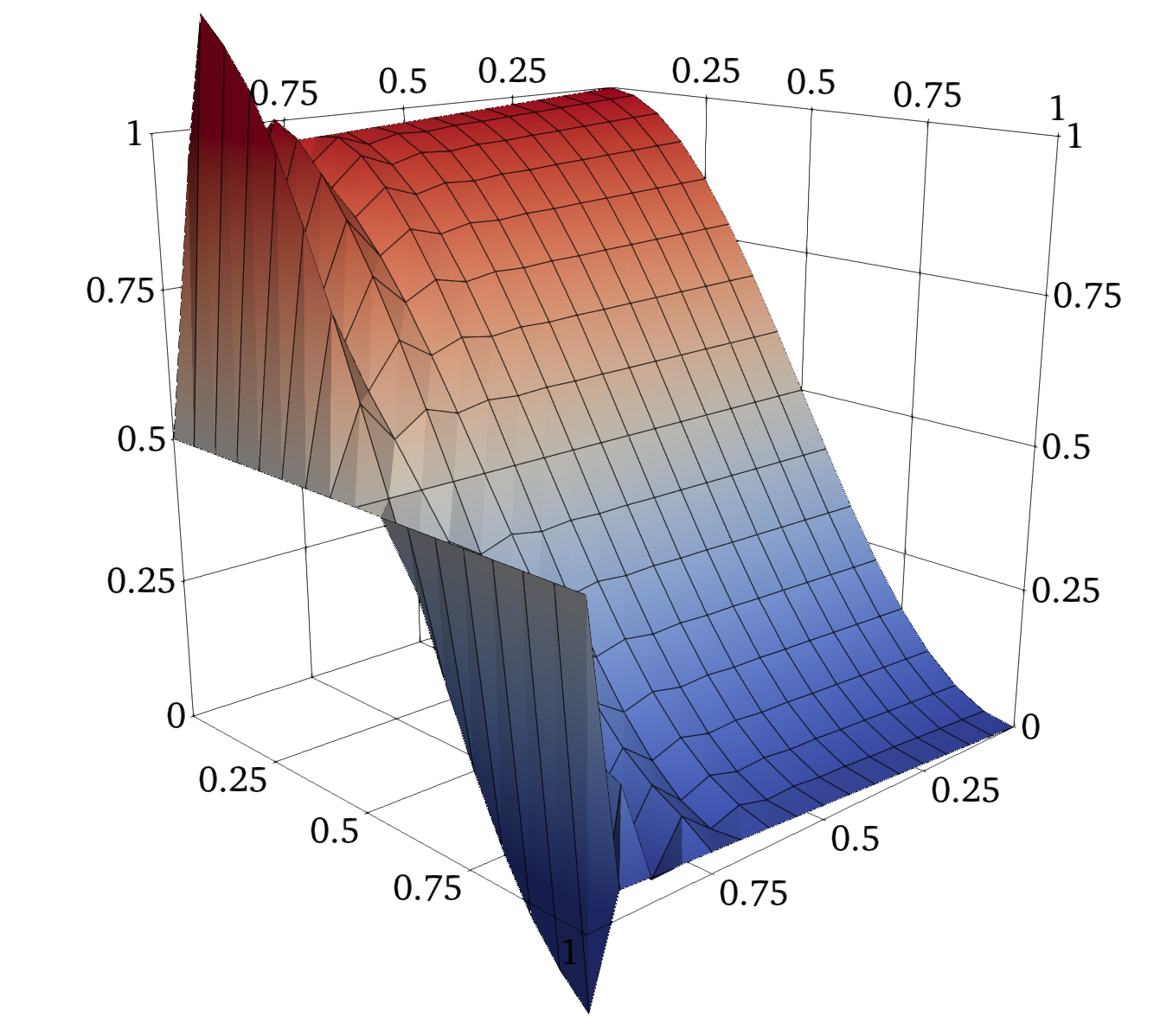}
		\\
		$u_h$
	\end{minipage}	~
	\begin{minipage}{0.3\textwidth}
	\small
		\centering
				\includegraphics[clip, trim = 0 2cm 0 0, width=\linewidth]{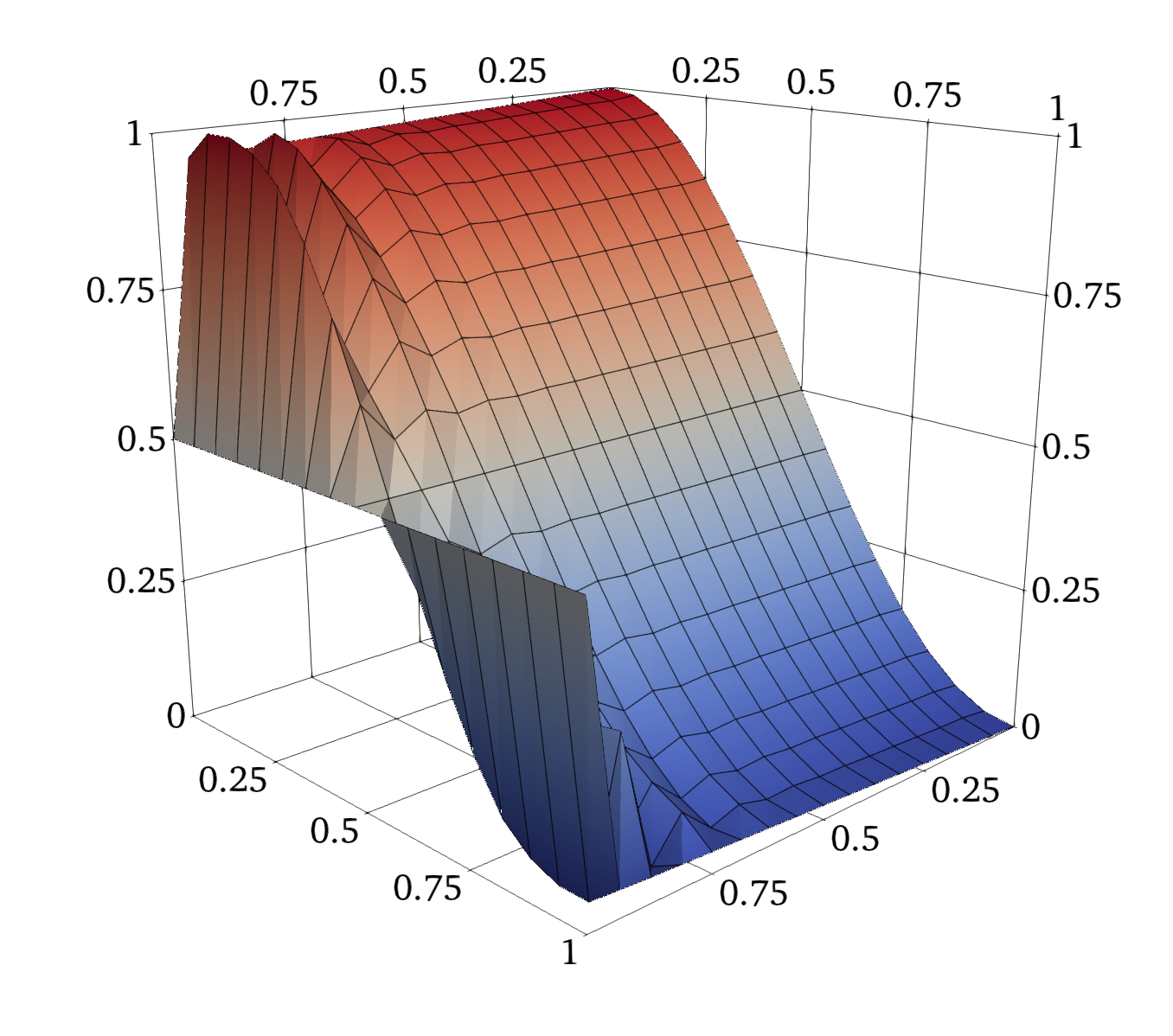}
		\\
		$\tilde{u}_h$
	\end{minipage}	\\[1em]
	\begin{minipage}{0.25\textwidth}
	\flushright
		With lumping:
	\end{minipage}
	~
	\begin{minipage}{0.3\textwidth}
	\small
		\centering
				\includegraphics[clip, trim = 0 3cm 0 0, width=\linewidth]{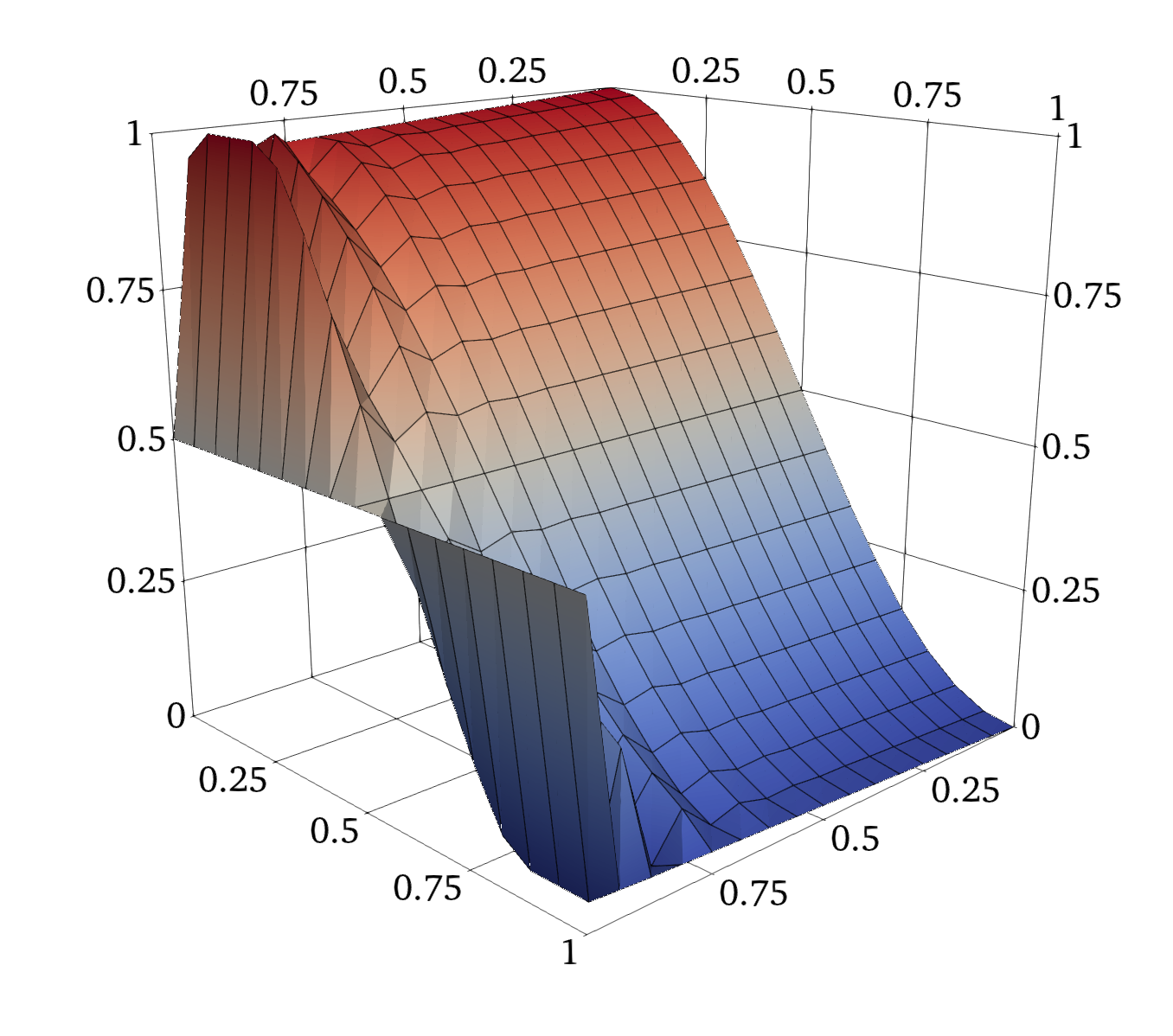}
		\\
		$u_h$
	\end{minipage}	~
	\begin{minipage}{0.3\textwidth}
	\small
		\centering
				\includegraphics[clip, trim = 0 3cm 0 0, width=\linewidth]{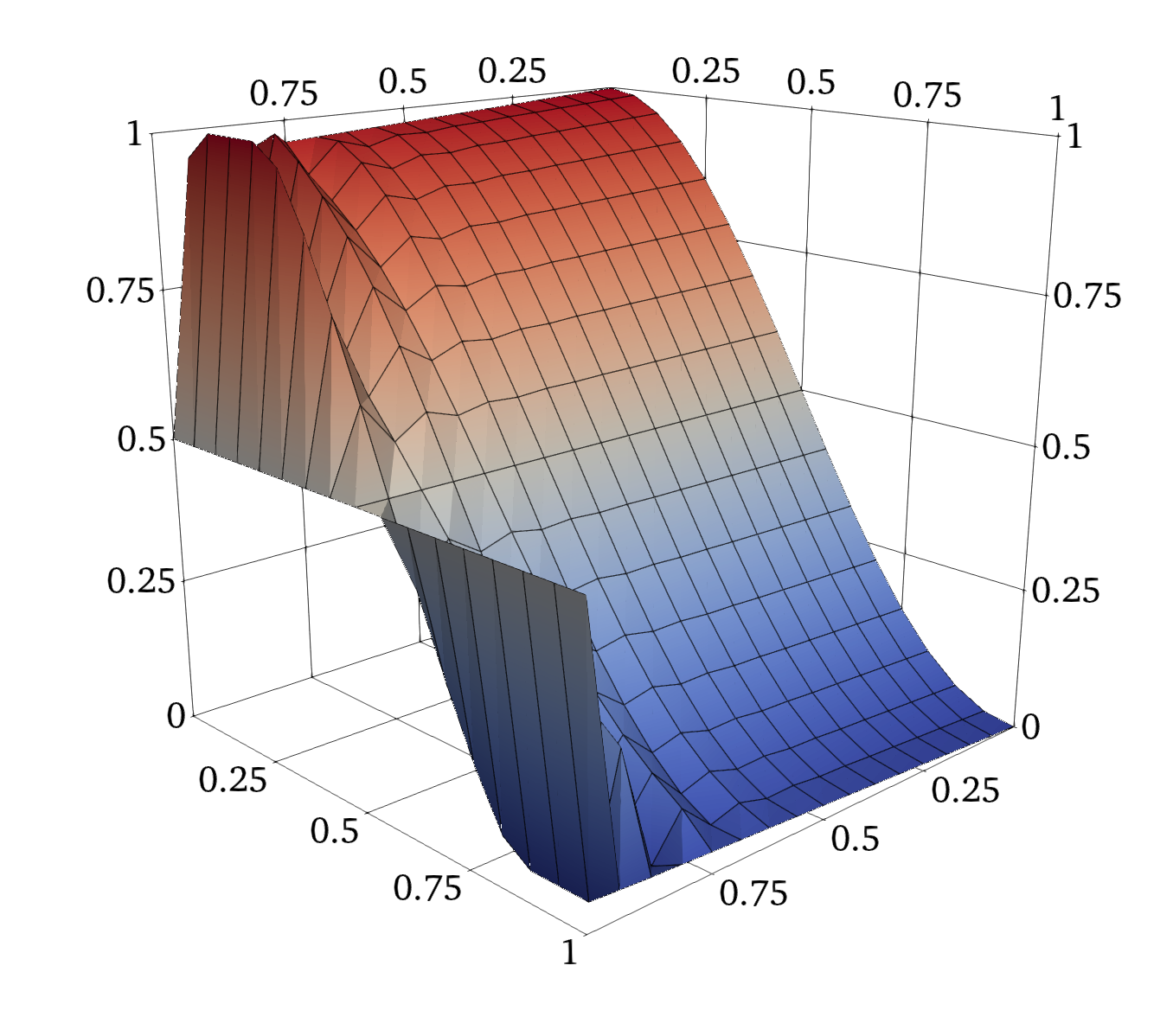}
		\\
		$\tilde{u}_h$
	\end{minipage}	\caption{
	Comparison of two $(\mathbb{Q}_1,\mathbb{Q}_1)$-proximal Galerkin discretizations of the Erikkson--Johnson problem~\cref{eq:EJProblem} also considered in~\cref{fig:EJProblem}.
	In the first row, we see a pair of solutions corresponding to a proximal Galerkin discretization where standard Gaussian quadrature is used to evaluate every integral.
	In the second row, we see a similar pair of solutions obtained from a discretization employing the nodal-quadrature mass lumping technique described in~\Cref{rem:MassLumping}.
	As argued above, the latter discretization delivers \emph{two} feasible discrete solutions.
	These results can be reproduced by running the FEniCSx code \texttt{advection\_diffusion.py} available at \cite{ZenodoCode}.
	\label{fig:EJProblem2}}
\end{figure}

Note that neither of the numerical approximations depicted in~\Cref{fig:EJProblem,fig:EJProblem2} were obtained with numerical stabilization techniques that are common for this class of singularly-perturbed problems and usually required to avoid spurious oscillations for smaller values of $\epsilon > 0$.
Just as conventional stabilized finite element methods do not necessarily preserve maximum principles \cite{evans2009enforcement}, we find that entropy regularization does not necessarily induce robustness with respect to the diffusion parameter $\epsilon$.
In turn, we observe spurious oscillations in the discrete solution for smaller values of $\epsilon$.
Our conclusion is, therefore, that future work is required to develop robust proximal Galerkin finite element methods for singularly-perturbed PDEs.

\section[Extensions II: Non-convex objective functions and a structure-preserving algorithm for topology optimization]{Extensions II: Non-convex objective functions and a structure-\linebreak{}preserving algorithm for topology optimization}
\label{sec:new_algorithms_for_topology_optimization}
The variational problems considered in the sections above share several common features. The most decisive feature is convexity. This raises the question as to whether entropy regularization can be as effective in a non-convex infinite-dimensional setting. We investigate this possibility here by providing a new proximal gradient (entropic mirror descent) framework for possibly non-convex, bounded-constrained optimization in infinite dimensions. Our benchmark problem for this setting is a well-known problem in topology optimization.  
As before, the section closes with an explicit algorithm and a brief account of numerical experiments. In the interest of completeness, we recall several details from abstract mirror descent methods. Although these methods are widely used in finite-dimensional convex optimization, and much of our treatment is inspired by the more recent works \cite{beck2003mirror,teboulle2018simplified}, it is important to note that Nemirovskij and Yudin did not restrict themselves to finite dimensions in their original works many decades ago \cite{nemirovski1979effective,nemirovskij1983problem}.

\subsection{Mirror descent} \label{sub:entropic_mirror_descent}

\Cref{sub:proximal_point} introduced a proximal framework that was applied to solve the obstacle problem.
\Cref{sub:non_symmetric_bilinear_forms} introduced a linearized proximal framework to solve variational inequalities with non-symmetric bilinear forms.
The purpose of the present subsection is to combine those two approaches into a general first-order framework for non-convex optimization problems,
\begin{equation}
\label{eq:NonConvexModelProblem}
	\min_{v \in V}~ F(v)
	~~
	\text{subject to~}
	v \in K \subset V
	\,,
\end{equation}
where $K$ is a nonempty, closed convex subset of a Banach space $V$ and $F\colon V \to \mathbb{R}$ is continuously Fr\'echet differentiable.
We closely follow \cite{beck2003mirror,teboulle2018simplified} below to provide intuition for the method. In several places, we are purposely vague. This is particularly the case for the assumption that a Bregman divergence $D_G$ induced by the derivative $G'$ is available or that $\interior \dom G$ is non-empty with respect to the topology on $V$.

We begin by introducing the Bregman gradient step operator,
\begin{equation}
\label{eq:BregmanGradientStep}
	P_\alpha(w)
		=
	\argmin_{v\in V} \big\{ \langle F^\prime(w), v \rangle + \alpha^{-1} D_G(v,w) \big\}
	\,,
	\quad
	w \in \interior \dom G
	\,,
\end{equation}
where $G\colon \dom G \to \mathbb{R}\cup\{\infty\}$ is strongly convex with derivative $G'(w) \in V^\prime$. 
When $V$ is a Hilbert space and $G(v) = \frac{1}{2}\|v\|_V^2 = \frac{1}{2}(v,v)_V$, the use of the 
gradient step operator leads to the standard gradient descent rule. This follows from
a straightforward computation of the first-order optimality criteria, which leads to
\begin{equation}
	P_\alpha(w)
		=
	w - \alpha \nabla F(w)
	\,,
\end{equation}
where $\nabla F\colon V \to V$ is the gradient of $F$ characterized by the variational equation
\begin{equation}
	(\nabla F(w),v)_V = \langle F^\prime(w), v \rangle
	~~
	\fa
	v \in V
	\,.
\end{equation}
More generally, assuming the minimizer exists and $G' : V \to V^\prime$ is invertible, \cref{eq:BregmanGradientStep} returns the formula
\begin{equation}\label{eq:bregman-prox-mirror}
	P_\alpha(w)
		=
	(G')^{-1} \big(G'(w) - \alpha F'(w)\big)
	\,.
\end{equation}
\begin{figure}
\centering
	\includegraphics[height=6cm]{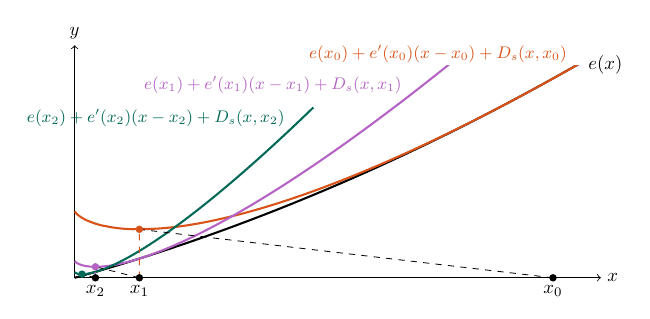}
	\caption{
	Illustration of convergence to the solution $x^\ast = 0$ for the constrained minimization problem $\min_{x\in[0,\infty)}\, e(x)$, where $e(x) = \frac{1}{2}x^2 + x$, by solving the sequence of minimization problems $x_{k+1} = \argmin_{x\in[0,\infty)}\, \{ e'(x_k)x + D_s(x,x_k) \}$ starting at $x_0 = 1$.
	\label{fig:MirrorPlot}}
\end{figure}

Recalling the classical steepest descent method, see, e.g., \cite{nocedal1999numerical}, it is not surprisingly that iterating~\cref{eq:BregmanGradientStep} can generate a convergent algorithm to solve~\cref{eq:NonConvexModelProblem} when $K = V$ and an appropriate step size rule for $\alpha$ is available.
Indeed, convergence of this algorithm is illustrated in~\cref{fig:MirrorPlot} for optimizing the  scalar objective function $e(x) = \frac{1}{2}x^2 + x$ with the Bregman divergence $D_s(x,x_k)$ from the scalar entropy function $s(x) = x\ln x - x$. This naturally leads to the so-called mirror descent method \cite{nemirovskij1983problem,beck2003mirror}, which, given a sequence of positive step sizes $\{\alpha_k\}$, generates a sequence of iterates $\{u^k\}$ according to the following scheme:
\[
u^0 \in \interior \dom G, \quad u^{k+1} = P_{\alpha_{k+1}}(u^k), \quad k=0,1,2\ldots
\]
Nemirovskij and Yudin point out that the motion of the iterates $\{u^k\}$, which takes place in the primal space $V$, is a ``shadow'' or ``image'', of the main motion: $G'(u^k) - \alpha_{k+1} F'(u^k)$, which by definition takes place in the dual space; whence the name ``method of mirror descent'' \cite[p. 88]{nemirovskij1983problem}.
This is easily witnessed by introducing a dual variable $\lambda := G'(w)$. Then under the assumptions that $G'$ is invertible, the new step in the dual space takes a somewhat more familiar form:
\[
\lambda^{k+1} = \lambda^k - \alpha_{k+1} (F'\circ (G')^{-1})(\lambda^{k}).
\]
This important distinction is often lost in finite dimensions and to some extent in the Hilbert space setting, where $G'$ and $F'$ are often identified with their Riesz representations in $V$; i.e., the gradients $\nabla G$ and $\nabla F$, respectively.

\subsection{Mirror descent with a linear equality constraint}

For constrained problems, it is essential that $G$ properly captures the geometry of the feasible set, as was done in the previous sections on the obstacle problem and advection-diffusion equations.  Many problems of interest have the following form:
\begin{equation}
\label{eq:convex-intersection-problem}
	\min_{v \in K_1 \cap K_2} F(v)
	\,,
\end{equation}
where $K_1$ and $K_2$ are nonempty, closed convex subsets of $V$ and $F$ is differentiable. For example, suppose that $K = K_1$ is a nonempty, closed convex set and $K_2 := \left\{v \in V \left|\; \ell(v) = c\right.\right\}$ for some linear functional $\ell \in V^\prime$ and constant $c \in \mathbb R$; i.e., $K_2 = \ell^{-1}(\left\{c\right\})$. Furthermore, suppose that $D_G$ is a Bregman divergence associated with a distance generating function $G$, which is a Legendre function whose critical domain is linked to the properties of $K$. In this setting, we fix $\alpha > 0$ and define the operator 
\[
T_{\alpha}(w) := \argmin_{v \in K \cap  \ell^{-1}(\left\{c\right\})} \{ F(w) + \langle F'(w),v - w\rangle + \alpha^{-1}D_{G}(v,w) \}
\,.
\]
We assume here that $D_{G}(\cdot,w)$ over $K \cap  \ell^{-1}(\left\{c\right\})$ has all the properties needed to ensure $T_{\alpha}$ is single-valued. Using standard optimality theory, e.g., \cite{ioffe2009theory}, we can argue that $u := T_{\alpha}(w)$ satisfies the inclusion
\begin{equation}\label{eq:fo-inclusion}
0 \in \alpha F'(w) + G'(u) - G'(w) + \mathcal{N}_{K \cap  \ell^{-1}(\left\{c\right\})}(u),
\end{equation}
where  $\mathcal{N}_{K \cap  \ell^{-1}(\left\{c\right\})}(u)$ is the normal cone from convex analysis \cite{ioffe2009theory}, defined by
\[
 \mathcal{N}_{K \cap  \ell^{-1}(\left\{c\right\})}(u) := \left\{
 \lambda \in V^\prime \left|\;
 \langle \lambda, v - u \rangle \le 0
 \quad \forall v \in K \cap  \ell^{-1}(\left\{c\right\})
 \right.
 \right\}
 .
\]
Note that if $w = T_{\alpha}(w)$, then \cref{eq:fo-inclusion} reduces to 
\[
0 \in \alpha F'(w) + \mathcal{N}_{K \cap  \ell^{-1}(\left\{c\right\})}(w),
\]
which indicates that $w$ is a first-order stationary point of \cref{eq:convex-intersection-problem}.

If we furthermore assume that $K$ contains a subset $\mathcal{B}$ such that $\ell(\mathcal{B}) \subset (c-\epsilon,c+\epsilon)$, for some $\epsilon > 0$, then $\{c\} - \ell(K)$ contains an open neighborhood of $0$. This constraint qualification \cite{ioffe2009theory} allows us to rewrite \cref{eq:fo-inclusion} as
\begin{equation}\label{eq:fo-inclusion-refine}
0 \in \alpha F'(w) + G'(u) - G'(w) + \mathcal{N}_{K}(u) +  \mathcal{N}_{\left\{c\right\}}(\ell(u)) \circ \ell,
\end{equation}
where $\mathcal{N}_{\left\{c\right\}}(\ell(u)) \circ \ell = \left\{ \mu \ell \in V' \left| \mu \in \mathbb R\right.\right\}$ provided $\ell(u) = c$.
Continuing on, we may assume for the sake of argument that the use of $D_{G}$ forces $\mathcal{N}_{K}(u) = \{0\}$ and $u \in K$. This happens, for example, if
$u$ remains away from the boundary of $K$. For pointwise bound constraints in $L^p$-spaces of the type $0 \le u \le 1$ considered below, we would also have $\mathcal{N}_{K}(u) = \left\{0\right\}$ when $0 < u < 1$ almost everywhere, even if the set $K$ has an empty interior. The remaining normal cone is trivial to compute and yields $\mathcal{N}_{\left\{c\right\}}(\ell(u)) = \mathbb R$.

These observations justify the following first-order optimality system that characterizes the map $w \mapsto u := T_{\alpha}(w)$: Find $(u,\mu) \in K \times \mathbb R$ such that 
\begin{equation}\label{eq:fo-order-mult-eq}
u = (G')^{-1}(G'(w) - \alpha F'(w) + \mu \ell) \text{ and } \ell(u) = c.
\end{equation}
In other words, given $w$ and $\alpha$, compute the increment $\widetilde{\lambda} := G'(w) - \alpha F'(w)$ and find $\mu \in \mathbb R$ by solving the equation 
\[
\ell( (G')^{-1}(\widetilde{\lambda} + \mu \ell)) = c.
\] 
Repeating the process
\[
u^0 \in \interior \dom G, \quad u^{k+1} = T_{\alpha_{k+1}}(u^k), \quad k=0,1,2\ldots
\]
generates a sequence of primal variables. Indeed, given a sequence of positive step sizes $\left\{\alpha_k\right\}$, we can generate $\left\{\lambda^k\right\}$ according to \Cref{alg:half-step-dual-update}.

\begin{algorithm2e}[H]
\DontPrintSemicolon
	\caption{\label{alg:half-step-dual-update} 
	Half-step mirror descent rule in Banach space
	}
	\SetKwInOut{Input}{Input}
	\SetKwInOut{Output}{Output}
	\BlankLine
	\Input{
	Initial dual variable $\lambda^0 \in V^\prime$ and sequence of step sizes $\alpha_k>0$.}
		\Output{Stationary dual variable $\overline{\lambda}$.}
	\BlankLine
	Initialize $k = 0$.\;
	\Repeat{a convergence test is satisfied}
	{
				\tcp*[l]{Dual space half step (gradient descent)}
		Assign  $\lambda^{k+1/2} \leftarrow  \lambda^{k} - \alpha_{k+1} (F' \circ (G')^{-1})(\lambda^k)$.\;
		\tcp*[l]{Compute Lagrange multiplier}
		Solve for $\mu^{k+1} \in \mathbb R$ such that $\ell( (G')^{-1}(\lambda^{k+1/2} + \mu^{k+1} \ell)) = c$.\;
		\tcp*[l]{Dual space feasibility correction}
		Assign  $\lambda^{k+1} \leftarrow  \lambda^{k+1/2} + \mu^{k+1}\ell$.\;
		Assign $k \leftarrow k+1$.\;
	}
\end{algorithm2e}

The pre-image of $G'$ is tacitly assumed to be contained in $K$. Therefore, the abstract scheme \Cref{alg:half-step-dual-update} theoretically provides a sequence of feasible primal iterates 
\[
u^{k+1} := (G')^{-1}(\lambda^{k+1/2} + \mu^{k+1} \ell). 
\]
Checking for optimality is rather difficult in general, as the evaluation of the residual of first-order optimality conditions may require the computation of a projection operator in non-trivial settings; recall the discussion in  \Cref{rem:ProximalMap} above. On the other hand, we demonstrated above that fixed points of $T_{\alpha}$ are stationary for the original problem. This motivates the simple stopping rule: $\| T_{\alpha_{k+1}}(u^k) - u^k \|_{V} < \mathtt{tol.}$, with $\mathtt{tol.} > 0$ sufficiently small, or some variant using absolute and relative tolerances. However, in order to remove the influence of $\alpha_k$, which will typically change with $k$, we advocate for rescaling the fixed point residual and also consider the relative quantities 
\[
\eta_{k} := \| T_{\alpha_{k+1}}(u^k) - u^k \|_{V}/\alpha_{k+1} =   \| u^{k+1}- u^k \|_{V}/\alpha_{k+1}.
\]
In the unconstrained Hilbert-space setting, we have $\eta_{k} = \| \nabla F(u^k) \|_{V}$. Therefore, if $\liminf_{k} \eta_{k} \to 0$, then $\liminf_{k} \| \nabla F(u_k) \|_{V} = 0$; i.e., we get a limiting stationarity condition. For example, 
if $\{u^{k+1} - u^k\}$ is a null sequence in $o(\alpha_k)$ for $\alpha_k \downarrow 0$, then clearly $\eta_k \downarrow 0$. We therefore also use $\eta_k$ as a heuristic stopping measure in our experiments below.
 
The abstract derivation above yields an iterative scheme in the dual space $V^\prime$. Implementing finite-dimensional approximations of negative-order Sobolev spaces can be challenging. However, the bound-constrained variational problem we have considered happens to have a substantial degree of useful structure, and entropy regularization of the associated bound constraints provides us with representations of $G'$, $G'^{-1}$ and $\ell$ that lead to a latent space reformulation of \Cref{alg:half-step-dual-update} that is readily treated with finite elements.

\subsection{An entropic mirror descent algorithm for topology optimization} \label{sub:top_opt_algorithm}
We consider the benchmark topology optimization problem of elastic compliance optimization of a cantilever beam; see, e.g., \cite{andreassen2011efficient}.
In particular, we use the two-field filtered density approach to topology optimization \cite[Section~3.1.2]{sigmund2013topology} to formulate the optimal cantilever beam problem.

\begin{subequations}
\label{eqs:elastic_compliance}
The purpose of the problem is to find a material density $0 \leq \rho \leq 1$, where zero indicates no material, and one indicates the complete presence of material, that induces a minimal elastic compliance, $\widehat{F}(\mathbf{u},\rho) = \int_{ \Omega} \mathbf{u}\cdot\mathbf{f} \dd \bm{x}$.
In this expression, the displacement $\mathbf{u} = \mathbf{u}(\rho)$ is determined by a variable material density $\rho$ and a fixed body force $\mathbf{f}$ through the classical linear elasticity equation \cite{marsden1994mathematical}, $-\mathrm{Div} \big( r(\tilde{\rho})\,\bm{\sigma} \big) = \mathbf{f}$.
In this equation, we are meant to understand that
\begin{equation}
\label{eq:elastic_compliance_cauchy_stress}
	\bm{\sigma} = \lambda \div(\mathbf{u})I + \mu(\nabla \mathbf{u} + (\nabla \mathbf{u})^\top)
	\,,
\end{equation}
with Lam\'e parameters $\lambda,\,\mu > 0$, is the Cauchy stress of a homogeneous, isotropic material, $\mathrm{Div}(\cdot)$ denotes the row-wise divergence operator, $\tilde{\rho}$ is a regularized (filtered) density function \cite{bruns2001topology,lazarov2011filters}, and $r(\tilde{\rho}) > 0$ is a local model for the Young's modulus.
For our work, we use the well-known (modified) solid isotropic material penalization (SIMP) model $r(\tilde{\rho}) = \underline{\rho} + \tilde{\rho}^3 (1-\underline{\rho})$, where $0<\underline{\rho} \ll 1$ is a nominal constant assigned to void regions in order to prevent the stiffness matrix from becoming singular \cite{andreassen2011efficient}.

The full problem formulation is written as follows:
\begin{equation}
\label{eq:elastic_compliance_objective}
    \min_{\rho \in L^1(\Omega)} \ \
    \bigg\{
    \,
    \widehat{F}(\mathbf{u},\rho)
    =
    \int_{ \Omega} \mathbf{u}\cdot\mathbf{f} \dd \bm{x}
    \,
    \bigg\}
    \, ,
\end{equation}
subject to the constraints
\begin{equation}
\label{eq:elastic_compliance_constraints}
\left\{\,\,
\begin{gathered}
    -\mathrm{Div} \big( r(\tilde{\rho})\,\bm{\sigma} \big) = \mathbf{f}
    ~~ \text{in }\Omega
    \quad\text{with}\quad
    \mathbf{u} = 0
    ~~\text{on }\Gamma_0
    \,,
    \quad
    \bm{\sigma} \mathbf{n} = 0
    ~~ \text{on }\partial\Omega \setminus \Gamma_0
    \,,
    \\
    -\epsilon^2\Delta \tilde{\rho} + \tilde{\rho}
	= \rho
	~~ \text{in }\Omega
	\quad\text{with}\quad
	\nabla \tilde{\rho}\cdot \mathbf{n} = 0
	~~ \text{on }\partial\Omega
    \,,
    \\
    \int_\Omega \rho(\bm{x}) \dd\bm{x}
    = \theta |\Omega|
    \,,\quad
    0 \leq \rho
    \leq 1
    \,,\quad
    r(\tilde{\rho})
    = \underline{\rho}+ \tilde{\rho}^3 (1-\underline{\rho})
    \,,
\end{gathered}
\right.
\end{equation}
\end{subequations}
where $\epsilon > 0$ is a \emph{length scale} and $0 < \theta < 1$ is the desired \emph{volume fraction}, which constrains the amount of the domain $\Omega$ occupied by the design.
The design domain $\Omega$ and associated boundary conditions are depicted in \Cref{fig:ElasticCompliance_mesh}.
\begin{figure}
\centering
	\includegraphics[width=7.5cm]{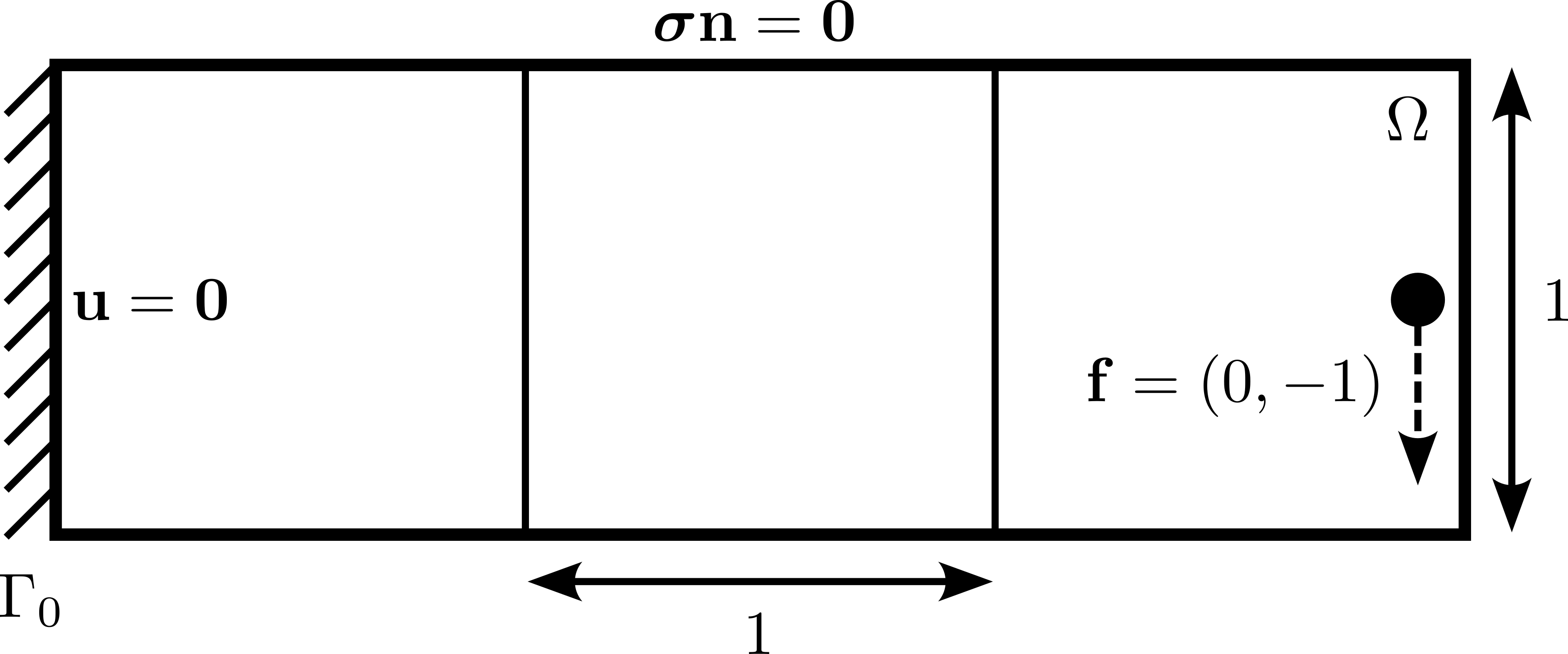}
	\caption{
	The design domain $\Omega$ for the cantilever beam problem~\cref{eqs:elastic_compliance} with corresponding boundary conditions,  and three-element initial mesh with length $h_0 = 1$.
	The circular load $\mathbf{f}$, is applied at the point $x_0 := (2.95, 0.5)$, and defined $\mathbf{f} = (0,-1)$ if $|x-x_0|\leq 0.05$ and $\mathbf{f} = (0,0)$ otherwise.
	\label{fig:ElasticCompliance_mesh}}
\end{figure}
We defer a rigorous mathematical discussion to the literature and simply note that it can be shown that $\mathbf{u}$ can be understood, via $\widetilde{\rho}$, as a differentiable mapping from $\rho$ into an 
appropriate regularity space; e.g., a subspace of $[H^1(\Omega)]^2$. Therefore, we replace the objective function in~\cref{eq:elastic_compliance_objective} by the reduced functional
\begin{subequations}
\begin{equation}
F(\rho) := \widehat{F}(\mathbf{u}(\rho),\rho)
\end{equation}
and arrive at the reduced space optimization problem
\begin{equation}\label{eq:reduced-to}
\aligned
	&\min_{\rho \in L^1(\Omega)} F(\rho)
	~~
			\text{subject to~} 0 \leq \rho \leq 1
	~\text{a.e.~and}~
	\int_\Omega \rho \dd x = \theta |\Omega|
	\,.
\endaligned
\end{equation}
\end{subequations}
We can now solve this problem with a custom version of \Cref{alg:half-step-dual-update} that employs the binary entropy-based Bregman divergence for the pointwise bound constraints found in~\cref{eq:reduced-to}; cf.~\Cref{sub:binary_entropy}.
In particular, the favorable structure of this problem 
lends itself nicely to a \textit{latent space} representation, given below, that makes use of the transformations
\[
\rho^{k} = \sigmoid(\psi^{k}) \quad \iff \quad \psi^k = \lnit(\rho^k),
\]
as well as the following variational characterization of the gradient $\nabla F(\rho^k)$:
\begin{equation}
\label{eq:TopOpt_VariationalCharacterizationOfTheGradient}
	\left\{
	\begin{aligned}
		\,&\text{Find}~
		\nabla F(\rho^k) := \tilde{w} \in H^1(\Omega)
		\text{~such that}
		\\
		&
		\epsilon^2(\nabla \tilde{w}, \nabla v) + (\tilde{w},v)
		=
		-(r^\prime(\tilde{\rho}^k)\, \bm{\sigma}(\mathbf{u}^k):\nabla \mathbf{u}^k , v)
				~\fa v \in H^1(\Omega)
		\,.
	\end{aligned}
	\right.
\end{equation}
A visual representation of a single iteration of \Cref{alg:TopologyOptimization} is given in~\Cref{fig:ProjectedMirrorDescent}.

{
\makeatletter
\makeatother
\begin{algorithm2e}[H]
\DontPrintSemicolon
	\caption{\label{alg:TopologyOptimization}
	Entropic mirror descent for topology optimization.
	}
	\SetKwInOut{Input}{Input}
	\SetKwInOut{Output}{Output}
	\BlankLine
	\Input{Initial latent variable $\psi^0 \in L^\infty(\Omega)$, sequence of step sizes $\alpha_k>0$, increment tolerance $\mathtt{itol.} > 0$, and normalized tolerance $\mathtt{ntol.} > 0$.}
	\Output{Optimized material density $\overline{\rho} = \sigmoid(\psi^k)$.}
	\BlankLine
	Initialize $k = 0$.\;
	\While{$\|\sigmoid(\psi^k) - \sigmoid(\psi^{k-1})\|_{L^1(\Omega)} > \min\{\alpha_k\, \mathtt{ntol.},\mathtt{itol.}\}$}
		{
		\tcp*[l]{Latent space gradient descent}
		Assign $\psi^{k+1/2} \leftarrow  \psi^{k} - \alpha_{k+1} \nabla F(\sigmoid(\psi^k))$.\;
		\tcp*[l]{Compute Lagrange multiplier}
		Solve for $c \in \mathbb R$ such that $\int_\Omega \sigmoid(\psi^{k+1/2} + c) \dd x = \theta |\Omega|$.\;
		\tcp*[l]{Latent space feasibility correction}
		Assign $\psi^{k+1} \leftarrow  \psi^{k+1/2} + c$.\;
		Assign $k \leftarrow k+1$.\;
	}
\end{algorithm2e}}

\begin{figure}
\centering
	\includegraphics[width=0.9\linewidth]{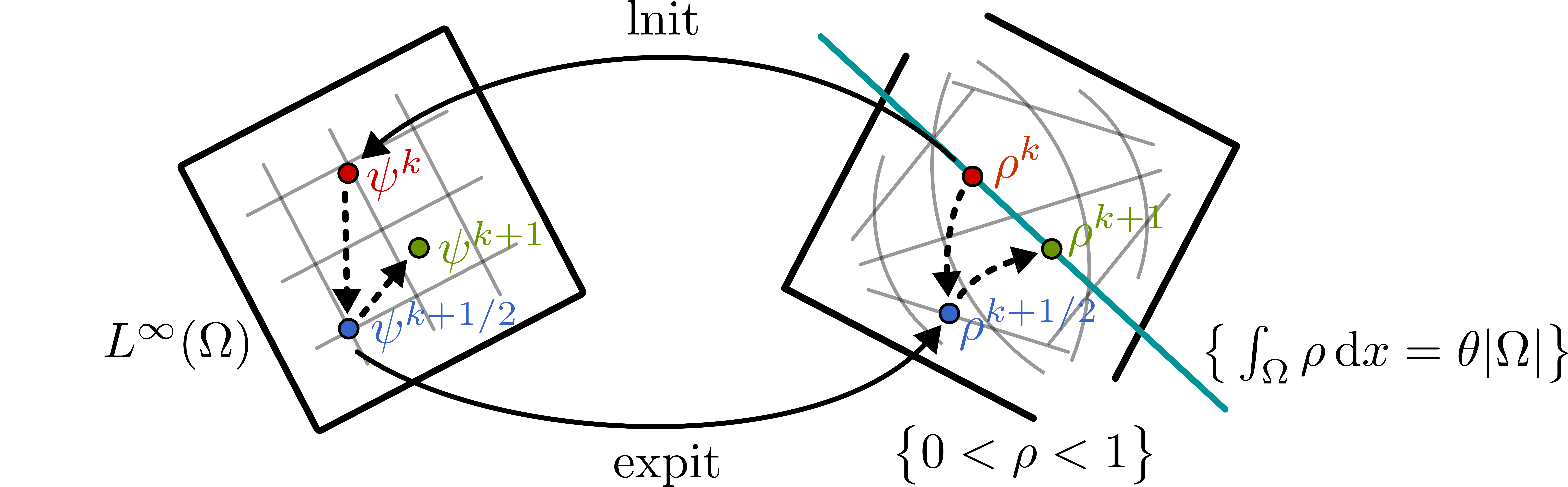}
	\caption{
	Illustration of motion of the primal and latent iterates in \Cref{alg:TopologyOptimization}.
	When viewed in the primal space, we find that both steps of the progression $\rho^{k} \mapsto \rho^{k+1/2} \mapsto \rho^{k+1}$ involve \emph{nonlinear} transformations of the primal variables.
					However, when viewed in the latent space, $L^\infty(\Omega)$, these transformations are simply just \emph{translations} of the latent variables; namely, $\psi^{k+1/2} = \psi^k - \alpha_{k+1} \nabla F(\rho^k)$ and $\psi^{k+1} = \psi^{k+1/2} + c$.
	\label{fig:ProjectedMirrorDescent}}
\end{figure}

\subsection{Numerical experiments} \label{sub:numerical_experiments_topopt}

In this set of experiments, we estimate the gradients $\nabla F(\rho^k)$ in~\Cref{alg:TopologyOptimization} by discretizing \cref{eq:TopOpt_VariationalCharacterizationOfTheGradient} with $C^0(\Omega)$-conforming, quadrilateral finite elements of degree $p\geq 1$.
Likewise, the discrete displacements $\mathbf{u}_h^k \approx \mathbf{u}^k$ and filtered densities, $\tilde{\rho}_h^k \approx \tilde{\rho}^k$, are also computed with conforming finite elements of degree $p$.
Finally, unlike the physical variables above, the latent variable $\psi^k$ is approximated by {discontinuous} piecewise polynomials $\psi^k_h$ of degree $p-1$.
Note that this induces a discontinuous primal variable $\rho^k_h := \sigmoid(\psi^k_h)$ satisfying $0 < \rho^k_h < 1$; see also \Cref{rem:TopOpt_DiscreteBoundConstraints}.
We then apply the resulting discretized version of \Cref{alg:TopologyOptimization} to solve~\cref{eqs:elastic_compliance} with $\underline{\rho} = 10^{-6}$, $\lambda=\mu = 1$, $\theta = 0.5$, and $\epsilon = 0.02$.
The boundary conditions and body force $\mathbf{f}$ are depicted in~\Cref{fig:ElasticCompliance_mesh}.
This experiment is an official part of MFEM 4.6.
We invite the reader to view MFEM Example~37 \cite{Keith2023TopOptCode} for full implementation details.
For sake of space, we have focused on presenting results with low-order discretizations (i.e., $p = 1,2$) of the above form and left the exploration of higher-order discretizations to future work.

\begin{figure}
\centering
	\centering
	\begin{minipage}[c]{0.96\textwidth}
		\begin{minipage}[c]{0.32\textwidth}
		\small
			\centering
			\includegraphics[width=\textwidth]{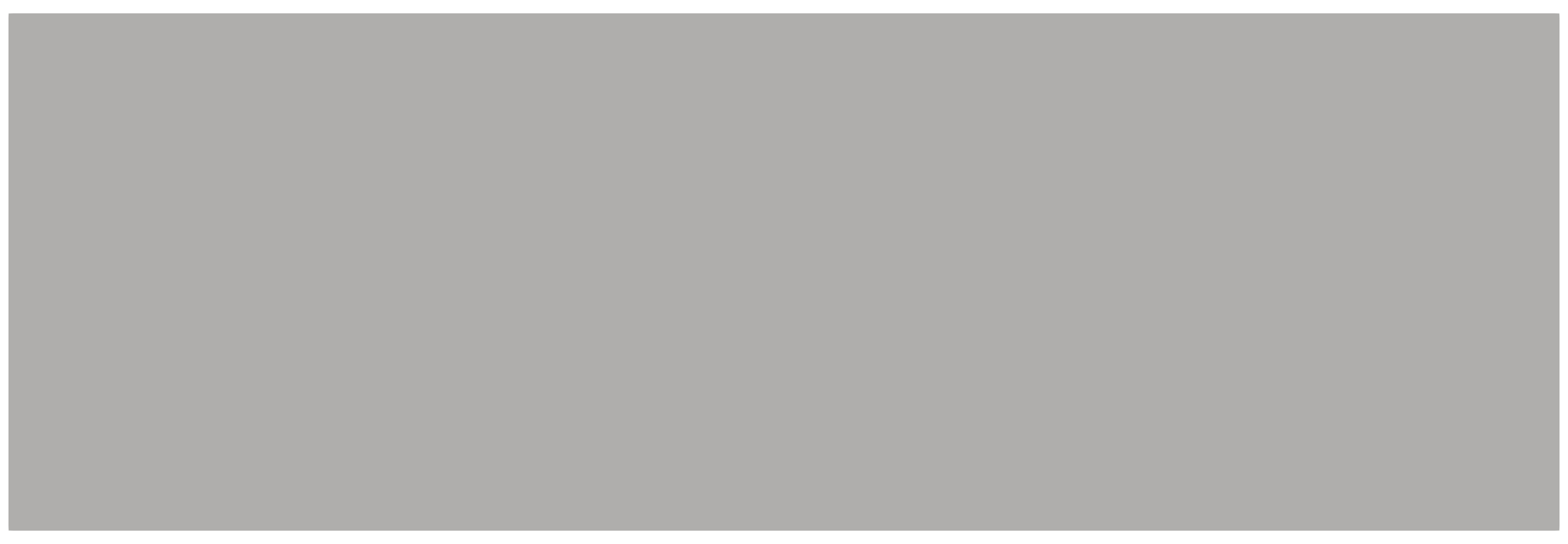}
			\\[3pt]
			$k = 0$
		\end{minipage}
		\begin{minipage}[c]{0.32\textwidth}
		\small
			\centering
			\includegraphics[width=\textwidth]{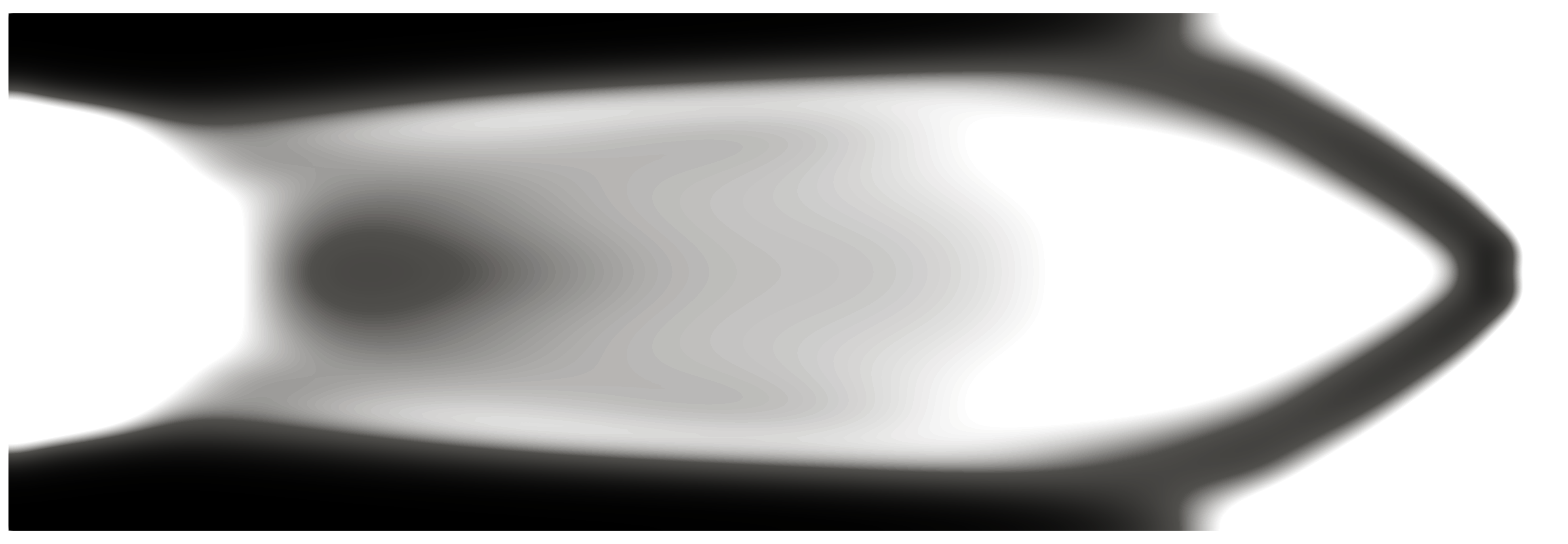}
			\\[3pt]
			$k = 6$
		\end{minipage}
		\begin{minipage}[c]{0.32\textwidth}
		\small
			\centering
			\includegraphics[width=\textwidth]{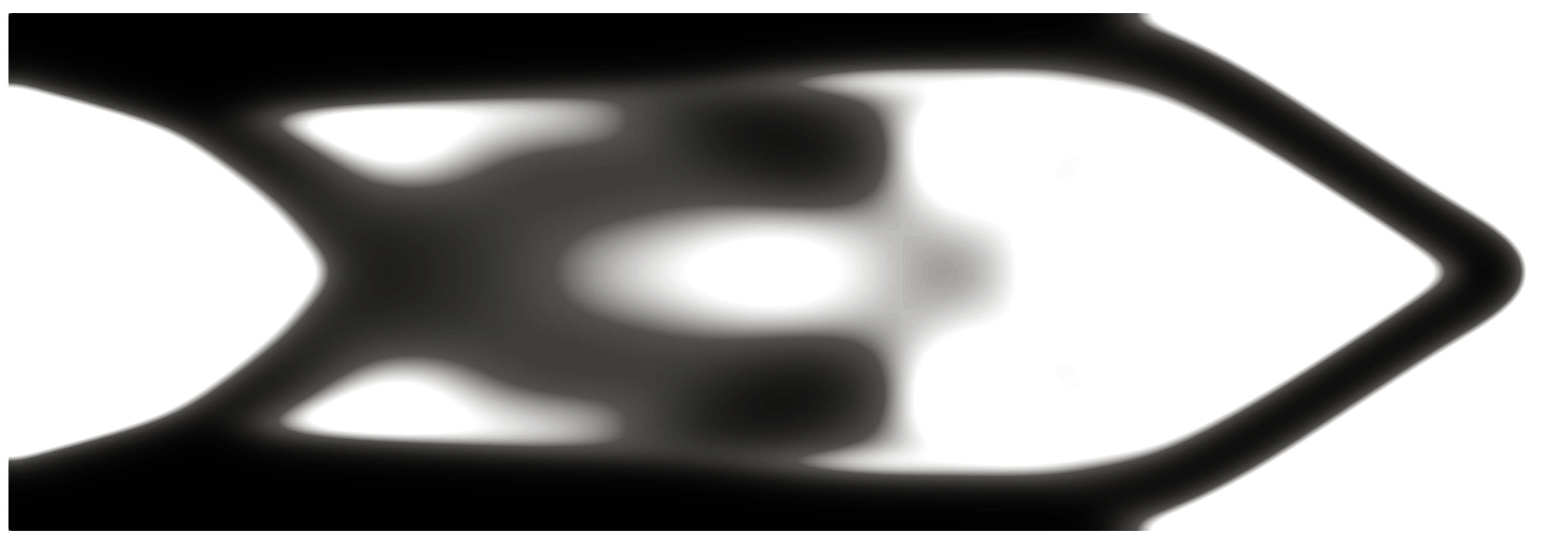}
			\\[3pt]
			$k = 12$
		\end{minipage}
		\\[7pt]
		\begin{minipage}[c]{0.32\textwidth}
		\small
			\centering
			\includegraphics[width=\textwidth]{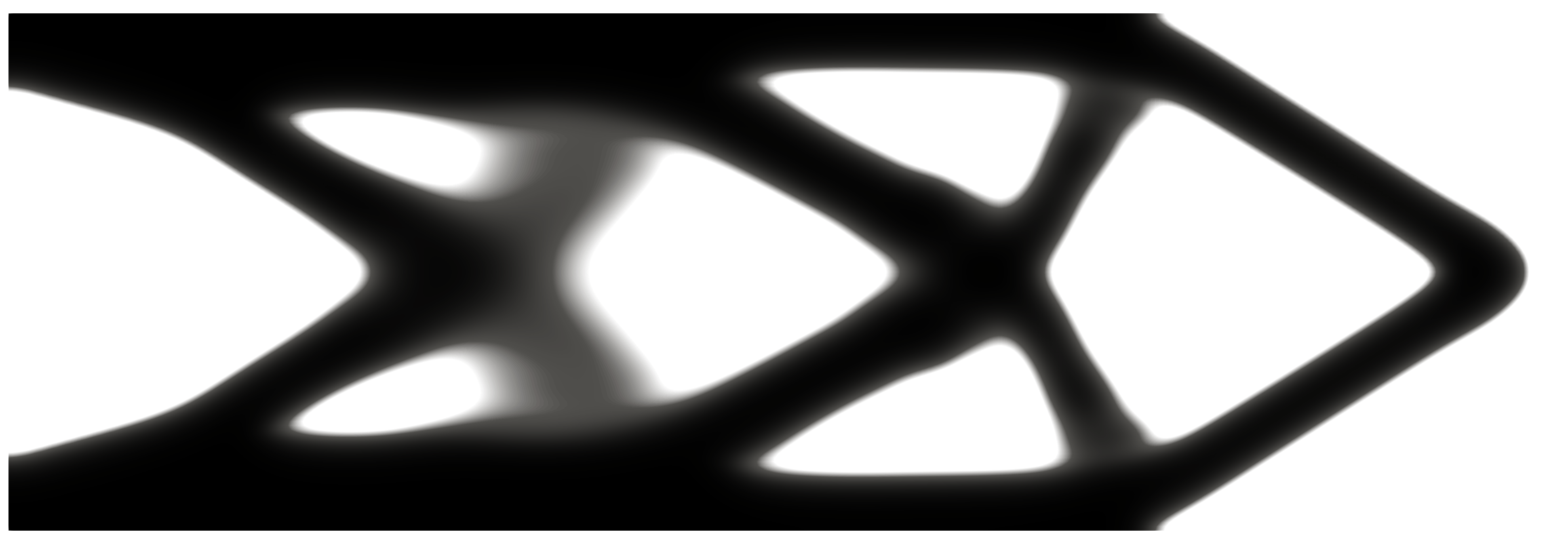}
			\\[3pt]
			$k = 18$
		\end{minipage}
		\begin{minipage}[c]{0.32\textwidth}
		\small
			\centering
			\includegraphics[width=\textwidth]{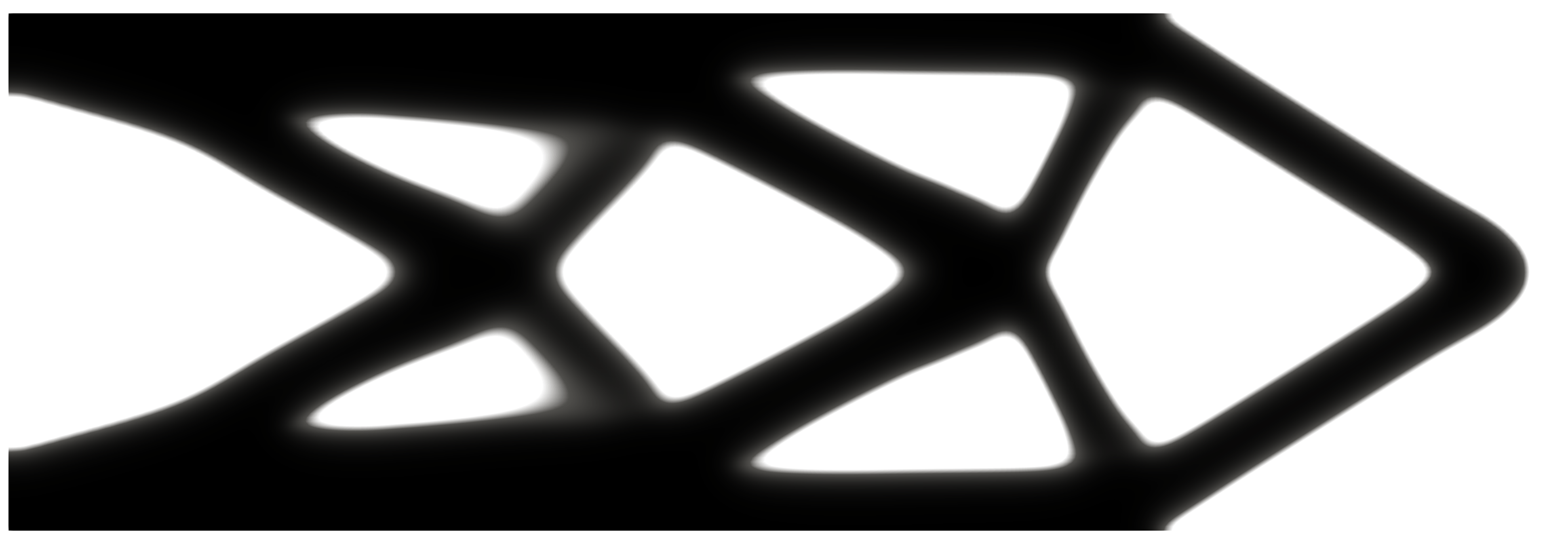}
			\\[3pt]
			$k = 24$
		\end{minipage}
		\begin{minipage}[c]{0.32\textwidth}
		\small
			\centering
			\includegraphics[width=\textwidth]{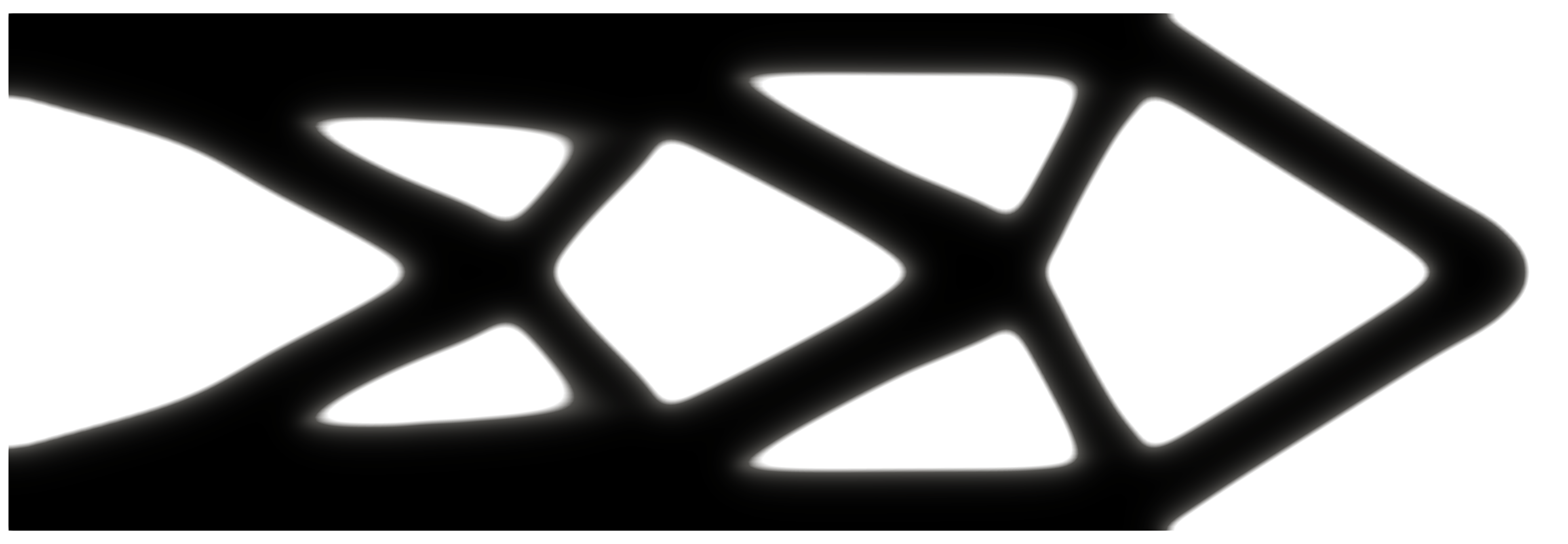}
			\\[3pt]
			$k = 29$ (final)
		\end{minipage}
	\end{minipage}
	\begin{minipage}[c]{0.032\textwidth}
		\centering
		\includegraphics[width=\textwidth]{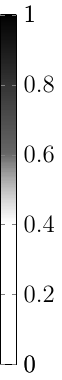}
		\\[4pt]~
	\end{minipage}
																																																			\caption{
		Subsequence of material densities $\tilde{\rho}_h^k$ from~\Cref{alg:TopologyOptimization} for selected iterations $k$.
	Results obtained with problem parameters $\epsilon = 2\cdot 10^{-2}$ and $\theta = 0.5$; algorithm parameters $\mathtt{itol.} = 10^{-2}$, $\mathtt{ntol.} = 10^{-5}$, and $\alpha_k = 25k$; and discretization parameters $h = h_0/128$ and $p = 1$.
	\label{fig:ElasticCompliance_sequence}}
\end{figure}

A sequence of iterates converging to a discrete solution with mesh size $h = h_0/128$ and polynomial degree $p = 1$ are depicted in \Cref{fig:ElasticCompliance_sequence}.
From this figure, we observe typical first-order convergence behavior to a standard truss-like structure.
To generate this figure, we used the heuristic step size sequence $\alpha_k = 25 k$ and tolerances $\mathtt{itol.} = 10^{-2}$ and $\mathtt{ntol.} = 10^{-5}$.
Although the conventional wisdom from finite-dimensional optimization theory would indicate that
$\alpha_k$ should tend to zero, or at the very least be less than the reciprocal of the Lipschitz constant of $\nabla F$, 
our experiments indicate that we can moderately increase the step sizes and still obtain satisfactory convergence
behavior. To be fair, ``satisfactory'' convergence is based on the heuristic stopping rule given in~\Cref{alg:TopologyOptimization}.

Future work is needed to develop an adaptive step size selection procedure. 
In turn, we include \Cref{fig:ElasticCompliance} to show the different effects the step size sequence can have on the final solution.
Here, we witness that different sequences --- e.g., $\alpha_k = 10k$, $\alpha_k = 25k$, and $\alpha_k = 50k$ --- can lead~\Cref{alg:TopologyOptimization} to converge to significantly different local optima.
This class of non-convex optimization problems is widely known to exhibit multiple local optima, though procedures are available to compute them \cite{papadopoulos2021computing}.
In particular, notice from the two top left images that different final designs are possible just by changing the step size rule.
The suspicious design on the bottom left is found because the $\alpha_k = 50k$ step size rule is too aggressive in the early iterations.
Thereafter, a ``design locking'' phenomenon that is common in topology optimization problems keeps the design close to its nearly-binary, early state.
To generate the results in \Cref{fig:ElasticCompliance}, we changed the length scale to $\epsilon = 10^{-2}$ because it invokes a higher parameter sensitivity.

\begin{figure}
	\centering
	\begin{minipage}[c]{0.38\textwidth}
	\small
		\centering
												\includegraphics[width=\textwidth]{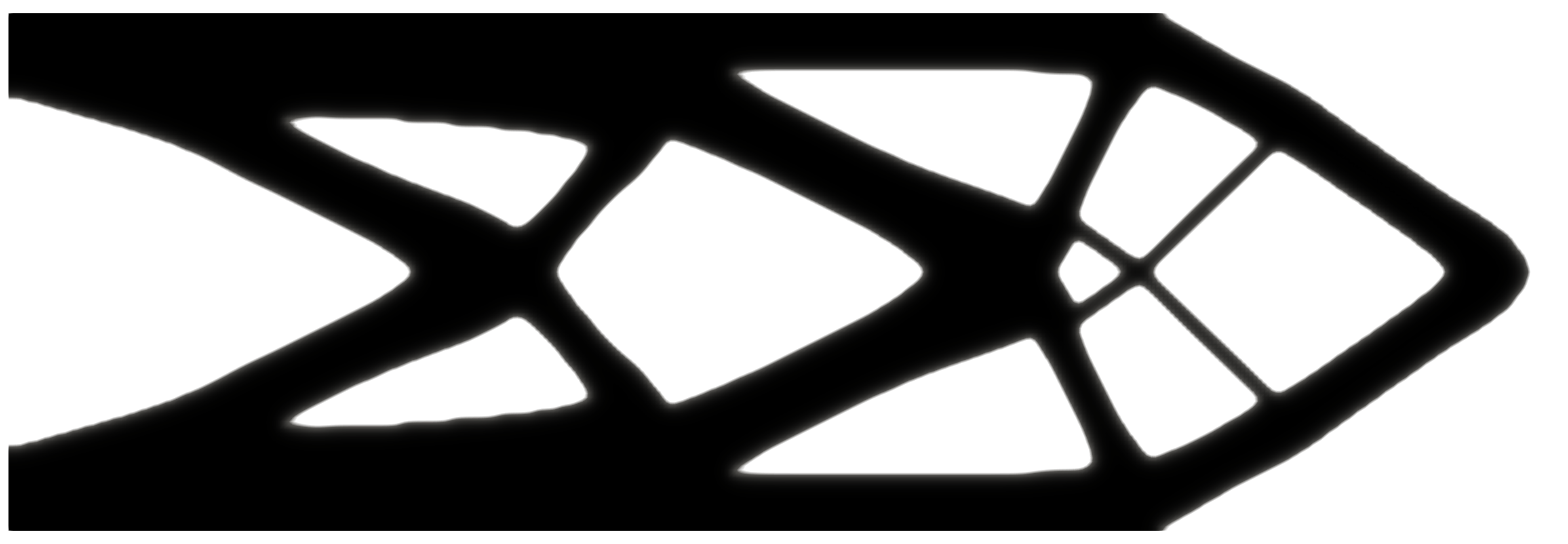}
		\\[2pt]
		$F(\overline{\rho}_h) = 3.75\cdot 10^{-3}$
		\\[6pt]
		\includegraphics[width=\textwidth]{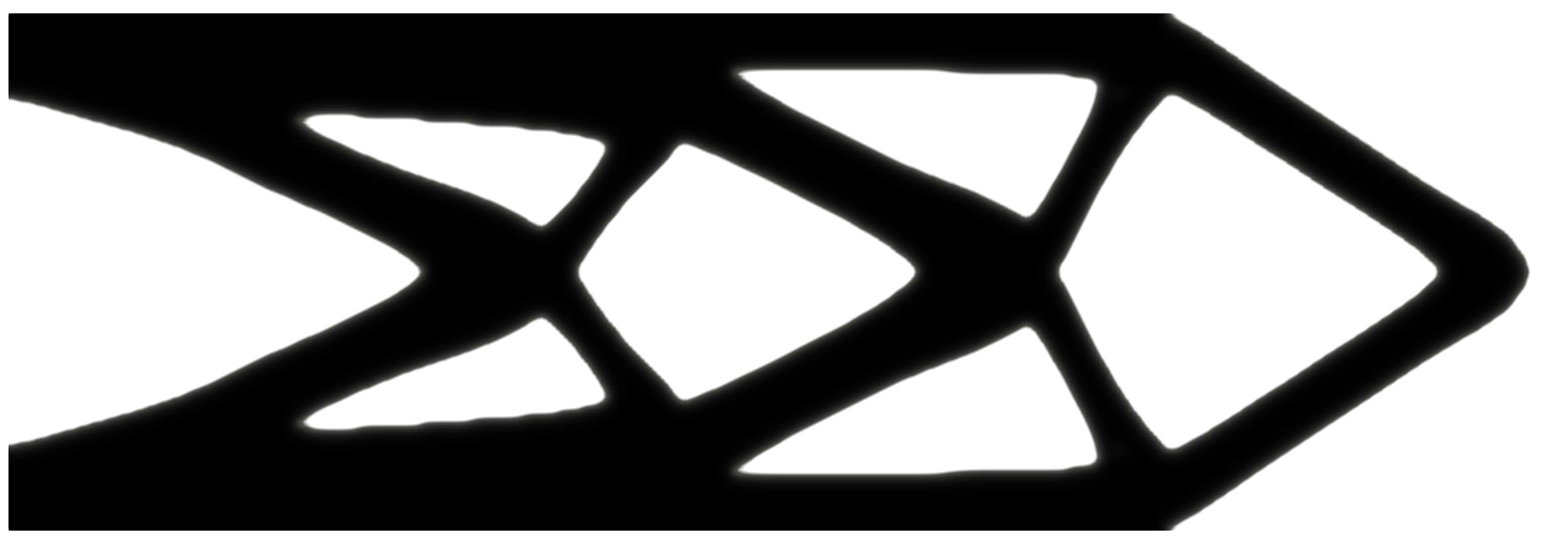}
		\\[2pt]
		$F(\overline{\rho}_h) = 3.72\cdot 10^{-3}$
		\\[6pt]
		\includegraphics[width=\textwidth]{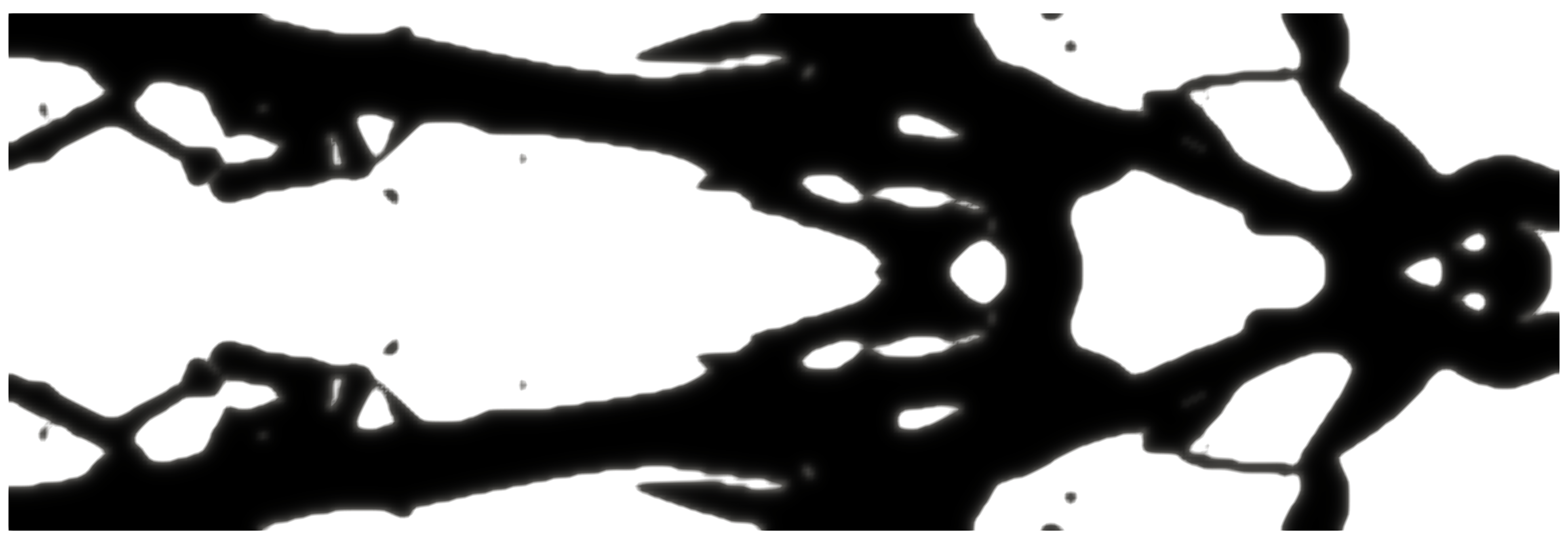}
		\\[2pt]
		$F(\overline{\rho}_h) \approx 7\cdot 10^{-3}$
											\end{minipage}
	\begin{minipage}[c]{0.6\textwidth}
	\small
		\centering
		\includegraphics[width=\linewidth]{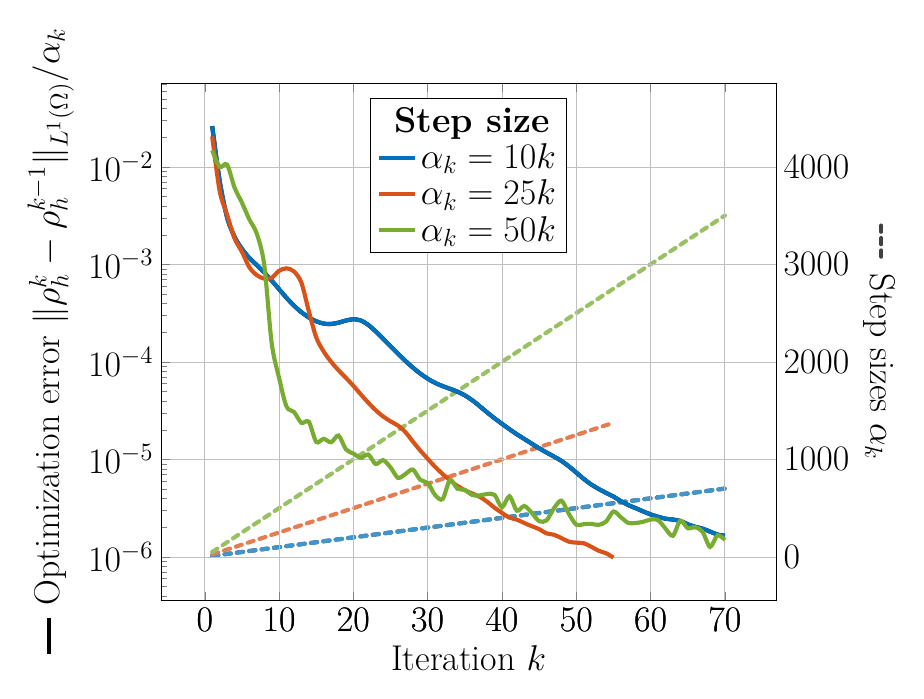}
					\end{minipage}
	\caption{
		An aggressive step size rule can lead to a better convergence rate.
	However, if the step size rule is too aggressive, the algorithm may convergence to a sub-optimal local minimum or even diverge.
	Left: The final densities $\tilde{\rho}_h$ and associated compliance values $F(\rho_h)$ for step sizes $\alpha_k = 10k$ (top), $\alpha_k = 25k$ (middle), and $\alpha_k = 50k$ (bottom).
	Right: The normalized error estimates for the various step size sequences $\alpha_k$, with $k = 1,2,\ldots$
	These results were obtained using the problem parameters $\epsilon = 10^{-2}$ and $\theta = 0.5$ and the discretization parameters $h = h_0/128$ and $p = 1$.
	\label{fig:ElasticCompliance}}
\end{figure}

Finally, we return to the case considered in \Cref{fig:ElasticCompliance_sequence} (i.e., we again set $\epsilon = 2\cdot 10^{-2}$ and $\alpha_k = 25k$) to record the sequence of error indicators $\eta_k = \|\rho_h^{k} - \rho_h^{k-1}\|_{L^1(\Omega)}/\alpha_k$ with different discretization parameters $h \in \{ h_0/64 , h_0/128 , h_0/256 \}$ and $p \in \{ 1 , 2 \}$.
The results are given in~\Cref{tab:TopOptMeshIndependence}.
From these results, we see that the number of iterations required to reach the tolerance $\|\rho_h^{k} - \rho_h^{k-1}\|_{L^1(\Omega)}/\alpha_k < 10^{-5}$ tends to a fixed value as the mesh is refined or the polynomial order is elevated.
Moreover, the individual values of $\eta_k$ appear to stabilize as $h \to 0$, for both $p = 1,2$.
Both of these properties suggest \emph{mesh-independence} of \Cref{alg:TopologyOptimization}.

\begin{table}
\centering
\footnotesize
\renewcommand{\arraystretch}{1.3}
\begin{tabular}{ |c|c|c|c|c|c|c| }
 \hhline{|>{\arrayrulecolor{white}}-->{\arrayrulecolor{black}}|-----|}
 \multicolumn{2}{c|}{} & \multicolumn{5}{c|}{\cellcolor{lightgray!15} \small\raisebox{5pt}{\vphantom{f}}Optimization error $\|\rho_h^{k} - \rho_h^{k-1}\|_{L^1(\Omega)}/\alpha_k$ for various $h$ and $p$}\\[3pt]
 \hhline{|>{\arrayrulecolor{white}}-->{\arrayrulecolor{black}}|-----|}
 \multicolumn{2}{c|}{}& \multicolumn{3}{c|}{\cellcolor{lightgray!05} Polynomial order $p = 1$} & \multicolumn{2}{c|}{\cellcolor{lightgray!10} Polynomial order $p = 2$}\\
 \hline
 \rowcolor{lightgray!15}
 $k$ & $\alpha_{k}$ & $h_0/64$ & $h_0/128$ & $h_0/256$ & $h_0/64$ & $h_0/128$ \\
 \hline
\cellcolor{lightgray!05}   1 & \cellcolor{lightgray!01} $25$  & $2.00\cdot 10^{-2}$  & $2.06\cdot 10^{-2}$  & $2.05\cdot 10^{-2}$  & $2.07\cdot 10^{-2}$  & $2.05\cdot 10^{-2}$ \\
\cellcolor{lightgray!05}   2 & \cellcolor{lightgray!01} $50$  & $5.42\cdot 10^{-3}$  & $5.80\cdot 10^{-3}$  & $5.76\cdot 10^{-3}$  & $5.88\cdot 10^{-3}$  & $5.74\cdot 10^{-3}$ \\
\cellcolor{lightgray!05}   3 & \cellcolor{lightgray!01} $75$  & $2.97\cdot 10^{-3}$  & $3.30\cdot 10^{-3}$  & $3.27\cdot 10^{-3}$  & $3.38\cdot 10^{-3}$  & $3.25\cdot 10^{-3}$ \\
\cellcolor{lightgray!05}   4 & \cellcolor{lightgray!01} $100$ & $1.61\cdot 10^{-3}$  & $1.87\cdot 10^{-3}$  & $1.85\cdot 10^{-3}$  & $1.94\cdot 10^{-3}$  & $1.83\cdot 10^{-3}$ \\
\cellcolor{lightgray!05}   5 & \cellcolor{lightgray!01} $125$ & $1.12\cdot 10^{-3}$  & $1.30\cdot 10^{-3}$  & $1.29\cdot 10^{-3}$  & $1.36\cdot 10^{-3}$  & $1.28\cdot 10^{-3}$ \\
\cellcolor{lightgray!05}   \vdots & \cellcolor{lightgray!01}   \vdots & \vdots  & \vdots & \vdots  & \vdots & \vdots \\
 \hline
 \rowcolor{lightgray!05}
 \multicolumn{2}{|c|}{\cellcolor{lightgray!15} Total iterations $k$} & $30$ & $29$ & $29$ & $29$ & $29$ \\
 \hline
 \rowcolor{lightgray!05}
 \multicolumn{2}{|c|}{\cellcolor{lightgray!15} Final compliance $F(\overline{\rho}_h)$} & $3.86\cdot 10^{-3}$ & $4.04\cdot 10^{-3}$ & $4.02\cdot 10^{-3}$ & $4.08\cdot 10^{-3}$ & $4.01\cdot 10^{-3}$ \\
 \hline
\end{tabular}
\caption{\label{tab:TopOptMeshIndependence}
Table of error estimates $\eta_k = \|\rho_h^{k} - \rho_h^{k-1}\|_{L^1(\Omega)}/\alpha_k$ for various mesh sizes $h$ and polynomial orders $p$.
The initial density was set to the constant function $\rho_h^0 = \theta$ (i.e., $\psi_h^0 = \lnit\theta$) at the beginning of each run and each run was stopped once $\eta_k < 10^{-5}$.
These results were obtained using the problem parameters $\epsilon = 2\cdot 10^{-2}$ and $\theta = 0.5$.
}
\end{table}

\begin{remark}[Preserving pointwise bound constraints at the discrete level]
\label{rem:TopOpt_DiscreteBoundConstraints}
No matter the polynomial degree $p \geq 1$, the discrete primal variable $\rho^k_h = \sigmoid(\psi^k_h)$ is guaranteed to satisfy the pointwise bound constraint $0 \leq \rho^k_h \leq 1$.
This is an immediate consequence of the sigmoid map $\sigmoid\colon \mathbb{R} \to (0,1)$ and the decision to discretize the \textit{latent variable} with finite elements.
Had we followed the literature and, instead, directly discretized the primal variable $\rho^k$ with finite elements, then the property $0 \leq \rho^k_h \leq 1$ would have to be enforced by introducing discrete-level pointwise bound constraints.
This is a common concern in standard topology optimization approaches since the number of discrete-level pointwise bound constraints must grow with the size of the finite element space; cf.~\Cref{sub:bound_constrained_finite_element_methods}.
\end{remark}

\section{Conclusion} \label{sec:conclusion}

We have introduced proximal Galerkin, a new nonlinear finite element method that hinges on a mathematical technique called entropy regularization and an infinite-dimensional optimization algorithm we were refer to as the latent variable proximal point method (LVPP).
The essential feature of the proximal Galerkin method is to provide robust, high-order, and \emph{pointwise} bound-preserving discretizations accompanied by a built-in, low iteration complexity, mesh-independent solution algorithm.
We have derived, analyzed, and implemented proximal Galerkin for the obstacle problem and used the advection-diffusion equation and topology optimization to motivate our wider vision for the method.
Each of our numerical experiments is accompanied by an open-source implementation to facilitate reproduction of our results and broader adoption of the proximal Galerkin method.

The upshot of this work is that computational techniques for variational inequalities, maximum principles, and bound constraints in optimal design can be unified with a rigorous mathematical framework rooted in nonlinear programming and nonlinear functional analysis.
We hope that the proximal Galerkin methods that arise from this framework will lead to new challenges and opportunities in optimization theory, analysis of PDEs, and numerical analysis, as well as provide promising alternatives to the more classical procedures used for industrial-scale problems.

\appendix

\section{Mathematical results I: Isomorphisms, regularity, characterizations, and convergence} \label{sec:preliminaries}
This appendix contains proofs and continuous-level structural results supporting the main sections of the paper.

\subsection[Structural results on the set]{Structural results on the set $\interior L^\infty_+(\Omega)$}\label{sub:GroupTheory}

The following concepts and results are not commonly used in the finite element literature.
Although they can be derived from diverse sources, such as \cite{conway2019course,borwein1994strong,glockner2002algebras}, we assemble them here for the reader’s convenience.

\begin{definition}[Group of units]
\label{def:units}
	Let $\mathcal{X}$ be a semiring equipped with two binary operations: addition $\oplus \colon \mathcal{X} \times \mathcal{X} \to \mathcal{X}$ and multiplication $\odot \colon \mathcal{X} \times \mathcal{X} \to \mathcal{X}$.
	An element $u$ of $\mathcal{X}$ is called a unit if there exists an inverse element in $\mathcal{X}$, denoted $\frac{1}{u}$, such that $u \odot \frac{1}{u} = \frac{1}{u} \odot u = 1$.
	The group of units of $\mathcal{X}$, denoted $\mathcal{X}^\times$, is the set of all units in $\mathcal{X}$.
\end{definition}

This work is largely centered around the group of units $(L^\infty_+(\Omega))^\times$.
We prove $(L^\infty_+(\Omega))^\times = \interior L^\infty_+(\Omega)$, along with several other algebraic/topological identities, at the end of this subsection; see~\Cref{prop:Equivalence}.

It is well-known that algebraic and topological structures are often entwined, as the following definition and result shows.
\begin{definition}[Banach algebra]
\label{def:BanachAlgebra}
			A Banach algebra is a complete normed vector space that is closed under multiplication.
	\end{definition}
\begin{proposition}[Topology of the group of units]
	For any Banach algebra $\mathcal{X}$, its group of units $\mathcal{X}^\times$ is open.
	Moreover, the inversion map $\mathcal{X}^\times \to \mathcal{X}^\times \colon u \mapsto \frac{1}{u}$ is continuous.
\end{proposition}
\begin{proof}
	See \cite[Theorem 2.2, p.~192]{conway2019course}.
\end{proof}
Notably, this result also implies that $\interior L^\infty_+(\Omega)$ is a Banach--Lie group.
\begin{definition}[Banach--Lie group]
\label{def:BanachLieGroup}
	A Banach manifold is a topological space $\mathcal{M}$ where each point $u \in \mathcal{M}$ has an open neighborhood that is homeomorphic to an open set in a Banach space.
	A set $\mathcal{G}$ is a Banach--Lie group if it is a Banach manifold that is closed under continuous multiplication and inversion operations.
\end{definition}

An important property of Lie groups is the existence of a smooth exponential map, $\exp\colon \mathcal{X} \to \mathcal{G}$, where $\mathcal{X}$ is the associated Lie algebra; cf.~\cite{lee2012introduction}.
\begin{definition}[Banach--Lie algebra]
\label{def:BanachLieAlgebra}
	A Lie algebra $\mathcal{X}$ is a vector space endowed with an antisymmetric bilinear form called the Lie bracket $[\cdot,\cdot] \colon \mathcal{X}\times \mathcal{X} \to \mathcal{X}$ satisfying the Jacobi identity $[\psi,[\varphi,\omega]] + [\varphi,[\omega,\psi]] + [\omega,[\psi,\varphi]] = 0$ for all $\psi,\varphi,\omega \in \mathcal{X}$.
	A set $\mathcal{X}$ is Banach--Lie algebra if it is both a Lie algebra and a Banach space.
\end{definition}

Using well-known results on Nemytskii operators between Lebesgue spaces, we can argue that $L^{\infty}(\Omega)$ is the Banach--Lie algebra associated to the Banach--Lie group $\interior L^\infty_+(\Omega)$; cf.~\Cref{prop:Equivalence}.
In particular, as a result of \Cref{lem:ExponentialMap}, the Nemytskii operator generated by the standard exponential function on $\mathbb R$ provides the exponential map from $L^{\infty}(\Omega)$ to $\interior L^\infty_+(\Omega)$.
Moreover, this map is surjective and, thus, the inverse of the entropy gradient $(\nabla S)^{-1} = \exp \colon L^{\infty}(\Omega) \to \interior L^\infty_+(\Omega)$ is a group isomorphism.
Finally, in our setting, $\interior L^\infty_+(\Omega)$ is commutative and so the Lie bracket is trivial; i.e., $[\psi,\varphi] = \psi\varphi - \varphi\psi = 0$ for all $\psi,\varphi \in L^\infty(\Omega)$.

\begin{lemma}\label{lem:ExponentialMap}
	The Nemytskii operator $\psi \mapsto \exp\psi$ is infinitely continuously Fr\'echet differentiable on $L^\infty(\Omega)$.
	\end{lemma}
\begin{proof}
The properties of the exponential function allow us to apply the required results in [79]; notably \cite[Theorem 1 (iv)]{goldberg1992nemytskij} and \cite[Theorem 5]{goldberg1992nemytskij} for continuity and \cite[Theorem 7]{goldberg1992nemytskij} for differentiability. Since $\exp' = \exp$, these results can be applied recursively indefinitely.
\end{proof}

\Cref{prop:Equivalence} summarizes various useful interpretations of the set $\interior L^\infty_+(\Omega)$.
We withhold the proof, which follows from elementary arguments, for sake of space.

\begin{proposition}
\label{prop:Equivalence}
	Nemytskii operator $\psi \mapsto \exp\psi$ is a diffeomorphism between $L^{\infty}(\Omega)$ and $\interior L^\infty_+(\Omega)$ for which the following definitions are equivalent:
	\begin{subequations}\label{eq:glequivs}
	\begin{enumerate}[label=(\alph*)]
		\item $\interior L^{\infty}_+(\Omega)$ is the set of all positive functions in $L^{\infty}(\Omega)$ whose reciprocals lie in $L^{\infty}(\Omega)$,
		\begin{equation}
			\interior L^{\infty}_+(\Omega) = \{ w \in L^{\infty}(\Omega) \mid 1/w \in L^{\infty}(\Omega) ~\text{and}~ w > 0 \}. \label{eq:glequivs1}
		\end{equation}
		In other words, $\interior L^{\infty}_+(\Omega) = (L^{\infty}_+(\Omega))^\times$ is the group of units in $L^{\infty}_+(\Omega)$.
		\item $\interior L^{\infty}_+(\Omega)$ is the set of all functions in $L^{\infty}(\Omega)$ whose logarithm is bounded in $L^{\infty}(\Omega)$,
		\begin{equation}
			\interior L^{\infty}_+(\Omega) = \ln^{-1}(L^{\infty}(\Omega)). \label{eq:glequivs2}
		\end{equation}
		\item $\interior L^{\infty}_+(\Omega)$ is the image of $L^\infty(\Omega)$ under the exponential map,
		\begin{equation}
			\interior L^{\infty}_+(\Omega) = \exp(L^\infty(\Omega)). \label{eq:glequivs3}
		\end{equation}
		\item $\interior L^{\infty}_+(\Omega)$ is the set of all positive functions in $L^{\infty}(\Omega)$ that are strictly bounded away from zero,
		\begin{equation}
			\interior L^{\infty}_+(\Omega) = \{ w \in L^{\infty}(\Omega) \mid \text{there exists~} \epsilon > 0 \text{~such that~} w > \epsilon \}. \label{eq:glequivs4}
		\end{equation}
															\end{enumerate}
	\end{subequations}
\end{proposition}

It is well-known that $W^{1,p}(\Omega) \cap L^\infty(\Omega)$ is a Banach algebra for every $1\leq p \leq \infty$; see, e.g., \cite[Proposition~9.4]{brezis2011functional}.
The following proposition shows that the space is also isomorphic to the Banach--Lie group $W^{1,p}(\Omega) \cap \interior L^\infty_+(\Omega)$.

\begin{proposition}
\label{prop:logexpChainRule}
	Let $\Omega$ be an open subset of $\mathbb{R}^n$ and $1\leq p \leq \infty$.
	Then $$\ln \colon W^{1,p}(\Omega) \cap \interior L^\infty_+(\Omega) \to W^{1,p}(\Omega) \cap L^\infty(\Omega)$$ and $$\exp \colon W^{1,p}(\Omega) \cap L^\infty(\Omega) \to W^{1,p}(\Omega) \cap \interior L^\infty_+(\Omega)$$ are isomorphisms.
	Moreover,
	\begin{equation}
		\nabla \ln u = \frac{1}{u}\nabla u
		\qquad
		\text{and}
		\qquad
		\nabla \exp \psi = \exp \psi \nabla \psi
		\,,
	\end{equation}
	for all $u \in W^{1,p}(\Omega) \cap \interior L^\infty_+(\Omega)$ and $\psi \in W^{1,p}(\Omega) \cap L^\infty(\Omega)$.
																		\end{proposition}
\begin{proof}
	We prove $\ln \colon W^{1,p}(\Omega) \cap \interior L^\infty_+(\Omega) \to W^{1,p}(\Omega) \cap L^\infty(\Omega)$ and $\nabla \ln u = 1/u{\nabla u}$ for the case that $\Omega$ is bounded.
	The corresponding statements for the exponential map are treated similarly. Similar to the proof of \cite[Proposition~9.4]{brezis2011functional}  and \cite[Corollary~8.10]{brezis2011functional},
	the case for $p = \infty$ needs to be considered separately. Hence, assume first that $1 \le p < \infty$.
	\smallskip

	\noindent\textsl{Step 0.}
	Let $u \in W^{1,p}(\Omega) \cap \interior L^\infty_+(\Omega)$.
	By~\cref{prop:Equivalence} we know that $\ln u \in L^\infty(\Omega)$ and, moreover, there exists $\epsilon > 0$ such that $\epsilon \leq u(x) \leq 1/\epsilon$ at a.e.~$x\in\Omega$.
	We now follow the proof technique used for \cite[Proposition~9.4]{brezis2011functional} to show that $\ln u \in W^{1,p}(\Omega)$.
	\smallskip

	\noindent\textsl{Step 1.}
	The first step involves constructing a sequence $u_k\in C^\infty_c(\Omega)$ such that
	\begin{subequations}
	\label{eq:logexpChainRule_Properties}
	\begin{alignat}{3}
	\label{eq:logexpChainRule_Property1}
		u_k &\to u \quad &&\text{in } L^p(\Omega) \text{~~and~pointwise~a.e.~in }\Omega
		\,,\\
	\label{eq:logexpChainRule_Property2}
		\nabla u_k &\to \nabla u \quad &&\text{in } [L^p(\omega)]^n ~\fa \omega \subset \subset \Omega
		\,.
	\end{alignat}
	Furthermore,
	\begin{equation}
	\label{eq:logexpChainRule_Property3}
		\|u_k\|_{L^\infty(\Omega)} \leq \|u\|_{L^\infty(\Omega)}
	\end{equation}
	and, for all $\omega \subset\subset \Omega$, it holds that
	\begin{equation}
	\label{eq:logexpChainRule_Property4}
		\|1/u_k\|_{L^\infty(\omega)} \leq \|1/u\|_{L^\infty(\Omega)}
		\,,
	\end{equation}
	\end{subequations}
	once $k$ is sufficiently large.
	For simplicity, we choose to focus on the case where $\Omega$ is bounded.
	This step may be modified by multiplying $u_k$ with a sequence of smooth cut-off functions to treat the case where $\Omega$ is unbounded; cf.~\cite[Proof of Theorem~9.2]{brezis2011functional}.

	Begin by defining
	\begin{equation}
	\overline{u}(x) =
		\begin{cases}
			u(x) & \text{if~} x \in \Omega\,,\\	
			0 & \text{if~} x \in \Omega\setminus \mathbb{R}^n	
			\,,
		\end{cases}
	\end{equation}
	and set $u_k = \rho_k \ast \overline{u}$, where $\rho_k \in C^\infty_c(\mathbb{R}^n)$ is a sequence of mollifier functions satisfying
	\begin{equation}
		\supp \rho_k \subset \overline{B(0,1/k)}
		\,,
		\quad
		\int_{\mathbb{R}^n} \rho_k = 1
		\,,
		\quad
		\rho_k \geq 0
		~~\text{a.e.~in } \mathbb{R}^n
		\,.
	\end{equation}
	Notice that $u_k(x) \leq 1/\epsilon$ at a.e.~$x\in\Omega$ since
	\begin{equation}
	 		 		 		 		 	u_k(x)
	 	=
	 	\int_{\mathbb{R}^n} \overline{u}(x-y) \rho_k(y) \dd y
	 	\leq
	 	\|u\|_{L^\infty(\Omega)} \int_{\mathbb{R}^n} \rho_k(y) \dd y
	 	\leq
	 	\frac{1}{\epsilon}
	 	\,.
	\end{equation}
	This proves~\cref{eq:logexpChainRule_Property3}.

			Now, take $\omega \subset \subset \Omega$ and let $\delta > 0$ be chosen small enough so that the open cover $\bigcup_{x\in \overline{\omega}} B(x,\delta)$ is contained in $\Omega$.
	Then, for all $k > 1/\delta$ and a.e.~$x\in \omega$, we have that
	\begin{equation}
	\label{eq:logexpChainRule_LowerBound}
		\epsilon
	 	=
	 	\int_{B(x,\delta)} \epsilon\,\rho_k(y) \dd y
	 	\leq
		\int_{B(x,\delta)} \overline{u}(x-y) \rho_k(y) \dd y
		=
		u_k(x)
		\,.
	\end{equation}
	We have thus shown~\cref{eq:logexpChainRule_Property4}.
	Properties~\cref{eq:logexpChainRule_Property1,eq:logexpChainRule_Property2} are proven for this sequence in \cite[Theorem~9.2]{brezis2011functional}.
	\smallskip

	\noindent\textsl{Step 2.}
	The next step is to consider a test function $\varphi \in C^1_c(\Omega)$.
	Observe that
	\begin{equation}
	\label{eq:logexpChainRule_integrationbyparts}
		\int_\Omega \ln(u_k) \nabla \varphi \dd x
		=
		-\int_\Omega (1/u_k \nabla u_k) \varphi \dd x
		\,.
	\end{equation}
													Let $\omega = \supp \varphi \subset\subset \Omega$ denote the support of $\varphi$.
	Clearly, $\ln u_k(x) \to \ln u(x)$ at a.e.~point $x\in \omega$.
	Moreover, it is a straightforward exercise to show that $|\ln u_k(x)| \leq \max\{ \ln \|u\|_{L^\infty(\Omega)}, \ln \|1/u\|_{L^\infty(\Omega)} \}$ at a.e.~$x\in\omega$.
	Therefore, by the dominated convergence theorem, we have that
	\begin{equation}
	\label{eq:logexpChainRule_limit1}
		\lim_{k\to\infty}
		\int_\Omega \ln(u_k) \nabla \varphi \dd x
		=
		\int_\Omega \ln(u) \nabla \varphi \dd x
		\,.
	\end{equation}

	To treat the right-hand side of~\cref{eq:logexpChainRule_integrationbyparts}, we apply a converse of the dominated convergence theorem to the sequence $\nabla u_k \to \nabla u$ given in~\cref{eq:logexpChainRule_Property2}.
	In particular, by \cite[Theorem~4.9]{brezis2011functional}, we know that there exists a subsequence $\{\nabla u_{k_l}\}_{l=1}^\infty$ and a function $h \in L^p(\omega)$ such that
	\begin{equation}
		| \nabla u_{k_l}(x) | \leq h(x)
		~~
		\fa l \geq 0
		\text{~and a.e.~} x \in \omega
		\,.
	\end{equation}
	Next, we use~\cref{eq:logexpChainRule_Property4} to conclude that $| 1/u_k \nabla u_{k_l}(x) | \leq \|1/u\|_{L^\infty(\Omega)} h(x) \in L^p(\omega)$.
	In turn, dominated convergence theorem implies that
	\begin{equation}
	\label{eq:logexpChainRule_limit2}
		\lim_{l\to\infty}
		\int_\Omega (1/u_{k_l} \nabla u_{k_l}) \varphi \dd x
		=
		\int_\Omega (1/u \nabla u) \varphi \dd x
	\end{equation}
	because $1/u_k \nabla u_{k_l}(x) \to 1/u \nabla u(x)$ as $l \to \infty$ at a.e.~point $x \in \omega$.
	The identity $\nabla \ln u = 1/u{\nabla u}$ immediately follows from~\cref{eq:logexpChainRule_integrationbyparts,eq:logexpChainRule_limit1,eq:logexpChainRule_limit2}.~

																																																																															For $p = \infty$, we can proceed analogously to the end of the proof of \cite[Corollary~8.10]{brezis2011functional}. The key observations are that 
	$\ln u, u^{-1}, \nabla u \in L^{\infty}(\Omega)$ holds and that the chain rule derived 
	above for $p < \infty$ works on compact subsets of $\Omega$.
\end{proof}

\subsection{Regularity of the entropy functional} \label{sub:regularity_of_H}

One of the important facts that arise from \cref{prop:Equivalence} is that 
$u \in \interior L^\infty_+(\Omega)$
implies $u \ge \| 1/u \|^{-1}_{L^\infty}$. Indeed, this property allows us to differentiate the negative entropy function
\begin{equation}
	S(u) =
	\begin{cases}
	\displaystyle
		\int_\Omega u \ln u - u \dd x & \text{if } u \in L^1_+(\Omega)\,,\\
		+\infty & \text{ otherwise, }
	\end{cases}
	\label{eq:NegEntropy}
\end{equation}
on the open set $\interior L^\infty_+(\Omega)$ with variations in $L^{\infty}(\Omega)$.
We proceed now with a proof of \Cref{lem:EntropyDifferentiability}, which makes use of the classical integral form of 
the mean value theorem for continuously differentiable functionals $g$ on a Banach space,
\[
g(u + v) - g(v) = \int_0^1 \frac{\dd}{\dd\sigma}[ g(u + \sigma v)] \dd\sigma
\,.
\]

\begin{proof}[Proof of \Cref{lem:EntropyDifferentiability}]

	\textsl{Case \ref{item:EntropyDifferentiability_Part_1}: $1\leq p \leq \infty$.} For $p =1$, the properties of strict convexity and lower semicontinuity 
	are easily verified from the properties of the scalar negative entropy function that generates $S$. For instance, strict convexity and the monotoncity of the 
	Lebesgue integral yield strict convexity of $S$ and lower semicontinuity follows from Fatou's lemma. Otherwise, we refer the reader to
	the seminal works  \cite{borwein1994strong,bauschke2001essential}, see e.g., \cite[Lemma 3.1]{borwein1994strong}. 
	Since $\Omega \subset \mathbb{R}^n$ is bounded, the continuous embedding of $L^p(\Omega)$ into $L^1(\Omega)$ imply these same properties for all $p \in (1,\infty]$.
		\smallskip
	
	\noindent\textsl{Case \ref{item:EntropyDifferentiability_Part_2}: $1< p \leq \infty$.}
	Our proof continues by considering the Nemytskii operator induced by the real-valued function
\[
\hat{s}(x) :=  x \ln |x| - x\,.
\]
	We will show that it is a continuous map from $L^p(\Omega)$ to $L^1(\Omega)$ when $p > 1$.
	In doing so, we first note that $\hat{s}$ is continuous when viewed as a real-valued function $x \mapsto x \ln |x| - x$ with $x\in\mathbb{R}$ and, moreover, for any $p > 1$, there exists a constant $C(p)$ such that
	\begin{equation}
	\label{eq:GrowthCondition}
		|\hat{s}(x)| \leq  C(p) + |x|^p
		\,.
	\end{equation}
	Note that $C(p)$ exists on the one hand since $|\hat{s}(x)| \le 1$ for $x \in [-e,e]$. Moreover, for $x \in (e,\infty)$ with $x \to +\infty$, we have $ |\hat{s}(x)|/x^p \to 0$ for all $p \in (1,\infty)$. By symmetry, the same argument holds for $x \in (-\infty,e)$ with $x \to -\infty$. Therefore, by Krasnosel'skii's theorem (see, e.g., \cite[Chapter~1, Theorem~2.2]{ambrosetti1995primer}), $\hat{s}\colon L^p(\Omega) \to L^1(\Omega)$ is continuous for $p \in (1,\infty)$. Clearly, if we restrict $\hat{s}$ to $L^p_+(\Omega)$, then we have a continuous mapping $\hat{s}|_{L^p_+(\Omega)}$ on $L^p_+(\Omega)$. This function coincides with 
	\begin{equation}
		s(x) = \left\{\begin{array}{cc}
				x \ln x - x, & x > 0,\\
				0, & x = 0,\\
				+\infty, & \text{ otherwise, }
				\end{array}
				\right.
	\end{equation}
	on $L^p_+(\Omega)$. Hence, continuity of $S$ on $L^p_+(\Omega)$ now follows from the continuity of the Lebesgue integral $u \mapsto \int_\Omega u \dd x$ and the fact that the composition of two continuous functions is also continuous. Finally, suppose $\left\{u_k\right\} \subset L^{\infty}_+(\Omega)$ converges to $u$ in $L^{\infty}(\Omega)$.
	Then $\left\{u_k\right\} \subset L^p_+(\Omega)$ for every for every $p \in (1,\infty)$ and, moreover, $u_k \to u$ in $L^p(\Omega)$ because $\Omega$ is a bounded domain.
		Consequently, $S(u_k) \to S(u)$ as $k \to +\infty$, as conjectured.
	\smallskip

	\noindent\textsl{Case \ref{item:EntropyDifferentiability_Part_3}: $p = \infty$.}
	In order to show that $S$ is Fr\'echet differentiable on $\interior L^{\infty}_+(\Omega)$ with respect to the $L^{\infty}(\Omega)$ topology, we will first prove that $S$ is G\^ateaux differentiable on $\interior L^{\infty}_+(\Omega)$ and, subsequently, that the G\^ateaux derivative $S^\prime_{\mathrm{G}}$ is continuous on $\interior L^{\infty}_+(\Omega)$.
	Fr\'echet differentiability of $S$ will then follow from \cite[Theorem~1.9]{ambrosetti1995primer}.

	To show that $S$ is G\^ateaux differentiable on $\interior L^{\infty}_+(\Omega)$, we must show that for any fixed $u \in \interior L^{\infty}_+(\Omega)$ and $v \in L^\infty(\Omega)$,
	\begin{equation}
		\lim_{\tau\to 0}
		\int_\Omega 
		\frac{s(u + \tau v) - s(u)}{\tau}
		\dd x
		=
		\int_\Omega 
		v\ln(u)
		\dd x
		\,.
	\end{equation}
	First observe that for almost every $x\in\Omega$, we have pointwise convergence of the associated integrands, namely,
	\begin{equation}
		\lim_{\tau\to 0}
		\frac{s(u(x) + \tau v(x)) - s(u(x))}{\tau}
		=
		v(x)\ln (u(x))
		\,.
	\label{eq:DCT1}
	\end{equation}
	Next, we know from the proof of \Cref{prop:Equivalence} that $u \ge \|\frac{1}{u}\|^{-1}_{L^{\infty}}$. This implies that for sufficiently small $\tau$, $u + \tau v >  \|\frac{1}{u}\|^{-1}_{L^{\infty}}/2$ holds a.e., and we have 
	\begin{equation}
		s(u + \tau v) - s(u)
		=
		\int_{u}^{u+\tau v} \ln \sigma  \dd \sigma
				=
		\tau v \int_0^1 \ln( u +\sigma\tau v ) \dd \sigma
		\,.
	\end{equation}
	The critical step is to see that for the $u$ and $v$ fixed above, we may find $w \in L^\infty(\Omega)$ where
	\begin{equation}
	\label{eq:EntropyDifferentiability_critical}
		v = u w
		\,.
	\end{equation}
	As such, for all sufficiently small $\tau$, we may rewrite
	\begin{equation}
		s(u + \tau v) - s(u)
		=
		\tau v\, \Big(
		\ln u + \int_0^1 \ln( 1 + \sigma\tau w ) \dd \sigma
		\Big)
		,
	\end{equation}
	and, consequently,
	\begin{equation}
		\bigg|
		\frac{s(u + \tau v) - s(u)}{\tau} - v\ln u
		\bigg|
		=
		\bigg|v\int_0^1 \ln( 1 + \sigma\tau w ) \dd \sigma \bigg|
		\leq
		|v| |\ln( 1 + \tau w )|
		.
	\label{eq:DCT2}
	\end{equation}
	To arrive at an upper bound that is independent of $\tau$, we use the following well-known inequality:
	\begin{equation}
		\frac{x}{x+1} \leq \ln(1+x)  \leq x
		.
	\end{equation}
	In turn, for all $|\tau| < (2\|w\|_{L^\infty})^{-1}$,
	\begin{equation}
		|\ln( 1 + \tau w )|
		\leq
		1
		\,.
	\label{eq:DCT3}
	\end{equation}
	The next argument owes to the function $|v|$ belonging to $L^1(\Omega)$ because $\Omega$ is bounded.
	Indeed, by \cref{eq:DCT1,eq:DCT2,eq:DCT3}, the dominated convergence theorem provides us with the following well-defined G\^ateaux derivative:
	\begin{equation}
		\langle S^\prime_{\mathrm{G}}(u), v\rangle
		=
		\lim_{\tau\to 0}
		\int_\Omega
		\frac{s(u + \tau v) - s(u)}{\tau}
		\dd x
		=
		\int_\Omega
		v\ln u
		\dd x
		.
	\end{equation}

	It remains to show that $S^\prime_{\mathrm{G}}\colon  \interior L^{\infty}_+(\Omega) \subset L^\infty(\Omega) \to [L^\infty(\Omega)]^\prime$ is continuous.
	To this end, let $u \in \interior L^{\infty}_+(\Omega)$ and consider any sequence $\{u_k\}$ in $\interior L^{\infty}_+(\Omega)$ where $u_k \to u$ in $L^\infty(\Omega)$.
	Consequently, we know there exists $C > 0$ such that $\|u_k\|_{L^\infty} \leq C$ and $u_k(x) \to u(x)$ for almost every $x\in\Omega$.
	Clearly,
	\begin{equation}
		|\langle S^\prime_{\mathrm{G}}(u) - S^\prime_{\mathrm{G}}(u_k), v\rangle|
		\leq
		\int_\Omega
		|v| |\ln |u/u_k||
		\dd x
		\leq
		\|v\|_{L^\infty}
		\int_\Omega
		|\ln |u/u_k||
		\dd x
		\,,
	\end{equation}
	where $|\ln |u/u_k|| \leq |\ln|u|| + |\ln C| \in L^1(\Omega)$.
	Therefore,
	\begin{equation}
		\|S^\prime_{\mathrm{G}}(u) - S^\prime_{\mathrm{G}}(u_k)\|_{[L^\infty]^\prime}
		\leq
		\int_\Omega
		|\ln |u/u_k||
		\dd x
	\end{equation}
	and, by the dominated convergence theorem,
	\begin{equation}
		\lim_{k\to\infty}
		\|S^\prime_{\mathrm{G}}(u) - S^\prime_{\mathrm{G}}(u_k)\|_{[L^\infty]^\prime}
		\leq
		\int_\Omega
		\lim_{k\to\infty}
		|\ln |u/u_k||
		\dd x
		=
		0
		,
	\end{equation}
	as necessary.
	\smallskip

	Finally, we show that $\nabla {S}(u)$ exists and is equal to $\ln u$.
	Let $u \in \interior L^{\infty}_+(\Omega)$. By \cref{eq:glequivs2}
	$\ln u \in L^\infty(\Omega)$.
	Next, we see that $\|S^\prime(u)\|_{[L^\infty]^\prime} \leq \|\ln u\|_{L^1}$, since
	\begin{equation}
		\langle S^\prime(u), v \rangle
		=
		\int_\Omega v \ln u \dd x \leq \|v\|_{L^\infty}\|\ln u\|_{L^1}
		.
	\end{equation}
	Moreover, since  the difference of characteristic functions associated with the sign of 
	$\ln u $, denoted by $\mathbbm{1}_{\{\ln u > 0\}} -  \mathbbm{1}_{\{\ln u < 0\}}$, are 
	in $L^{\infty}(\Omega)$ with norm one, we obtain the lower bound on the dual norm 
	of $S'(u)$:
	\begin{equation}
		\|S^\prime(u)\|_{[L^\infty]^\prime} 
		\geq
		\int \mathbbm{1}_{\{\ln u > 0\}} \ln u \dd x
		-
		\int \mathbbm{1}_{\{\ln u < 0\}} \ln u \dd x
		=
		\|\ln u\|_{L^1}
		,
	\end{equation}
	and so $\|S^\prime(u)\|_{[L^\infty]^\prime} = \|\ln u\|_{L^1}$.
	\end{proof}

We complete this subsection with a proof of the gradient representation theorem for the shifted entropy functional, $S_{\phi}(u) = S(u - \phi)$.

\begin{proof}[Proof of~\Cref{cor:shift-ent}]

By \Cref{lem:EntropyDifferentiability}, $S$ is continuous on $L^\infty_+(\Omega)$.  
If $\phi \in L^\infty(\Omega)$, then the shift operator $T_{\phi}u := u - \phi$ is continuous on $L^\infty(\Omega)$
for $u \in L^{\infty}(\Omega)$ with $u \ge \phi$, as well; 
continuity of the composition follows. 

As argued in \Cref{lem:EntropyDifferentiability}, 
$S$ is strictly convex on $L^{\infty}_+(\Omega)$.
Taking 
$w_i \in L^{\infty}_{\phi,+}(\Omega)$,
with $i =1,2$ and $w_1 \ne w_2$, we see that
$T_{\phi}w_i = w_i - \phi \ge 0$ a.e.\ for $i = 1,2$. Moreover, $T_{\phi} w_1 = T_{\phi} w_2$ iff $w_1 = w_2$ and for $\lambda \in (0,1)$ we  have
$T_{\phi}(\lambda w_1 + (1-\lambda) w_2) = 
 \lambda (w_1 -  \phi) + (1 - \lambda)(w_2 - \phi) = 
 \lambda T_{\phi} w_1 + (1-\lambda) T_{\phi} w_2$. Then since 
 $T_{\phi}w_i \in L^{\infty}_+(\Omega)$
    for $i=1,2$, the
 strict convexity of the composition follows. 
 
We proceed with the characterization of $\interior L^\infty_{\phi,+}(\Omega)$.   
  First, we can easily show the elementary properties 
 $
 L^\infty_{\phi,+}(\Omega) = \phi + L^{\infty}_+(\Omega)
 $
and
 \[
 \interior L^\infty_{\phi,+}(\Omega)  = \interior ( \phi + L^{\infty}_+(\Omega)) = \phi + \interior  L^{\infty}_+(\Omega).
 \]
 Therefore,  $w \in \interior L^\infty_{\phi,+}(\Omega)$ implies $w - \phi \in  \interior L^{\infty}_+(\Omega)$. By \Cref{prop:Equivalence},
  $w \ge \phi$ and $\essinf (w - \phi) > 0$. Conversely, if $w \in L^{\infty}(\Omega)$ such that $w \ge \phi$ and $\essinf (w - \phi) > 0$, then \Cref{prop:Equivalence} implies $w - \phi \in \interior L^{\infty}_{+}(\Omega)$. Hence, $w \in  \interior L^\infty_{\phi,+}(\Omega)$
 
Let 
$
 w_1  \in 
 \interior L^{\infty}_{\phi,+}(\Omega).
$
 Then $ w_1 - \phi \in \interior L^{\infty}_{+}(\Omega)$. It follows from 
 \Cref{lem:EntropyDifferentiability} that $S_{\phi}$ is Fr\'echet differentiable at $w_1$.
 
 The formula for the derivative of $S'_{\phi}$ can be viewed as an application of the chain 
 rule. Indeed, $S_{\phi} = S \circ T_{\phi}$, $S$ is differentiable with respect to the
 $L^{\infty}$-norm at $T_{\phi}w$ with $w
 \in \interior L^{\infty}_{\phi,+}(\Omega)$ and $T_{\phi}$ is differentiable with respect to the
 $L^{\infty}$-norm at (any) $w \in L^{\infty}$ with derivative $A'_{\phi}(w)$ given by
 the identity on $L^{\infty}$. Therefore, we have
 \cref{eq:EntropyDifferentiability_variations_inhom}.
Since $u - \phi \in \interior L^{\infty}_{+}(\Omega)$  the rest of the computations for the gradient
remain unchanged; in particular, we obtain \cref{eq:EntropyDifferentiability_gradient_inhom}
and \cref{eq:EntropyDifferentiability_primal_inho}.
\end{proof}

\subsection{Deriving the entropic Poisson equation} \label{sub:Characterization}

We devote this subsection to proofs of the characterization theorem and its corollary for non-zero obstacles (i.e., \Cref{cor:PrimalProblem_inhom}).
This section also includes a short remark about a weak maximum principle for the entropic Poisson equation that arises from the first of these proofs.

\begin{proof}[Proof of \Cref{thm:PrimalProblem}]
The proof proceeds in five steps. We suppress the $w$-argument in $A_{\alpha}$ as it plays no role here.
\medskip 

\noindent\textsl{Step 1.}
Show that there exists a unique solution.

The proof of existence is standard. We sketch the main points here; 
see, e.g., \cite[Chap. 3.2]{attouch2014variational} for details.
By \cite[Lem. 3.30]{ern2021finite}, we have that
\begin{equation}\label{eq:pw-ineq}
	\| v \|_{L^2} - \int_{\partial \Omega} g \dd\mathcal{H}_{n-1} \le c\| \nabla v \|_{L^2}
	\,,
	~
	\fa v \in H^1_g(\Omega)
	\,,
\end{equation}
for some constant $c$ that depends on $\Omega$.
Clearly, $A_{\alpha}$ is finite on $K$. This yields a minimizing sequence $\{u_k\}$.
The form of $A_{\alpha}$ consequently yields the boundedness of $\{\| \nabla u_k \|_{L^2}\}$. 
Combined with \cref{eq:pw-ineq} we deduce boundedness of $\{u_k\}$ in $H^1(\Omega)$.
We can readily show that $A_{\alpha}$ is weakly lower-semicontinuous and $K$ weakly sequentially closed
both in $H^1(\Omega)$.
This yields the existence of a minimizer $u \in K$.
The minimizer $u$ is unique because $A_{\alpha}$ is strictly convex on $K$.
\medskip

\noindent\textsl{Step 2.} Show that $u \leq \max\{\|g\|_{L^\infty(\partial\Omega)}, \exp(\|\ln w + \alpha f\|_{L^\infty(\Omega)})\}$.
	
	For all $N > 1$, define the set $R_N = \{ x \in \Omega \mid u(x) > N \}$.
	By way of contradiction, we assume that $|R_N| > 0$ for all $N > 1$.
	Now, consider the following function in $L^\infty$:
	\begin{equation}
		u_N(x)
		=
		\min\{ N, u(x) \} \ge 0.
	\end{equation}
	We claim that if $N > \esssup_{x\in\partial \Omega} g(x)$, then $u_N \in K$.
																											{
		Begin by choosing $\{u_m\} \subset C^{1}(\overline{\Omega})$ 
		such that $u_m \to u$ (strongly in $H^1(\Omega)$)
		and define $u^N_m := \min\{N,u_m\}$.  The existence of $u_m$ 
		follows from the assumption
		that $\partial \Omega$ is Lipschitz; see, e.g., \cite[3.22 Theorem]{adams2003sobolev}.
		Recall that $\gamma$ denotes the trace operator.
		Next let $v \in L^{\infty}(\partial \Omega)$ and consider that
		\[
		\int_{\partial \Omega} \gamma(u^N_m) v \dd\mathcal{H}_{n-1} = \int_{\{\gamma(u_m) \le N\}} \gamma(u_m) v \dd\mathcal{H}_{n-1} + N\int_{\{\gamma(u_m) > N\}} v \dd\mathcal{H}_{n-1}
		\,.
		\] 
		Along a subsequence, denoted still by $m$, $\gamma(u_m)$ converges pointwise almost everywhere to $\gamma(u) = g$. Then, by hypothesis,
		the sequence of characteristic functions $f_m := \chi_{\{\gamma(u_m) > N\}} \to 0$ pointwise almost everywhere. It follows from Lebesgue's
		dominated convergence theorem that 
		\[
		N\int_{\{\gamma(u_m) > N\}} v \dd\mathcal{H}_{n-1} \to 0 \text{~~as~} m \to +\infty.
		\]
		Continuing, we appeal to the proof of \cite[Thm A.1]{kinderlehrer2000introduction} and, e.g., \cite[Cor. 18.4]{leoni2009first}, to argue that
		$\min\{N,u_m\} \to \min\{N,u\}$ weakly in $H^1(\Omega)$ and 
		$\gamma(\min\{N,u_m\}) \to \gamma(\min\{N,u\})$ strongly in $L^2(\partial \Omega)$, which in turn yields
		\[
		\lim_{m\to+\infty}\int_{\{\gamma(u_m) \le N\}} \gamma(u_m) v \dd\mathcal{H}_{n-1}
		=\int_{\partial \Omega} \gamma(u_N) v \dd\mathcal{H}_{n-1}
		. 
		\]
		Since $v$ is essentially bounded, we have 
		\[
		|\gamma(u_m)v - \gamma(u)v| = |v| |\gamma(u_m) - \gamma(u)| \le \| v\|_{L^{\infty}}|\gamma(u_m) - \gamma(u)|
		\,.
		\]
		Hence, $\gamma(u_m)v$ converges strongly in $L^2(\partial \Omega)$ to $\gamma(u) v$. Similar to the above, we can argue that 
		$f'_m := \chi_{\{\gamma(u_m) \le N\}} \to \chi_{\partial \Omega} = 1$ in $L^2(\partial \Omega)$. It follows that
		\[
		\int_{\partial \Omega} \gamma(u_N) v \dd\mathcal{H}_{n-1} = 
		\int_{\partial \Omega} \gamma(u) v \dd\mathcal{H}_{n-1} = 
		\int_{\partial \Omega} g v \dd\mathcal{H}_{n-1}.
		\]
		By the density of $L^{\infty}(\partial \Omega)$ in $L^2(\partial \Omega)$ and the fundamental lemma of the calculus of variations \cite[Theorem~1.32]{ern2021finite}, 
		we deduce $\gamma(u_N) = g$ a.e.~on $\partial \Omega$. }Consequently, $u_N \in K$ and for sufficiently large $N$, it holds that
	\begin{equation}
		D(u,w) + \alpha E(u)
		<
		D(u_N,w) + \alpha E(u_N)
	\end{equation}	
	because $u$ is the unique global minimizer of $A_{\alpha}$ over $K$. 	
	
	Note, however, that
	\begin{align*}
		\alpha E(u)
		-
		\alpha E(u_N)
		&=
		\alpha\int_{R_N} 
		\frac{1}{2}|\nabla u|^2 - (u - N) f \dd x
	\end{align*}
	and
	\begin{align*}
		D(u,w)
		-
		D(u_N,w)
    		&=
    		\int_{R_N} u \ln u - N \ln N - (1 + \ln w)(u  - N) \dd x
    		\\
		&=
		\int_{R_N} (u - N) \bigg(
			\int_0^1 \ln (N + t(u - N)) \dd t
			- \ln w
		\bigg) \dd x
		\\
		&\geq
		\int_{R_N} (\ln N - \ln w) (u - N) \dd x
				\,.
	\end{align*}
	Combining these observations, we see that
	\begin{multline*}
	D(u,w) + \alpha E(u) - 	D(u_N,w) - \alpha E(u_N)
	\ge \\
	\frac{\alpha}{2} \| \nabla u\|^2_{L^2(R_N)} + (\ln N - \alpha f - \ln w,u- N)_{L^2(R_N)}.
	\end{multline*}
	Therefore, for any $N > \max\{\|g\|_{L^\infty(\partial\Omega)},\exp(\|\alpha f + \ln w\|_{L^\infty(\Omega)})\}$, we have
	\begin{equation}
		D(u,w) + \alpha E(u)
		>
		D(u_N,w) + \alpha  E(u_N)
		,
	\end{equation}
	which contradicts the optimality of $u$.
	Hence, there exists some $N_0 > 0$ such that $|R_N| = 0$ for all $N > N_0$, and, in turn, $u \in L^\infty$
	with $u \le \max\{\|g\|_{L^\infty(\partial\Omega)},\exp(\|\alpha f + \ln w\|_{L^\infty(\Omega)})\}$.
	\medskip
	
\noindent\textsl{Step 3.}
	Show that $u \geq \min\{\essinf_{x\in\partial \Omega} g(x), \exp(-\|\ln w + \alpha f\|_{L^\infty(\Omega)})\}$.
	Thus, $u \in H^1_g(\Omega) \cap \interior L^\infty_+(\Omega)$.

	For all $\epsilon > 0$, define the set $S_\epsilon = \{ x \in \Omega \mid u(x) < \epsilon \}$.
	By way of contradiction, we assume that $|S_\epsilon| > 0$ for all $\epsilon > 0$.
	Now, consider the following function in $H^1(\Omega)\cap \interior L^\infty_+(\Omega)$:
	\begin{equation}
		u_\epsilon(x)
		=
										\max\{ \epsilon, u(x) \}.
	\end{equation}
	The fact that $u_{\epsilon} \in \interior L^\infty_+$ follows from \Cref{prop:Equivalence}.
	
	Continuing, we can emulate the arguments of Step 2.~above to show that if 
	$\epsilon < \essinf_{x\in\partial \Omega} g(x)$, then $u_\epsilon \in H^1_g \cap \interior L^\infty_+$.
	The steps and justifications are almost identical and are therefore omitted.
	In turn, for sufficiently small $\epsilon > 0$, it holds that
	\begin{equation}
	\label{eq:GlobalMinStep}
		D(u,w) + \alpha E(u)
		\leq
		D(u_\epsilon,w) + \alpha E(u_\epsilon)
		.
	\end{equation}
		As above, we obtain the lower bound
	\begin{multline*}
	D(u,w) + \alpha E(u) - D(u_{\epsilon},w) - \alpha E(u_{\epsilon})
	\ge \\
	\alpha \| \nabla u\|^2_{L^2(S_{\epsilon})} + (\|\alpha f + \ln w\|_{L^{\infty}} + \ln \epsilon,u- \epsilon)_{L^2(S_{\epsilon})}
	\end{multline*}	
	As a result, once $\epsilon < \exp(-\| \alpha f + \ln w\|_{L^{\infty}})$, we again contradict the optimality of $u$. 
	Thus, there exists some $\epsilon_0 > 0$ such that $|S_\epsilon| = 0$ for all $\epsilon < \epsilon_0$ and $u \in H^1_g(\Omega)\cap \interior L^\infty_+(\Omega)$ by~\cref{prop:Equivalence}.
	\medskip	
	
\noindent\textsl{Step 4.} Derive the variational equation.

	Let $t > 0$ and $v \in K \cap L^{\infty}(\Omega)$. Then, by definition,
	\[
	\alpha E(u) + D(u,w) \le \alpha E(v) + D(v,w).
	\]
	Clearly, $u + t(v - u) \in K$,
	and consequently,
	\[
	0 \le \frac{ \alpha E(u + t (v - u)) - \alpha E(u)}{t} + \frac{ D(u + t (v - u),w) - D(u,w)}{t}
	\]
	Since, $u \in H^1(\Omega) \cap \interior L^{\infty}_+(\Omega)$ and $v - u \in H^1(\Omega) \cap L^{\infty}(\Omega)$, 
	\Cref{lem:EntropyDifferentiability} allows us to expand and pass to the limit as $t \downarrow 0$. This yields the 
	variational inequality:
	\begin{equation}
	\label{eq:Step4_VI}
	0 \le \alpha E'(u)(v - u) + S'(u)(v-u) - (\ln w,v-u)
	\end{equation}
	for all $v \in H^1_g(\Omega) \cap L^{\infty}_+(\Omega)$.
	As a result of Step 3, there exists a constant $\epsilon >0$ such that $u \ge \epsilon$ a.e.\ in $\Omega$.
	In particular, for all $y \in H^1_0(\Omega) \cap L^{\infty}(\Omega)$ there exists a sufficiently small $\delta > 0$ such that $u + ty \in H^1_g(\Omega) \cap L^{\infty}_+(\Omega)$ for all $t \in [-\delta,\delta]$.
	Setting $v = u \pm \delta y$ in the variational inequality~\cref{eq:Step4_VI} and rescaling the result by $1/\delta$, we arrive at the following
	variational \emph{equation} with test functions $y \in H^1_0(\Omega) \cap L^{\infty}(\Omega)$:
	 \[
	0 = \alpha E'(u)y+ S'(u)y- (\ln w,y).
	\]
										The first summand is equivalent to $(\alpha\nabla u,\nabla y) - (\alpha f,y)$ and
	the second and third summands together have the form $(\ln u - \ln w,y)$. 
	Since $u,w \in \interior L^{\infty}_+(\Omega)$, the map $y \mapsto (\ln u - \ln w,y)$ defines a 
	bounded linear functional on $H^1_0(\Omega)$. Finally, by virtue of the inclusion 
	$C^{\infty}_c(\Omega) \subset H^1_0(\Omega) \cap L^{\infty}(\Omega)$, we deduce
	\begin{equation}\label{eq:ent-poisson}
	(\alpha\nabla u,\nabla y)  + (\ln u,y) = (\alpha f,y) +  (\ln w,y)~\fa y \in H^1_0(\Omega),
	\end{equation}
	as was to be shown.
	\medskip	
	
\noindent\textsl{Step 5.} Prove that the entropic Poisson equation has a unique solution in $H^1_g(\Omega) \cap \interior L^{\infty}_+(\Omega)$.

	Conversely, suppose $u \in H^1_g(\Omega) \cap \interior L^{\infty}_+(\Omega)$ such that \cref{eq:ent-poisson} holds. 
	Then for any $y \in H^1_g(\Omega) \cap L^{\infty}(\Omega)$, setting 
	$v = y - u \in H^1_0(\Omega) \cap L^{\infty}(\Omega) \subset H^1_0(\Omega)$ in \cref{eq:ent-poisson} yields
	\[
	A'_{\alpha}(u)(y - u) = (\alpha\nabla u,\nabla [y-u])  + (\ln u,y-u) - (\alpha f,y-u) -  (\ln w,y-u) = 0.
	\]
	Since $A_{\alpha}$ is differentiable at $u$ with respect to variations in $H^1(\Omega) \cap L^{\infty}(\Omega)$
	and convex on $H^1_g(\Omega)$, we have
	\[
	A_{\alpha}(y) \ge A_{\alpha}(u) + A'_{\alpha}(u)(y -u ) = A_{\alpha}(u) ~\fa y \in H^1_g(\Omega)\cap L^{\infty}(\Omega).
	\]
	Taking the closure of $H^1_g(\Omega) \cap L^{\infty}(\Omega)$ with respect to the $H^1$-norm, it follows from the 
	continuity of $A_{\alpha}$ on $H^1(\Omega)$ that $u$ is a minimizer of $A_{\alpha}$ over $H^1_g(\Omega)$. Uniqueness
	follows from strict convexity of $A_{\alpha}$.
	\end{proof}

\begin{remark}[Maximum principle]
\label{rem:ExplicitBounds}
	Inspecting the proof above, we see that
	\begin{subequations}
	\label{eqs:ExplicitBounds}
	\begin{equation}
		\min\{g_{\min}, \exp(-\|\ln w + \alpha f\|_{L^\infty})\}
		\leq
		u
		\leq
		\max\{g_{\max}, \exp(\|\ln w + \alpha f\|_{L^\infty})\}
		,
	\end{equation}
	or, equivalently,
	\begin{equation}
		\min\{\ln g_{\min}, -\|\ln w + \alpha f\|_{L^\infty}\}
		\leq
		\ln u
		\leq
		\max\{\ln g_{\max}, \|\ln w + \alpha f\|_{L^\infty}\}
		,
	\end{equation}
	\end{subequations}
	where $g_{\min} = \essinf_{x\in\partial \Omega} g(x)$ and $g_{\max} = \esssup_{x\in\partial \Omega} g(x)$.
\end{remark}

\begin{proof}[Proof of \Cref{cor:PrimalProblem_inhom}]
The proof follows that of \Cref{thm:PrimalProblem}. Here, \Cref{cor:shift-ent} plays the same role as \Cref{lem:EntropyDifferentiability}.

Existence and uniqueness
follows the homogeneous case in light of 
the implications of \Cref{cor:shift-ent}.
We only need argue that $K_{\phi}$ is nonempty. Since $g, \phi \in 
H^1(\Omega) \cap C(\overline{\Omega})$ the function
$
v := \max\{g,\phi\} = g + \max\{0,\phi-g\}
$
is in $H^1(\Omega) \cap C(\overline{\Omega})$ and satisfies $v \ge \phi$. The trace of $w$ is merely the
evaluation on the boundary. Then since  $\essinf \gamma(g - \phi) > 0$ on $\partial \Omega$ by assumption, we have 
$\gamma(v) = \gamma(g)$ and consequently $v \in K_{\phi}$.

Setting $\tilde{w} = w - \phi$, we can now readily argue that $u = \tilde{u} + \phi$ where $\tilde{u}$ is the solution of
\begin{multline}\label{eq:subproblem_inhom_shift}
	\min
		\frac{1}{2} \| \nabla \tilde{v} \|^2_{L^2} - (f +  \Delta \phi,\tilde{v}) + \alpha^{-1} D(\tilde{v},\tilde{w})
	\\\text{ over }
	\tilde{v} \in H^1_{g - \phi}(\Omega)
	~~\text{subject to~}
	\tilde{v} \geq 0
	~\text{in~}\Omega
	\,.
\end{multline}
\Cref{thm:PrimalProblem} then guarantees that $\tilde{u}$ solves
	\begin{equation*}
		(\alpha\nabla \tilde{u}, \nabla v)
				+
		(\ln \tilde{u}, v)
		=
		(\alpha f + \alpha \Delta \phi, v) + (\ln \tilde{w}, v)
				\quad
		\text{for all~}
		v \in H^1_0(\Omega)
		.
	\end{equation*}
Substituting $u - \phi = \tilde{u}$ and $w - \phi = \tilde{w}$ yields \cref{eq:PrimalProblemVE_inhom}.
\end{proof}

\subsection{Convergence of the latent variable proximal point method}
\label{sub:proximal_methods}

In this subsection, we establish arbitrary convergence rates for the continuous-level proximal point algorithm~\cref{eq:ConvergenceContinuousLevel_VE} to solve the obstacle problem.
We begin by proving the following lemma.

\begin{lemma}
\label{lem:PrimalSaddlePointEquivalence}

Under the assumptions of~\Cref{thm:PrimalProblem}, the second-order problem
\begin{subequations}
\begin{equation}
\label{eq:Step1_ProofOfConvergenceContinuousLevel_PrimalProblem}
	\text{Find}~
	u\in  H^1_g(\Omega)\cap \interior L^\infty_+(\Omega)
	~\text{such that~}
	-\Delta u + \ln u = f ~~\text{in~}H^{-1}(\Omega)
	\,,
\end{equation}
is equivalent to the saddle-point problem
\begin{equation}
\label{eq:Step1_ProofOfConvergenceContinuousLevel_SaddlePointProblem}
	\text{Find}~
	\tilde{u}\in  H^1_g(\Omega) ~\text{and}~\tilde{\psi} \in L^\infty(\Omega)
	~\text{such that~}
	\left\{
	\begin{alignedat}{4}
		-\Delta \tilde{u} + \tilde{\psi} &= f~~ &&\text{in~}H^{-1}(\Omega)
		\,,
		\\
		\tilde{u} - \exp\tilde{\psi} &= 0 &&\text{in~}L^2(\Omega)
		\,.
	\end{alignedat}
	\right.
\end{equation}
\end{subequations}
More specifically, both problems admit unique solutions that coincide in the sense that $u = \tilde{u}$ and $\ln u = \tilde{\psi}$ a.e~in $\Omega$.
\end{lemma}

\begin{proof}
	First of all, we know from~\Cref{thm:PrimalProblem} that there exists a unique solution to~\cref{eq:Step1_ProofOfConvergenceContinuousLevel_PrimalProblem}.
	Using~\cref{prop:Equivalence}, we know that $\exp \colon L^\infty(\Omega) \to \interior L^\infty_+ (\Omega)$ is an isomorphism.
	Thus, $\tilde{\psi} = \ln u$ and $\tilde{u} = u$ form a solution to~\cref{eq:Step1_ProofOfConvergenceContinuousLevel_SaddlePointProblem}.
	Now, assume that $\tilde{u}\in  H^1_g(\Omega)$ and $\tilde{\psi} \in L^\infty(\Omega)$ form an arbitrary solution to~\cref{eq:Step1_ProofOfConvergenceContinuousLevel_SaddlePointProblem}.
	By the second equation in~\cref{eq:Step1_ProofOfConvergenceContinuousLevel_SaddlePointProblem}, we know that
	\begin{equation}
	\label{eq:PrimalSaddlePointEquivalence_ae}
		\tilde{u} = \exp\tilde{\psi}
		~~
		\text{a.e.~in~}\Omega
		\,.
	\end{equation}
	Now, by~\cref{prop:Equivalence}, we know that $\exp\tilde{\psi} \in \interior L^\infty_+(\Omega)$.
	Thus, $\tilde{u} \in H^1_g(\Omega)\cap \interior L^\infty_+(\Omega)$.
	Moreover, by applying $\ln$ to both sides of~\cref{eq:PrimalSaddlePointEquivalence_ae}, we find that $\tilde{\psi} = \ln \tilde{u}$.
	Thus, $\tilde{u} \in H^1_g(\Omega)\cap \interior L^\infty_+(\Omega)$ solves~\cref{eq:Step1_ProofOfConvergenceContinuousLevel_PrimalProblem}.
	Since the solution of~\cref{eq:Step1_ProofOfConvergenceContinuousLevel_PrimalProblem} is unique, we find that $\tilde{u} = u$ and, in turn, $\tilde{\psi} = \ln u$.
\end{proof}

We now move on to proving the convergence theorem.

\begin{proof}[Proof of~\Cref{thm:ConvergenceContinuousLevel}]
The proof  has three main steps, the first two of which build off of~\Cref{lem:PrimalSaddlePointEquivalence}.
Without loss of generality, we focus on the case where $\phi = 0$.
The statement for general obstacles $\phi \neq 0$ can be recovered by making the change of variables $u - \phi = \tilde{u}$ and $w - \phi = \tilde{w}$ used in the proof of \Cref{cor:PrimalProblem_inhom}.
\smallskip

\noindent\textsl{Step 0.}
By~\Cref{lem:PrimalSaddlePointEquivalence}, the sequence of iterates $u^k$ coming from~\cref{eq:ConvergenceContinuousLevel_VE} and
\begin{equation}
\label{eq:ConvergenceContinuousLevel_VE_primal}
	(\alpha_{k+1}\nabla u^{k+1}, \nabla v)
	+
	(\ln u^{k+1}, v)
	=
	(\alpha_{k+1} f + \ln u^k, v)
	\quad
	\text{for all~}
	v \in H^1_0(\Omega)
	\,
\end{equation}
are equal a.e.~in $\Omega$.
We take advantage of this fact throughout the proof below.
In particular, given $u^0 = \exp \psi^0$ as in the hypotheses, we suppose that sequence $\{u^k\}$ is generated by the proximal point method, where each $u^k$ solves \cref{eq:ConvergenceContinuousLevel_VE_primal}, $k = 1,2,\dots$ By \Cref{thm:PrimalProblem}, $u^k \in H^1_g(\Omega)\cap \interior L^\infty_{+}(\Omega)$ for all $k = 0,1,2,\dots$.
\smallskip

\noindent\textsl{Step 1.}
Inequality~\cref{eq:ConvergenceContinuousLevel_Monotonicity} is proved by exploiting the fact that $D(u,w) \geq 0$ with equality if and only if $u = w$.
In particular,
\begin{equation}
	\begin{aligned}
		E(u^{k+1})
		&\leq
		E(u^{k+1}) + D(u^{k+1},u^k)/\alpha_{k+1}
		\\
		&\leq
		E(u^k) + D(u^k,u^k)/\alpha_{k+1}
		=
		E(u^k)
		\,,
	\end{aligned}
\end{equation}
where the second inequality follows from the optimality of $u_{k+1}$.

\noindent\textsl{Step 2.}
Since $u^j \in H^1_g(\Omega)\cap \interior L^\infty_{+}(\Omega)$ for all $j=1,2,\ldots,k$, the definition of $D$ as a true Bregman distance according to \cref{eq:bregman-trad} is justified and consequently, the three-points identity \cref{eq:CosineIdentity} is as well. This leads to 
	\begin{equation}
	\label{eq:CosineIdentity_proof}
		D(w,u^j) - D(w,u^{j-1}) + D(u^j,u^{j-1})
		=
		\langle S^\prime(u^j) - S^\prime(u^{j-1}), u^j - w \rangle
		\,,
	\end{equation}
	where $w \in H^1_g(\Omega)$ such that $w \ge 0$ a.e.
	Next, notice that~\cref{eq:ConvergenceContinuousLevel_VE_primal} is equivalent to
	\begin{equation}
	\label{eq:RelativeEntropyIdentity}
		\langle S^\prime(u^j),v\rangle  - \langle S^\prime(u^{j-1}),v\rangle
		=
		-\alpha_{j} E^\prime(u^j)v ~\fa v \in H^1_0(\Omega).
	\end{equation}
	Clearly, we have $u^j - w \in H^1_0(\Omega)$. Therefore, \cref{eq:RelativeEntropyIdentity} and the subgradient inequality for $E$ at $u^j$ imply 
	\begin{equation}
	\label{eq:ConvergenceContinuousLevel_SubgradientIdentity}
		\langle S^\prime(u^j) - S^\prime(u^{j-1}), u^j - w \rangle
		=
		\langle \alpha_{j} E^\prime(u^j), w - u^j \rangle
		\leq
		\alpha_{j} E(w) - \alpha_{j} E(u^j)
		\,.
	\end{equation}
	Together,~\cref{eq:CosineIdentity_proof,eq:ConvergenceContinuousLevel_SubgradientIdentity} imply that
	\begin{equation}
	\label{eq:ConvergenceContinuousLevel_Step3}
		D(w,u^k)
		+
		\sum_{j=1}^{k}
			D(u^{j},u^{j-1}) 
			+ \sum_{j=1}^{k}\alpha_{j}[E(u^j) -  E(w)]
		\leq
		D(w,u^0)
		\,.
	\end{equation}
	Now, given $D(w,u^k) \geq 0$, $D(u^{k+1},u^{k}) \geq 0$ for all $k$ and $E(u^k) \leq E(u^j)$ for all $j\leq k$, by~\cref{eq:ConvergenceContinuousLevel_Monotonicity}, along with ~\cref{eq:ConvergenceContinuousLevel_Step3}, we deduce the
	bound
	\begin{equation}
		E(u^k)
		\leq
		E(w) + \frac{D(w,u^0)}{\sum_{j=1}^k \alpha_{j}}
		\,,
	\end{equation}
	for all $w \in H^1_g(\Omega)$ satisfying $w \geq 0$.
	This step is completed by setting $w = u^\ast$ in the inequality above and using strong convexity of $E\colon H^1_0(\Omega) \to \mathbb{R}$.
	In particular, observe that
	\begin{equation}
	\label{eq:ConvergenceContinuousLevel_StrongConvexity}
		\begin{aligned}
			\frac{D(u^\ast,u^0)}{\sum_{j=1}^k \alpha_{j}}
			\geq
			E(u^k)
			-
			E(u^\ast)
			&\geq
			\langle E^\prime(u^\ast), u^k - u^\ast \rangle
			+
			\frac{1}{2}\|\nabla u^\ast - \nabla u^k\|_{L^2}^2
			\\
			&\geq
			\frac{1}{2}\|\nabla u^\ast - \nabla u^k\|_{L^2}^2
			\,,
		\end{aligned}
	\end{equation}
	where, we have used the first-order optimality condition $\langle E^\prime(u^\ast), v - u^\ast \rangle \geq 0$ for all $v \in K$ in the final inequality.
	\smallskip

	\noindent\textsl{Step 3.}
	Finally, we prove the first equality in~\cref{eq:ConvergenceContinuousLevel_Rate}.
	To this end, consider the two equations
	\begin{subequations}
	\begin{equation}
		(\nabla u^k, \nabla v) - (f,v)
		=
		(\lambda^k,v)
		\fa v \in H^1_0(\Omega)
		\,,
	\end{equation}
	and
	\begin{equation}
		(\nabla u^\ast, \nabla v) - (f,v)
		=
		(\lambda^\ast,v)
		\fa v \in H^1_0(\Omega)
		\,.
	\end{equation}
	\end{subequations}
	Combining these two equations, we find that
	\begin{equation}
		\|\lambda^\ast - \lambda^{k}\|_{H^{-1}(\Omega)}
		=
		\sup_{v \in H^1_0(\Omega)}
		\frac{(\lambda^\ast - \lambda^k,v)}{\|\nabla v\|_{L^2(\Omega)}}
		=
		\sup_{v \in H^1_0(\Omega)}
		\frac{(\nabla u^\ast - \nabla u^k,\nabla v)}{\|\nabla v\|_{L^2(\Omega)}}
										\,.
	\end{equation}
	We now find $\|\lambda^\ast - \lambda^{k}\|_{H^{-1}(\Omega)} \leq \|\nabla u^\ast - \nabla u^k\|_{L^2(\Omega)}$ by applying the triangle inequality to the numerator of the third expression above.
	Likewise, we find $\|\nabla u^\ast - \nabla u^k\|_{L^2(\Omega)} \leq \|\lambda^\ast - \lambda^{k}\|_{H^{-1}(\Omega)}$ by considering the candidate function $v = u^\ast - u^k$.~
\end{proof}

We now turn to studying the iteration complexity of LVPP for various step size sequences.
To this end, we first recall the standard definitions of $\mathrm{Q}$- and $\mathrm{R}$-convergence.

\begin{definition}[Convergence orders and rates]
	Let $\mathcal{X}$ be a Banach space with norm $\|\cdot\|_{\mathcal{X}}$.
	We say that a sequence $\{x_k\}_{k=0}^\infty \subset \mathcal{X}$ converges to $x^\ast \in \mathcal{X}$ with order $q \geq 1$ and rate $r \geq 0$ if
	\begin{equation}
		\lim_{k\to\infty}
		\frac{\|x_{k+1} - x^\ast\|_{\mathcal{X}}}{\|x_k - x^\ast\|_{\mathcal{X}}^q}
		=
		r
		\,.
	\end{equation}
	\\\indent
	If $q = 1$ and $r = 1$, then we say $x_k$ converges $\mathrm{Q}$-sublinearly to $x^\ast$.
	If $q = 1$ and $r \in (0,1)$, then we say $x_k$ converges $\mathrm{Q}$-linearly to $x^\ast$.
	If $q>1$ or $q = 1$ and $r = 0$, then we say $x_k$ converges $\mathrm{Q}$-superlinearly to $x^\ast$.
	\\\indent
	If $\|x_k - x^\ast\|_{\mathcal{X}} \leq \epsilon_k$ for all $k$, where $\epsilon_k$ converges $\mathrm{Q}$-sublinearly (linearly, superlinearly) to zero, then we say that $x_k$ converges $\mathrm{R}$-sublinearly (linearly, superlinearly) to $x^\ast$.
\end{definition}

The following corollary establishes convergence orders associated to various step size sequences.

\begin{corollary}[Prescribed convergence orders]
\label{cor:ConvergenceRates}
Fix $C > 0$.
Under the assumptions of \Cref{thm:ConvergenceContinuousLevel}, consider the following candidate sequences of step sizes:
\smallskip
\begin{subequations}
\begin{enumerate}[itemindent=.7cm]
	\item[\textsl{Case 1:}]
	Fix $m \in \mathbb{N}$ and set
	\begin{equation}
	\label{eq:ArithmeticSeries}
		\alpha_{k} = C k(k+1)\cdots(k+m) ~ \fa k = 1,2,\ldots
	\end{equation}
	\item[\textsl{Case 2:}]
	Fix $\mu > 1$ and set
	\begin{equation}
	\label{eq:GeometricSeries}
		\alpha_{k} = C \mu^{k-1} ~ \fa k = 1,2,\ldots
	\end{equation}
	\item[\textsl{Case 3:}]
	Set $\alpha_1 = C$ and
	\begin{equation}
	\label{eq:FactorialSeries}
		\alpha_{k+1} = C k k! ~ \fa k = 1,2,\ldots
	\end{equation}
	\item[\textsl{Case 4:}]
	Fix $\mu, q, r > 1$ and set $\alpha_1 = r^{1/(q-1)} \mu$ and
	\begin{equation}
	\label{eq:DoubleExponential}
		\alpha_{k+1} = r^{1/(q-1)} \mu^{q^k} - \alpha_{k} ~ \fa k = 2,3,\ldots
	\end{equation}
\end{enumerate}
\end{subequations}
Then sequence \cref{eq:ArithmeticSeries} delivers a sublinear $\mathrm{R}$-convergence; sequence \cref{eq:GeometricSeries} delivers $\mathrm{R}$-linear convergence with rate $1/\mu$; sequence \cref{eq:FactorialSeries} delivers $\mathrm{R}$-superlinear convergence with order $1$ and rate $0$; and sequence \cref{eq:DoubleExponential} delivers $\mathrm{R}$-superlinear convergence with order $q$ and rate $r$. 
\end{corollary}
\begin{proof}
	Throughout the proof, we use the definition $\epsilon_k = D(u^\ast,u^0)/{\sum_{j=1}^k \alpha_{j}}$.
	\smallskip

	\noindent\textsl{Case 1:}
	For this case, we can use Chu's theorem \cite[Theorem~1.5.2]{merris2003combinatorics} to show that
	\begin{equation}
		(m+2)
		\sum_{j=1}^k
		j(j+1)\cdots(j+m)
		=
		k(k+1)\cdots(k+m+1)
		\,.
	\end{equation}
	Thus, if $\alpha_k = C k(k+1)\cdots(k+m)$ for all $k$, then
							\begin{equation}
		\frac{\epsilon_{k+1}}{\epsilon_k}
		=
		\frac{k(k+1)\cdots(k+m)}{(k+1)(k+2)\cdots(k+m+1)}
		=
		\frac{k}{k+m+1}
		\to 1
		~~
		\text{as~} k \to \infty
		\,.
	\end{equation}
	\smallskip
	
	\noindent\textsl{Case 2:}
	For this case, we use the well-known finite geometric series identity
	\begin{equation}
		\sum_{j=1}^{k} \mu^{j-1}
		=
		\frac{\mu^{k} - 1}{\mu-1}
		\,.
	\end{equation}
							Thus, if $\alpha_k = C \mu^{k-1}$ for all $k$, then
	\begin{equation}
		\frac{\epsilon_{k+1}}{\epsilon_k}
		=
		\frac{\mu^{k} - 1}{\mu^{k+1} - 1}
		\to \frac{1}{\mu}
		~~
		\text{as~} k \to \infty
		\,.
	\end{equation}
	\smallskip
	
	\noindent\textsl{Case 3:}
	For this case, we use the identity
	\begin{equation}
		\sum_{j=1}^k j j!
		=
		(k+1)! - 1
		\,,
	\end{equation}
	which is readily verified by mathematical induction.
							Thus, if $\alpha_k = C (k-1) (k-1)!$ for all $k \geq 2$ and $\alpha_1 = C$, then
	\begin{equation}
		\frac{\epsilon_{k+1}}{\epsilon_k}
		=
		\frac{k!}{(k+1)!}
		=
		\frac{1}{k+1}
		\to 0
		~~
		\text{as~} k \to \infty
		\,.
	\end{equation}
	\smallskip
	
	\noindent\textsl{Case 4:}
	In this case, we use the fact that $\sum_{j=1}^k \alpha_{j-1} = \alpha_{k-1}$ is a telescoping sum by design.
							Thus, if $\alpha_{k+1} = r^{1/(q-1)} \mu^{q^k} - \alpha_{k}$ for all $k\geq 1$ and $\alpha_1 = r^{1/(q-1)} \mu$, then
	\begin{equation}
		\frac{\epsilon_{k+1}}{\epsilon_k^q}
		=
		\frac{r^{q/(q-1)}(\mu^{q^{k-1}})^q}{r^{1/(q-1)}\mu^{q^{k}}}
		=
		r
		~~
		\fa k\geq 1
		\,.
	\end{equation}
\end{proof}

\subsection{The entropic Poisson equation in the zero-temperature limit} \label{sub:convergence_of_the_penalty_method}

We close this section by showing that the one-parameter family of solutions to the entropic Poisson equation with ``temperature'' $\theta = \alpha^{-1}$ converge strongly (in $H^1(\Omega)$) to the solution of the obstacle problem as $\theta \to 0$.

\begin{theorem}
\label{cor:EntropicPoissonConvergence}
	Assume $\Omega \subset \mathbb{R}^n$ is an open, bounded Lipschitz domain, $n \ge 1$, and let $g \in H^1(\Omega) \cap C(\overline{\Omega})$ such that $\min g_{|\partial \Omega} > 0$.
	Let $u_\theta \in H^1_g(\Omega)\cap \interior L^\infty_+(\Omega)$ denote the solution of the entropic Poisson equation,
	\begin{equation}
	\label{eq:EntropicPoissonConvergence_EPE}
		(\nabla u_\theta, \nabla w)
		+
		\theta (\ln u_\theta, w)
		=
		(f, v)
		\quad
		\text{for all~}
		w \in H^1_0(\Omega)
		,
	\end{equation}
	and let
	\begin{equation}
		u^\ast
		=
		\argmin_{u\in H^1(\Omega)}
		~
		E(u)
		~~\text{subject to~}
		u \geq 0
		~\text{in~}\Omega
		~\text{and~}
		u = g
		~\text{on~}\partial\Omega
		.
	\end{equation}
	Then $u_\theta \to u^\ast$ in $H^1(\Omega)$ linearly with respect to $\theta$.
	In particular,
	\begin{equation}
		\frac12\|\nabla u^\ast - \nabla u_\theta\|_{L^2(\Omega)}^2
		\leq
		\theta (S(u^\ast) + |\Omega|)
		.
	\end{equation}
\end{theorem}
\begin{proof}
	For all functions $u \in L^{\infty}_+(\Omega)$, $v,w \in \interior L^\infty_+(\Omega)$, the representation theorem, \Cref{lem:EntropyDifferentiability}, and the three-points identity~\cref{eq:CosineIdentity}, together, give us
	\begin{equation}
	\label{eq:EntropicPoissonConvergence_TPI}
		D(u,v) - D(u,w) + D(v,w)
		=
		(\nabla S(v) - \nabla S(w), v - u )
		\,.
	\end{equation}
	Moreover, the characterization theorem, \Cref{thm:PrimalProblem}, tell us that~\cref{eq:EntropicPoissonConvergence_EPE} is equivalent to
	\begin{equation}
		S^\prime(u_\theta) = -\frac{1}{\theta} E^\prime(u_\theta)
		\,.
	\end{equation}
	Next, notice that $\nabla S(1) = \ln 1 = 0$ and, provided $u \in H^1(\Omega)$,
	\begin{equation}
		\langle E^\prime(u_\theta), u - u_\theta \rangle
		\leq
		E(u) - E(u_\theta)
		\,,
	\end{equation}
	by convexity.
	Thus, taking $u \in H^1_g(\Omega) \cap L^{\infty}_+(\Omega)$ and setting $v = u_\theta$ and $w = 1$ in~\cref{eq:EntropicPoissonConvergence_TPI} leads to
	\begin{align*}
		D(u,u_\theta) - D(u,1) + D(u_\theta,1)
		&=
		( \nabla S(u_\theta) - \nabla S(1), u_\theta - u )
		\\
		&=
		\langle S^\prime(u_\theta), u_\theta - u \rangle
		\\
		&=
		\frac{1}{\theta} \langle E^\prime(u_\theta), u - u_\theta \rangle
		\\
		&\leq
		\frac{1}{\theta} ( E(u) - E(u_\theta) )
		\,.
	\end{align*}
		Rerranging the inequality above and invoking~\Cref{prop:BregmanDivergenceProperties}, we find that
	\begin{equation}
		E(u_\theta) - E(u)
		\leq
		\theta (D(u,1) - D(u,u_\theta) - D(u_\theta,1))
		\leq
		\theta D(u,1)
		\,,
	\end{equation}
	where the second inequality arises because both $D(u,u_\theta)$ and $D(u_\theta,1)$ are non-negative.
	The proof is completed by setting $u = u^\ast$ and exploiting the strong convexity of $E$, as done in~\cref{eq:ConvergenceContinuousLevel_StrongConvexity}, and noting that $D(u,1) = S(u) + |\Omega|$.
\end{proof}

\section{Mathematical results II: Elementary finite element error analysis} \label{sec:the_entropic_finite_element_method}
The purpose of this appendix is to establish certain minor \textit{a priori} error analysis results related to the stability of the subproblems encountered upon linearizing~\cref{eq:ObstacleDiscreteNonlinearSaddlePoint}.
These results are necessary to motivate the finite elements proposed in~\cref{eq:SubspacePair1,eq:SubspacePair2}.
In short, the main outcome of this appendix is that the finite elements are stable and, therefore, we expect optimal high-order convergence rates for the solutions of certain \emph{linearized} subproblems.
We intentionally stop short of providing a full \textit{a priori} error analysis of the \emph{nonlinear} subproblems or the complete proximal Galerkin method.
Such analysis is planned for a forthcoming paper in which we additionally aim to prove that the Proximal Galerkin method is \emph{mesh-independent}.

\subsection{Stability of the linearized subproblems} \label{sub:linear_stability_theory}

The first result of this section is that linearizations of the subproblems in~\Cref{alg:ObstacleProblem} are stable at the continuous level.
We begin with analyzing linearizations of the continuous-level algorithm given in~\cref{eq:ConvergenceContinuousLevel_VE}.
The proof uses standard Hilbert space arguments for singularly-perturbed saddle-point problems; cf.~\cite[Chapter~4.3.2]{boffi2013mixed}.

\begin{theorem}
\label{thm:ContinuousStability}
	Let $\psi \in L^\infty(\Omega)$.
	Then, for every $f \in H^{-1}(\Omega)$, $g \in L^2(\Omega)$, and $\phi \in H^1_0(\Omega)$, the saddle-point problem
	\begin{equation}
	\label{eq:ContinuousStability}
		\left\{
		\begin{aligned}
															&\begin{alignedat}{4}
			( \nabla u, \nabla v) + (\delta, v) &= \langle f, v\rangle
			&&~\fa v \in H_0^1(\Omega),
			\\
			(u, w) - (\delta \exp\psi, w) &= (\phi, w) + (g, w)
			&&~\fa w \in L^2(\Omega),
			\end{alignedat}
		\end{aligned}
		\right.
	\end{equation}
	has a unique solution $u \in H_0^1(\Omega)$, $\delta \in L^2(\Omega)$ that satisfies the stability bound
	\begin{multline}
	\label{eq:StabilityBounds}
		\|\nabla u\|_{L^2(\Omega)}
		+
		\|\delta\|_{H^{-1}(\Omega)}
		+
		\|\exp(\psi/2)\delta\|_{L^2(\Omega)}
						\\\leq
		C\Big(\|f\|_{H^{-1}(\Omega)} + \|\nabla \phi\|_{L^2(\Omega)} + \|\exp(-\psi/2)g\|_{L^2(\Omega)} \Big)
		\,,
	\end{multline}
	for a constant $C>0$ that is independent of $\psi$ and $\Omega$.
																		\end{theorem}

\begin{remark}[Choice of norms]
	We are generally interested in the setting where $g = \tilde{g}\exp\psi$, for some $\tilde{g} \in L^2(\Omega)$, as this is the setting that appears in Newton methods, as well as in the \textit{a priori} error analysis carried out in \Cref{sub:convergence_of_the_linearized_subproblems} below.
	In this case, notice that the upper bounds on $\|\nabla u\|_{L^2(\Omega)}$ and $\|\delta\|_{H^{-1}(\Omega)}$ are uniform in the limit $\essinf \psi \to -\infty$.
	On the other hand, the stability bound on the (unweighted) $L^2(\Omega)$-norm of $\delta$, arising from~\cref{eq:StabilityBounds} and the inequality
	\begin{equation}
	\label{eq:UnweightedL2Bound}
		\|\delta\|_{L^2(\Omega)}^2 \leq \|\exp(-\psi)\|_{L^\infty(\Omega)}\|\exp(\psi/2)\delta\|_{L^2(\Omega)}^2
				\,,
	\end{equation}
	degenerates as $\essinf \psi \to -\infty$, even if $g = 0$.
	\end{remark}

\begin{proof}[Proof of~\Cref{thm:ContinuousStability}]
	The proof proceeds in two steps.
	\smallskip

	\noindent\textsl{Step 1.}
	Existence and uniqueness of $u \in H_0^1(\Omega)$, $\delta \in L^2(\Omega)$ follows readily from the Lax--Milgram theorem.
	Indeed, notice that~\cref{eq:ContinuousStability} may be rewritten as
	\begin{equation}
	\label{eq:LinearizedSaddleRewrite}
		B((u,\delta),(v,w)) = \langle f, v\rangle - (\phi, w) - (g, w)
		~
		\fa v \in H_0^1(\Omega)
		\text{ and }
		w \in L^2(\Omega)
		\,,
	\end{equation}
	where $B((u,\delta),(v,w)) = ( \nabla u, \nabla v) + (\delta, v) - (u, w) + (\delta \exp\psi, w)$ is a bounded bilinear map $B \colon (H_0^1(\Omega) \times L^2(\Omega)) \times (H_0^1(\Omega) \times L^2(\Omega)) \to \mathbb{R}$.
	Setting $v = u$ and $w = \delta$ in~\cref{eq:LinearizedSaddleRewrite}, we find that
	\begin{equation}
	\label{eq:WeightedL2LowerBound}
		\|\nabla u\|^2_{L^2(\Omega)} + \|\exp(\psi/2)\delta\|_{L^2(\Omega)}^2
		=
		\langle f, u\rangle - ( \phi, \delta) - (g, \delta)
		\,.
	\end{equation}
	Applying~\cref{eq:UnweightedL2Bound}, we find that
	\begin{equation}
	\label{eq:LinearizedSaddleCoercivity}
		\|\nabla u\|^2_{L^2(\Omega)} + \|\exp(-\psi)\|_{L^\infty(\Omega)}^{-1}\|\delta\|_{L^2(\Omega)}^2
		\leq
		B((u,\delta),(u,\delta))
		\,,
	\end{equation}
	for all $(u,\delta) \in H_0^1(\Omega) \times L^2(\Omega)$, which establishes the coercivity condition necessary to apply the theorem.
	\smallskip
	
	\noindent\textsl{Step 2.}
	The remainder of the proof centers on~\cref{eq:WeightedL2LowerBound} assuming $u \in H_0^1(\Omega)$, $\delta \in L^2(\Omega)$ are the unique solutions of~\cref{eq:ContinuousStability}.
	From the first equation in~\cref{eq:ContinuousStability}, we observe that
	\begin{equation}
	\label{eq:incrementHminus1normidentity}
	\begin{aligned}
		\|\delta\|_{H^{-1}(\Omega)}
		=
		\sup_{v \in H^1_0(\Omega)} \frac{(\delta, v)}{\|\nabla v\|_{L^2(\Omega)}}
		&\leq
		\sup_{v \in H^1_0(\Omega)} \frac{(\nabla u, \nabla v)}{\|\nabla v\|_{L^2(\Omega)}}
		+
		\sup_{v \in H^1_0(\Omega)} \frac{\langle f, v\rangle}{\|\nabla v\|_{L^2(\Omega)}}
		\\
		&=
		\|\nabla u\|_{L^2(\Omega)} + \|f\|_{H^{-1}(\Omega)}
		\,.
	\end{aligned}
	\end{equation}
	Thus, we deduce from submultiplicativity and Young's inequality that
	\begin{align*}
		(\phi,\delta)
		\leq
		\|\nabla \phi\|_{L^2(\Omega)}\|\delta\|_{H^{-1}(\Omega)}
		&\leq
		\|\nabla \phi\|_{L^2(\Omega)}\|\nabla u\|_{L^2(\Omega)} + \|\nabla \phi\|_{L^2(\Omega)}\|f\|_{H^{-1}(\Omega)}
				\\
		&\leq
		\gamma\|\nabla \phi\|_{L^2(\Omega)}^2
		+
		\frac{1}{2\gamma}
		\|\nabla u\|_{L^2(\Omega)}^2
		+
		\frac{1}{2\gamma}
		\|f\|_{H^{-1}(\Omega)}^2
		\,,
	\end{align*}
	for every $\gamma > 0$.
	By similar arguments, we notice that
	\[
		\langle f, u\rangle
		\leq
		\|f\|_{H^{-1}(\Omega)}\|\nabla u\|_{L^2(\Omega)}
		\leq
		\frac{1}{2}\|f\|_{H^{-1}(\Omega)}^2 + \frac{1}{2}\|\nabla u\|_{L^2(\Omega)}^2
	\]
	and
	\begin{multline*}
		(g, \delta)
		\leq
		\|\exp(-\psi/2)g\|_{L^2(\Omega)}\|\exp(\psi/2)\delta\|_{L^2(\Omega)}
		\\
		\leq
		\frac{1}{2}\|\exp(-\psi/2)g\|_{L^2(\Omega)}^2 + \frac{1}{2}\|\exp(\psi/2)\delta\|_{L^2(\Omega)}^2
		\,.
	\end{multline*}
	Together with~\cref{eq:WeightedL2LowerBound}, these inequalities imply that
	\begin{multline*}
		\|\nabla u\|_{L^2(\Omega)}^2
		+
		\frac{1}{2}\|\exp(\psi/2)\delta\|_{L^2(\Omega)}^2
		\leq
		\gamma\|\nabla \phi\|_{L^2(\Omega)}^2
		+
		\frac{1}{2}\bigg(1 + \frac{1}{\gamma}\bigg)\|f\|_{H^{-1}(\Omega)}^2
		\\
		+
		\frac{1}{2}\bigg(1 + \frac{1}{\gamma}\bigg)\|\nabla u\|_{L^2(\Omega)}^2
		+
		\frac{1}{2}\|\exp(-\psi/2)g\|_{L^2(\Omega)}^2
		\,.
	\end{multline*}
	Setting $\gamma > 1$, we find that
	\begin{multline*}
		\|\nabla u\|_{L^2(\Omega)}
		+
		\|\exp(\psi/2)\delta\|_{L^2(\Omega)}
		\\
		\leq
		C \big(
			\|\nabla \phi\|_{L^2(\Omega)} + \|f\|_{H^{-1}(\Omega)} + \|\exp(-\psi/2)g\|_{L^2(\Omega)}
		\big)
		\,,
	\end{multline*}
	for some $C > 0$.
	The required result follows from combining this inequality\linebreak with~\cref{eq:incrementHminus1normidentity}.
\end{proof}

Our next result is that the finite elements~\cref{eq:SubspacePairs} are uniformly stable in $H^1(\Omega) \times H^{-1}(\Omega)$; i.e., they satisfy the Ladyzhenskaya--Babu\v{s}ka--Brezzi (LBB) stability condition
\begin{equation}
\label{eq:LBB}
	\beta_h :=
	\inf_{w \in W_h} \sup_{v \in V_h}
	\frac{(v,w)}{\|\nabla v\|_{L^2(\Omega)}\|w\|_{H^{-1}(\Omega)}}
	> 0
	\,,
\end{equation}
and, furthermore, $\beta_h$ is strictly bounded away from zero for all mesh sizes $h>0$ and (clearly) independent of $\psi$.
Note that here and throughout, we typically treat the symbol $C>0$ as a generic mesh-independent constant.

\begin{lemma}
\label{lem:Fortin}
	Assume that $\mathcal{T}_h$ is a shape-regular sequence of affine meshes covering $\overline{\Omega} = \bigcup_{T \in \mathcal{T}_h} T$.
	Let $V_h$ and $W_h$ be the finite element spaces defined in~\cref{eq:SubspacePairs}.
	Then there is a constant $\beta_0$ such that for all $h>0$,
	\begin{equation}
	\label{eq:UniformStabilityConstant}
		\inf_{w \in W_h} \sup_{v \in V_h}
		\frac{(v,w)}{\|\nabla v\|_{L^2(\Omega)}\|w\|_{H^{-1}(\Omega)}}
		\geq
		\beta_0
		> 0
		\,.
	\end{equation}
\end{lemma}

\begin{remark}[Idea of the proof]
\label{rem:FortinProof}
	The proof proceeds independently for each finite element pairing by constructing a so-called Fortin operator $\Pi_h \colon H^1_0(\Omega) \to V_h$ satisfying
	\begin{subequations}
	\begin{equation}
	\label{eq:FortinBoundednessCondition}
		\|\Pi_h v\|_{H^1(\Omega)}
		\leq
		C
		\|v\|_{H^1(\Omega)}
		\,,
	\end{equation}
	for some $h$-independent constant $C>0$, and 
	\begin{equation}
	\label{eq:FortinProjectionCondition}
		( \Pi_h v, w) = (v, w)
				~
		\fa v \in H^1_0(\Omega)
		\text{~and~} w\in W_h
		\,.
	\end{equation}
	It is well-known that the existence of such an operator  on a fixed mesh $\mathcal{T}_h$ implies the LBB stability condition~\cref{eq:LBB}; see, e.g., \cite{fortin1977analysis} and \cite[Proposition~5.4.3]{boffi2013mixed}. See also \cite[Theorem~1]{ern2016converse} for the converse.
	Likewise, $h$-independence of the constant $C$ in~\cref{eq:FortinBoundednessCondition} implies the existence of the uniform discrete stability constant $\beta_0$ in \cref{eq:UniformStabilityConstant}.

	We employ a standard technique to construct our Fortin operators that involves splitting the operator into two terms; cf.~\cite[Section~5.4.4]{boffi2013mixed}.
	In particular, for each pair of subspaces $(V_h,W_h)$, we define $\Pi_h = \tilde{\mathcal{I}}_h + \tilde{\Pi}_{h}(I - \tilde{\mathcal{I}}_h)$, where $\tilde{\mathcal{I}}_h \colon H^1_0(\Omega) \to V_h$ is a quasi-interpolation operator (see, e.g., \cite[Section~22.4]{ern2021finite}) and $\tilde{\Pi}_{h} \colon L^2(\Omega) \to V_h$ is a linear operator satisfying
	\begin{equation}
	\label{eq:FortinProjectionCondition2}
		( \tilde{\Pi}_{h} v, w) = (v, w)
				~
		\fa v \in L^2(\Omega)
		\text{~and~} w\in W_h
		\,,
	\end{equation}
	\end{subequations}
	and $\|\tilde{\Pi}_{h}(I-\tilde{\mathcal{I}}_h) v\|_{H^1(\Omega)} \leq C \|v\|_{H^1(\Omega)}$ for all $v \in H^1_0(\Omega)$.
	It is easy to check that such an operator satisfies~\cref{eq:FortinBoundednessCondition,eq:FortinProjectionCondition}.
	Indeed,
	\begin{equation}
		\|\Pi_h v\|_{H^1(\Omega)} \leq \|\tilde{\mathcal{I}}_h v\|_{H^1(\Omega)} + \|\tilde{\Pi}_{h}(I-\tilde{\mathcal{I}}_h) v\|_{H^1(\Omega)} \leq C \|v\|_{H^1(\Omega)}
		\,,
	\end{equation}
	and, moreover,
	\begin{equation}
		(\Pi_h v, w)
		=
		(\tilde{\mathcal{I}}_h v, w) + (\tilde{\Pi}_{h}(I-\tilde{\mathcal{I}}_h) v, w)
		=
		(\tilde{\mathcal{I}}_h v, w) + ((I-\tilde{\mathcal{I}}_h) v, w)
		=
		(v,w)
		\,,
	\end{equation}
	where the second equality follows from~\cref{eq:FortinProjectionCondition2}.

																	\end{remark}

\begin{proof}[Proof of~\Cref{lem:Fortin}]
	For simplicity, we consider only the two-dimensional setting $n = 2$.

	\textit{Case 1.} We first consider the $(\mathbb{P}_{p}\text{-bubble},\mathbb{P}_{p-1}\text{-broken})$ finite elements defined in~\cref{eq:SubspacePair1}.
	In this setting, we define $\tilde{\Pi}_{h}$ satisfying~\cref{eq:FortinProjectionCondition2} element-wise by solving the following local variational problem at each element $T \in \mathcal{T}_h$:
	\begin{equation}
	\label{eq:FortinSubproblemSubspace1}
		\left\{
			\begin{alignedat}{3}
				&\text{Find}~(\tilde{\Pi}_{h} v)_{|T} := v_T \in \mathring{\mathbb{P}}_{p+2}(T)
								~\text{such that}
				\\
				&
				(v_T, \varphi)_T = (v, \varphi)_T
				~
				\fa \varphi \in \mathbb{P}_{p-1}(T)
				\,.
			\end{alignedat}
		\right.
	\end{equation}
							
	Recall that every function in $\mathring{\mathbb{P}}_{p+2}(T)$ is a polynomial vanishing on $\partial T$; cf.~\cite[Section 5.1.3]{fuentes2015orientation}.
	Thus, every function in this set is divisible by linear functions vanishing on some edge of $T$. 
	For instance, consider the three linear vertex functions, which vanish at two vertices and evaluate to one at the third vertex.
	Define $b_T$ to be the product of these three functions and notice that it is positive everywhere, except on $\partial T$, and that every function in $\mathring{\mathbb{P}}_{p+2}(T)$ is divisible by it.
	We conclude that $\varphi / b_T \in \mathbb{P}_{p-1}(T)$ for all $\varphi \in \mathring{\mathbb{P}}_{p+2}(T)$ and, further, $| \mathring{\mathbb{P}}_{p+2}(T) | = p(p+1)/2 = | \mathbb{P}_{p-1}(T) |$, implying
	\[
		\mathring{\mathbb{P}}_{p+2}(T)
		=
		\{
			b_T \varphi \mid \varphi \in \mathbb{P}_{p-1}(T)
		\}
		\,.
	\]

	We now consider $\tilde{\Pi}_{h}$ defined in~\cref{eq:FortinSubproblemSubspace1}.
	Let $\{\varphi_i\}$ be a basis for $\mathbb{P}_{p-1}(T)$ and $\{b_T\varphi_i\}$ be the corresponding basis for $\mathring{\mathbb{P}}_{p+2}(T)$.
	Upon writing $v_T = b_T\sum_{j=1}^{p(p+1)/2} \mathsf{c}_j \varphi_j$, we see that the variational problem~\cref{eq:FortinSubproblemSubspace1} is equivalent to the invertible linear system
	\begin{equation}
		\mathsf{M}_{ij}\mathsf{c}_j = \mathsf{d}_i
		\,,
		\qquad
		i = 1,2,\ldots,p(p+1)/2,
	\end{equation}
	where $\mathsf{M}_{ij} = (b_T\varphi_j,\varphi_i)_T$, and $\mathsf{d}_i = (v,\varphi_i)_T$.

	Norm equivalence on $\mathbb{R}^{p(p+1)/2}$ and standard scaling arguments (see, e.g., the proof of \cite[Proposition~28.5]{ern2021finiteII}) can be used to show that $\|\mathsf{M}^{-1}\|_{\ell^2} \leq C | T |^{-1}$.
	Meanwhile, H\"older's inequality can be used to show that
	\begin{equation}
		\mathsf{d}_i = \int_T v \varphi_i \dd x \leq \max_i \|\varphi_i\|_{L^\infty(T)} |T|^{1/2} \|v\|_{L^2(T)}
		\,,
	\end{equation}
	for each $i = 1,2,\ldots,p(p+1)/2$.
	Thus, we conclude that
	\begin{equation}
	\label{eq:lem_Fortin_scalingbound1}
		\|\mathsf{c}\|_{\ell^2}
		\leq
		\|\mathsf{M}^{-1}\|_{\ell^2}
		\|\mathsf{d}\|_{\ell^2}
		\leq
		C |T|^{-1/2} \|v\|_{L^2(T)}
		\,.
	\end{equation}
	Shape-regularity and a similar scaling argument (see, e.g., \cite[Lemmas~11.1 and 11.7]{ern2021finite}) implies that
	\begin{equation}
	\label{eq:lem_Fortin_scalingbound2}
		| b_T\varphi_i |_{H^1(T)}
		\leq
		C h_T^{-1}|T|^{1/2}
		\,.
	\end{equation}
	Combining \cref{eq:lem_Fortin_scalingbound1,eq:lem_Fortin_scalingbound2}, we find that
	\begin{equation}
		| \tilde{\Pi}_{h} v |_{H^1(T)}
		=
		| v_T |_{H^1(T)}
		\leq C h_T^{-1} \| v\|_{L^2(T)}
		\,.
	\end{equation}

	The next step is to specify the quasi-interpolation operator $\tilde{\mathcal{I}}_h \colon H^1_0(\Omega) \to V_h$.
	We choose to use the operator defined in~\cite[Equation~6.10]{ern2017finite},
		which, by \cite[Theorem~6.4]{ern2017finite}, has the property that
	\begin{equation}
	 	\|(I - \tilde{\mathcal{I}}_h) v \|_{L^2(T)}
	 	\leq
	 	C h_T\|\nabla v\|_{L^2(\Omega_T)}
	 	~
		\fa v \in H^1_0(\Omega)
		\,,
	\end{equation}
	where $\Omega_T \subset \Omega$ is the union of mesh cells neighboring $T$.
	We now find that
	\begin{equation}
		|\tilde{\Pi}_{h}(I-\tilde{\mathcal{I}}_h) v|_{H^1(T)}
		\leq
		C h_T^{-1} \|(I-\tilde{\mathcal{I}}_h) v\|_{L^2(T)}
		\leq
		C\|\nabla v\|_{L^2(\Omega_T)}
		\,.
	\end{equation}
	Note that the maximum number of elements in $\Omega_T$ is bounded uniformly in $h$ owing to the regularity of the mesh sequence.
	Likewise, we find that
	\begin{equation}
		|\tilde{\Pi}_{h}(I-\tilde{\mathcal{I}}_h) v|_{H^1(\Omega)}^2
		\leq
		C
		\sum_{T\in\mathcal{T}}
		\|\nabla v\|_{L^2(\Omega_T)}^2
		\leq
		C
		\|\nabla v\|_{L^2(\Omega)}^2
		\,.
	\end{equation}
	We have succeeded in checking the conditions outlined in~\cref{rem:FortinProof} and the proof is complete.

	\textit{Case 2.} We now consider the $(\mathbb{Q}_{p}\text{-bubble},\mathbb{Q}_{p-1}\text{-broken})$ finite elements defined in~\cref{eq:SubspacePair2}.
	In this case, we define $\tilde{\Pi}_h v := v_h \in V_h$ element-wise by solving the local variational problems
	\begin{equation}
	\label{eq:FortinSubproblemSubspace2}
		\left\{
			\begin{alignedat}{3}
																												&\text{Find}~{v_h}_{|T} \in \mathring{\mathbb{Q}}_{p+1}(T)
				~\text{such that}
				\\
				&
				(v_h, \varphi)_T = (v, \varphi)_T
				~
				\fa \varphi \in \mathbb{Q}_{p-1}(T)
												\,.
			\end{alignedat}
		\right.
	\end{equation}
	Here, we notice that $| \mathring{\mathbb{Q}}_{p+1}(T) |=p^2=| \mathbb{Q}_{p-1}(T) |$ and arrive at the conclusion that
	\[
		\mathring{\mathbb{Q}}_{p+1}(T)
		=
		\{
			d_T \varphi \mid \varphi \in \mathbb{Q}_{p-1}(T)
		\}
		\,,
	\]
	where $d_T$ a non-zero bubble function in $\mathbb{Q}_{2}(T)\cap H^1_0(T)$.
	Thus, for every element $T\in \mathcal{T}_h$, the variational problem~\cref{eq:FortinSubproblemSubspace2} is equivalent to an invertible $p^2 \times p^2$ linear system, and so $\tilde{\Pi}_h v := v_h \in V_h$ is well-posed.
	The remainder of the proof proceeds as done in \textit{Case 1}.
\end{proof}

\begin{remark}[Alternative subspaces]
\label{rem:AlternativeSubspaces2}
	Notice that \Cref{lem:Fortin} implies that the pairing $(\tilde{V}_h, W_h)$ is uniformly stable for any subspace $\tilde{V}_h\subset H^1_0(\Omega)$ containing $V_h$.
	Indeed, observe that
	\begin{equation}
		\sup_{v \in \tilde{V}_h}
		\frac{(v,w)}{\|\nabla v\|_{L^2(\Omega)}}
		\geq
		\sup_{v \in V_h}
		\frac{(v,w)}{\|\nabla v\|_{L^2(\Omega)}}
		\geq
		\beta_0\|w\|_{H^{-1}(\Omega)}
		\,,
	\end{equation}
	for all $w \in W_h$.
	Thus, other elements such as the $(\mathbb{Q}_{p+1},\mathbb{Q}_{p-1}\text{-broken})$ pair proposed in \Cref{rem:AlternativeSubspaces1}, are also stable owing to the embedding $\mathbb{Q}^{p+1}_p(T) = \hat{\mathbb{Q}}_{p}(T) \oplus \mathring{\mathbb{Q}}_{p+1}(T) \subset \mathbb{Q}_{p+1}(T)$.
\end{remark}

\subsection{Convergence of the linearized subproblems} \label{sub:convergence_of_the_linearized_subproblems}

This subsection is devoted to a proof that the $(\mathbb{P}_{p}\text{-bubble},\mathbb{P}_{p-1}\text{-broken})$ and $(\mathbb{Q}_{p}\text{-bubble},\mathbb{Q}_{p-1}\text{-broken})$ elements defined in~\cref{eq:SubspacePairs} converge at optimal rates toward the solutions of the linearized subproblems~\cref{eq:ContinuousStability}.
We note that the existence and uniqueness of $u_h$ and $\delta_h$ in the theorem below follow from a Lax--Milgram argument similar to Step 1 of the proof of~\cref{thm:ContinuousStability}.

\begin{theorem}
\label{thm:APrioriLinearized}
	Let $u_{h} \in V_h$ and $\delta_{h} \in W_h$ to be the discrete solutions of the saddle-point problem
	\begin{equation}
	\label{eq:APrioriLinearized_discrete}
		\left\{
		\begin{aligned}
			\,&\text{Find}~
			u_{h}\in V_{h} ~\text{and}~\delta_{h} \in W_{h}
			~\text{such that~}
			\\
			&\begin{alignedat}{4}
			( \nabla u_{h}, \nabla v) + (\delta_{h}, v) &= \langle f, v\rangle
			&&~\fa v \in V_{h}\,,
			\\
			(u_{h}, w) - (\delta_{h} \exp\psi, w) &= (\phi, w) + (g, w)
			&&~\fa w \in W_{h}\,,
			\end{alignedat}
		\end{aligned}
		\right.
	\end{equation}
	where $V_h$ and $W_h$ are the $(\mathbb{P}_{p}\text{-bubble},\mathbb{P}_{p-1}\text{-broken})$ and $(\mathbb{Q}_{p}\text{-bubble},\mathbb{Q}_{p-1}\text{-broken})$ finite element spaces defined in~\cref{eq:SubspacePairs}.
	Likewise, let $r \geq 1$ be an integer and assume that the unique solutions of the continuous-level variational problem
	\begin{equation}
		\left\{
		\begin{aligned}
			\,&\text{Find}~
			u\in H_0^1(\Omega) ~\text{and}~\delta \in L^2(\Omega) 
			~\text{such that~}
			\\
			&\begin{alignedat}{4}
			( \nabla u, \nabla v) + (\delta, v) &= \langle f, v\rangle
			&&~\fa v \in H_0^1(\Omega)\,,
			\\
			(u, w) - (\delta \exp\psi, w) &= (\phi, w) + (g, w)
			&&~\fa w \in L^2(\Omega)\,,
			\end{alignedat}
		\end{aligned}
		\right.
	\end{equation}
	are sufficiently regular that $u \in H^{r+1}(\Omega)$ and $\delta \in H^{r}(\Omega)$.
	Then, if $\mathcal{T}_h$ is a shape-regular sequence of affine meshes, it holds that
	\begin{equation}
	\label{eq:RobustConvergence_LinearizedSubproblem}
		\|u - u_{h}\|_{H^1(\Omega)} + \|\delta - \delta_{h}\|_{H^{-1}(\Omega)}
		\leq
		C_1 h^s \big(|u|_{H^{s+1}(\Omega)} + |\delta|_{H^{s}(\Omega)}\big)
		\,,
	\end{equation}
	for all $1 \leq s \leq \min\{r,p\}$, where $C_1$ is a mesh-independent constant that remains bounded as $\essinf \psi \to -\infty$.
		Furthermore, there exists a mesh-independent constant $C_2$ such that
	\begin{equation}
	\label{eq:L2Convergence_LinearizedSubproblem}
		\|\delta - \delta_{h}\|_{L^2(\Omega)}
		\leq
		C_2 h^s \big(|u|_{H^{s+1}(\Omega)} + |\delta|_{H^{s}(\Omega)}\big)
				\,.
	\end{equation}
	However, $C_2 \to \infty$ as $\essinf \psi \to -\infty$.
	\end{theorem}
\begin{proof}
	The proof consists of two steps.
	\smallskip

	\noindent\textsl{Step 1.}
	Deriving discrete stability estimates.

	Let $u_I \in V_h$ and $\delta_I \in W_h$ and define $e_u = u_h - u_I$ and $e_\delta = \delta_h - \delta_I$.
	Observe that $e_u \in V_h$ and $e_\delta \in W_h$ are the unique solutions to
	\begin{alignat*}{4}
		(\nabla e_u, \nabla v) + (e_\delta, v) &= (\nabla (u - u_I), \nabla v)+ (\delta - \delta_I, v)
		&&~\fa v \in V_{h}\,,
		\\
		(e_u, w) - (e_\delta \exp\psi, w) &= (u - u_I, w) + ((\delta - \delta_I)\exp\psi, w)
		&&~\fa w \in W_{h}\,.
	\end{alignat*}
	Setting $v = e_u$ and $w = -e_\delta$ and summing the two equations, we find that
	\[
		\|\nabla e_u\|_{L^2(\Omega)}^2
		+
		\|\exp(\psi/2)e_\delta\|_{L^2(\Omega)}^2
		\leq
		\langle R_1, e_u \rangle - (R_2, e_\delta) - (R_3, e_\delta)
		\,,
	\]
	where $\langle R_1, v \rangle = (\nabla (u - u_I), \nabla v) + (\delta - \delta_I, v)$, $R_2 = u - u_I$, and $R_3 = (\delta - \delta_I)\exp\psi$.
	From~\Cref{lem:Fortin}, we find that there exists a constant $\beta_0 > 0$ such that
	\begin{align*}
		\beta_0\|e_\delta\|_{H^{-1}(\Omega)}
		\leq
		\sup_{v \in V_h} \frac{(e_\delta,w)}{\|\nabla v\|_{L^2(\Omega)}}
		&\leq
		\sup_{v \in V_h} \frac{(\nabla e_u, \nabla v)}{\|\nabla v\|_{L^2(\Omega)}}
		+
		\sup_{v \in V_h} \frac{\langle R_1, v\rangle}{\|\nabla v\|_{L^2(\Omega)}}
		\\
		&\leq
		\|\nabla e_u\|_{L^2(\Omega)}
		+
		\|R_1\|_{H^{-1}(\Omega)}
								\,.
	\end{align*}
	We may now proceed to bound $(R_2, e_\delta)$, $\langle R_1, e_u \rangle$, and $(R_3, e_\delta)$ as was done for $(\phi,\delta)$, $\langle f, u\rangle$, and $(g, \delta)$, respectively, in the proof of~\cref{thm:ContinuousStability}.
	By these means, we arrive at the following upper bound:
	\begin{align*}
		\|\nabla e_u &\|_{L^2(\Omega)}
		+
		\|e_\delta\|_{H^{-1}(\Omega)}
		+
		\|\exp(\psi/2)e_\delta\|_{L^2(\Omega)}
		\\
		&\leq
		C\Big(\|R_1\|_{H^{-1}(\Omega)} + \|\nabla R_2\|_{L^2(\Omega)} + \|\exp(-\psi/2)R_3\|_{L^2(\Omega)} \Big)
		\\
		&\leq
		C\Big(\|\nabla(u - u_I)\|_{L^2(\Omega)} + \|\delta - \delta_I\|_{H^{-1}(\Omega)} + \|\exp(\psi/2)(\delta - \delta_I)\|_{L^2(\Omega)}\Big)
		\,,
	\end{align*}
	where $C > 0$ is a generic constant depending on $\beta_0$ but not on $h$ or $\psi$.
	Finally, using the triangle inequality and the fact that $u_I$ and $\delta_I$ were arbitrary, we arrive at 
	\begin{align}
\nonumber
		\|&u - u_{h}\|_{H^1(\Omega)}
		+
		\|\delta - \delta_{h}\|_{H^{-1}(\Omega)}
		+
		\|\exp(\psi/2)(\delta - \delta_{h})\|_{L^2(\Omega)}
		\\
\nonumber
		&\leq
		C
		\left(
			\inf_{v \in V_h} \|u - v \|_{H^1(\Omega)}
			+
			\inf_{w \in W_h} \Big(\|\delta - w \|_{H^{-1}(\Omega)} + \|\exp(\psi/2)(\delta - w)\|_{L^2(\Omega)}\Big)
		\right)
		\\
	\label{eq:BAE}
		&\leq
		C
		\left(
			\inf_{v \in V_h} \|u - v \|_{H^1(\Omega)}
			+
			\big(1+\|\exp(\psi/2)\|_{L^\infty(\Omega)}\big)
			\inf_{w \in W_h} \|\delta - w\|_{L^2(\Omega)}
		\right)
		\,.
	\end{align}

	\noindent\textsl{Step 2.}
	We now bound the two terms on the right-hand side of~\cref{eq:BAE}.

	The remaining part of the proof is standard, so we only consider the setting of the $(\mathbb{P}_{p}\text{-bubble}, \mathbb{P}_{p-1}\text{-broken})$ triangular finite elements defined in~\cref{eq:SubspacePair1}.
	Given that $\mathbb{P}_p(\mathcal{T}_h)\cap H^1_0(\Omega) \subset V_h$, we employ the order-$p$ nodal interpolation operator $\mathcal{I}_h \colon H^{r+1}(\Omega) \cap H^1_0(\Omega) \to \mathbb{P}_p(\mathcal{T}_h)\cap H^1_0(\Omega)$; see, e.g., \cite[Section~19.3]{ern2021finite}.
	By shape-regularity of the mesh sequence and \cite[Corollary~19.8]{ern2021finite}, we have that
	\begin{equation}
		\|v - \mathcal{I}_h v\|_{H^1(\Omega)}
		\leq
		C h^s |v|_{H^{s+1}(\Omega)}
		\,,
		\qquad
		1 \leq s \leq \min\{r,p\}
		\,.
	\end{equation}
	Thus, we find that
	\begin{equation}
	\label{eq:Approximability_primalvariable}
		\inf_{v \in V_h} \|u - v \|_{H^1(\Omega)}
		\leq
		\|u - \mathcal{I}_h u \|_{H^1(\Omega)}
		\leq
		C h^s|u|_{H^{s+1}}
		\,.
	\end{equation}
	The second term on the right-hand side of~\cref{eq:BAE} is likewise treated with the $L^2(\Omega)$-orthogonal projection operator $\mathcal{P}_h \colon L^2(\Omega) \to W_h = \mathbb{P}_{p-1}(\mathcal{T}_h)$, which has the following property \cite[Theorem~18.16]{ern2021finite}:
	\begin{equation}
	\label{eq:ProjectionConvergenceRates}
		\| w - \mathcal{P}_h w \|_{L^2(\Omega)}
		\leq
		C h^{s}|w |_{H^{s}(\Omega)}
		\,,
		\qquad
		0\leq s \leq \min\{r,p\}
		\,.
	\end{equation}
			Error estimates~\cref{eq:RobustConvergence_LinearizedSubproblem,eq:L2Convergence_LinearizedSubproblem} now follow by collecting the above bounds and applying~\cref{eq:UnweightedL2Bound}.
	\end{proof}

\subsection{Approximability result} \label{sub:nonlinear_approximability}

We finish this appendix with a proof of \Cref{lem:UtildeBound}.

\begin{proof}[Proof of~\Cref{lem:UtildeBound}]
	Recall that $u = \exp \psi$ and, therefore,
	\begin{align}
		\tilde{u}_h - u 
		&=
		\exp\psi_h - \exp\psi
		=
		\int_{\psi}^{\psi_h} \exp s \dd s
		\\
		&=
		(\psi_h - \psi)\int_0^1 \exp (\psi + s(\psi_h - \psi)) \dd s
		\\
		&=
		(\psi_h - \psi)\exp \psi \int_0^1 (\exp(\psi_h - \psi))^s \dd s
		.
	\end{align}
	As such, it holds that
	\begin{align}
		\|u - \tilde{u}_h\|_{L^\infty}
		&\leq
		\|\psi - \psi_h\|_{L^\infty}
		\|\exp\psi\|_{L^\infty}
		\int_0^1
		(\exp\|\psi - \psi_h\|_{L^\infty})^s
		\dd s
		\\
		&=
		\|\exp\psi\|_{L^\infty}(\exp\|\psi - \psi_h\|_{L^\infty} - 1)
		,
	\end{align}
	where the last line follows from the identity $\int_0^1 a^s \dd s = (a - 1)/\ln a$.
\end{proof}

{
\section*{Extended dedication from B.~Keith} \label{sec:dedication}
\it
Feynman once said that calculus is ``the language God talks'' \cite{wouk2010language}.
Expanding on this mystical statement, Strogatz has suggested that all physical laws are ``sentences'' in this ``language of the universe'' \cite{strogatz2019infinite}.
From my perspective, if the above is true, it must be the case that God has written his sentences in variational form.

The present work deals centrally with variational methods and a seemingly divine entropy functional that has never failed to surprise me since the day I began this project with Thomas.
In turn, I have frequently been reminded of von Neumann's famous quote to Shannon: ``No one knows what entropy really is'' \cite{tribus1971energy}.
Reflecting back on Feynman and Strogatz's perspectives, it is helpful to think that at least God knows, even if we do not.
As such, and on the occasion of his 70th birthday, it feels only fitting that I dedicate this work to Leszek Demkowicz; the kind and deeply religious man who not so long ago taught me a variational perspective on the universe.
}

\subsection*{Acknowledgements} \label{sub:acknowledgements}

The authors gratefully thank Jesse Chan, J\'er\^{o}me Darbon, Tarik Dzanic, Alexandre Ern, Patrick Farrell, Caroline Geiersbach, Dohyun Kim, Boyan Lazarov, Michael Hinterm\"uller, Thomas J.R.~Hughes, Karl-Hermann Neeb, Chi-Wang Shu, and N.~Sukumar for their time and helpful discussions during the development of this work.
We also thank Dohyun Kim, Tzanio Kolev, Boyan Lazarov, Socratis Petrides, and Jingyi Wang for peer-reviewing our MFEM implementations in order to merge them into the source code as official MFEM examples.
Next, we thank J{\o}rgen Dokken for generously responding to our many inquiries about FEniCSx, as well as his help refactoring and preparing our FEniCSx implementations for public release.
We also thank Rami Masri for his meticulous reading of the manuscript, pointing out several statements that were unclear and a mistake we corrected before publication.
Finally, we thank Markus Bachmayr and Endre S\"uli for carefully and patiently handling the manuscript and the anonymous reviewers for their helpful and detailed comments.

\phantomsection\bibliographystyle{siamplain}
\bibliography{main.bib}

\end{document}